\documentclass[12pt,a4paper,draft]{amsart}

\usepackage{latexsym}
\usepackage{ifthen}
\usepackage[leqno]{amsmath}
\usepackage{enumerate}
\usepackage{calc}
\usepackage{amstext,amsbsy,amsopn,amsthm,amsgen,amsfonts,amscd,amsxtra,upref}
\usepackage{mathrsfs}\usepackage{euscript}\usepackage{amssymb}

\swapnumbers
\theoremstyle{plain}
\newtheorem{Thm}{Theorem}[section]
\newtheorem{Lem}[Thm]{Lemma}
\newtheorem{Cor}[Thm]{Corollary}
\newtheorem{Pro}[Thm]{Proposition}
\newtheorem{Prp}[Thm]{Properties}
\newtheorem{Sub}[Thm]{Sublemma}

\theoremstyle{definition}
\newtheorem{Def}[Thm]{Definition}
\newtheorem{Exm}[Thm]{Example}
\newtheorem{Exs}[Thm]{Examples}

\theoremstyle{remark}
\newtheorem{Rem}[Thm]{Remark}
\newtheorem{Rms}[Thm]{Remarks}
\newtheorem*{Com}{Commentary}

\newcommand{\myEmail}{piotr.niemiec@uj.edu.pl}
\newcommand{\myAddress}{\noindent{}Piotr Niemiec\\{}Jagiellonian University\\{}Institute of Mathematics\\{}
   ul. \L{}ojasiewicza 6\\{}30-348 Krak\'{o}w\\{}Poland}
\newcommand{\myData}{\author[P. Niemiec]{Piotr Niemiec}\address{\myAddress}\email{\myEmail}}

\newcommand{\AAA}{\mathbb{A}}\newcommand{\BBB}{\mathbb{B}}\newcommand{\CCC}{\mathbb{C}}
\newcommand{\EEE}{\mathbb{E}}\newcommand{\FFF}{\mathbb{F}}\newcommand{\GGG}{\mathbb{G}}
\newcommand{\III}{\mathbb{I}}\newcommand{\JJJ}{\mathbb{J}}
\newcommand{\MMM}{\mathbb{M}}\newcommand{\OOO}{\mathbb{O}}\newcommand{\PPP}{\mathbb{P}}
\newcommand{\QQQ}{\mathbb{Q}}\newcommand{\RRR}{\mathbb{R}}\newcommand{\SSS}{\mathbb{S}}\newcommand{\TTT}{\mathbb{T}}
\newcommand{\XXX}{\mathbb{X}}
\newcommand{\YYY}{\mathbb{Y}}\newcommand{\ZZZ}{\mathbb{Z}}
\newcommand{\AAa}{\CMcal{A}}\newcommand{\BBb}{\CMcal{B}}\newcommand{\CCc}{\CMcal{C}}
\newcommand{\EEe}{\CMcal{E}}\newcommand{\FFf}{\CMcal{F}}\newcommand{\GGg}{\CMcal{G}}\newcommand{\HHh}{\CMcal{H}}
\newcommand{\IIi}{\CMcal{I}}\newcommand{\KKk}{\CMcal{K}}
\newcommand{\MMm}{\CMcal{M}}\newcommand{\NNn}{\CMcal{N}}\newcommand{\PPp}{\CMcal{P}}

\newcommand{\UUu}{\CMcal{U}}\newcommand{\WWw}{\CMcal{W}}
\newcommand{\ZZz}{\CMcal{Z}}
\newcommand{\AaA}{\EuScript{A}}\newcommand{\BbB}{\EuScript{B}}\newcommand{\CcC}{\EuScript{C}}
\newcommand{\DdD}{\EuScript{D}}\newcommand{\EeE}{\EuScript{E}}\newcommand{\FfF}{\EuScript{F}}
\newcommand{\GgG}{\EuScript{G}}\newcommand{\HhH}{\EuScript{H}}\newcommand{\IiI}{\EuScript{I}}
\newcommand{\LlL}{\EuScript{L}}
\newcommand{\MmM}{\EuScript{M}}\newcommand{\NnN}{\EuScript{N}}
\newcommand{\PpP}{\EuScript{P}}\newcommand{\QqQ}{\EuScript{Q}}\newcommand{\RrR}{\EuScript{R}}
\newcommand{\SsS}{\EuScript{S}}\newcommand{\UuU}{\EuScript{U}}
\newcommand{\XxX}{\EuScript{X}}
\newcommand{\ZzZ}{\EuScript{Z}}
\newcommand{\Bb}{\mathfrak{B}}\newcommand{\Dd}{\mathfrak{D}}
\newcommand{\Ff}{\mathfrak{F}}

\newcommand{\Mm}{\mathfrak{M}}\newcommand{\Nn}{\mathfrak{N}}\newcommand{\Pp}{\mathfrak{P}}

\newcommand{\Ww}{\mathfrak{W}}\newcommand{\Xx}{\mathfrak{X}}
\newcommand{\Zz}{\mathfrak{Z}}
\newcommand{\aA}{\mathfrak{a}}\newcommand{\bB}{\mathfrak{b}}
\newcommand{\fF}{\mathfrak{f}}
\newcommand{\jJ}{\mathfrak{j}}
\newcommand{\mM}{\mathfrak{m}}\newcommand{\pP}{\mathfrak{p}}
\newcommand{\sS}{\mathfrak{s}}
\newcommand{\uU}{\mathfrak{u}}

\newcommand{\aaA}{\mathscr{A}}\newcommand{\bbB}{\mathscr{B}}\newcommand{\ccC}{\mathscr{C}}\newcommand{\ddD}{\mathscr{D}}
\newcommand{\eeE}{\mathscr{E}}\newcommand{\ffF}{\mathscr{F}}

\newcommand{\rrR}{\mathscr{R}}\newcommand{\ssS}{\mathscr{S}}
\newcommand{\wwW}{\mathscr{W}}\newcommand{\xxX}{\mathscr{X}}
\newcommand{\yyY}{\mathscr{Y}}\newcommand{\zzZ}{\mathscr{Z}}
\newcommand{\aAA}{\pmb{A}}\newcommand{\bBB}{\pmb{B}}
\newcommand{\eEE}{\pmb{E}}\newcommand{\fFF}{\pmb{F}}\newcommand{\gGG}{\pmb{G}}
\newcommand{\jJJ}{\pmb{J}}\newcommand{\lLL}{\pmb{L}}
\newcommand{\mMM}{\pmb{M}}
\newcommand{\qQQ}{\pmb{Q}}\newcommand{\sSS}{\pmb{S}}\newcommand{\tTT}{\pmb{T}}
\newcommand{\xXX}{\pmb{X}}
\newcommand{\yYY}{\pmb{Y}}

\newcommand{\SECT}[1]{\section{#1}\renewcommand{\theequation}{\thesection-\arabic{equation}}\setcounter{equation}{0}}

\newcounter{help}
\newcommand{\ITE}[3]{\ifthenelse{#1}{#2}{#3}}\newcommand{\ITEE}[3]{\ITE{\equal{#1}{#2}}{#3}{}}

\newcommand{\supp}{\operatorname{supp}}

\newcommand{\id}{\operatorname{id}}
\newcommand{\cll}{\operatorname{cl}}\newcommand{\intt}{\operatorname{int}}\newcommand{\lin}{\operatorname{lin}}
\newcommand{\card}{\operatorname{card}}

\newcommand{\Card}{\operatorname{Card}}
\newcommand{\red}{\operatorname{red}}\newcommand{\tr}{\operatorname{tr}}
\newcommand{\GCD}{\operatorname{GCD}}

\newcommand{\scalar}[2]{\langle #1,#2\rangle}\newcommand{\leqsl}{\leqslant}\newcommand{\geqsl}{\geqslant}

\newcommand{\varempty}{\varnothing}\newcommand{\dd}{\colon}
\newcommand{\dint}[1]{\,\textup{d} #1}

\newcommand{\TFCAE}{The following conditions are equivalent:}
\newcommand{\tfcae}{the following conditions are equivalent:}\newcommand{\iaoi}{if and only if}

\newcommand{\THM}[1]{Theorem~\textup{\ref{thm:#1}}}
\newcommand{\DEF}[1]{Definition~\textup{\ref{def:#1}}}\newcommand{\COR}[1]{Corollary~\textup{\ref{cor:#1}}}
\newcommand{\LEM}[1]{Lemma~\textup{\ref{lem:#1}}}\newcommand{\PRO}[1]{Proposition~\textup{\ref{pro:#1}}}
\newcommand{\REM}[1]{Remark~\textup{\ref{rem:#1}}}
\newcommand{\EXM}[1]{Example~\textup{\ref{exm:#1}}}\newcommand{\EXS}[1]{Examples~\textup{\ref{exs:#1}}}

\newenvironment{thm}[1]{\begin{Thm}\label{thm:#1}}{\end{Thm}}\newenvironment{lem}[1]{\begin{Lem}\label{lem:#1}}{\end{Lem}}
\newenvironment{cor}[1]{\begin{Cor}\label{cor:#1}}{\end{Cor}}\newenvironment{pro}[1]{\begin{Pro}\label{pro:#1}}{\end{Pro}}
\newenvironment{dfn}[1]{\begin{Def}\label{def:#1}}{\end{Def}}
\newenvironment{exm}[1]{\begin{Exm}\label{exm:#1}}{\end{Exm}}\newenvironment{exs}[1]{\begin{Exs}\label{exs:#1}}{\end{Exs}}
\newenvironment{rem}[1]{\begin{Rem}\label{rem:#1}}{\end{Rem}}

\newenvironment{thmm}[2]{\begin{Thm}[#2]\label{thm:#1}}{\end{Thm}}

\newcommand{\bibITEM}[2]{\ITE{\equal{#2}{}}{\bibitem{#1} }{\bibitem[#2]{#1} }}
\newcommand{\BIB}[8]{
   \bibITEM{#1}{#8} #2, \textit{#3}, #4{} \textbf{#5} (#6), #7.}
\newcommand{\myBIB}[6][P. Niemiec]{#1, \textit{#2}, #3{}\ITE{\equal{#4}{}}{}{ \textbf{#4}} (#5), #6.}
\newcommand{\BIb}[6]{
   \bibITEM{#1}{#6} #2, \textit{#3}, #4, #5.}
\newcommand{\BiB}[9]{
   \bibITEM{#1}{#9} #2, \textit{#3}, #4{} \textit{#5}, #6, #7, #8.}

\newcommand{\myBAPP}[3][P. Niemiec]{
   #1, \textit{#2}, #3}

\newcommand{\jRN}[2][]{
   \ITEE{#2}{ActaM}{\ITE{\equal{#1}{+}}
      {Acta Mathematica}{Acta Math.}}
   \ITEE{#2}{AdvM}{\ITE{\equal{#1}{+}}
      {Advances in Mathematics}{Adv. in Math.}}
   \ITEE{#2}{ACS}{\ITE{\equal{#1}{+}}
      {Applied Categorical Structures}{Appl. Categor. Struct.}}
   \ITEE{#2}{ActaSM}{\ITE{\equal{#1}{+}}
      {Acta Scientiarum Mathematicarum}{Acta Sci. Math.}}
   \ITEE{#2}{AmJM}{\ITE{\equal{#1}{+}}
      {American Journal of Mathematics}{Amer. J. Math.}}
   \ITEE{#2}{AmMMon}{\ITE{\equal{#1}{+}}
      {American Mathematical Monthly}{Amer. Math. Mon.}}
   \ITEE{#2}{AnnSciEcNormSupT}{\ITE{\equal{#1}{+}}
      {Annales Scientifiques de l'\'{E}cole Normale Sup\'{e}rieure (3)}{Ann. Sci. \'{E}c. Norm. Sup\'{e}r. (3)}}
   \ITEE{#2}{AnnM}{\ITE{\equal{#1}{+}}
      {Annals of Mathematics}{Ann. Math.}}
   \ITEE{#2}{AnnProb}{\ITE{\equal{#1}{+}}
      {The Annals of Probability}{Ann. Probab.}}
   \ITEE{#2}{AnnPALog}{\ITE{\equal{#1}{+}}
      {Annals of Pure and Applied Logic}{Ann. Pure Appl. Logic}}
   \ITEE{#2}{APM}{\ITE{\equal{#1}{+}}
      {Annales Polonici Mathematici}{Ann. Polon. Math.}}
   \ITEE{#2}{ArchM}{\ITE{\equal{#1}{+}}
      {Archiv der Mathematik}{Arch. Math.}}
   \ITEE{#2}{AttiAccLincRendNat}{\ITE{\equal{#1}{+}}
      {Atti della Accademia Nazionale dei Lincei. Rendiconti. Classe di Scienze Fisiche, Matematiche e Naturali}
      {Atti Accad. Naz. Lincei Rend. Cl. Sci. Fis. Mat. Nat.}}
   \ITEE{#2}{BAMS}{\ITE{\equal{#1}{+}}
      {Bulletin of the American Mathematical Society}{Bull. Amer. Math. Soc.}}
   \ITEE{#2}{BAustrMS}{\ITE{\equal{#1}{+}}
      {Bulletin of the Australian Mathematical Society}{Bull. Austral. Math. Soc.}}
   \ITEE{#2}{BLondMS}{\ITE{\equal{#1}{+}}
      {Bulletin of the London Mathematical Sociecy}{Bull. Lond. Math. Soc.}}
   \ITEE{#2}{BAPolSSSM}{\ITE{\equal{#1}{+}}
      {Bulletin de l'Acad\'{e}mie Polonaise des Sciences. S\'{e}rie des Sciences Math\'{e}matiques}
      {Bull. Acad. Pol. Sci. S\'{e}r. Sci. Math.}}
   \ITEE{#2}{BullSM}{\ITE{\equal{#1}{+}}
      {Bulletin des Sciences Math\'{e}matiques}{Bull. Sci. Math.}}
   \ITEE{#2}{BullPol}{\ITE{\equal{#1}{+}}
      {Bulletin of the Polish Academy of Sciences: Mathematics}{Bull. Polish Acad. Sci. Math.}}
   \ITEE{#2}{CanadJM}{\ITE{\equal{#1}{+}}
      {Canadian Journal Mathematics}{Canad. J. Math.}}
   \ITEE{#2}{CollectM}{\ITE{\equal{#1}{+}}
      {Collectanea Mathematica}{Collect. Math.}}
   \ITEE{#2}{CMUC}{\ITE{\equal{#1}{+}}
      {Commentationes Mathematicae Universitatis Carolinae}{Comment. Math. Univ. Carolin.}}
   \ITEE{#2}{CRParis}{\ITE{\equal{#1}{+}}
      {C. R. Paris}{C. R. Paris}}
   \ITEE{#2}{CRASParis}{\ITE{\equal{#1}{+}}
      {Comptes Rendus de l'Acad\'{e}mie des Sciences. Paris}{C. R. Acad. Sci. Paris}}
   \ITEE{#2}{CEurJM}{\ITE{\equal{#1}{+}}
      {Central European Journal of Mathematics}{Cent. Eur. J. Math.}}
   \ITEE{#2}{CMHelv}{\ITE{\equal{#1}{+}}
      {Commentarii Mathematici Helvetici}{Comment. Math. Helv.}}
   \ITEE{#2}{CollM}{\ITE{\equal{#1}{+}}
      {Colloquium Mathematicum}{Coll. Math.}}
   \ITEE{#2}{ComposM}{\ITE{\equal{#1}{+}}
      {Compositio Mathematica}{Compos. Math.}}
   \ITEE{#2}{CzMJ}{\ITE{\equal{#1}{+}}
      {Czechoslovak Mathematical Journal}{Czech. Math. J.}}
   \ITEE{#2}{DissM}{\ITE{\equal{#1}{+}}
      {Dissertationes Mathematicae}{Diss. Math.}}
   \ITEE{#2}{DANSSSR}{\ITE{\equal{#1}{+}}
      {Doklady Akademii Nauk SSSR}{Dokl. Akad. Nauk SSSR}}
   \ITEE{#2}{DukeMJ}{\ITE{\equal{#1}{+}}
      {Duke Mathematical Journal}{Duke Math. J.}}
   \ITEE{#2}{ELA}{\ITE{\equal{#1}{+}}
      {The Electronic Journal of Linear Algebra}{Electron. J. Linear Algebra}}
   \ITEE{#2}{ExtrM}{\ITE{\equal{#1}{+}}
      {Extracta Mathematicae}{Extracta Math.}}
   \ITEE{#2}{FM}{\ITE{\equal{#1}{+}}
      {Fundamenta Mathematicae}{Fund. Math.}}
   \ITEE{#2}{FAA}{\ITE{\equal{#1}{+}}
      {Functional Analysis and its Applications}{Funct. Anal. Appl.}}
   \ITEE{#2}{FunkAnalPril}{\ITE{\equal{#1}{+}}
      {Funktsional'ny\u{\i} Analiz i Ego Prilozheniya}{Funkts. Anal. Prilozh.}}
   \ITEE{#2}{GTopA}{\ITE{\equal{#1}{+}}
      {General Topology and its Applications}{General Topol. Appl.}}
   \ITEE{#2}{IllinoisJM}{\ITE{\equal{#1}{+}}
      {Illinois Journal of Mathematics}{Illinois J. Math.}}
   \ITEE{#2}{IndagMP}{\ITE{\equal{#1}{+}}
      {Indagationes Mathematicae (Proceedings)}{Indagationes Math. Proc.}}
   \ITEE{#2}{IndianaUMJ}{\ITE{\equal{#1}{+}}
      {Indiana University Mathematical Journal}{Indiana Univ. Math. J.}}
   \ITEE{#2}{InHauEtSPM}{\ITE{\equal{#1}{+}}
      {Inst. Hautes \'{E}tudes Sci. Publ. Math.}{Inst. Hautes \'{E}tudes Sci. Publ. Math.}}
   \ITEE{#2}{IEOT}{\ITE{\equal{#1}{+}}
      {Integral Equations and Operator Theory}{Integr. Equ. Oper. Theory}}
   \ITEE{#2}{IsraelJM}{\ITE{\equal{#1}{+}}
      {Israel Journal of Mathematics}{Israel J. Math.}}
   \ITEE{#2}{JAusMSA}{\ITE{\equal{#1}{+}}
      {Journal of the Australian Mathematical Society. Series A}{J. Aust. Math. Soc. Ser. A}}
   \ITEE{#2}{JCA}{\ITE{\equal{#1}{+}}
      {Journal of Convex Analysis}{J. Convex Anal.}}
   \ITEE{#2}{JChinUST}{\ITE{\equal{#1}{+}}
      {J. China Univ. Sci. Tech.}{J. China Univ. Sci. Tech.}}
   \ITEE{#2}{JFA}{\ITE{\equal{#1}{+}}
      {Journal of Functional Analysis}{J. Funct. Anal.}}
   \ITEE{#2}{JKoreanMS}{\ITE{\equal{#1}{+}}
      {Journal of the Korean Mathematical Society}{J. Korean Math. Soc.}}
   \ITEE{#2}{JMAnApp}{\ITE{\equal{#1}{+}}
      {J. Math. Anal. Appl.}{J. Math. Anal. Appl.}}
   \ITEE{#2}{JOT}{\ITE{\equal{#1}{+}}
      {Journal of Operator Theory}{J. Operator Theory}}
   \ITEE{#2}{KodaiMSemRep}{\ITE{\equal{#1}{+}}
      {Kodai Math. Sem. Rep.}{Kodai Math. Sem. Rep.}}
   \ITEE{#2}{LAA}{\ITE{\equal{#1}{+}}
      {Linear Algebra and its Applications}{Linear Algebra Appl.}}
   \ITEE{#2}{LMLA}{\ITE{\equal{#1}{+}}
      {Linear and Multilinear Algebra}{Linear Multilinear Algebra}}
   \ITEE{#2}{MathJap}{\ITE{\equal{#1}{+}}
      {Math. Japon.}{Math. Japon.}}
   \ITEE{#2}{MLQ}{\ITE{\equal{#1}{+}}
      {Mathematical Logic Quarterly}{Math. Log. Q.}}
   \ITEE{#2}{MProcCambPhS}{\ITE{\equal{#1}{+}}
      {Mathematical Proceedings of the Cambridge Philosophical Society}{Math. Proc. Cambridge Phil. Soc.}}
   \ITEE{#2}{MMag}{\ITE{\equal{#1}{+}}
      {Mathematics Magazine}{Math. Mag.}}
   \ITEE{#2}{MSb}{\ITE{\equal{#1}{+}}
      {Matematicheski\u{\i} Sbornik}{Mat. Sb.}}
   \ITEE{#2}{MStud}{\ITE{\equal{#1}{+}}
      {Matematychni Studi\"{\i}}{Mat. Stud.}}
   \ITEE{#2}{MScand}{\ITE{\equal{#1}{+}}
      {Mathematica Scandinavica}{Math. Scand.}}
   \ITEE{#2}{MAnn}{\ITE{\equal{#1}{+}}
      {Mathematische Annalen}{Math. Ann.}}
   \ITEE{#2}{MAMS}{\ITE{\equal{#1}{+}}
      {Memoirs of the American Mathematical Society}{Mem. Amer. Math. Soc.}}
   \ITEE{#2}{MichMJ}{\ITE{\equal{#1}{+}}
      {Michigan Mathematical Journal}{Mich. Math. J.}}
   \ITEE{#2}{MonatM}{\ITE{\equal{#1}{+}}
      {Monatshefte f\"{u}r Mathematik}{Mh. Math.}}
   \ITEE{#2}{NonlinA}{\ITE{\equal{#1}{+}}
      {Nonlinear Analysis: Theory, Methods \& Applications}{Nonlinear Anal.}}
   \ITEE{#2}{NAMS}{\ITE{\equal{#1}{+}}
      {Notices of the American Mathematical Society}{Notices Amer. Math. Soc.}}
   \ITEE{#2}{OpusM}{\ITE{\equal{#1}{+}}
      {Opuscula Mathematica}{Opuscula Math.}}
   \ITEE{#2}{PacJM}{\ITE{\equal{#1}{+}}
      {Pacific Journal of Mathematics}{Pacific J. Math.}}
   \ITEE{#2}{PeriodMHung}{\ITE{\equal{#1}{+}}
      {Periodica Mathematica Hungarica}{Period. Math. Hungarica}}
   \ITEE{#2}{PAMS}{\ITE{\equal{#1}{+}}
      {Proceedings of the American Mathematical Society}{Proc. Amer. Math. Soc.}}
   \ITEE{#2}{ProcCambPhS}{\ITE{\equal{#1}{+}}
      {Proceedings of the Cambridge Philosophical Society}{Proc. Cambridge Phil. Soc.}}
   \ITEE{#2}{ProcImpAcadTokyo}{\ITE{\equal{#1}{+}}
      {Proc. Imp. Acad. Tokyo}{Proc. Imp. Acad. Tokyo}}
   \ITEE{#2}{ProcKonink}{\ITE{\equal{#1}{+}}
      {Proceedings of the Koninklijke Nederlandse Akademie van Wetenschappen}{Nederl. Akad. Wetensch. Proc. Ser. A}}
   \ITEE{#2}{PLondMS}{\ITE{\equal{#1}{+}}
      {Proceedings of the London Mathematical Society}{Proc. London Math. Soc.}}
   \ITEE{#2}{PNatlUSA}{\ITE{\equal{#1}{+}}
      {Proceedings of the National Academy of Sciences of the United States of America}{Proc. Natl. Acad. Sci. USA}}
   \ITEE{#2}{PublRIMSKyoto}{\ITE{\equal{#1}{+}}
      {Publ. Res. Inst. Math. Sci. Kyoto Univ.}{Publ. Res. Inst. Math. Sci.}}
   \ITEE{#2}{PWN}{\ITE{\equal{#1}{+}}
      {PWN -- Polish Scientific Publishers, Warszawa}{PWN -- Polish Scientific Publishers, Warszawa}}
   \ITEE{#2}{RCMP}{\ITE{\equal{#1}{+}}
      {Rendiconti del Circolo Matematico di Palermo}{Rend. Circ. Mat. Palermo}}
   \ITEE{#2}{RussMS}{\ITE{\equal{#1}{+}}
      {Russian Mathematical Surveys}{Russian Math. Surveys}}
   \ITEE{#2}{SbM}{\ITE{\equal{#1}{+}}
      {Sbornik: Mathematics}{Sb. Math.}}
   \ITEE{#2}{SciRepTokyoA}{\ITE{\equal{#1}{+}}
      {Science Reports of Tokyo Kyoiku Daigaku, Section A}{Sci. Rep. Tokyo Kyoiku Daigaku Sect. A}}
   \ITEE{#2}{SeminProbStras}{\ITE{\equal{#1}{+}}
      {S\'{e}minaire de probabilit\'{e}s de Strasbourg}{S\'{e}min. Prob. Strasbourg}}
   \ITEE{#2}{SIAMJMAA}{\ITE{\equal{#1}{+}}
      {SIAM Journal on Matrix Analysis and Applications}{SIAM J. Matrix Anal. Appl.}}
   \ITEE{#2}{SibirMZ}{\ITE{\equal{#1}{+}}
      {Sibirski\v{\i} Mat. \v{Z}hurnal}{Sibirsk. Mat. \v{Z}.}}
   \ITEE{#2}{SM}{\ITE{\equal{#1}{+}}
      {Studia Mathematica}{Studia Math.}}
   \ITEE{#2}{TAMS}{\ITE{\equal{#1}{+}}
      {Transactions of the American Mathematical Society}{Trans. Amer. Math. Soc.}}
   \ITEE{#2}{TohokuMJ}{\ITE{\equal{#1}{+}}
      {T\^{o}hoku Mathematical Journal}{T\^{o}hoku Math. J.}}
   \ITEE{#2}{TomskUnivRev}{\ITE{\equal{#1}{+}}
      {Tomsk Universitet Review}{Tomsk. Univ. Rev.}}
   \ITEE{#2}{TopA}{\ITE{\equal{#1}{+}}
      {Topology and its Applications}{Topology Appl.}}
   \ITEE{#2}{TsukubaJM}{\ITE{\equal{#1}{+}}
      {Tsukuba Journal of Mathematics}{Tsukuba J. Math.}}
   \ITEE{#2}{UspekhiMN}{\ITE{\equal{#1}{+}}
      {Uspekhi Matem. Nauk}{Uspekhi Mat. Nauk}}
   }

\newcommand{\paplist}[3][]{
   \ITEE{#3}{NIAkhiezer,IMGlazman1993}{
      \BIb{#2}{N.I. Akhiezer and I.M. Glazman}
         {Theory of Linear Operators in Hilbert Space}
         {Dover Publications, Inc., New York}{1993}{#1}}
   \ITEE{#3}{RDAnderson1967}{
      \BIB{#2}{R.D. Anderson}
         {On topological infinite deficiency}
         {\jRN{MichMJ}}{14}{1967}{365--383}{#1}}
   \ITEE{#3}{RDAnderson,JMcCharen1970}{
      \BIB{#2}{R.D. Anderson and J. McCharen}
         {On extending homeomorphisms to Fr\'{e}chet manifolds}
         {\jRN{PAMS}}{25}{1970}{283--289}{#1}}
   \ITEE{#3}{RDAnderson,DWCurtis,JVanMill1982}{
      \BIB{#2}{R.D. Anderson, D.W. Curtis, J. van Mill}
         {A fake topological Hilbert space}
         {\jRN{TAMS}}{272}{1982}{311--321}{#1}}
   \ITEE{#3}{RArens,JEells1956}{
      \BIB{#2}{R. Arens and J. Eells}
         {On embedding uniform and topological spaces}
         {\jRN{PacJM}}{6}{1956}{397--403}{#1}}
   \ITEE{#3}{NAronszajn,PPanitchpakdi1956}{
      \BIB{#2}{N. Aronszajn and P. Panitchpakdi}
         {Extension of uniformly continuous transformations and hyperconvex metric spaces}
         {\jRN{PacJM}}{6}{1956}{405--439}{#1}}
   \ITEE{#3}{KJBabenko1948}{
      \BIB{#2}{K.J. Babenko}
         {On conjugate functions}
         {\jRN{DANSSSR}}{62}{1948}{157--160}{#1}}
   \ITEE{#3}{TBanakh1995}{
      \BIB{#2}{T.O. Banakh}
         {Topology of spaces of probability measures, I}
         {\jRN{MStud}}{5}{1995}{65--87 (Russian)}{#1}}
   \ITEE{#3}{TBanakh1995a}{
      \BIB{#2}{T.O. Banakh}
         {Topology of spaces of probability measures, II}
         {\jRN{MStud}}{5}{1995}{88--106 (Russian)}{#1}}
   \ITEE{#3}{TBanakh1998}{
      \BIB{#2}{T. Banakh}
         {Characterization of spaces admitting a homotopy dense embedding into a Hilbert manifold}
         {\jRN{TopA}}{86}{1998}{123--131}{#1}}
   \ITEE{#3}{TBanakh,TNRadul1997}{
      \BIB{#2}{T.O. Banakh and T.N. Radul}
         {Topology of spaces of probability measures}
         {\jRN{SbM}}{188}{1997}{973--995}{#1}}
   \ITEE{#3}{TBanakh,TRadul,MZarichnyi1996}{
      \BIb{#2}{T. Banakh, T. Radul, M. Zarichnyi}
         {Absorbing sets in infinite-dimensional manifolds}
         {VNTL Publishers, Lviv}{1996}{#1}}
   \ITEE{#3}{TBanakh,IZarichnyy2008}{
      \BIB{#2}{T. Banakh and I. Zarichnyy}
         {Topological groups and convex sets homeomorphic to non-separable Hilbert spaces}
         {\jRN{CEurJM}}{6}{2008}{77--86}{#1}}
   \ITEE{#3}{HBecker,ASKechris1996}{
      \BIb{#2}{H. Becker and A.S. Kechris}
         {The Descriptive Set Theory of Polish Group Actions \textup{(London Math. Soc. Lecture Note Series, vol. 232)}}
         {University Press, Cambridge}{1996}{#1}}
   \ITEE{#3}{NEBenamara,NNikolski1999}{
      \BIB{#2}{N.E. Benamara and N. Nikolski}
         {Resolvent tests for similarity to a normal operator}
         {\jRN{PLondMS}}{78}{1999}{585--626}{#1}}
   \ITEE{#3}{YBenyamini,JLindenstrauss2000}{
      \BIb{#2}{Y. Benyamini and J. Lindenstrauss}
         {Geometric nonlinear functional analysis I}
         {AMS Colloquium Publications 48}{2000}{#1}}
   \ITEE{#3}{SKBerberian1974}{
      \BIb{#2}{S.K. Berberian}
         {Lectures in Functional Analysis and Operator Theory}
         {Graduate Texts in Mathematics 15, Springer-Verlag, New York}{1974}{#1}}
   \ITEE{#3}{SNBernstein1954}{
      \BIb{#2}{S.N. Bernstein}
         {Collected Works II}
         {Akad. Nauk SSSR, Moscow}{1954 (Russian)}{#1}}
   \ITEE{#3}{CzBessaga,APelczynski1972}{
      \BIB{#2}{Cz. Bessaga and A. Pe\l{}czy\'{n}ski}
         {On spaces of measurable functions}
         {\jRN{SM}}{44}{1972}{597--615}{#1}}
   \ITEE{#3}{CzBessaga,APelczynski1975}{
      \BIb{#2}{Cz. Bessaga and A. Pe\l{}czy\'{n}ski}
         {Selected topics in infinite-dimensional topology}
         {\jRN{PWN}}{1975}{#1}}
   \ITEE{#3}{MBestvina,JMogilski1986}{
      \BIB{#2}{M. Bestvina and J. Mogilski}
         {Characterizing certain incomplete infinite-dimensional absolute retracts}
         {\jRN{MichMJ}}{33}{1986}{291--313}{#1}}
   \ITEE{#3}{MBestvina,PBowers,JMogilsky,JWalsh1986}{
      \BIB{#2}{M. Bestvina, P. Bowers, J. Mogilsky, J. Walsh}
         {Characterization of Hilbert space manifolds revisited}
         {\jRN{TopA}}{24}{1986}{53--69}{#1}}
   \ITEE{#3}{RBhatia1997}{
      \BIb{#2}{R. Bhatia}
         {Matrix Analysis}
         {Springer, New York}{1997}{#1}}
   \ITEE{#3}{GBirkhoff1936}{
      \BIB{#2}{G. Birkhoff}
         {A note on topological groups}
         {\jRN{ComposM}}{3}{1936}{427--430}{#1}}
   \ITEE{#3}{MSBirman,MZSolomjak1987}{
      \BIb{#2}{M.S. Birman and M.Z. Solomjak}
         {Spectral Theory of Self-Adjoint Operators in Hilbert Space}
         {D. Reidel Publishing Co., Dordrecht}{1987}{#1}}
   \ITEE{#3}{EBishop1961}{
      \BIB{#2}{E. Bishop}
         {A generalization of the Stone-Weierstrass theorem}
         {\jRN{PacJM}}{11}{1961}{777--783}{#1}}
   \ITEE{#3}{BBlackadar2006}{
      \BIb{#2}{B. Blackadar}{Operator Algebras. Theory of $\CCc^*$-algebras and von Neumann algebras 
         \textup{(Encyclopaedia of Mathematical Sciences, vol. 122: Operator Algebras and Non-Commutative Geometry III)}}
         {Springer-Verlag, Berlin-Heidelberg}{2006}{#1}}
   \ITEE{#3}{JBlass,WHolsztynski1972}{
      \BIB{#2}{J. Blass and W. Holszty\'{n}ski}
         {Cubical polyhedra and homotopy III}
         {\jRN{AttiAccLincRendNat}}{53}{1972}{275--279}{#1}}
   \ITEE{#3}{FFBonsall,NJDuncan1973}{
      \BIb{#2}{F.F. Bonsall and N.J. Duncan}
         {Complete Normed Algebras}
         {Springer Verlag, Berlin}{1973}{#1}}
   \ITEE{#3}{NBourbaki2002}{
      \BIb{#2}{N. Bourbaki}
         {Lie Groups and Lie Algebras, Chapters 4--6}
         {Springer, New York}{2002}{#1}}
   \ITEE{#3}{PLBowers1989}{
      \BIB{#2}{P.L. Bowers}
         {Limitation topologies on function spaces}
         {\jRN{TAMS}}{314}{1989}{421--431}{#1}}
   \ITEE{#3}{ABrown1953}{
      \BIB{#2}{A. Brown}
         {On a class of operators}
         {\jRN{PAMS}}{4}{1953}{723--728}{#1}}
   \ITEE{#3}{ABrown,CKFong,DWHadwin1978}{
      \BIB{#2}{A. Brown, C.-K. Fong, D.W. Hadwin}
         {Parts of operators on Hilbert space}
         {\jRN{IllinoisJM}}{22}{1978}{306--314}{#1}}
   \ITEE{#3}{AMBruckner,JBBruckner,BSThomson1997}{
      \BIb{#2}{A.M. Bruckner, J.B. Bruckner, B.S. Thomson}
         {Real Analysis}
         {Prentice-Hall, New Jersey}{1997}{#1}}
   \ITEE{#3}{PJCameron,AMVershik2006}{
      \BIB{#2}{P.J. Cameron and A.M. Vershik}
         {Some isometry groups of Urysohn space}
         {\jRN{AnnPALog}}{143}{2006}{70--78}{#1}}
   \ITEE{#3}{CCastaing1966}{
      \BIB{#2}{C. Castaing}
         {Quelques probl\`{e}mes de mesurabilit\'{e} li\'{e}es \`{a} la th\'{e}orie de la commande}
         {\jRN{CRParis}}{262}{1966}{409--411}{#1}}
   \ITEE{#3}{JAVanCasteren1980}{
      \BIB{#2}{J.A. van Casteren}
         {A problem of Sz.-Nagy}
         {\jRN{ActaSM}}{42}{1980}{189--194}{#1}}
   \ITEE{#3}{JAVanCasteren1983}{
      \BIB{#2}{J.A. van Casteren}
         {Operators similar to unitary or selfadjoint ones}
         {\jRN{PacJM}}{104}{1983}{241--255}{#1}}
   \ITEE{#3}{XCatepillan,MPtak,WSzymanski1994}{
      \BIB{#2}{X. Catepill\'{a}n, M. Ptak, W. Szyma\'{n}ski}
         {Multiple canonical decompositions of families of operators and a model of quasinormal families}
         {\jRN{PAMS}}{121}{1994}{1165--1172}{#1}}
   \ITEE{#3}{RCauty1994}{
      \BIB{#2}{R. Cauty}
         {Un espace m\'{e}trique lin\'{e}aire qui n'est pas un r\'{e}tracte absolu}
         {\jRN{FM}}{146}{1994}{85--99, (French)}{#1}}
   \ITEE{#3}{TAChapman1971}{
      \BIB{#2}{T.A. Chapman}
         {Deficiency in infinite-dimensional manifolds}
         {\jRN{GTopA}}{1}{1971}{263--272}{#1}}
   \ITEE{#3}{TAChapman1976}{
      \BIb{#2}{T.A. Chapman}
         {Lectures on Hilbert cube manifolds}
         {C.B.M.S. Regional Conference Series in Math. No 28, Amer. Math. Soc.}{1976}{#1}}
   \ITEE{#3}{RBChuaqui1977}{
      \BIB{#2}{R.B. Chuaqui}
         {Measures invariant under a group of transformations}
         {\jRN{PacJM}}{68}{1977}{313--329}{#1}}
   \ITEE{#3}{JBConway1985}{
      \BIb{#2}{J.B. Conway}
         {A Course in Functional Analysis}
         {Springer-Verlag, New York}{1985}{#1}}
   \ITEE{#3}{JBConway2000}{
      \BIb{#2}{J.B. Conway}
         {A Course in Operator Theory}
         {(Graduate Studies in Mathematics, vol. 21) Amer. Math. Soc., Providence}{2000}{#1}}
   \ITEE{#3}{GCorach,AMaestripieri,MMbekhta2009}{
      \BIB{#2}{G. Corach, A. Maestripieri, M. Mbekhta}
         {Metric and homogeneous structure of closed range operators}
         {\jRN{JOT}}{61}{2009}{171--190}{#1}}
   \ITEE{#3}{MJCowen,RGDouglas1978}{
      \BIB{#2}{M.J. Cowen and R.G. Douglas}
         {Complex geometry and operator theory}
         {\jRN{ActaM}}{141}{1978}{187--261}{#1}}
   \ITEE{#3}{DWCurtis1985}{
      \BIB{#2}{D.W. Curtis}
         {Boundary sets in the Hilbert cube}
         {\jRN{TopA}}{20}{1985}{201--221}{#1}}
   \ITEE{#3}{MMDay1958}{
      \BIb{#2}{M.M. Day}
         {Normed Linear Spaces}
         {Springer Verlag, Berlin}{1958}{#1}}
   \ITEE{#3}{CDellacherie1967}{
      \BIB{#2}{C. Dellacherie}
         {Un compl\'{e}ment au th\'{e}or\`{e}me de Weierstrass-Stone}
         {\jRN{SeminProbStras}}{1}{1967}{52--53}{#1}}
   \ITEE{#3}{JJDijkstra1987}{
      \BIB{#2}{J.J. Dijkstra}
         {Strong negligibility of $\sigma$-compacta does not characterize Hilbert space}
         {\jRN{PacJM}}{127}{1987}{19--30}{#1}}
   \ITEE{#3}{JJDijkstra1990}{
      \BIB{#2}{J.J. Dijkstra}
         {Characterizing Hilbert space topology in terms of strong negligibility}
         {\jRN{ComposM}}{75}{1990}{299--306}{#1}}
   \ITEE{#3}{TDobrowolski,WMarciszewski2002}{
      \BIB{#2}{T. Dobrowolski and W. Marciszewski}
         {Failure of the Factor Theorem for Borel pre-Hilbert spaces}
         {\jRN{FM}}{175}{2002}{53--68}{#1}}
   \ITEE{#3}{TDobrowolski,JMogilski1990}{
      \BiB{#2}{T. Dobrowolski and J. Mogilski}
         {Problems on Topological Classification of Incomplete Metric Spaces}{Chapter 25 in:}
         {Open Problems in Topology}{J. van Mill and G.M. Reed (eds.), North-Holland Amsterdam}{1990}{411--429}{#1}}
   \ITEE{#3}{TDobrowolski,HTorunczyk1981}{
      \BIB{#2}{T. Dobrowolski and H. Toru\'{n}czyk}
         {Separable complete ANR's admitting a group structure are Hilbert manifolds}
         {\jRN{TopA}}{12}{1981}{229--235}{#1}}
   \ITEE{#3}{RGDouglas1966}{
      \BIB{#2}{R.G. Douglas}
         {On majorization, factorization and range inclusion of operators in Hilbert space}
         {\jRN{PAMS}}{17}{1966}{413--416}{#1}}
   \ITEE{#3}{CHDowker1947}{
      \BIB{#2}{C.H. Dowker}
         {Mapping theorems for non-compact spaces}
         {\jRN{AmJM}}{69}{1947}{200--242}{#1}}
   \ITEE{#3}{CHDowker1952}{
      \BIB{#2}{C.H. Dowker}
         {Topology of metric complexes}
         {\jRN{AmJM}}{74}{1952}{555--577}{#1}}
   \ITEE{#3}{JDugundji1951}{
      \BIB{#2}{J. Dugundji}
         {An extension of Tietze's theorem}
         {\jRN{PacJM}}{1}{1951}{353--367}{#1}}
   \ITEE{#3}{JDugundji1958}{
      \BIB{#2}{J. Dugundji}
         {Absolute neighborhood retracts and local connectedness for arbitrary metric spaces}
         {\jRN{ComposM}}{13}{1958}{229--246}{#1}}
   \ITEE{#3}{JDugundji1965}{
      \BIB{#2}{J. Dugundji}
         {Locally equiconnected spaces and absolute neighborhood retracts}
         {\jRN{FM}}{57}{1965}{187--193}{#1}}
   \ITEE{#3}{NDunford,JTSchwartz1958}{
      \BIb{#2}{N. Dunford and J.T. Schwartz}
         {Linear Operators, part I}
         {Interscience Publishers, New York}{1958}{#1}}
   \ITEE{#3}{NDunford,JTSchwartz1963}{
      \BIb{#2}{N. Dunford and J.T. Schwartz}
         {Linear Operators, part II}
         {Interscience Publishers, New York}{1963}{#1}}
   \ITEE{#3}{NDunford,JTSchwartz1971}{
      \BIb{#2}{N. Dunford and J.T. Schwartz}
         {Linear Operators, part III}
         {Wiley-Interscience, New York}{1971}{#1}}
   \ITEE{#3}{MLEaton,MDPerlman1977}{
      \BIB{#2}{M.L. Eaton and M.D. Perlman}
         {Reflection groups, generalized Schur functions and the geometry of majorization}
         {\jRN{AnnProb}}{5}{1977}{829--860}{#1}}
   \ITEE{#3}{EGEffros1965}{
      \BIB{#2}{E.G. Effros}
         {The Borel space of von Neumann algebras on a separable Hilbert space}
         {\jRN{PacJM}}{15}{1965}{1153--1164}{#1}}
   \ITEE{#3}{EGEffros1966}{
      \BIB{#2}{E.G. Effros}
         {Global structure in von Neumann algebras}
         {\jRN{TAMS}}{121}{1966}{434--454}{#1}}
   \ITEE{#3}{REspinola,MAKhamsi2001}{
      \BiB{#2}{R. Espinola and M.A. Khamsi}
         {Introduction to hyperconvex spaces}{Chapter XIII in:}{Handbook of Metric Fixed Point Theory}
         {W.A. Kirk and B. Sims (editors), Kluwer Academic Publishers}{2001}{391--435}{#1}}
   \ITEE{#3}{PAFillmore,JPWilliams1971}{
      \BIB{#2}{P.A. Fillmore and J.P. Williams}
         {On operator ranges}
         {\jRN{AdvM}}{7}{1971}{254--281}{#1}}
   \ITEE{#3}{JEells,NHKuiper1969}{
      \BIB{#2}{J. Eells and N.H. Kuiper}
         {Homotopy negligible subsets in infinite-dimensional manifolds}
         {\jRN{ComposM}}{21}{1969}{151--161}{#1}}
   \ITEE{#3}{REngelking1977}{
      \BIb{#2}{R. Engelking}
         {General Topology}
         {\jRN{PWN}}{1977}{#1}}
   \ITEE{#3}{REngelking1978}{
      \BIb{#2}{R. Engelking}
         {Dimension Theory}
         {\jRN{PWN}}{1978}{#1}}
   \ITEE{#3}{REngelking1989}{
      \BIb{#2}{R. Engelking}
         {General Topology. Revised and completed edition \textup{(Sigma series in pure mathematics, vol. 6)}}
         {Heldermann Verlag, Berlin}{1989}{#1}}
   \ITEE{#3}{PErdos,RDMauldin1976}{
      \BIB{#2}{P. Erd\"{o}s and R.D. Mauldin}
         {The nonexistence of certain invariant measures}
         {\jRN{PAMS}}{59}{1976}{321--322}{#1}}
   \ITEE{#3}{JErnest1976}{
      \BIB{#2}{J. Ernest}
         {Charting the operator terrain}
         {\jRN{MAMS}}{171}{1976}{207 pp}{#1}}
   \ITEE{#3}{RHFox1943}{
      \BIB{#2}{R.H. Fox}
         {On fiber spaces, II}
         {\jRN{BAMS}}{49}{1943}{733--735}{#1}}
   \ITEE{#3}{NAFriedman1970}{
      \BIb{#2}{N.A. Friedman}
         {Introduction to ergodic theory}
         {Van Nostrand Reinhold Company}{1970}{#1}}
   \ITEE{#3}{MFujii,MKajiwara,YKato,FKubo1976}{
      \BIB{#2}{M. Fujii, M. Kajiwara, Y. Kato, F. Kubo}
         {Decompositions of operators in Hilbert spaces}
         {\jRN{MathJap}}{21}{1976}{117--120}{#1}}
   \ITEE{#3}{SGao,ASKechris2003}{
      \BIB{#2}{S. Gao and A.S. Kechris}
         {On the classification of Polish metric spaces up to isometry}
         {\jRN{MAMS}}{161}{2003}{viii+78}{#1}}
   \ITEE{#3}{MIGarrido,FMontalvo1991}{
      \BIB{#2}{M.I. Garrido and F. Montalvo}
         {On some generalizations of the Kakutani-Stone and Stone-Weierstrass theorems}
         {\jRN{ExtrM}}{6}{1991}{156--159}{#1}}
   \ITEE{#3}{LGe,JShen2002}{
      \BIB{#2}{L. Ge and J. Shen}
         {Generator problem for certain property T factors}
         {\jRN{PNAS}}{99}{2002}{565--567}{#1}}
   \ITEE{#3}{IMGelfand,MANaimark1943}{
      \BIB{#2}{I.M. Gelfand and M.A. Naimark}
         {On the embedding of normed rings into the ring of operators in Hilbert space}
         {\jRN{MSb}}{12}{1943}{197--213}{#1}}
   \ITEE{#3}{FGesztesy,MMalamud,MMitrea,SNaboko2009}{
      \BIB{#2}{F. Gesztesy, M. Malamud, M. Mitrea, S. Naboko}
         {Generalized polar decompositions for closed operators in Hilbert spaces and some applications}
         {\jRN{IEOT}}{64}{2009}{83--113}{#1}}
   \ITEE{#3}{LGillman,MJerison1960}{
      \BIb{#2}{L. Gillman and M. Jerison}
         {Rings of continuous functions}
         {New York}{1960}{#1}}
   \ITEE{#3}{JGlimm1960}{
      \BIB{#2}{J. Glimm}
         {A Stone-Weierstrass theorem for $\CCc^*$-algebras}
         {\jRN{AnnM}}{72}{1960}{216--244}{#1}}
   \ITEE{#3}{GGodefroy,NJKalton2003}{
      \BIB{#2}{G. Godefroy and N.J. Kalton}
         {Lipschitz-free Banach spaces}
         {\jRN{SM}}{159}{2003}{121--141}{#1}}
   \ITEE{#3}{ICGohberg,MGKrein1967}{
      \BIB{#2}{I.C. Gohberg and M.G. Krein}
         {On a description of contraction operators similar to unitary ones}
         {\jRN{FunkAnalPril}}{1}{1967}{38--60}{#1}}
   \ITEE{#3}{ELGriffinJr1953}{
      \BIB{#2}{E.L. Griffin Jr.}
         {Some contributions to the theory of rings of operators}
         {\jRN{TAMS}}{75}{1953}{471--504}{#1}}
   \ITEE{#3}{ELGriffinJr1955}{
      \BIB{#2}{E.L. Griffin Jr.}
         {Some contributions to the theory of rings of operators II}
         {\jRN{TAMS}}{79}{1955}{389--400}{#1}}
   \ITEE{#3}{MGromov1981}{
      \BIB{#2}{M. Gromov}
         {Groups of polynomial growth and expanding maps}
         {\jRN{InHauEtSPM}}{53}{1981}{53--73}{#1}}
   \ITEE{#3}{MGromov1999}{
      \BIb{#2}{M. Gromov}
         {Metric Structures for Riemannian and Non-Riemannian Spaces}
         {Progress in Math. \textbf{152}, Birkh\"{a}user}{1999}{#1}}
   \ITEE{#3}{JDeGroot1956}{
      \BIB{#2}{J. de Groot}
         {Non-archimedean metrics in topology}
         {\jRN{PAMS}}{7}{1956}{948--953}{#1}}
   \ITEE{#3}{LCGrove,CTBenson1985}{
      \BIb{#2}{L.C. Grove and C.T. Benson}
         {Finite Reflection Group}
         {2nd ed., Springer-Verlag}{1985}{#1}}
   \ITEE{#3}{VIGurarii1966}{
      \BIB{#2}{V.I. Gurari\v{\i}}{Spaces of universal placement, isotropic spaces and a problem of Mazur 
         on rotations of Banach spaces \textup{(Russian)}}
         {\jRN{SibirMZ}}{7}{1966}{1002--1013}{#1}}
   \ITEE{#3}{DWHadwin1976}{
      \BIB{#2}{D.W. Hadwin}
         {An operator-valued spectrum}
         {\jRN{NAMS}}{23}{1976}{A-163}{#1}}
   \ITEE{#3}{DWHadwin1977}{
      \BIB{#2}{D.W. Hadwin}
         {An operator-valued spectrum}
         {\jRN{IndianaUMJ}}{26}{1977}{329--340}{#1}}
   \ITEE{#3}{HHahn1932}{
      \BIb{#2}{H. Hahn}
         {Reelle Funktionen I}
         {Leipzig}{1932}{#1}}
   \ITEE{#3}{PRHalmos1950}{
      \BIb{#2}{P.R. Halmos}
         {Measure theory}
         {Van Nostrand, New York}{1950}{#1}}
   \ITEE{#3}{PRHalmos1951}{
      \BIb{#2}{P.R. Halmos}
         {Introduction to Hilbert Space and the Theory of Spectral Multiplicity}
         {Chelsea Publishing Company, New York}{1951}{#1}}
   \ITEE{#3}{PRHalmos1956}{
      \BIb{#2}{P.R. Halmos}
         {Lectures on Ergodic Theory}
         {Publ. Math. Soc. Japan, Tokyo}{1956}{#1}}
   \ITEE{#3}{PRHalmos1982}{
      \BIb{#2}{P.R. Halmos}
         {A Hilbert Space Problem Book}
         {Springer-Verlag New York Inc.}{1982}{#1}}
  \ITEE{#3}{PRHalmos,JEMcLaughlin1963}{
      \BIB{#2}{P.R. Halmos and J.E. McLaughlin}
         {Partial isometries}
         {\jRN{PacJM}}{13}{1963}{585--596}{#1}}
   \ITEE{#3}{RWHansell1972}{
      \BIB{#2}{R.W. Hansell}
         {On the nonseparable theory of Borel and Souslin sets}
         {\jRN{BAMS}}{78}{1972}{236--241}{#1}}
   \ITEE{#3}{FHausdorff1930}{
      \BIB{#2}{F. Hausdorff}
         {Erweiterung einer Hom\"{o}omorphie}
         {\jRN{FM}}{16}{1930}{353--360}{#1}}
   \ITEE{#3}{FHausdorff1934}{
      \BIB{#2}{F. Hausdorff}
         {\"{U}ber innere Abbildungen}
         {\jRN{FM}}{23}{1934}{279--291}{#1}}
   \ITEE{#3}{FHausdorff1938}{
      \BIB{#2}{F. Hausdorff}
         {Erweiterung einer stetigen Abbildung}
         {\jRN{FM}}{30}{1938}{40--47}{#1}}
   \ITEE{#3}{DWHenderson1971}{
      \BIB{#2}{D.W. Henderson}
         {Corrections and extensions of two papers about infinite-dimensional manifolds}
         {\jRN{GTopA}}{1}{1971}{321--327}{#1}}
   \ITEE{#3}{DWHenderson1975}{
      \BIB{#2}{D.W. Henderson}
         {$Z$-sets in ANR's}
         {\jRN{TAMS}}{213}{1975}{205--216}{#1}}
   \ITEE{#3}{DWHenderson,RMSchori1970}{
      \BIB{#2}{D.W. Henderson and R.M. Schori}
         {Topological classification of infinite-dimensional manifolds by homotopy type}
         {\jRN{BAMS}}{76}{1970}{121--124}{#1}}
   \ITEE{#3}{DWHenderson,JEWest1970}{
      \BIB{#2}{D.W. Henderson and J.E. West}
         {Triangulated infinite-dimensional manifolds}
         {\jRN{BAMS}}{76}{1970}{655--660}{#1}}
   \ITEE{#3}{BHoffmann1979}{
      \BIB{#2}{B. Hoffmann}
         {A compact contractible topological group is trivial}
         {\jRN{ArchM}}{32}{1979}{585--587}{#1}}
   \ITEE{#3}{DHofmann2002}{
      \BIB{#2}{D. Hofmann}
         {On a generalization of the Stone-Weierstrass theorem}
         {\jRN{ACS}}{10}{2002}{569--592}{#1}}
   \ITEE{#3}{GHognas,AMukherjea1995}{
      \BIb{#2}{G. H\"ogn\"as and A. Mukherjea}
         {Probability Measures on Semigroups. Convolution Products, Random Walks, and Random Matrices}
         {Plenum Press, New York}{1995}{#1}}
   \ITEE{#3}{MRHolmes1992}{
      \BIB{#2}{M.R. Holmes}
         {The universal separable metric space of Urysohn and isometric embeddings thereof in Banach spaces}
         {\jRN{FM}}{140}{1992}{199--223}{#1}}
   \ITEE{#3}{MRHolmes2008}{
      \BIB{#2}{M.R. Holmes}
         {The Urysohn space embeds in Banach spaces in just one way}
         {\jRN{TopA}}{155}{2008}{1479--1482}{#1}}
   \ITEE{#3}{RRHolmes,TYTam1999}{
      \BIB{#2}{R.R. Holmes and T.Y. Tam}
         {Distance to the convex hull of an orbit under the action of a compact group}
         {\jRN{JAusMSA}}{66}{1999}{331--357}{#1}}
   \ITEE{#3}{RHorn,RMathias1990}{
      \BIB{#2}{R. Horn and R. Mathias}
         {Cauchy-Schwartz inequalities associated with positive semidefinite matrices}
         {\jRN{LAA}}{142}{1990}{63--82}{#1}}
   \ITEE{#3}{GEHuhunaisvili1955}{
      \BIB{#2}{G.E. Huhunai\v{s}vili}
         {On a property of Urysohn's universal metric space}
         {\jRN{DANSSSR}}{101}{1955}{607--610 (Russian)}{#1}}
   \ITEE{#3}{JEHumphreys1990}{
      \BIb{#2}{J.E. Humphreys}
         {Reflection Groups and Coxeter Groups}
         {Cambridge University Press}{1990}{#1}}
   \ITEE{#3}{JRIsbell1964}{
      \BIB{#2}{J.R. Isbell}
         {Six theorems about injective metric spaces}
         {\jRN{CMHelv}}{39}{1964}{65--76}{#1}}
   \ITEE{#3}{SIzumino,YKato1985}{
      \BIB{#2}{S. Izumino and Y. Kato}
         {The closure of invertible operators on Hilbert space}
         {\jRN{ActaSM}}{49}{1985}{321--327}{#1}}
   \ITEE{#3}{CJiang2004}{
      \BIB{#2}{C. Jiang}
         {Similarity classification of Cowen-Douglas operators}
         {\jRN{CanadJM}}{56}{2004}{742--775}{#1}}
   \ITEE{#3}{WBJohnson,JLindenstrauss2001}{
      \BiB{#2}{W.B. Johnson and J. Lindenstrauss}{Basic Concepts in the Geometry of Banach Spaces}
         {Chapter 1 in:}{Handbook of the Geometry of Banach Spaces, Vol. 1}
         {W.B. Johnson and J. Lindenstrauss (editors), Elsevier Science B.V., Amsterdam}{2001}{1--84}{#1}}
   \ITEE{#3}{IBJung,JStochel2008}{
      \BIB{#2}{I.B. Jung and J. Stochel}
         {Subnormal operators whose adjoints have rich point spectrum}
         {\jRN{JFA}}{255}{2008}{1797--1816}{#1}}
   \ITEE{#3}{RVKadison,JRRingrose1983}{
      \BIb{#2}{R.V. Kadison and J.R. Ringrose}
         {Fundamentals of the Theory of Operator Algebras. Volume I: Elementary Theory}
         {Academic Press, Inc., New York-London}{1983}{#1}}
   \ITEE{#3}{RVKadison,JRRingrose1986}{
      \BIb{#2}{R.V. Kadison and J.R. Ringrose}
         {Fundamentals of the Theory of Operator Algebras. Volume II: Advanced Theory}
         {Academic Press, Inc., Orlando-London}{1986}{#1}}
   \ITEE{#3}{SKakutani1936}{
      \BIB{#2}{S. Kakutani}
         {\"{U}ber die Metrisation der topologischen Gruppen}
         {\jRN{ProcImpAcadTokyo}}{12}{1936}{82--84}{#1}}
   \ITEE{#3}{SKakutani1938}{
      \BIB{#2}{S. Kakutani}
         {Two fixed-point theorems concerning bicompact convex sets}
         {\jRN{ProcImpAcadTokyo}}{14}{1938}{242--245}{#1}}
   \ITEE{#3}{SKakutani1941}{
      \BIB{#2}{S. Kakutani}
         {Concrete representation of abstract L-spaces}
         {\jRN{AnnM}}{42}{1941}{523--537}{#1}}
   \ITEE{#3}{SKakutani1941a}{
      \BIB{#2}{S. Kakutani}
         {Concrete representation of abstract M-spaces}
         {\jRN{AnnM}}{42}{1941}{994--1024}{#1}}
   \ITEE{#3}{NKalton2007}{
      \BIB{#2}{N. Kalton}
         {Extending Lipschitz maps into $\CCc(K)$-spaces}
         {\jRN{IsraelJM}}{162}{2007}{275--315}{#1}}
   \ITEE{#3}{RKane2001}{
      \BIb{#2}{R. Kane}
         {Reflection Groups and Invariant Theory}
         {Canadian Mathematical Society, Springer}{2001}{#1}}
   \ITEE{#3}{VKannan,SRRaju1980}{
      \BIB{#2}{V. Kannan and S.R. Raju}
         {The nonexistence of invariant universal measures on semigroups}
         {\jRN{PAMS}}{78}{1980}{482--484}{#1}}
   \ITEE{#3}{IKaplansky1951}{
      \BIB{#2}{I. Kaplansky}
         {A theorem on rings of operators}
         {\jRN{PacJM}}{1}{1951}{227--232}{#1}}
   \ITEE{#3}{MKatetov1988}{
      \BiB{#2}{M. Kat\v{e}tov}{On universal metric spaces}{in: Frolik (ed.),}
         {General Topology and its Relations to Modern Analysis and Algebra VI. Proceedings of the Sixth Prague 
         Topological Symposium 1986}{Heldermann Verlag Berlin}{1988}{323--330}{#1}}
   \ITEE{#3}{YKatznelson1960}{
      \BIB{#2}{Y. Katznelson}
         {Sur les alg\'{e}bres dont les \'{e}l\'{e}ments non n\'{e}gatifs admettent des racines carr\'{e}es}
         {\jRN{AnnSciEcNormSupT}}{77}{1960}{167--174}{#1}}
   \ITEE{#3}{OHKeller1931}{
      \BIB{#2}{O.H. Keller}
         {Die Homoiomorphie der kompakten konvexen Mengen in Hilbertschen Raum}
         {\jRN{MAnn}}{105}{1931}{748--758}{#1}}
   \ITEE{#3}{MAKhamsi,WAKirk,CMartinez2000}{
      \BIB{#2}{M.A. Khamsi, W.A. Kirk, C. Martinez}
         {Fixed point and selection theorems in hyperconvex spaces}
         {\jRN{PAMS}}{128}{2000}{3275--3283}{#1}}
   \ITEE{#3}{ABKhararazishvili1998}{
      \BIb{#2}{A.B. Khararazishvili}
         {Transformation groups and invariant measures. Set-theoretic aspects}
         {World Scientific Publishing Co., Inc., River Edge, NJ}{1998}{#1}}
   \ITEE{#3}{YKijima1987}{
      \BIB{#2}{Y. Kijima}
         {Fixed points of nonexpansive self-maps of a compact metric space}
         {\jRN{JMAnApp}}{123}{1987}{114--116}{#1}}
  \ITEE{#3}{JSKim,ChRKim,SGLee1980}{
      \BIB{#2}{J.S. Kim, Ch.R. Kim, S.G. Lee}
         {Reducing operator valued spectra of a Hilbert space operator}
         {\jRN{JKoreanMS}}{17}{1980}{123--129}{#1}}
   \ITEE{#3}{JKindler1995}{
      \BIB{#2}{J. Kindler}
         {Minimax theorems with applications to convex metric spaces}
         {\jRN{CollM}}{68}{1995}{179--186}{#1}}
   \ITEE{#3}{WAKirk1998}{
      \BIB{#2}{W.A. Kirk}
         {Hyperconvexity of $\RRR$-trees}
         {\jRN{FM}}{156}{1998}{67--72}{#1}}
   \ITEE{#3}{VLKleeJr1952}{
      \BIB{#2}{V.L. Klee Jr.}
         {Invariant metrics in groups (solution of a problem of Banach)}
         {\jRN{PAMS}}{3}{1952}{484--487}{#1}}
   \ITEE{#3}{HJKowalsky1957}{
      \BIB{#2}{H.J. Kowalsky}
         {Einbettung metrischer R\"{a}ume}
         {\jRN{ArchM}}{8}{1957}{336--339}{#1}}
   \ITEE{#3}{WKubis,MRubin2010}{
      \BIB{#2}{W. Kubi\'{s} and M. Rubin}
         {Extension and reconstruction theorems for the Urysohn universal metric space}
         {\jRN{CzMJ}}{60}{2010}{1--29}{#1}}
   \ITEE{#3}{KKuratowski1966}{
      \BIb{#2}{K. Kuratowski}
         {Topology. \textup{Vol. I}}
         {\jRN{PWN}}{1966}{#1}}
   \ITEE{#3}{KKuratowski,BKnaster1927}{
      \BIB{#2}{K. Kuratowski and B. Knaster}
         {A connected and connected im kleinen point set which contains no perfect subset}
         {\jRN{BAMS}}{33}{1927}{106--109}{#1}}
   \ITEE{#3}{KKuratowski,AMostowski1976}{
      \BIb{#2}{K. Kuratowski and A. Mostowski}
         {Set Theory with an Introduction to Descriptive Set Theory}
         {\jRN{PWN}}{1976}{#1}}
   \ITEE{#3}{GLewicki1992}{
      \BIB{#2}{G. Lewicki}
         {Bernstein's ``lethargy'' theorem in metrizable topological linear spaces}
         {\jRN{MonatM}}{113}{1992}{213--226}{#1}}
   \ITEE{#3}{ASLewis1996}{
      \BIB{#2}{A.S. Lewis}
         {Group invariance and convex matrix analysis}
         {\jRN{SIAMJMAA}}{17}{1996}{927--949}{#1}}
   \ITEE{#3}{C-KLi,N-KTsing1991}{
      \BIB{#2}{C.-K. Li and N.-K. Tsing}
         {$G$-invariant norms and $G(c)$-radii}
         {\jRN{LAA}}{150}{1991}{179--194}{#1}}
   \ITEE{#3}{AJLazar,JLindenstrauss1971}{
      \BIB{#2}{A.J. Lazar and J. Lindenstrauss}
         {Banach spaces whose duals are $L_1$ spaces and their representing matrices}
         {\jRN{ActaM}}{126}{1971}{165--193}{#1}}
   \ITEE{#3}{EHLieb,MLoss1997}{
      \BIb{#2}{E.H. Lieb and M. Loss}
         {Analysis \textup{(Graduate Studies in Mathematics, vol. 14)}}
         {Amer. Math. Soc., Providence, RI}{1997}{#1}}
   \ITEE{#3}{ALindenbaum1926}{
      \BIB{#2}{A. Lindenbaum}
         {Contributions \`{a} l'\'{e}tude de l'espace m\'{e}trique I}
         {\jRN{FM}}{8}{1926}{209--222}{#1}}
   \ITEE{#3}{DLindenstrauss,LTzafriri1971}{
      \BIB{#2}{D. Lindenstrauss and L. Tzafriri}
         {On the complemented subspaces problem}
         {\jRN{IsraelJM}}{9}{1971}{263--269}{#1}}
   \ITEE{#3}{RILoebl1986}{
      \BIB{#2}{R.I. Loebl}
         {A note on containment of operators}
         {\jRN{BAustrMS}}{33}{1986}{279--291}{#1}}
   \ITEE{#3}{LHLoomis1945}{
      \BIB{#2}{L.H. Loomis}
         {Abstract congruence and the uniqueness of Haar measure}
         {\jRN{AnnM}}{46}{1945}{348--355}{#1}}
   \ITEE{#3}{LHLoomis1949}{
      \BIB{#2}{L.H. Loomis}
         {Haar measure in uniform structures}
         {\jRN{DukeMJ}}{16}{1949}{193--208}{#1}}
   \ITEE{#3}{ERLorch1939}{
      \BIB{#2}{E.R. Lorch}
         {Bicontinuous linear transformation in certain vector spaces}
         {\jRN{BAMS}}{45}{1939}{564--569}{#1}}
   \ITEE{#3}{ATLundell,SWeingram1969}{
      \BIb{#2}{A.T. Lundell and S. Weingram}
         {The topology of CW-complexes}
         {Litton Educ. Publ.}{1969}{#1}}
   \ITEE{#3}{WLusky1976}{
      \BIB{#2}{W. Lusky}
         {The Gurarij spaces are unique}
         {\jRN{ArchM}}{27}{1976}{627--635}{#1}}
   \ITEE{#3}{WLusky1977}{
      \BIB{#2}{W. Lusky}
         {On separable Lindenstrauss spaces}
         {\jRN{JFA}}{26}{1977}{103--120}{#1}}
   \ITEE{#3}{DMaharam1942}{
      \BIB{#2}{D. Maharam}
         {On homogeneous measure algebras}
         {\jRN{PNatlUSA}}{28}{1942}{108--111}{#1}}
   \ITEE{#3}{MMalicki,SSolecki2009}{
      \BIB{#2}{M. Malicki and S. Solecki}
         {Isometry groups of separable metric spaces}
         {\jRN{MProcCambPhS}}{146}{2009}{67--81}{#1}}
   \ITEE{#3}{PMankiewicz1972}{
      \BIB{#2}{P. Mankiewicz}
         {On extension of isometries in normed linear spaces}
         {\jRN{BAPolSSSM}}{20}{1972}{367--371}{#1}}
   \ITEE{#3}{JMartinezMaurica,MTPellon1987}{
      \BIB{#2}{J. Martinez-Maurica and M.T. Pell\'{o}n}
         {Non-archimedean Chebyshev centers}
         {\jRN{IndagMP}}{90}{1987}{417--421}{#1}}
   \ITEE{#3}{KMaurin1980}{
      \BIb{#2}{K. Maurin}
         {Analysis, Part II}
         {D. Reidel, Dordrecht-Boston-London}{1980}{#1}}
   \ITEE{#3}{SMazur,SUlam1932}{
      \BIB{#2}{S. Mazur and S. Ulam}
         {Sur les transformationes isom\'{e}triques d'espaces vectoriels norm\'{e}s}
         {\jRN{CRASParis}}{194}{1932}{946--948}{#1}}
   \ITEE{#3}{SMazurkiewicz1920}{
      \BIB{#2}{S. Mazurkiewicz}
         {Sur les lignes de Jordan}
         {\jRN{FM}}{1}{1920}{166--209}{#1}}
   \ITEE{#3}{SMazurkiewicz,WSierpinski1920}{
      \BIB{#2}{S. Mazurkiewicz and W. Sierpi\'{n}ski}
         {Contributions a la topologie des ensembles denombrables}
         {\jRN{FM}}{1}{1920}{17--27}{#1}}
   \ITEE{#3}{MMbekhta1992}{
      \BIB{#2}{M. Mbekhta}
         {Sur la structure des composantes connexes semi-Fredholm de $B(H)$}
         {\jRN{PAMS}}{116}{1992}{521--524}{#1}}
   \ITEE{#3}{JEMcCarthy1996}{
      \BIB{#2}{J.E. McCarthy}
         {Boundary values and Cowen-Douglas curvature}
         {\jRN{JFA}}{137}{1996}{1--18}{#1}}
   \ITEE{#3}{JMelleray2007}{
      \BIB{#2}{J. Melleray}
         {Computing the complexity of the relation of isometry between separable Banach spaces}
         {\jRN{MLQ}}{53}{2007}{128--131}{#1}}
   \ITEE{#3}{JMelleray2007a}{
      \BIB{#2}{J. Melleray}
         {On the geometry of Urysohn's universal metric space}
         {\jRN{TopA}}{154}{2007}{384--403}{#1}}
   \ITEE{#3}{JMelleray2008}{
      \BIB{#2}{J. Melleray}
         {Some geometric and dynamical properties of the Urysohn space}
         {\jRN{TopA}}{155}{2008}{1531--1560}{#1}}
   \ITEE{#3}{JMelleray,FVPetrov,AMVershik2008}{
      \BIB{#2}{J. Melleray, F.V. Petrov, A.M. Vershik}
         {Linearly rigid metric spaces and the embedding problem}
         {\jRN{FM}}{199}{2008}{177--194}{#1}}
   \ITEE{#3}{EMichael1953}{
      \BIB{#2}{E. Michael}
         {Some extension theorems for continuous functions}
         {\jRN{PacJM}}{3}{1953}{789--806}{#1}}
   \ITEE{#3}{EMichael1954}{
      \BIB{#2}{E. Michael}
         {Local properties of topological spaces}
         {\jRN{DukeMJ}}{21}{1954}{163--171}{#1}}
   \ITEE{#3}{EMichael1956}{
      \BIB{#2}{E. Michael}
         {Selected selection theorems}
         {\jRN{AmMMon}}{58}{1956}{233--238}{#1}}
   \ITEE{#3}{EMichael1956a}{
      \BIB{#2}{E. Michael}
         {Continuous selections. I}
         {\jRN{AnnM}}{63}{1956}{361--382}{#1}}
   \ITEE{#3}{EMichael1956b}{
      \BIB{#2}{E. Michael}
         {Continuous selections. II}
         {\jRN{AnnM}}{64}{1956}{562--580}{#1}}
   \ITEE{#3}{EMichael1959}{
      \BIB{#2}{E. Michael}
         {A theorem on semi-continuous set-valued functions}
         {\jRN{DukeMJ}}{26}{1959}{647--652}{#1}}
   \ITEE{#3}{JVanMill1986}{
      \BIB{#2}{J. van Mill}
         {Another counterexample in ANR theory}
         {\jRN{PAMS}}{97}{1986}{136--138}{#1}}
   \ITEE{#3}{JVanMill2001}{
      \BIb{#2}{J. van Mill}
         {The Infinite-Dimensional Topology of Function Spaces 
         \textup{(North-Holland Mathematical Library, vol. 64)}}
         {Elsevier, Amsterdam}{2001}{#1}}
   \ITEE{#3}{WMlak1991}{
      \BIb{#2}{W. Mlak}
         {Hilbert Spaces and Operator Theory}
         {PWN --- Polish Scientific Publishers and Kluwer Academic Publishers, Warszawa-Dordrecht}{1991}{#1}}
   \ITEE{#3}{JMogilski1979}{
      \BIB{#2}{J. Mogilski}
         {$CE$-decomposition of $l_2$-manifolds}
         {\jRN{BAPolSSSM}}{27}{1979}{309--314}{#1}}
   \ITEE{#3}{RLMoore1916}{
      \BIB{#2}{R.L. Moore}
         {On the foundations of plane analysis situs}
         {\jRN{TAMS}}{17}{1916}{131--164}{#1}}
   \ITEE{#3}{KMorita1955}{
      \BIB{#2}{K. Morita}
         {A condition for the metrizability of topological spaces and for $n$-dimensionality}
         {\jRN{SciRepTokyoA}}{5}{1955}{33--36}{#1}}
   \ITEE{#3}{AMukherjea,NATserpes1976}{
      \BIb{#2}{A. Mukherjea and N.A. Tserpes}
         {Measures on topological semigroups}
         {Springer Lecture Notes in Math. Vol. 547, Berlin}{1976}{#1}}
   \ITEE{#3}{JMycielski1974}{
      \BIB{#2}{J. Mycielski}
         {Remarks on invariant measures in metric spaces}
         {\jRN{CollM}}{32}{1974}{105--112}{#1}}
   \ITEE{#3}{SNNaboko1984}{
      \BIB{#2}{S.N. Naboko}
         {Conditions for similarity to unitary and selfadjoint operators}
         {\jRN{FunkAnalPril}}{18}{1984}{16--27}{#1}}
   \ITEE{#3}{LNachbin1965}{
      \BIb{#2}{L. Nachbin}
         {The Haar Integral}
         {D. Van Nostrand Company, Inc., Princeton-New Jersey-Toronto-New York-London}{1965}{#1}}
   \ITEE{#3}{TDNarang,SKGarg1991}{
      \BIB{#2}{T.D. Narang and S.K. Garg}
         {On the uniqueness of best approximation in non-archimedian spaces}
         {\jRN{PeriodMHung}}{22}{1991}{121--124}{#1}}
   \ITEE{#3}{JVonNeumann1930}{
      \BIB{#2}{J. von Neumann}
         {Zur Algebra der Funktionaloperationen und Theorie der normalen Operatoren}
         {\jRN{MAnn}}{102}{1930}{370--427}{#1}}
   \ITEE{#3}{JVonNeumann1934}{
      \BIB{#2}{J. von Neumann}
         {Zum Haarschen Mass in topologischen Gruppen}
         {\jRN{ComposM}}{1}{1934}{106--114}{#1}}
   \ITEE{#3}{JVonNeumann1937}{
      \BiB{#2}{J. von Neumann}
         {Some matrix-inequalities and metrization of matrix-space}{\jRN{TomskUnivRev}{} \textbf{1} (1937), 286--300; 
         in }{Collected Works}{Pergamon, New York}{1962}{Vol. 4, 205--219}{#1}}
   \ITEE{#3}{JVonNeumann1949}{
      \BIB{#2}{J. von Neumann}
         {On Rings of Operators. Reduction Theory}
         {\jRN{AnnM}}{50}{1949}{401--485}{#1}}
   \ITEE{#3}{ONielson1973}{
      \BIB{#2}{O. Nielson}
         {Borel sets of von Neumann algebras}
         {\jRN{AmJM}}{95}{1973}{145--164}{#1}}
   \ITEE{#3}{pn2002}{\bibITEM{#2}{#1} \mypaplist{pn1}}
   \ITEE{#3}{pn2006a}{\bibITEM{#2}{#1} \mypaplist{pn2}}
   \ITEE{#3}{pn2006b}{\bibITEM{#2}{#1} \mypaplist{pn3}}
   \ITEE{#3}{pn2007}{\bibITEM{#2}{#1} \mypaplist{pn4}}
   \ITEE{#3}{pn2008a}{\bibITEM{#2}{#1} \mypaplist{pn5}}
   \ITEE{#3}{pn2008b}{\bibITEM{#2}{#1} \mypaplist{pn6}}
   \ITEE{#3}{pn2009a}{\bibITEM{#2}{#1} \mypaplist{pn7}}
   \ITEE{#3}{pn2009b}{\bibITEM{#2}{#1} \mypaplist{pn8}}
   \ITEE{#3}{pn2009c}{\bibITEM{#2}{#1} \mypaplist{pn9}}
   \ITEE{#3}{pn2010a}{\bibITEM{#2}{#1} \mypaplist{pn12}}
   \ITEE{#3}{pn2010b}{\bibITEM{#2}{#1} \mypaplist{pn13}}
   \ITEE{#3}{pn2011a}{\bibITEM{#2}{#1} \mypaplist{pn10}}
   \ITEE{#3}{pn2011b}{\bibITEM{#2}{#1} \mypaplist{pn15}}
   \ITEE{#3}{pn2011c}{\bibITEM{#2}{#1} \mypaplist{pn16}}
   \ITEE{#3}{pn2009x}{
      \bibITEM{#2}{#1} \mypaplist{pn11}}
   \ITEE{#3}{pn2010x}{
      \bibITEM{#2}{#1} \mypaplist{pn14}}
   \ITEE{#3}{pnXXXXa}{
      \bibITEM{#2}{#1} \mypaplist{pnX1}}
   \ITEE{#3}{pnXXXXb}{
      \bibITEM{#2}{#1} \mypaplist{pnX2}}
   \ITEE{#3}{pnXXXXc}{
      \bibITEM{#2}{#1} \mypaplist{pnX3}}
   \ITEE{#3}{pnXXXXd}{
      \bibITEM{#2}{#1} \mypaplist{pnX13}}
   \ITEE{#3}{MNiezgoda1998}{
      \BIB{#2}{M. Niezgoda}
         {Group majorization and Schur type inequalities}
         {\jRN{LAA}}{268}{1998}{9--30}{#1}}
   \ITEE{#3}{MNiezgoda1998a}{
      \BIB{#2}{M. Niezgoda}
         {An analytical characterization of effective and of irreducible groups inducing cone orderings}
         {\jRN{LAA}}{269}{1998}{105--114}{#1}}
   \ITEE{#3}{MNiezgoda,TYTam2001}{
      \BIB{#2}{M. Niezgoda and T.Y. Tam}
         {On norm property of $G(c)$-radii and Eaton triples}
         {\jRN{LAA}}{336}{2001}{119--130}{#1}}
   \ITEE{#3}{APazy1983}{
      \BIb{#2}{A. Pazy}{Semigroups of Linear Operators 
         and Applications to Partial Differential Equations \textup{(Applied Mathematical Sciences, vol. 44)}}
         {Springer-Verlag, New York}{1983}{#1}}
   \ITEE{#3}{APelc1982}{
      \BIB{#2}{A. Pelc}
         {Semiregular invariant measures on abelian groups}
         {\jRN{PAMS}}{86}{1982}{423--426}{#1}}
   \ITEE{#3}{RPenrose1955}{
      \BIB{#2}{R. Penrose}
         {A generalized inverse for matrices}
         {\jRN{ProcCambPhS}}{51}{1955}{406--413}{#1}}
   \ITEE{#3}{VPestov2006}{
      \BIb{#2}{V. Pestov}
         {Dynamics of infinite-dimensional groups. The Ramsey-Dvoretzky-Milman phenomenon}
         {University Lecture Series \textbf{40}, AMS, Providence, RI}{2006}{#1}}
   \ITEE{#3}{VPestov2007}{
      \BiB{#2}{V. Pestov}
         {Forty-plus annotated questions about large topological groups}
         {in:}{Open Problems in Topology II}{Elliot Pearl (editor), Elsevier B.V., Amsterdam}{2007}{439--450}{#1}}
   \ITEE{#3}{PVPetersen1993}{
      \BiB{#2}{P.V. Petersen}
         {Gromov-Hausdorff convergence of metric spaces}{in book:}{Differential Geometry: Riemannian Geometry 
         (Los Angeles, CA, 1990)}{Amer. Math. Soc., Providence, RI}{1993}{489--504}{#1}}
   \ITEE{#3}{DRamachandran,MMisiurewicz1982}{
      \BIB{#2}{D. Ramachandran and M. Misiurewicz}
         {Hopf's theorem on invariant measures for a group of transformations}
         {\jRN{SM}}{74}{1982}{183--189}{#1}}
   \ITEE{#3}{JMRosenblatt1974}{
      \BIB{#2}{J.M. Rosenblatt}
         {Equivalent invariant measures}
         {\jRN{IsraelJM}}{17}{1974}{261--270}{#1}}
   \ITEE{#3}{HLRoyden1963}{
      \BIb{#2}{H.L. Royden}
         {Real Analysis}
         {The Macmillan Co., New York}{1963}{#1}}
   \ITEE{#3}{WRudin1962}{
      \BIb{#2}{W. Rudin}
         {Fourier Analysis on Groups \textup{(Interscience Tracts in Pure and Applied Mathematics, Number 12)}}
         {Interscience Publishers, New York}{1962}{#1}}
   \ITEE{#3}{WRudin1991}{
      \BIb{#2}{W. Rudin}
         {Functional Analysis}
         {McGraw-Hill Science}{1991}{#1}}
   \ITEE{#3}{TSaito1972}{
      \BiB{#2}{T. Sait\^{o}}{Generations of von Neumann algebras}
         {Lecture Notes in Math. vol. 247}{\textup{(}Lecture on Operator Algebras\textup{)}}
         {Springer, Berlin-Heidelberg-New York}{1972}{435--531}{#1}}
   \ITEE{#3}{KSakai,MYaguchi2003}{
      \BIB{#2}{K. Sakai and M. Yaguchi}
         {Characterizing manifolds modeled on certain dense subspaces of non-separable Hilbert spaces}
         {\jRN{TsukubaJM}}{27}{2003}{143--159}{#1}}
   \ITEE{#3}{SSakai1971}{
      \BIb{#2}{S. Sakai}
         {$\CCc^*$-Algebras and $\WWw^*$-Algebras}
         {Springer-Verlag, Berlin-Heidelberg-New York}{1971}{#1}}
   \ITEE{#3}{RSchori1971}{
      \BIB{#2}{R. Schori}
         {Topological stability for infinite-dimensional manifolds}
         {\jRN{ComposM}}{23}{1971}{87--100}{#1}}
   \ITEE{#3}{JTSchwartz1967}{
      \BIb{#2}{J.T. Schwartz}
         {$\WWw^*$-algebras}
         {Gordon and Breach, Science Publishers Inc., New York-London-Paris}{1967}{#1}}
   \ITEE{#3}{ZSemadeni1971}{
      \BIb{#2}{Z. Semadeni}
         {Banach Spaces of Continuous Functions (Vol. I)}
         {\jRN{PWN}}{1971}{#1}}
   \ITEE{#3}{JPSerre1951}{
      \BIB{#2}{J.-P. Serre}
         {Homologie singuli\`{e}re des espaces fibr\'{e}s}
         {\jRN{AnnM}}{54}{1951}{425--505}{#1}}
   \ITEE{#3}{DSherman2007}{
      \BIB{#2}{D. Sherman}
         {On the dimension theory of von Neumann algebras}
         {\jRN{MScand}}{101}{2007}{123--147}{#1}}
   \ITEE{#3}{WSierpinski1928}{
      \BIB{#2}{W. Sierpi\'{n}ski}
         {Sur les projections des ensembles compl\'{e}mentaires aux ensembles \textup{(A)}}
         {\jRN{FM}}{11}{1928}{117--122}{#1}}
   \ITEE{#3}{MSlocinski1980}{
      \BIB{#2}{M. S\l{}oci\'{n}ski}
         {On the Wold-type decomposition of a pair of commuting isometries}
         {\jRN{APM}}{37}{1980}{255--262}{#1}}
   \ITEE{#3}{RCSteinlage1975}{
      \BIB{#2}{R.C. Steinlage}
         {On Haar measure in locally compact $T_2$ spaces}
         {\jRN{AmJM}}{97}{1975}{291--307}{#1}}
   \ITEE{#3}{JStochel,FHSzafraniec1989}{
      \BIB{#2}{J. Stochel and F.H. Szafraniec}
         {On normal extensions of unbounded operators. III. Spectral properties}
         {\jRN{PublRIMSKyoto}}{25}{1989}{105--139}{#1}}
   \ITEE{#3}{JStochel,FHSzafraniec1989a}{
      \BIB{#2}{J. Stochel and F.H. Szafraniec}
         {The normal part of an unbounded operator}
         {\jRN{ProcKonink}}{92}{1989}{495--503}{#1}}
   \ITEE{#3}{AHStone1962}{
      \BIB{#2}{A.H. Stone}
         {Absolute $\FFf_{\sigma}$-spaces}
         {\jRN{PAMS}}{13}{1962}{495--499}{#1}}
   \ITEE{#3}{AHStone1962a}{
      \BIB{#2}{A.H. Stone}
         {Non-separable Borel sets}
         {\jRN{DissM}}{28}{1962}{41 pages}{#1}}
   \ITEE{#3}{AHStone1972}{
      \BIB{#2}{A.H. Stone}
         {Non-separable Borel sets II}
         {\jRN{GTopA}}{2}{1972}{249--270}{#1}}
   \ITEE{#3}{MHStone1937}{
      \BIB{#2}{M.H. Stone}
         {Application of the theory of Boolean rings to general topology}
         {\jRN{TAMS}}{41}{1937}{375--481}{#1}}
   \ITEE{#3}{MHStone1948}{
      \BIB{#2}{M.H. Stone}
         {The generalized Weierstrass approximation theorem}
         {\jRN{MMag}}{21}{1948}{167--184}{#1}}
   \ITEE{#3}{BSz-Nagy1947}{
      \BIB{#2}{B. Sz.-Nagy}
         {On uniformly bounded linear transformations in Hilbert space}
         {\jRN{ActaSM}}{11}{1947}{152--157}{#1}}
   \ITEE{#3}{WTakahashi1970}{
      \BIB{#2}{W. Takahashi}
         {A convexity in metric space and nonexpansive mappings, I}
         {\jRN{KodaiMSemRep}}{22}{1970}{142--149}{#1}}
   \ITEE{#3}{MTakesaki2002}{
      \BIb{#2}{M. Takesaki}
         {Theory of Operator Algebras I \textup{(Encyclopaedia of Mathematical Sciences, Volume 124)}}
         {Springer-Verlag, Berlin-Heidelberg-New York}{2002}{#1}}
   \ITEE{#3}{MTakesaki2003}{
      \BIb{#2}{M. Takesaki}
         {Theory of Operator Algebras II \textup{(Encyclopaedia of Mathematical Sciences, Volume 125)}}
         {Springer-Verlag, Berlin-Heidelberg-New York}{2003}{#1}}
   \ITEE{#3}{MTakesaki2003a}{
      \BIb{#2}{M. Takesaki}
         {Theory of Operator Algebras III \textup{(Encyclopaedia of Mathematical Sciences, Volume 127)}}
         {Springer-Verlag, Berlin-Heidelberg-New York}{2003}{#1}}
   \ITEE{#3}{TYTam1999}{
      \BIB{#2}{T.Y. Tam}
         {An extension of a result of Lewis}
         {\jRN{ELA}}{5}{1999}{1--10}{#1}}
   \ITEE{#3}{TYTam2000}{
      \BIB{#2}{T.Y. Tam}
         {Group majorization, Eaton triples and numerical range}
         {\jRN{LMLA}}{47}{2000}{11--28}{#1}}
   \ITEE{#3}{TYTam2002}{
      \BIB{#2}{T.Y. Tam}
         {Generalized Schur-concave functions and Eaton triples}
         {\jRN{LMLA}}{50}{2002}{113--120}{#1}}
   \ITEE{#3}{TYTam,WCHill2001}{
      \BIB{#2}{T.Y. Tam and W.C. Hill}
         {On $G$-invariant norms}
         {\jRN{LAA}}{331}{2001}{101--112}{#1}}
   \ITEE{#3}{AFTiman,IAVestfrid1983}{
      \BIB{#2}{A.F. Timan and I.A. Vestfrid}
         {Any separable ultrametric space can be isometrically imbedded in $l_2$}
         {\jRN{FAA}}{17}{1983}{70--71}{#1}}
   \ITEE{#3}{JTomiyama1958}{
      \BIB{#2}{J. Tomiyama}
         {Generalized dimension function for $\WWw^*$-algebras of infinite type}
         {\jRN{TohokuMJ} (2)}{10}{1958}{121--129}{#1}}
   \ITEE{#3}{HTorunczyk1970}{
      \BIB{#2}{H. Toru\'{n}czyk}
         {Remarks on Anderson's paper ``On topological infinite deficiency''}
         {\jRN{FM}}{66}{1970}{393--401}{#1}}
   \ITEE{#3}{HTorunczyk1970a}{
      \BIb{#2}{H. Toru\'{n}czyk}
         {$G$-$K$-absorbing and skeletonized sets in metric spaces}
         {Ph.D. thesis, Inst. Math. Polish Acad. Sci., Warszawa}{1970}{#1}}
   \ITEE{#3}{HTorunczyk1972}{
      \BIB{#2}{H. Toru\'{n}czyk}
         {A short proof of Hausdorff's theorem on extending metrics}
         {\jRN{FM}}{77}{1972}{191--193}{#1}}
   \ITEE{#3}{HTorunczyk1974}{
      \BIB{#2}{H. Toru\'{n}czyk}
         {Absolute retracts as factors of normed linear spaces}
         {\jRN{FM}}{86}{1974}{53--67}{#1}}
   \ITEE{#3}{HTorunczyk1975}{
      \BIB{#2}{H. Toru\'{n}czyk}
         {On Cartesian factors and the topological classification of linear metric spaces}
         {\jRN{FM}}{88}{1975}{71--86}{#1}}
   \ITEE{#3}{HTorunczyk1978}{
      \BIB{#2}{H. Toru\'{n}czyk}
         {Concerning locally homotopy negligible sets and characterization of $l_2$-manifolds}
         {\jRN{FM}}{101}{1978}{93--110}{#1}}
   \ITEE{#3}{HTorunczyk1980}{
      \BiB{#2}{H. Toru\'{n}czyk}{Characterization of infinite-dimensional manifolds}{in:}
         {Proceedings of the International Conference on Geometric Topology (Warsaw, 1978)}
         {\jRN{PWN}}{1980}{431--437}{#1}}
   \ITEE{#3}{HTorunczyk1981}{
      \BIB{#2}{H. Toru\'{n}czyk}
         {Characterizing Hilbert space topology}
         {\jRN{FM}}{111}{1981}{247--262}{#1}}
   \ITEE{#3}{HTorunczyk1985}{
      \BIB{#2}{H. Toru\'{n}czyk}
         {A correction of two papers concerning Hilbert manifolds}
         {\jRN{FM}}{125}{1985}{89--93}{#1}}
   \ITEE{#3}{KTsuda1985}{
      \BIB{#2}{K. Tsuda}
         {A note on closed embeddings of finite dimensional metric spaces}
         {\jRN{BLondMS}}{17}{1985}{273--278}{#1}}
   \ITEE{#3}{PSUrysohn1925}{
      \BIB{#2}{P.S. Urysohn}
         {Sur un espace m\'{e}trique universel}
         {\jRN{CRASParis}}{180}{1925}{803--806}{#1}}
   \ITEE{#3}{PSUrysohn1927}{
      \BIB{#2}{P.S. Urysohn}
         {Sur un espace m\'{e}trique universel}
         {\jRN{BullSM}}{51}{1927}{43--64, 74--96}{#1}}
   \ITEE{#3}{VVUspenskij1986}{
      \BIB{#2}{V.V. Uspenskij}
         {A universal topological group with a countable basis}
         {\jRN{FAA}}{20}{1986}{86--87}{#1}}
   \ITEE{#3}{VVUspenskij1990}{
      \BIB{#2}{V.V. Uspenskij}
         {On the group of isometries of the Urysohn universal metric space}
         {\jRN{CMUC}}{31}{1990}{181--182}{#1}}
   \ITEE{#3}{VVUspenskij2004}{
      \BIB{#2}{V.V. Uspenskij}
         {The Urysohn universal metric space is homeomorphic to a Hilbert space}
         {\jRN{TopA}}{139}{2004}{145--149}{#1}}
   \ITEE{#3}{VVUspenskij2008}{
      \BIB{#2}{V.V. Uspenskij}
         {On subgroups of minimal topological groups}
         {\jRN{TopA}}{155}{2008}{1580--1606}{#1}}
   \ITEE{#3}{VSVaradarajan1963}{
      \BIB{#2}{V.S. Varadarajan}
         {Groups of automorphisms of Borel spaces}
         {\jRN{TAMS}}{109}{1963}{191--220}{#1}}
   \ITEE{#3}{AMVershik1998}{
      \BIB{#2}{A.M. Vershik}
         {The universal Urysohn space, Gromov's metric triples, and random metrics on the series of natural numbers}
         {\jRN{UspekhiMN}}{53}{1998}{57--64}{#1} English translation: \jRN{RussMS}{} \textbf{53} (1998), 921--928. 
         Correction: \jRN{UspekhiMN}{} \textbf{56} (2001), p. 207. English translation: \jRN{RussMS}{} \textbf{56} 
         (2001), p. 1015.}
   \ITEE{#3}{AMVershik2002}{
      \BIb{#2}{A.M. Vershik}
         {Random metric spaces and the universal Urysohn space}
         {Fundamental Mathematics Today. 10th anniversary of the Independent Moscow University. MCCME Publ.}{2002}{#1}}
   \ITEE{#3}{NWeaver1999}{
      \BIb{#2}{N. Weaver}
         {Lipschitz Algebras}
         {World Scientific}{1999}{#1}}
   \ITEE{#3}{JWeidmann1980}{
      \BIb{#2}{J. Weidmann}
         {Linear Operators in Hilbert Spaces}
         {(Graduate Texts in Mathematics, vol. 68) Springer-Verlag New York Inc.}{1980}{#1}}
   \ITEE{#3}{JEWest1969}{
      \BIB{#2}{J.E. West}
         {Approximating homotopies by isotopies in Fr\'{e}chet manifolds}
         {\jRN{BAMS}}{75}{1969}{1254--1257}{#1}}
   \ITEE{#3}{JEWest1969a}{
      \BIB{#2}{J.E. West}
         {Fixed-point sets of transformation groups on infinite-product spaces}
         {\jRN{PAMS}}{21}{1969}{575--582}{#1}}
   \ITEE{#3}{JEWest1970}{
      \BIB{#2}{J.E. West}
         {The ambient homeomorphy of infinite-dimensional Hilbert spaces}
         {\jRN{PacJM}}{34}{1970}{257--267}{#1}}
   \ITEE{#3}{JHCWhitehead1949}{
      \BIB{#2}{J.H.C. Whitehead}
         {Combinatorial homotopy I}
         {\jRN{BAMS}}{55}{1949}{213--245}{#1}}
   \ITEE{#3}{GTWhyburn1942}{
      \BIb{#2}{G. T. Whyburn}
         {Analytic Topology}
         {Amer. Math. Soc. Colloquium Publications (vol. XXVIII), New York}{1942}{#1}}
   \ITEE{#3}{WWogen1969}{
      \BIB{#2}{W. Wogen}
         {On generators for von Neumann algebras}
         {\jRN{BAMS}}{75}{1969}{95--99}{#1}}
   \ITEE{#3}{RYTWong1967}{
      \BIB{#2}{R.Y.T. Wong}
         {On homeomorphisms of certain infinite dimensional spaces}
         {\jRN{TAMS}}{128}{1967}{148--154}{#1}}
   \ITEE{#3}{LYang,JZhang1987}{
      \BIB{#2}{L. Yang and J. Zhang}
         {Average distance constants of some compact convex space}
         {\jRN{JChinUST}}{17}{1987}{17--23}{#1}}
   \ITEE{#3}{PZakrzewski1993}{
      \BIB{#2}{P. Zakrzewski}
         {The existence of invariant $\sigma$-finite measures for a group of transformations}
         {\jRN{IsraelJM}}{83}{1993}{275--287}{#1}}
   \ITEE{#3}{PZakrzewski2002}{
      \BIb{#2}{P. Zakrzewski}
         {Measures on Algebraic-Topological Structures, Handbook of Measure Thoery}
         {E. Pap, ed., Elsevier, Amsterdam}{2002, 1091--1130}{#1}}
   \ITEE{#3}{KZhu2000}{
      \BIB{#2}{K. Zhu}
         {Operators in Cowen-Douglas classes}
         {\jRN{IllinoisJM}}{44}{2000}{767--783}{#1}}
   }

\newcommand{\mypaplist}[2][]{
   \ITEE{#2}{pn1}{
      \myBIB{Separate and joint similarity to families of normal operators}
         {\jRN[#1]{SM}}{149}{2002}{39--62}}
   \ITEE{#2}{pn2}{
      \myBIB{Locally arcwise connected metrizable spaces with the fixed point property are complete-metrizable}
         {\jRN[#1]{TopA}}{153}{2006}{1639--1642}}
   \ITEE{#2}{pn3}{
      \myBIB{Invariant measures for equicontinuous semigroups of continuous transformations 
         of a compact Hausdorff space}{\jRN[#1]{TopA}}{153}{2006}{3373--3382}}
   \ITEE{#2}{pn4}{
      \myBIB{Approximation of the Hausdorff distance by the distance of continuous surjections}
         {\jRN[#1]{TopA}}{154}{2007}{655--664}}
   \ITEE{#2}{pn5}{
      \myBIB{Generalized Haar integral}
         {\jRN[#1]{TopA}}{155}{2008}{1323--1328}}
   \ITEE{#2}{pn6}{
      \myBIB{Integration and Lipschitz functions}
         {\jRN[#1]{RCMP}}{57}{2008}{391--399}}
   \ITEE{#2}{pn7}{
      \myBIB{Canonical Banach function spaces generated by Urysohn universal spaces. Measures as Lipschitz maps}
         {\jRN[#1]{SM}}{192}{2009}{97--110}}
   \ITEE{#2}{pn8}{
      \myBIB{Urysohn universal spaces as metric groups of exponent $2$}
         {\jRN[#1]{FM}}{204}{2009}{1--6}}
   \ITEE{#2}{pn9}{
      \myBIB{Central subsets of Urysohn universal spaces}
         {\jRN[#1]{CMUC}}{50}{2009}{445--461}}
   \ITEE{#2}{pn10}{
      \myBIB[P. Niemiec and T.Y. Tam]{A representation of $G$-in\-variant norms for Eaton triple}
         {\jRN[#1]{JCA}}{18}{2011}{59--65}}
   \ITEE{#2}{pn11}{
      \myBIB{Functor of extension of contractions on Urysohn universal spaces}
         {\jRN[#1]{ACS}}{}{2009}{\texttt{DOI: 10.1007/s10485-009-9218-z}}}
   \ITEE{#2}{pn12}{
      \myBIB{Ultra-$\mM$-separability}
         {\jRN[#1]{TopA}}{157}{2010}{669--673}}
   \ITEE{#2}{pn13}{
      \myBIB{Functor of extension of $\Lambda$-isometric maps between central subsets 
         of the unbounded Urysohn universal space}{\jRN[#1]{CMUC}}{51}{2010}{541--549}}
   \ITEE{#2}{pn14}{
      \myBIB{Normed topological pseudovector groups}{\jRN[#1]{ACS}}{}{2010}
         {\ITE{\equal{#1}{}}{\texttt{DOI: 10.1007/s10485\-010-9239-7}}{\texttt{DOI: 10.1007/s10485-010-9239-7}}}}
   \ITEE{#2}{pn15}{
      \myBIB{Topological structure of Urysohn universal spaces}
         {\jRN[#1]{TopA}}{158}{2011}{352--359}}
   \ITEE{#2}{pn16}{
      \myBIB{A note on invariant measures}
         {\jRN[#1]{OpusM}}{31}{2011}{425--431}}
   \ITEE{#2}{pnX1}{
      \myBAPP{Strengthened Stone-Weierstrass type theorem}
         {accepted for publication in \jRN[#1]{OpusM}}}
   \ITEE{#2}{pnX2}{
      \myBAPP{Functor of continuation in Hilbert cube and Hilbert space}
         {to appear in \jRN[#1]{FM}}}
   \ITEE{#2}{pnX3}{
      \myBAPP{Norm closures of orbits of bounded operators}
         {to appear.}}
   \ITEE{#2}{pnX6}{
      \myBAPP{Extending maps by injective $\sigma$-$Z$-maps in Hilbert manifolds}
         {to appear in \jRN[#1]{BullPol}}}
   \ITEE{#2}{pnX7}{
      \myBAPP{Spaces of measurable functions}
         {submitted to \jRN[#1]{CollectM}}}
   \ITEE{#2}{pnX8}{
      \myBAPP{Normal systems over ANR's, rigid embeddings and nonseparable absorbing sets}
         {submitted to \jRN[#1]{}}}
   \ITEE{#2}{pnX9}{
      \myBAPP{Borel structure of the spectrum of a closed operator}
         {submitted to \jRN[#1]{JFA}}}
   \ITEE{#2}{pnX10}{
      \myBAPP{Central points and measures and dense subsets of compact metric spaces}
         {submitted to \jRN[#1]{NonlinA}}}
   \ITEE{#2}{pnX11}{
      \myBAPP{Generalized absolute values and polar decompositions of a bounded operator}
         {submitted to \jRN[#1]{IEOT}.}}
   \ITEE{#2}{pnX12}{
      \myBAPP{Ultrametrics, extending of Lipschitz maps and nonexpansive selections}
         {submitted to \jRN[#1]{HJM}}}
   \ITEE{#2}{pnX13}{
      \myBAPP{A note on ANR's}
         {submitted to \jRN[#1]{TopA}}}
   \ITEE{#2}{pnX14}{
      \myBAPP{Problem with almost everywhere equality}
         {submitted to \jRN[#1]{ArchM}}}
   \ITEE{#2}{pnX15}{
      \myBAPP{Universal valued Abelian groups}
         {submitted to \jRN[#1]{MProcCambPhS}}}
   }

\newcommand{\CDDc}{\operatorname{CDD}}\newcommand{\CDD}{\CcC\DdD\DdD}\newcommand{\RED}{\operatorname{cred}}
\newcommand{\zero}{\OOO}\newcommand{\disj}{\perp_u}\newcommand{\sqplus}{\boxplus}\newcommand{\sqminus}{\boxminus}
\newcommand{\fIN}{\FfF\IiI\NnN}\newcommand{\bigsqplus}{\operatornamewithlimits{\raisebox{-0.5ex}{\Large$\boxplus$}}}
\newcommand{\Bigsqplus}{\operatornamewithlimits{\raisebox{-1ex}{\Huge$\boxplus$}}}\newcommand{\Tr}{\operatorname{Tr}}
\newcommand{\contS}[3][]{\left.\bigoplus_{#2}#1\!\right.^{#3}}\newcommand{\Dim}{\operatorname{Dim}}
\newcommand{\contSQ}[3][]{\left.\Bigsqplus_{#2}#1\!\right.^{#3}}\newcommand{\Div}{\operatorname{Div}}
\newcommand{\tII}{I\!I}\newcommand{\tIII}{I\!I\!I}\newcommand{\sdisj}{\perp_s}\newcommand{\rgM}{\operatorname{rgm}}
\newcommand{\rgS}{\operatorname{RGS}}\newcommand{\DiNT}[2]{\frac{\dint{#1}}{\dint{#2}}}
\newcommand{\stSt}{\stackrel{*s}{\to}}

\newcommand{\pREF}[1]{(page~\pageref{#1})}\newcommand{\PREF}[1]{page~\pageref{#1}}
\newcommand{\EQp}[1]{\eqref{eqn:#1} \pREF{eqn:#1}}\newcommand{\EQP}[1]{\eqref{eqn:#1}, \PREF{eqn:#1}}
\newcommand{\THMp}[1]{\THM{#1} \pREF{thm:#1}}\newcommand{\THMP}[1]{\THM{#1}, \PREF{thm:#1}}
\newcommand{\PROp}[1]{\PRO{#1} \pREF{pro:#1}}\newcommand{\PROP}[1]{\PRO{#1}, \PREF{pro:#1}}
\newcommand{\LEMp}[1]{\LEM{#1} \pREF{lem:#1}}\newcommand{\LEMP}[1]{\LEM{#1}, \PREF{lem:#1}}
\newcommand{\CORp}[1]{\COR{#1} \pREF{cor:#1}}\newcommand{\CORP}[1]{\COR{#1}, \PREF{cor:#1}}
\newcommand{\EXMp}[1]{\EXM{#1} \pREF{exm:#1}}
\newcommand{\DEFp}[1]{\DEF{#1} \pREF{def:#1}}
\begin{document}

\title[Unitary equivalence of systems of operators]
   {Unitary equivalence and decompositions of finite systems of closed densely defined operators in Hilbert spaces}
\myData
\begin{abstract}
An \textit{ideal} of $N$-tuples of operators is a class invariant with respect to unitary equivalence which contains
direct sums of arbitrary collections of its members as well as their (reduced) parts. New decomposition theorems (with
respect to ideals) for $N$-tuples of closed densely defined linear operators acting in a common (arbitrary) Hilbert space
are presented. Algebraic and order (with respect to containment) properties of the class $\CDD_N$ of all unitary
equivalence classes of such $N$-tuples are established and certain ideals in $\CDD_N$ are distinguished. It is proved that
infinite operations in $\CDD_N$ may be reconstructed from the direct sum operation of a pair. \textit{Prime decomposition}
in $\CDD_N$ is proposed and its (in a sense) uniqueness is established. The issue of classification of ideals in $\CDD_N$
(up to isomorphism) is discussed. A model for $\CDD_N$ is described and its concrete realization is presented. A new
partial order of $N$-tuples of operators is introduced and its fundamental properties are established. Extremal importance
of unitary disjointness of $N$-tuples and the way how it `tidies up' the structure of $\CDD_N$ are emphasized.
\end{abstract}
\maketitle


\SECT{Introduction}

Criterions for unitary equivalence of two (bounded linear) operators (acting on Hilbert spaces) and the classification
of operators up to unitary equivalence are subjects which focused an attention of many mathematicians inspired by methods
and ideas proven in practice with quite well explored normal operators. The literature dealing with these and related
topics is still growing up, let us mention here only a few: Brown \cite{bro} classified quasi-normal operators; Halmos and
McLaughlin \cite{hml} reduced the issue of unitary equivalence of arbitrary bounded operators to partial isometries; Ernest
\cite{e}, Hadwin \cite{ha1,ha2} and others (e.g. \cite{kkl}) investigated operator-valued spectra which generalized
standard (scalar) spectrum of a normal operator; Ernest \cite{e}, Brown, Fong and Hadwin \cite{bfh} and Loebl \cite{loe}
studied parts (that is, suboperators) of operators. It was Ernest \cite{e} who first shown that --- in a sense ---
the classification of all operators up to unitary equivalence is an essentially unattainable objective, although he gave
an equivalent condition for two (totally arbitrary) bounded operators to be unitarily equivalent. It was formulated
by means of certain (operator-valued) spectra of operators and multiplicity theory extended from normal to all bounded
operators (roughly speaking, he adapted and generalized the classical Hahn-Hellinger theorem). The recent paper is
motivated by his approach to this subject. One of aims of the paper is to finish Ernest's program of exploring the realm
of unitary equivalence classes of closed densely defined operators by making no assumptions neither on the dimension
of Hilbert spaces nor on boundedness of operators (this solves the problem posed by Ernest in point c of \S7 of Chapter~5
of \cite{e}). Even more, we study the class $\CDD_N$ of finite systems ($N$-tuples) of closed densely defined operators
acting in (totally arbitrary) common Hilbert spaces. Surprisingly, such general considerations lead to more elegant results
and reveal features which become invisible when one restricts only to separable spaces. Although $\CDD_N$ is not a set but
a class, we shall show that it is `controlled' by a single $N$-tuple (acting in a nonseparable space; cf. \PROP{unity})
and this observation will enable us to find an (algebraic as well as order) model for $\CDD_N$ (\THMP{model}).
An elementary form of the model will enable us to establish several interesting properties of $\CDD_N$
(e.g. (AO13)--(AO14), \PREF{AO13}). Also the central decomposition introduced by Ernest for an operator acting
in a separable space may be extended to a general context and translated into a more attractive (at least for us) form
of the `prime decomposition' similar to the one for natural numbers (\THMP{prime}).\par
Another aspect discussed in the treatise concerns various (known) results on decompositions of operators. There is a quite
large number of results stating that a certain operator may be uniquely decomposed into two (or more) parts the first
of which is of special type (kind, class, etc.) and the second has no nontrivial part of this type. The latter part is
often named by a phrase of the form `completely non-\textit{sth}' or `purely \textit{sth}'. Let us mention only a few such
results:
\begin{enumerate}[(DC1)]
\item a contraction operator may be decomposed into a unitary part and a completely non-unitary one,
\item a bounded operator may be decomposed into a normal part and a completely non-normal one,
\item a closed densely defined operator admits a unique decomposition into a normal, a purely formally normal
   and a completely non-formally normal part (\cite{s-sz})
\end{enumerate}
(other results in this fashion are e.g. \cite{fk3}, \cite{slo}, \cite{cps}). There is a striking resemblance in the above
statements. And this is not a coincidence. In this paper we put \textbf{all} results of this type in a \textit{one} general
frame. To be more precise, let us introduce the notion of an \textit{ideal}. It is any nonempty class $\AaA$ of closed
densely defined operators which satisfies the following three axioms:
\begin{itemize}
\item if $A$ and $B$ are unitarily equivalent, then $A \in \AaA \iff B \in \AaA$,
\item every part (including the trivial one acting on a zero-dimen\-sional Hilbert space) of a member of $\AaA$ belongs
   to $\AaA$,
\item $\bigoplus_{s \in S} A_s \in \AaA$ whenever $\{A_s\}_{s \in S} \subset \AaA$ (and $S$ is a nonempty set).
\end{itemize}
For every ideal $\AaA$ we denote by $\AaA^{\perp}$ the class of all operators $A$ whose no nontrivial part belongs
to $\AaA$. In \THMp{1} we show that whenever $\AaA$ and $\BbB$ are ideals, so is $\AaA^{\perp}$ and every (closed densely
defined) operator $T$ acting in a (completely arbitrary) Hilbert space $\HHh$ induces a unique decomposition $\HHh =
\HHh_{11} \oplus \HHh_{10} \oplus \HHh_{01} \oplus \HHh_{00}$ such that $\HHh_{jk}$ are reducing subspaces for $T$
and $T\bigr|_{\HHh_{11}} \in \AaA \cap \BbB$, $T\bigr|_{\HHh_{10}} \in \AaA \cap \BbB^{\perp}$, $T\bigr|_{\HHh_{01}} \in
\AaA^{\perp} \cap \BbB$ and $T\bigr|_{\HHh_{00}} \in \AaA^{\perp} \cap \BbB^{\perp}$. This result covers (DC1)--(DC3)
and all above mentioned theorems on decompositions.\par
Ernest \cite[Definition~1.7]{e} has introduced the notion of \textit{disjoint} operators, say $A$ and $B$. In this paper
we express this by writing `$A \disj B$' and call $A$ and $B$ \textit{unitarily disjoint}. (Unitary disjointness,
as a relation, behaves as singularity of measures or orthogonality in Hilbert spaces. Beside this, unitary disjointness
is formulated in order-theoretic terms in the same way as disjointness in Banach lattices, where the disjointness of two
vectors $x$ and $y$ is expressed by writing $x \perp y$. This is why we prefer symbol `$\disj$' for disjointness
of operators than Ernest's original notation.) For Ernest the disjointness was only one of possible relations between
operators. His \textit{Lebesgue decomposition} of a one operator with respect to the other (Proposition~2.12
and Definition~2.13 in \cite{e}) seems to be a secondary result rather than a `serious' theorem. Another aim of our
treatise is to underline the great importance of (unitary) disjointness (for example, we demonstrate how Ernest's central
decomposition, or our prime one, may be translated into the `intrinsic' language of operators, with use of unitary
disjointness; also the proof of above mentioned \THM{1} depends on properties of unitary disjointness). Roughly speaking,
composing direct sums from arbitrary collections of operators is very \textit{chaotic}, while the direct sum of a family
of mutually unitarily disjoint operators is well `arranged'. We may compare this with representing a simple Borel function
(i.e. whose range is finite) as a linear combination of the characteristic functions of Borel sets---this may be done
in infinitely many ways; there is however only one such a representation in which all appearing sets form a partition
of the domain of the function. The latter form of a simple function tells us \textit{everything} about the function.
The same occurs in the class $\CDD_N$ (see e.g. \THMP{decomp}) when an $N$-tuple is written as the direct sum
of a collection of mutually unitarily disjoint $N$-tuples. To distinguish between these specific decompositions
and `chaotic' ones, we call every direct sum (as well as any collection) of pairwise unitarily disjoint $N$-tuples
\textit{regular}. The notion of regularity may easily be adapt for the `continuous' versions of direct sums (defined
in Section~21 by means of direct integrals). This generalization turns out to be crucial for formulating our Prime
Decomposition Theorem (\THMP{prime}).\par
The main tools of the treatise are, as in Ernest's work \cite{e}, techniques of von Neumann algebras. In Sections~1--17
and 23 we involve dimension theory of $\WWw^*$-algebras, especially a property discovered recently by Sherman \cite{sh}.
All results of these sections may be formulated and proved in the language of a `semigroup' $\CDD_N$ with the direct sum
of a pair as the only available operation (cf. Section~13). The remainder (Sections~18--22) depends on the reduction theory
due to von Neumann \cite{vN2}. This deals with topological and measure-theoretic aspects which are introduced
in Sections~18--20. It is assumed that the reader is well oriented in basics of von Neumann algebras (it is enough to know
the materials of \cite{sak}, \cite{k-r1,k-r2} and \cite{tk1}).\par
The main results of the paper are Theorems~\ref{thm:1} \pREF{thm:1}, \ref{thm:decomp} \pREF{thm:decomp},
\ref{thm:model} \pREF{thm:model}, \ref{thm:prime} \pREF{thm:prime} and \ref{thm:isom} \pREF{thm:isom}.
\vspace{0.3cm}

\textbf{Notation and terminology.} In this paper $\RRR_+ = [0,\infty)$ and all Hilbert spaces are over the complex field.
$\HHh$ and $\KKk$ denote (possibly trivial) Hilbert spaces. By an \textit{operator} we mean a linear function between
linear subspaces of Hilbert spaces. The Hilbert space dimension of $\HHh$ is denoted by $\dim \HHh$. $\BBb(\HHh,\KKk)$ and
$\UUu(\HHh,\KKk)$ denote, respectively, the Banach space of all bounded operators from $\HHh$ into $\KKk$ and the set
of all unitary operators of $\HHh$ onto $\KKk$, and $\BBb(\HHh) = \BBb(\HHh,\HHh)$ and $\UUu(\HHh) = \UUu(\HHh,\HHh)$.
Whenever $A$ is an operator, $\DdD(A)$, $\RrR(A)$, $\overline{\DdD}(A)$ and $\overline{\RrR}(A)$ stand for, respectively,
the domain and the range of $A$ and their closures. Additionally, $\NnN(A)$ denotes the kernel of $A$. The direct sum
of a collection of Hilbert spaces $\{\HHh_s\}_{s \in S}$ is denoted by $\bigoplus_{s \in S} \HHh_s$ and $\oplus_s x_s$ is
a member of $\bigoplus_{s \in S} \HHh_s$ corresponding to a family $\{x_s\}_{s \in S}$ of vectors such that $x_s \in
\HHh_s$ and $\sum_{s \in S} \|x_s\|^2 < \infty$. The same notation is used for direct sums of operators:
if $\{A_s\}_{s \in S}$ is a family of operators, $A = \bigoplus_{s \in S} A_s$ is an operator with
$$
\DdD(A) = \bigl\{\oplus_s x_s \in \bigoplus_{s \in S} \overline{\DdD}(A_s)\dd\ x_s \in \DdD(A_s)\ (s \in S),\
\sum_{s \in S} \|A_s x_s\|^2 < \infty\bigr\}
$$
and for $x = \oplus_s x_s \in \DdD(A)$, $A x = \oplus_s (A_s x_s) \in \bigoplus_{s \in S} \overline{\RrR}(A_s)$.\par
For two operators $A$ and $B$ acting in a common Hilbert space we write $A \subset B$ provided $\DdD(A) \subset \DdD(B)$
and $Bx = Ax$ for $x \in \DdD(A)$.\par
Let $A$ be a closed densely defined operator in $\HHh$. A closed linear subspace $E$ of $\HHh$ is said to be
\textit{reducing} for $A$ if $P_E A \subset A P_E$ where $P_E$ is the orthogonal projection onto $E$ and $\DdD(A P_E) =
P_E^{-1}(\DdD(A))$. The \textit{reduced part} of $A$ to $E$ is denoted by $A \bigr|_E$ and it is the restriction of $A$ to
$\DdD(A) \cap E$. The set of all reducing subspaces for $A$ is denoted by $\red(A)$. $E \in \red(A)$ is \textit{centrally}
reducing iff $P_E P_K = P_K P_E$ for any $K \in \red(A)$. The collection of all centrally reducing subspaces is denoted
by $\RED(A)$. The \textit{$*$-commutant} of $A$ is the set $\WWw'(A)$ consisting of all $T \in \BBb(\HHh)$ such that
$T A \subset A T$ and $T^* A \subset A T^*$; and $\WWw''(A) = (\WWw'(A))'$ is the \textit{$*$-bicommutant} of $A$. When
$A$ is bounded, we may also use $\WWw(A)$ to denote the smallest von Neumann algebra containing $A$, in that case $\WWw(A)
= \WWw''(A)$ (thanks to von Neumann's bicommutant theorem). The \textit{polar decomposition} of $A$ has the form
$A = Q |A|$ where $|A|$ is the square root of $A^* A$ (obtained e.g. by the functional calculus for unbounded selfadjoint
operators) and $Q$ is a partial isometry with $\NnN(Q) = \NnN(A)$. Whenever we use the notation `$Q_T$' with $T$ being
a closed densely defined operator, this denotes the partial isometry appearing in the polar decomposition of $T$.

\SECT{Preliminaries}

In the whole paper, $N$ is a fixed positive integer corresponding to the length of tuples of operators. Whenever $\HHh$ is
a Hilbert space, $\CDDc(\HHh)$ is the collection of all closed densely defined linear operators acting in $\HHh$ and
$\CDDc_N(\HHh) = [\CDDc(\HHh)]^N$. That is, $\CDDc_N(\HHh)$ consists of all $N$-tuples of members of $\CDDc(\HHh)$.
Further, we put
$$
\CDDc_N = \bigcup_{\HHh} \CDDc_N(\HHh)
$$
where $\HHh$ runs over all Hilbert spaces (including zero-dimensional). For simplicity, we shall write $\CDDc$ in place
of $\CDDc_1$. For every $\aAA = (A_1,\ldots,A_N) \in \CDDc_N$ there is a unique Hilbert space, denoted
by $\overline{\DdD}(\aAA)$, such that $\aAA \in \CDDc_N(\overline{\DdD}(\aAA))$. In particular, $\overline{\DdD}(\aAA) =
\overline{\DdD}(A_j)$ for $j=1,\ldots,N$.\par
Suppose $\aAA = (A_1,\ldots,A_N) \in \CDDc_N$. We define $\aAA^*$, $|\aAA|$ and $\qQQ_{\aAA}$ (as members of $\CDDc_N$)
in a coordinatewise manner: $\aAA^* = (A_1^*,\ldots,A_N^*)$, $|\aAA| = (|A_1|,\ldots,|A_N|)$ and $\qQQ_{\aAA} = (Q_{A_1},
\ldots,Q_{A_N})$. In the same way we may define other operations on $N$-tuples, if only they can be made on each of their
entries. For example, if each of $A_j$'s is one-to-one and has dense image, we may define $\aAA^{-1}$ as $(A_1^{-1},\ldots,
A_N^{-1})$.\par
Everywhere below in points (DF1)--(DF11), $\aAA = (A_1,\ldots,A_N)$, $\bBB = (B_1,\ldots,B_N)$ and $\aAA^{(s)} =
(A^{(s)}_1,\ldots,A^{(s)}_N)$ represent arbitrary members of $\CDDc_N$. For a single operator, some of notions stated below
are well-known and some of them were introduced in \cite{e} (with different notation). Probably the only new notion
is the \textit{strong} order `$\leqsl^s$' defined in (DF8) below.
\begin{enumerate}[(DF1)]
\item Let $\bigoplus_{s \in S} \aAA^{(s)} = (\bigoplus_{s \in S} A^{(s)}_1,\ldots,\bigoplus_{s \in S} A^{(s)}_N)$. For
   a positive cardinal $\alpha$ let $\alpha \odot \aAA = \bigoplus_{\xi < \xi_{\alpha}} \aAA^{(\xi)}$ where $\xi_{\alpha}$
   is the first ordinal of cardinality $\alpha$ and $\aAA^{(\xi)} = \aAA$ for any $\xi < \xi_{\alpha}$.
\item $\aAA$ is \textit{trivial} provided $\overline{\DdD}(\aAA)$ is zero-dimensional; otherwise $\aAA$ is
   \textit{nontrivial}.
\item $\aAA$ is \textit{bounded} iff each of $A_1,\ldots,A_N$ is a bounded operator; $\|\aAA\| :=
   \max(\|A_1\|,\ldots,\|A_N\|)$ provided $\aAA$ is bounded, otherwise $\|\aAA\| := \infty$. We say a bounded $\aAA$
   \textit{assumes its norm} provided there is $x \in \overline{\DdD}(\aAA)$ of norm $1$ with $\max(\|A_1 x\|,\ldots,
   \|A_N x\|) = \|\aAA\|$.
\item Let $\red(\aAA) = \bigcap_{j=1}^N \red(A_j)$ and for $E \in \red(\aAA)$,
   $$\aAA\bigr|_E = (A_1\bigr|_E,\ldots,A_N\bigr|_E);$$
   $\RED(\aAA)$ consists of all $E \in \red(\aAA)$ such that $P_E P_K = P_K P_E$ for every $K \in \red(\aAA)$.
\item The \textit{$*$-commutant} of $\aAA$ is the set $\WWw'(\aAA) = \bigcap_{j=1}^N \WWw'(A_j) \subset
   \BBb(\overline{\DdD}(\aAA))$ and $\WWw''(\aAA) = (\WWw'(\aAA))'$ is the \textit{$*$-bicommutant} of $\aAA$. When $\aAA$
   is bounded, we may also use $\WWw(\aAA)$ to denote the smallest von Neumann algebra including $\{A_1,\ldots,A_N\}$,
   in that case $\WWw(\aAA) = \WWw''(\aAA)$.
\item $\aAA \equiv \bBB$ (or, $\aAA$ and $\bBB$ are \textit{unitarily equivalent}) iff there is $U \in
   \UUu(\overline{\DdD}(\aAA),\overline{\DdD}(\bBB))$ such that $A_j = U^{-1} B_j U$ for $j = 1,\ldots,N$.
\item $\aAA \leqsl \bBB$ iff $\aAA \equiv \bBB\bigr|_E$ for some $E \in \red(\bBB)$.
\item $\aAA \leqsl^s \bBB$ iff $\aAA \equiv \bBB\bigr|_E$ for some $E \in \RED(\bBB)$.
\item $\aAA$ and $\bBB$ are \textit{unitarily disjoint}, in symbol $\aAA \disj \bBB$, if there is no nontrivial $N$-tuple
   $\xXX \in \CDDc_N$ with $\xXX \leqsl \aAA$ and $\xXX \leqsl \bBB$.
\item $\aAA$ is \textit{covered} by $\bBB$, in symbol $\aAA \ll \bBB$, if $\aAA \leqsl \alpha \odot \bBB$ for some cardinal
   $\alpha$.
\item The symbols `$\sqplus$' and `$\bigsqplus$' shall often be used instead of `$\oplus$' and `$\bigoplus$' in situations
   when all summands are mutually unitarily disjoint. So, whenever in the sequel notation $\aAA \sqplus \bBB$
   or $\bigsqplus_{s \in S} \aAA^{(s)}$ appears, this will always imply that $\aAA \disj \bBB$ or, respectively,
   $\aAA^{(s')} \disj \aAA^{(s'')}$ for any distinct indices $s',s'' \in S$. The direct sum (a collection) is called
   \textit{regular} provided all its summands (elements) are mutually unitarily disjoint.
\end{enumerate}
The reader should notice that a function $\red(\aAA) \ni E \mapsto P_E \in \WWw'(\aAA)$ establishes a one-to-one
correspondence between $\red(\aAA)$ and the set $E(\WWw'(\aAA))$ of all orthogonal projections belonging to $\WWw'(\aAA)$.
What is more, this map sends $\RED(\aAA)$ onto $E(\WWw'(\aAA)) \cap \ZZz(\WWw'(\aAA))$ where $\ZZz(\WWw'(\aAA))$ is
the center of $\WWw'(\aAA)$.\par
It is quite easy to check that `$\equiv$' is an equivalence on $\CDDc_N$ and thus for each $\aAA \in \CDDc_N$ we may
consider the equivalence class of $\aAA$ with respect to `$\equiv$', which we shall denote by $\AAA$. Let $\CDD_N$ be
the class of (all) equivalence classes of all members of $\CDDc_N$ and let $\CDD = \CDD_1$. Elements of $\CDD_N$ will be
denoted by $\AAA,\BBB,\XXX,\YYY$ and so on and their corresponding representatives by $\aAA,\bBB,\xXX,\yYY$. The symbol
$\zero$ is reserved to denote the equivalence class of a trivial element of $\CDDc_N$. $\zero$ is the unique member
of $\CDD_N$ whose representatives act on zero-dimensional Hilbert spaces. (It is also the neutral element for the action
`$\oplus$'.) For every $\AAA \in \CDD_N$, the following are well defined, in an obvious manner: $\AAA^*$, $|\AAA|$,
$\QQQ_{\AAA}$ (the latter corresponds to $\qQQ_{\aAA}$) and $\dim(\AAA) = \dim \overline{\DdD}(\aAA)$. For simplicity,
we shall use the term `$N$-tuple' for members of $\CDDc_N$ as well as for members of $\CDD_N$.\par
Some of notions in (DF1)--(DF11) may be adapted for members of $\CDD_N$ as follows.
\begin{enumerate}[(UE1)]
\item Let $\bigoplus_{s \in S} \AAA^{(s)} = \XXX$ where $\xXX = \bigoplus_{s \in S} \aAA^{(s)}$. For any cardinal
   $\mM > 0$, put $\mM \odot \AAA = \YYY$ where $\yYY = \mM \odot \aAA$. Additionally, let $0 \odot \AAA = \OOO$.
\item $\AAA$ is bounded, nontrivial, trivial iff so is $\aAA$. $\|\AAA\| = \|\aAA\|$; $\AAA$ assumes its norm iff so does
   $\aAA$.
\item $\AAA \leqsl \BBB$, $\AAA \leqsl^s \BBB$, $\AAA \disj \BBB$, $\AAA \ll \BBB$ iff corresponding relation holds true
   for $\aAA$ and $\bBB$. Note that $\AAA \leqsl^s \BBB \implies \AAA \leqsl \BBB \implies \AAA \ll \BBB$.
\item\label{UE4} Notation $\AAA \sqplus \BBB$ or $\bigsqplus_{s \in S} \AAA^{(s)}$ includes information that $\AAA \disj
   \BBB$ or, respectively, $\AAA^{(s')} \disj \AAA^{(s'')}$ for any distinct indices $s',s'' \in S$. The direct sum
   (a family) of members of $\CDD_N$ is \textit{regular} iff all its summands (elements) are pairwise unitarily disjoint.
\end{enumerate}

A starting point for all of our investigations is the following classical result (see e.g. \cite[Theorem~1.3]{e}).

\begin{pro}{order}
`$\leqsl$' and `$\leqsl^s$' are partial orders on $\CDD_N$. More precisely, if $\AAA \leqsl \BBB$ and $\BBB \leqsl \AAA$,
then $\AAA = \BBB$.
\end{pro}

\SECT{The $\bB$-transform}

This part is mainly devoted to single operators.\par
We fix a Hilbert space $\HHh$ and an operator $T \in \CDDc(\HHh)$. Let $I$ be the identity operator on $\HHh$.

\begin{dfn}{b}
The \textit{$\bB$-transform} of $T$ is the operator $$\bB(T) = T (I + |T|)^{-1} \in \BBb(\HHh).$$
\end{dfn}

The reader should verify with no difficulties that

\begin{pro}{b}
Let $S = \bB(T)$.
\begin{enumerate}[\upshape(A)]
\item $\bB(|T|) = |S| = |T| (I + |T|)^{-1}$ and $Q_T = Q_S$.
\item $\|S x\| < \|x\|$ for each $x \in \HHh \setminus \{0\}$.
\item $T = S (I - |S|)^{-1} =: \uU\bB(S)$.
\item $\WWw'(T) = \WWw'(S)$. Consequently, $\red(T) = \red(S)$ and $\RED(T) = \RED(S)$. For every $E \in \red(T)$,
   $\bB(T\bigr|_E) = S\bigr|_E$.
\item The $\bB$-transform establishes a one-to-one correspondence between members of $\CDDc(\HHh)$ and operators
   $S \in \BBb(\HHh)$ satisfying \textup{(B)}.
\item $\bB(\bigoplus_{s \in S} T_s) = \bigoplus_{s \in S} \bB(T_s)$ for arbitrary family $\{T_s\}_{s \in S} \subset \CDDc$.
\end{enumerate}
\end{pro}

The following result is a little bit surprising.

\begin{thm}{b*}
For every $T \in \CDDc$, $\bB(T^*) = [\bB(T)]^*$.
\end{thm}
\begin{proof}
Let $T = Q |T|$ be the polar decomposition of $T$. Then $T^* = Q^* |T^*|$ is the polar decomposition of $T^*$. Put $\HHh =
\overline{\DdD}(T)$, $S = \bB(T)$ and $S' = \bB(T^*)$. Fix $x, y \in \HHh$, put $u = (I + |T|)^{-1}x \in \DdD(T)$ and $v =
(I + |T^*|)^{-1}y \in \DdD(T^*)$ and observe that
\begin{multline*}
\scalar{Sx}{y} = \scalar{Tu}{(I + |T^*|)v} = \scalar{Tu}{v} + \scalar{Q|T|u}{|T^*|v} =\\
= \scalar{u}{T^*v} + \scalar{|T|u}{T^*v} = \scalar{(I + |T|)u}{T^*v} = \scalar{x}{S'y}
\end{multline*}
which finishes the proof.
\end{proof}

Involving $\bB$-transform we now easily prove

\begin{thm}{common}
Let $\HHh$ be a nonseparable Hilbert space and $\{T_s\}_{s \in S} \subset \CDDc(\HHh)$ be a \textbf{countable} family
of operators. For every nonzero $x \in \HHh$ there is a separable space $E \subset \HHh$ containing $x$ such that
$E \in \red(T_s)$ for each $s \in S$.
\end{thm}
\begin{proof}
Thanks to point (D) of \PRO{b}, we may assume each of $T_s$'s is bounded (because we may replace $T_s$ by $\bB(T_s)$).
Now it suffices to put $E = \overline{\lin}\{S_1 \cdot \ldots \cdot S_n x|\ n \geqsl 1,\ S_1,\ldots,S_n \in \{T_s\dd\
s \in S\} \cup \{T_s^*\dd\ s \in S\} \cup \{I\}\}$ where $I$ is the identity operator on $\HHh$.
\end{proof}

Now for $\aAA = (A_1,\ldots,A_N) \in \CDDc_N$ put $\bB(\aAA) = (\bB(A_1),\ldots,\bB(A_N))$ and $\bB(\AAA) = \XXX$ where
$\xXX = \bB(\aAA)$. Below we list most important properties of the $\bB$-transform on $\CDD_N$ and $\CDDc_N$.
\begin{enumerate}[(BT1)]
\item $\bB(\AAA) = \zero \iff \AAA = \zero$.
\item $\bB(\AAA)$ is bounded, $\bB(\AAA^*) = [\bB(\AAA)]^*$, $|\bB(\AAA)| = \bB(|\AAA|)$ and $\QQQ_{\bB(\AAA)} =
   \QQQ_{\AAA}$.
\item\label{BT3} $\WWw'(\aAA) = \WWw'(\bB(\aAA))$, $\WWw''(\aAA) = \WWw(\bB(\aAA))$; $\red(\aAA) = \red(\bB(\aAA))$
   and $\RED(\aAA) = \RED(\bB(\aAA))$; for every $E \in \red(\aAA)$, $\bB(\aAA\bigr|_E) = \bB(\aAA)\bigr|_E$.
\item $\bB(\bigoplus_{s \in S} \AAA^{(s)}) = \bigoplus_{s \in S} \bB(\AAA^{(s)})$.
\item\label{BT5} If `$\sim$' denotes one of the relations $=$, $\leqsl$, $\leqsl^s$, $\ll$, $\disj$, then $\AAA \sim \BBB
   \iff \bB(\AAA) \sim \bB(\BBB)$.
\end{enumerate}

\SECT{Background of von Neumann algebras}

Let $\MMm$ be a von Neumann subalgebra of $\BBb(\HHh)$. Denote by $E(\MMm)$ the set of all orthogonal projections in $\MMm$
and by $\ZZz(\MMm)$ the center of $\MMm$. By `$\sim$' we shall denote the Murray-von Neumann equivalence on $E(\MMm)$.
Further, put $\EEe(\MMm) = E(\MMm) / \sim$ and let `$\preccurlyeq$' denote the Murray-von Neumann order on $\EEe(\MMm)$.
Finally, for each $p \in E(\MMm)$, $c_p \in E(\ZZz(\MMm))$ stands for the central support of $p$.\par
It was observed by several mathematicians that the order `$\leqsl$' on $\CDD$ translates into the Murray-von Neumann
order between (equivalence classes of) projections in a suitable von Neumann algebra. This was explicitly stated and proved
in \cite[Proposition~1.35]{e}. It is nothing new that the same idea works for tuples of operators. We formulate this
precisely in the next result which is the main tool of the paper.

\begin{pro}{transl}
Let $\tTT \in \CDDc_N(\HHh)$, $E, F \in \red(\tTT)$, $\aAA = \tTT\bigr|_E$ and $\bBB = \tTT\bigr|_F$. Further, let $\MMm =
\WWw'(\tTT)$, $p = P_E$ and $q = P_F$ ($p, q \in E(\MMm)$). Then
\begin{enumerate}[\upshape(a)]
\item $\aAA \equiv \bBB \iff p \sim q$,
\item $\aAA \leqsl \bBB \iff p \preccurlyeq q$,
\item $\aAA \leqsl^s \bBB \iff p \sim c_p q$,
\item $\aAA \disj \bBB \iff c_p c_q = 0$,
\item $\aAA \ll \bBB \iff p \leqsl c_q$.
\end{enumerate}
\end{pro}
\begin{proof}
We shall only prove (c), since the other points are covered by \cite[Proposition~1.35]{e} ((d) is stated there in other
form, its recent form may be deduced e.g. from \cite[Lemma~1.7]{tk1}). For this purpose put $\MMm_0 = q \MMm q$, $z_0 =
c_p q \in E(\ZZz(\MMm_0))$ and let $K \in \RED(\bBB)$ be the range of $z_0$. If $z_0 \sim p$, then by (a), $\aAA \equiv
\bBB\bigr|_K$ and thus $\aAA \leqsl^s \bBB$. Conversely, if the latter inequality is fulfilled, there is $z_0 \in
E(\ZZz(\MMm_0))$ such that $p \sim z_0$ (again by (a)). But $\ZZz(\MMm_0) = \ZZz(\MMm) q$ and hence $z_0 = z q$ for some
$z \in E(\ZZz(\MMm))$. To this end, note that $c_p = c_{z q}$ (since $p \sim z q$) and $c_{z q} = z c_q$ and therefore
$z q = z c_q q = c_p q$.
\end{proof}

Some of consequences of \PRO{transl} are formulated below (these are adaptations of suitable results of \cite{e}).
\begin{enumerate}[(PR1)]
\item\label{PR1} $\AAA \sim \XXX \oplus \YYY$ and $\AAA \disj \YYY$ imply $\AAA \sim \XXX$ when `$\sim$' is replaced by one
   of $\leqsl,\leqsl^s,\ll$.
\item\label{PR2} If $\AAA^{(s)} \disj \BBB^{(t)}$ for all $s \in S$ and $t \in T$, then $\bigoplus_{s \in S} \AAA^{(s)}
   \disj \bigoplus_{t \in T} \BBB^{(t)}$.
\item\label{PR3} A function $\RED(\aAA) \ni E \mapsto \XXX(E) \in \{\BBB \in \CDD_N\dd\ \BBB \leqsl^s \AAA\}$ where
   $\xXX(E) = \aAA\bigr|_E$ is a (well defined) bijection.
\item For every $E \in \red(\aAA)$, $\aAA\bigr|_E \disj \aAA\bigr|_{E^{\perp}} \iff E \in \RED(\aAA)$.
\item For every pair $(\AAA,\BBB)$ such that $\AAA \leqsl^s \BBB$ there is a unique $\XXX \in \CDD_N$ such that $\BBB =
   \AAA \sqplus \XXX$. \underline{Notation}: $\BBB \sqminus \AAA := \XXX$. (So, $\BBB \sqminus \AAA$ makes sense iff $\AAA
   \leqsl^s \BBB$.)
\item\label{PR6} For every $\XXX \in \CDD_N$ and a cardinal $\alpha$, $\{\YYY \in \CDD_N\dd\ \YYY \leqsl^s \alpha \odot
   \XXX\} = \{\alpha \odot \YYY\dd\ \YYY \leqsl^s \XXX\}$.
\end{enumerate}
Following (PR5), let us agree with the following convention: whenever for a pair $(\AAA,\BBB)$ there is a unique $\XXX$
for which $\BBB = \AAA \oplus \XXX$, we shall denote this unique $\XXX$ by $\BBB \ominus \AAA$. Observe that $\AAA \leqsl
\BBB$ provided $\BBB \ominus \AAA$ makes sense.

Combining \PRO{transl} with Sherman's theorem \cite{sh}, we obtain an interesting

\begin{thm}{complete}
$(\CDD_N,\leqsl)$ is an order-complete lattice. Precisely, for every nonempty family (i.e. a set) $\{\AAA^{(s)}\}_{s \in S}
\subset \CDD_N$ there are members $\XXX$ and $\YYY$ of $\CDD_N$ such that $\XXX \leqsl \AAA^{(s)} \leqsl \YYY$ for each
$s \in S$ and $\XXX' \leqsl \XXX$ (respectively $\YYY \leqsl \YYY'$) whenever $\XXX' \leqsl \AAA^{(s)}$ (respectively
$\AAA^{(s)} \leqsl \YYY'$) for all $s \in S$.
\end{thm}
\begin{proof}
Put $\AAA = \bigoplus_{s \in S} \AAA^{(s)}$ and $\MMm = \WWw'(\aAA)$. By \cite{sh}, $(\EEe(\MMm),\preccurlyeq)$ is
an order-complete lattice. So, using \PRO{transl} we see that there are $\XXX$ and $\YYY$ (both less than or equal
to $\AAA$) which correspond to the g.l.b. and l.u.b. (with respect to `$\preccurlyeq$') of the projections corresponding
to $\AAA^{(s)}$'s. Now if $\XXX'$ and $\YYY'$ are as in the statement of the theorem, consider $\widetilde{\AAA} = \AAA
\oplus \XXX' \oplus \YYY'$ and $\widetilde{\MMm} = \WWw'(\widetilde{\aAA})$ and repeat the above argument to get
the assertion. We skip the details.
\end{proof}

As it is usually done when working with lattices, for every nonempty collection $\AAa = \{\AAA^{(s)}\}_{s \in S}$ we shall
denote by $\bigvee_{s \in S} \AAA^{(s)}$ and $\bigwedge_{s \in S} \AAA^{(s)}$ the l.u.b. and g.l.b of $\AAa$. Observe that
$\AAA \disj \BBB$ iff $\AAA \wedge \BBB = \zero$.

\SECT{Decompositions relative to ideals}

Let $\AAa$ be a subclass of $\CDDc_N$. We call $\AAa$ an \textit{ideal} iff $\AAa$ satisfies the following four conditions:
\begin{enumerate}[({I}D1)]\label{ID}
\item $\AAa$ is nonempty,
\item whenever $\aAA \in \AAa$ and $\aAA \equiv \bBB \in \CDDc_N$, then $\bBB \in \AAa$,
\item for every $\aAA \in \AAa$ and $E \in \red(\aAA)$, $\aAA\bigr|_E \in \AAa$,
\item $\bigoplus_{s \in S} \aAA_s \in \AAa$ for any nonempty family $\{\aAA_s\}_{s \in S} \subset \AAa$.
\end{enumerate}
Classical examples of ideals the reader may find in \EXS{id} below.\par
For every subclass $\FFf$ of $\CDDc_N$ put
$$
\FFf^{\perp} = \{\tTT \in \CDDc_N\dd\ \tTT \disj \fFF \textup{ for every } \fFF \in \FFf\}.
$$
It is easily seen that $\FFf^{\perp}$ is an ideal for any $\FFf \subset \CDDc_N$ (thanks to (PR2)). As we will see
later, the `converse' is also true, that is, $\AAa$ is an ideal iff $\AAa = (\AAa^{\perp})^{\perp}$. This reminds anologous
characterization of closed linear subspaces of Hilbert spaces. However, the just defined `orthogonal complement' is more
familiar to the orthogonality in spaces of measures than to that in Hilbert spaces.\par
One of main results of the paper is the following

\begin{thm}{main}
Let $\AAa \subset \CDDc_N$ be an ideal. For every $\tTT \in \CDDc_N$ there is a unique $E \in \red(\tTT)$ such that
\begin{equation}\label{eqn:d1}
\tTT\bigr|_E \in \AAa \qquad \textup{and} \qquad \tTT\bigr|_{E^{\perp}} \in \AAa^{\perp}.
\end{equation}
Moreover, $E \in \RED(\tTT)$ and
\begin{equation}\label{eqn:1}
E = \bigvee \{K \in \red(\tTT)\dd\ \tTT\bigr|_K \in \AAa\}.
\end{equation}
\end{thm}
\begin{proof}
First we shall show the existence of $E$. We may assume that $\tTT \notin \AAa^{\perp}$. By Zorn's lemma, there is
a maximal family $\{E_s\}_{s \in S}$ of mutually orthogonal nontrivial reducing (for $\tTT$) subspaces with
$\tTT\bigr|_{E_s} \in \AAa$ for every $s \in S$. It is clear that \eqref{eqn:d1} is satisfied with $E = \bigvee_{s \in S}
E_s$.\par
Now assume that $E \in \red(\tTT)$ is as in \eqref{eqn:d1}. By (PR4), $E \in \RED(\tTT)$. To establish uniqueness and
finish the proof it is enough to check \eqref{eqn:1}. But this simply follows from (PR1) and \PRO{transl}. (Indeed,
if $K \in \red(\tTT)$ is such that $\tTT\bigr|_K \in \AAa$, then $\tTT\bigr|_K \leqsl \tTT\bigr|_E \oplus
\tTT\bigr|_{E^{\perp}}$ and $\tTT\bigr|_K \disj \tTT\bigr|_{E^{\perp}}$. So, we conclude from (PR1) that $\tTT\bigr|_K
\leqsl \tTT\bigr|_E$. Thus, by \PRO{transl}, $P_K \preccurlyeq P_E$ in $\MMm = \WWw'(\tTT)$. But $P_E \in \ZZz(\MMm)$
and hence $P_K \leqsl P_E$ which means that $K \subset E$.)
\end{proof}

For simplicity, let us introduce the following notation. For every ideal $\AAa \subset \CDDc_N$, $\AAa^{(0)} = \AAa$ and
$\AAa^{(1)} = \AAa^{\perp}$. Under such a notation, by a simple induction argument we obtain

\begin{thm}{1}
Let $\AAa_1,\ldots,\AAa_k \subset \CDDc_N$ be ideals. For every $\tTT \in \CDDc_N(\HHh)$ there is a unique system
$\{E_{\delta}\}_{\delta \in \{0,1\}^k}$ of reducing subspaces for $\tTT$ such that
\begin{enumerate}[\upshape(i)]
\item $E_{\delta} \perp E_{\delta'}$ for distinct $\delta,\delta' \in \{0,1\}^k$; and $\HHh =
   \bigoplus_{\delta \in \{0,1\}^k} E_{\delta}$,
\item $\tTT\bigr|_{E_{\delta}} \in \bigcap_{j=1}^k \AAa_j^{(\delta_j)}$ for each $\delta \in \{0,1\}^k$.
\end{enumerate}
Moreover, $E_{\delta} \in \RED(\tTT)$ for every $\delta \in \{0,1\}^k$.
\end{thm}

We leave the proof of \THM{1} for the reader.\par
\THM{1} covers any known result on decomposition of a single operator into two parts with first of them of a special class
and the other `completely' (or `hereditarily') not of this class. Examples to this are given below.

\begin{exs}{id}
\begin{enumerate}[\upshape(A)]
\item Let $F$ be a closed subset of the complex plane $\CCC$. Let $\NNn(F)$ be the class of all normal operators whose
   spectrum is contained in $F$. (Here we assume that operators on zero-dimensional Hilbert spaces are normal and have
   empty spectra.) It is easily checked that $\NNn(F)$ is an ideal. Thus, every operator $T \in \CDDc$ admits a unique
   decomposition into a part in $\NNn(F)$ and the remainder in $\NNn(F)^{\perp}$. This means that there is a unique $E \in
   \red(T)$ such that $T\bigr|_E$ is normal, $\sigma(T\bigr|_E) \subset F$ and $T\bigr|_{E^{\perp}}$ has no nontrivial
   reduced part which belongs to $\NNn(F)$. When $F = \CCC$, this is the decomposition into the normal part
   and the completely non-normal one. When $F = \RRR$, we get the decomposition into the selfadjoint part
   and the completely non-selfadjoint one. Finally, when $F = \{z \in \CCC\dd\ |z| = 1\}$, the operator decomposes into
   the unitary part and the completely non-unitary one. These three cases are most classical. (Compare with \cite{e},
   page 179.)
\item Single operators of each of the following classes form an ideal: formally normal (for definition see e.g.
   \cite{s-sz}); quasinormal; hyponormal; subnormal; contractions. As we will see in \PRO{id}, also the following class
   $\AAa$ is an ideal: $T \in \AAa$ iff $T$ is the direct sum of bounded operators.
\item Stochel and Szafraniec \cite{s-sz} have shown that every operator $T \in \CDDc$ admits a unique decomposition
   of the form $T = T_{nor} \oplus T_{pfn} \oplus T_{cnfn}$ where $T_{nor}$ is normal, $T_{pfn}$ is purely formally normal
   (here `purely' means that $T_{pfn}$ is in addition completely non-normal) and $T_{cnfn}$ is completely non-formally
   normal. Their result is a special case of \THM{1}.
\item Ernest \cite{e} distinguishes an important class of bounded operators on separable Hilbert spaces, the so-called
   \textit{smooth} operators (see \S6 of Chapter~1 in \cite{e}). Let us say that an operator $T \in \CDDc(\HHh)$ where
   $\HHh$ is separable is \textit{$\sigma$-smooth} iff $\bB(T)$ is the direct sum of countably (finitely or infinitely)
   many smooth operators. By Proposition~1.52 of \cite{e} and \PRO{id} below, operators which are direct sums
   of $\sigma$-smooth operators form an ideal. In particular, every closed densely defined operator acting on a separable
   Hilbert space admits a unique decomposition into a $\sigma$-smooth operator and a completely non-smooth one.
\item\label{E} Let us give some examples dealing with systems of operators. Let $\NNn_N$ and $\widetilde{\NNn}_N$ consist
   of all $N$-tuples (belonging to $\CDDc_N$) of, respectively, commuting normal and arbitrary normal operators
   (commutativity may be defined by means of the spectral measures or, equivalently, of the $\bB$-transforms). It is clear
   that both $\NNn_N$ and $\widetilde{\NNn}_N$ are ideals. So, every $\tTT \in \CDDc_N$ has a unique decomposition
   in the form $\tTT = \tTT_{jn} \oplus \tTT_{psn} \oplus \tTT_{cnsn}$ where $\tTT_{jn} \in \NNn_N$, $\tTT_{psn} \in
   \widetilde{\NNn}_N$ and no nontrivial reduced part of $\tTT_{psn}$ is a member of $\NNn_N$, and no nontrivial reduced
   part of $\tTT_{cnsn}$ belongs to $\widetilde{\NNn}_N$. (The labels `jn', `psn' and `cnsn' appearing here are
   the abbreviations for \textit{jointly normal}, \textit{purely separately normal} and \textit{completely non-separately
   normal}.) We call an $N$-tuple $\AAA$ \textit{normal} iff $\aAA \in \NNn_N$.
\item If $\AAa \subset \CDDc$ is an ideal, so are $\Delta_N(\AAa) \subset \CDDc_N$ and $\AAa^{[N]} \subset \CDDc_N$ where
   $\AAa^{[N]}$ consists of all $N$-tuples $(A_1,\ldots,A_N)$ with $A_1,\ldots,A_N \in \AAa$ acting in a common Hilbert
   space, and $$\Delta_N(\AAa) = \{(A_1,\ldots,A_N)\dd\ A_1 = \ldots = A_N \in \AAa\}.$$
\item \THM{main} may be shortly reformulated in the following significant way: $\CDDc_N = \AAa \oplus \AAa^{\perp}$ for
   every ideal $\AAa \subset \CDDc_N$. Using this notation, \THM{1} with $k = 2$ asserts that
   \begin{equation}\label{eqn:aux2}
   \CDDc_N = (\AAa \cap \BBb) \oplus (\AAa \cap \BBb^{\perp}) \oplus (\AAa^{\perp} \cap \BBb) \oplus
   (\AAa^{\perp} \cap \BBb^{\perp})
   \end{equation}
   for any two ideals $\AAa$ and $\BBb$ in $\CDDc_N$. The counterpart of \eqref{eqn:aux2} for linear subspaces $K$ and $L$
   of a Hilbert space $\HHh$ is fulfilled only when $P_K$ and $P_L$ commute. Thus, as we have said earlier, the `orthogonal
   complement' for ideals behaves in a similar manner as the orthogonal complement of lattices of measures (or of more
   general structures such as abstract $L$-spaces).
\end{enumerate}
\end{exs}

The next result is useful for producing ideals.

\begin{pro}{id}
Let $\AAa$ be a subclass of $\CDDc_N$ and $\Theta_N$ be the class of all trivial members of $\CDDc_N$.
\begin{enumerate}[\upshape(a)]
\item The class
   \begin{multline*}
   J(\AAa) = \{\tTT \in \CDDc_N\dd\ \textup{for some set } S,\ \tTT = \bigoplus_{s \in S} \xXX^{(s)}\\
   \textup{with } \xXX^{(s)} \leqsl \yYY^{(s)} \in \AAa \cup \Theta_N\}
   \end{multline*}
   is an ideal and it is the smallest ideal which contains $\AAa$.
\item $\AAa$ is an ideal iff $\AAa = (\AAa^{\perp})^{\perp}$.
\end{enumerate}
\end{pro}
\begin{proof}
To show (a), we only need to check that $\aAA \in J(\AAa)$ provided $\aAA \leqsl \bigoplus_{s \in S} \yYY^{(s)}$ with
$\yYY^{(s)} \in \AAa$. Assuming $\aAA$ is nontrivial, take a maximal family $\EeE = \{E_{\gamma}\}_{\gamma \in \Gamma}$
of mutually orthogonal nontrivial reducing subspaces for $\aAA$ such that $\aAA\bigr|_{E_{\gamma}} \leqsl \xXX^{(\gamma)}$
for some $\xXX^{(\gamma)} \in \AAa$ ($\gamma \in \Gamma$). Let $F$ be the orthogonal complement
of $\bigoplus_{\gamma \in \Gamma} E_{\gamma}$ (in $\overline{\DdD}(\aAA)$). We only need to check that $F$ is trivial.
We infer from the maximality of $\EeE$ that $\aAA\bigr|_F \in \AAa^{\perp}$. Thus, thanks to (PR2), $\aAA\bigr|_F \disj
\bigoplus_{s \in S} \yYY^{(s)}$ and hence, by (PR1), $\aAA\bigr|_F$ is trivial and we are done.\par
The `if' part of (b) is immediate, while the `only if' one follows from \THM{main}.
\end{proof}

\begin{rem}{unity}
In \PROp{unitid} we shall show that for every ideal $\AAa$ there is a (unique up to unitary equivalence under some
additional properties of $\aAA$) $N$-tuple $\aAA$ such that $\AAa = \{\bBB\dd\ \bBB \ll \aAA\}$. Thus, our \THM{main} is
a generalization of Ernest's Proposition~2.12 in \cite{e}.
\end{rem}

The rest of the paper is devoted to the class $\CDD_N$.

\SECT{Order `$\leqsl^s$'}

Everywhere below the prefix `$\leqsl^s$' says that the suitable term is understood with respect to this order.
The aim of this part is to prove

\begin{thm}{ords}
Let $\Bb$ be a nonempty set of members of $\CDD_N$ and let $\AAA, \BBB \in \CDD_N$.
\begin{enumerate}[\upshape(A)]
\item $\Bb$ has the $\leqsl^s$-g.l.b.
\item $\Bb$ has the $\leqsl^s$-l.u.b. \iaoi{} every two-point subset of $\Bb$ is $\leqsl^s$-upper bounded. If the latter
   happens, $\inf{}_{\leqsl^s} \Bb = \bigwedge \Bb$ and $\sup{}_{\leqsl^s} \Bb = \bigvee \Bb$.
\item \TFCAE
   \begin{enumerate}[\upshape(i)]
   \item the set $\{\AAA,\BBB\}$ is $\leqsl^s$-upper bounded,
   \item $\AAA \leqsl^s \AAA \vee \BBB$ and $\BBB \leqsl^s \AAA \vee \BBB$,
   \item $\AAA$ and $\BBB$ may be written in the forms $\AAA = \EEE \sqplus \XXX$ and $\BBB = \EEE \sqplus \YYY$ for some
      $\EEE, \XXX, \YYY \in \CDD_N$ such that $\XXX \disj \YYY$.
   \end{enumerate}
\item If $\{\AAA,\BBB\}$ is $\leqsl^s$-upper bounded, then $\AAA \leqsl \BBB \iff \AAA \leqsl^s \BBB$.
\end{enumerate}
\end{thm}
\begin{proof}
We begin with (C). Implications (iii)$\implies$(ii)$\implies$(i) are immediate (indeed, if (iii) is fulfilled, $\AAA \vee
\BBB = \EEE \sqplus \XXX \sqplus \YYY$). To see that (iii) follows from (i), let $\FFF \in \CDD_N$ $\leqsl^s$-majorizes
$\AAA$ and $\BBB$. This means that $\aAA \equiv \fFF\bigr|_K$ and $\bBB \equiv \fFF\bigr|_L$ for some $K, L \in
\RED(\fFF)$. Then $P_K$ and $P_L$ commute and therefore $K = M \oplus K'$ and $L = M \oplus L'$ where $M = K \cap L$, $K' =
M^{\perp} \cap K$ and $L' = M^{\perp} \cap L$. Note that then $\eEE = \fFF\bigr|_M$, $\xXX = \fFF\bigr|_{K'}$ and $\yYY =
\fFF\bigr|_{L'}$ are pairwise unitarily disjoint and $\AAA = \EEE \sqplus \XXX$ and $\BBB = \EEE \sqplus \YYY$.\par
Now we pass to (B). Suppose every two-point subset of $\Bb$ is $\leqsl^s$-upper bounded. Let $\MMM$ be such that $\BBB
\leqsl \MMM$ for every $\BBB \in \Bb$. Put $\MMm = \WWw'(\mMM)$. For every $\BBB \in \Bb$ take $K(\BBB) \in \red(\mMM)$
such that $\bBB \equiv \mMM\bigr|_{K(\BBB)}$ and put $p_{\BBB} = P_{K(\BBB)} \in \MMm$.\par
For a moment fix $\AAA, \BBB \in \Bb$. By (C), there is $\FFF \leqsl \MMM$ such that $\AAA \leqsl^s \FFF$ and $\BBB
\leqsl^s \FFF$. We infer from this, involving \PRO{transl}, that there is a projection $q \in E(\MMm)$ for which $p_{\AAA}
\sim c_{p_{\AAA}} q$ and $p_{\BBB} \sim c_{p_{\BBB}} q$. Notice that then $c_{p_{\BBB}} p_{\AAA} \sim c_{p_{\BBB}}
c_{p_{\AAA}} q$ and $c_{p_{\AAA}} p_{\BBB} \sim c_{p_{\AAA}} c_{p_{\BBB}} q$. This proves that
\begin{equation}\label{eqn:aux10}
c_{p_{\BBB}} p_{\AAA} \sim c_{p_{\AAA}} p_{\BBB}
\end{equation}
for all $\AAA, \BBB \in \Bb$. Now put $w = \bigvee \{c_{p_{\AAA}}\dd\ \AAA \in \Bb\} \in \ZZz(\MMm)$. There is a family
$\{z_{\AAA}\}_{\AAA \in \Bb}$ of mutually orthogonal central projections in $\MMm$ such that $z_{\AAA} \leqsl c_{p_{\AAA}}$
for every $\AAA \in \Bb$ and $\sum_{\AAA \in \Bb} z_{\AAA} = w$. Put
$$
q = \sum_{\AAA \in \Bb} z_{\AAA} p_{\AAA} \in E(\MMm).
$$
For $\AAA, \BBB \in \Bb$ we have, by \eqref{eqn:aux10}, $z_{\BBB} c_{p_{\AAA}} q = z_{\BBB} c_{p_{\AAA}} p_{\BBB} \sim
z_{\BBB} c_{p_{\BBB}} p_{\AAA} = z_{\BBB} p_{\AAA}$ and consequently (since $w \geqsl c_{p_{\AAA}}$),
$$
p_{\AAA} = \sum_{\BBB \in \Bb} z_{\BBB} p_{\AAA} \sim \sum_{\BBB \in \Bb} z_{\BBB} c_{p_{\AAA}} q = c_{p_{\AAA}} q.
$$
Now if $E \in \red(\mMM)$ is the range of $q$ and $\mMM' = \mMM\bigr|_E$, \PRO{transl} shows that $\AAA \leqsl^s \MMM'$
for every $\AAA \in \Bb$. Hence, replacing $\MMM$ by $\MMM'$, we may assume that $p_{\AAA} \in \ZZz(\MMm)$. It is known
that in such a case $\bigvee_{\AAA \in \Bb} p_{\AAA}$ and $\bigwedge_{\AAA \in \Bb} p_{\AAA}$ are, respectively, the l.u.b.
and the g.l.b. with respect to `$\preccurlyeq$' in $E(\MMm)$. It is left as an exercise that the assertion of (B) now
follows.\par
Finally, (A) is implied by (B), and (D) is left for the reader.
\end{proof}

As a very special case of \THM{ords} we get

\begin{cor}{disjoint}
If $\{\AAA^{(s)}\}_{s \in S}$ is a nonempty family of mutually unitarily disjoint $N$-tuples, then $\bigvee_{s \in S}
\AAA^{(s)} = \bigsqplus_{s \in S} \AAA^{(s)}$.
\end{cor}
\begin{proof}
One easily checks that $\bigsqplus_{s \in S} \AAA^{(s)}$ is the $\leqsl^s$-l.u.b. of $\{\AAA^{(s)}\}_{s \in S}$. Thus
the assertion follows from \THM{ords}.
\end{proof}

\begin{exm}{ord-ords}
[$N=1$] Let $I_j$ for $j=1,2$ be the identity operator on a $j$-dimensional Hilbert space. It is clear that $\III_1 \leqsl
\III_2$, $\III_1 \wedge \III_2 = \III_1$ and $\III_1 \vee \III_2 = \III_2$, while $\inf{}_{\leqsl^s} \{\III_1,\III_2\} =
\zero$ and $\{\III_1,\III_2\}$ is not $\leqsl^s$-upper bounded. This shows that the $\leqsl^s$-g.l.b. in general differs
from the $\leqsl$-g.l.b. (although both of them always exist).
\end{exm}

We end the section with a useful

\begin{pro}{leqsl-leqsls}
\begin{enumerate}[\upshape(A)]
\item If $\AAA \leqsl \bigsqplus_{s \in S} \BBB^{(s)}$, then $\AAA = \bigsqplus_{s \in S} (\AAA \wedge \BBB^{(s)})$.
\item Suppose $\AAA^{(s)} \leqsl \XXX$ ($s \in S \neq \varempty$) and $\BBB \leqsl^s \XXX$. Then
   $$
   \bigl[\bigvee_{s \in S} \AAA^{(s)}\bigr] \wedge \BBB = \bigvee_{s \in S} [\AAA^{(s)} \wedge \BBB].
   $$
   If in addition $\bigoplus_{s \in S} \AAA^{(s)} \leqsl \XXX$, then
   $$
   \bigl[\bigoplus_{s \in S} \AAA^{(s)}\bigr] \wedge \BBB = \bigoplus_{s \in S} [\AAA^{(s)} \wedge \BBB].
   $$
\end{enumerate}
\end{pro}
\begin{proof}
To prove (A), put $\BBB = \bigsqplus_{s \in S} \BBB^{(s)}$. Since each of $\BBB^{(s)}$'s corresponds to a central
projection in $\WWw'(\bBB)$, the assertion easily follows. The same argument works in (B)---here $\BBB$ corresponds
to a central projection in $\WWw'(\xXX)$.
\end{proof}

A counterpart of a part of \PRO{leqsl-leqsls} for the order `$\leqsl$' will be proved in \THMp{Boole}. However, this will
be much more complicated.

\SECT{Steering projections in $\WWw^*$-algebras}

We would like to propose a little bit new approach to the so-called \textit{dimension theory} of $\WWw^*$-algebras
(see e.g. \cite[Chapter~5, \S5]{k-r1} and \cite[Chapter~6]{k-r2}; \cite[Chapter~5, \S1]{tk1}; \cite{gr1,gr2}; \cite{tom};
\cite{sh}). Usually one decomposes a projection in a $\WWw^*$-algebra into (in a sense) `homogenous' parts, as it was
done by Griffin \cite{gr1,gr2}, Tomiyama \cite{tom} and Sherman \cite{sh}. In the next section we will do essentially
the same but in a different manner, convenient for applications to the class $\CDD_N$. In every $\WWw^*$-algebra $\MMm$
we shall distinguish a projection, called \textit{steering}, and next we shall show how this projection `controls'
the Murray-von Neumann order on $\EEe(\MMm)$. As we will see, the steering projection is defined in different ways for type
II$_1$; type II$_{\infty}$; and type I or III algebras. Therefore we shall distinguish our investigations into these three
cases.\par
\textbf{Type II$_{\pmb{1}}$.} When $\MMm$ is a type II$_1$ $\WWw^*$-algebra, it seems to be most appropriate to call
the unit of $\MMm$ the steering projection.\par
\textbf{Types I and III.}
We assume that $\MMm$ is a type I or III $\WWw^*$-algebra. We say $\MMm$ is \textit{quasi-commutative} iff $p \sim c_p$
for every $p \in E(\MMm)$. A projection $p \in E(\MMm)$ is \textit{quasi-abelian} iff $p = 0$ or $p \MMm p$ is
quasi-commutative.

\begin{lem}{q-a}
For $p \in E(\MMm)$ \tfcae
\begin{enumerate}[\upshape(i)]
\item $p$ is quasi-abelian,
\item for every $q \in E(\MMm)$ with $q \leqsl p$, $q \sim c_q p$,
\item for every $q \in E(\MMm)$, $p \preccurlyeq q \iff p \leqsl c_q$.
\end{enumerate}
\end{lem}
\begin{proof}
The equivalence of (i) and (ii) follows from the fact that the central support of $q \in E(p \MMm p)$ with respect
to $p \MMm p$ coincides with $c_q p$ (where $c_q$ is the central support of $q$ in $\MMm$).\par
To show that (iii) follows from (ii), assume that $0 \neq p \leqsl c_q$ and take a maximal family $\{p_s\}_{s \in S}
\subset E(\MMm)$ of nonzero projections such that $p_s \leqsl p$, $p_s \preccurlyeq q$ for $s \in S$ and $c_{p_s} c_{p_t}
p = 0$ for distinct $s,t \in S$. Notice that then $p = \sum_{s \in S} c_{p_s} p$ and $c_{p_s} c_{p_t} c_p = 0$ for
different $s,t \in S$. Now we infer from (ii) that $c_{p_s} p \preccurlyeq q$ and consequently $c_{p_s} p \preccurlyeq
c_{p_s} c_p q$. So, $p = \sum_{s \in S} c_{p_s} p \preccurlyeq (\sum_{s \in S} c_{p_s} c_p) q \leqsl q$.\par
Finally, under the assumption of (iii), for $q \leqsl p$ put $r = q + (1 - c_q) p$, notice that $c_r = c_p \geqsl p$
and thus, by (iii), $p \preccurlyeq r$. Consequently, $c_q p \preccurlyeq c_q r = q$ and we are done.
\end{proof}

A \textit{steering} projection in (a type I or III $\WWw^*$-algebra) $\MMm$ is a quasi-abelian projection
$p \in E(\MMm)$ such that $c_p = 1$.

\begin{thm}{st1,3}
\begin{enumerate}[\upshape(I)]
\item Suppose $\MMm$ is type I. A projection $p \in E(\MMm)$ with $c_p = 1$ is steering iff $p$ is abelian. In particular,
   $\MMm$ has a steering projection and any two steering projections are Murray-von Neumann equivalent.
\item Suppose $\MMm$ is type III. $\MMm$ has a steering projection and any two steering projections are Murray-von Neumann
   equivalent.
\end{enumerate}
\end{thm}
\begin{proof}
Point (I) is left for the reader. We shall give a sketch of proof of (II). If $p$ and $q$ are steering, then $c_p = c_q =
1$ and thus $p \preccurlyeq q$ and $q \preccurlyeq p$, by \LEM{q-a}. This establishes uniqueness up to Murray-von Neumann
equivalence. To show the existence, take a maximal family $\{p_s\}_{s \in S} \subset E(\MMm)$ of mutually centrally
orthogonal nonzero projections each of which is countably decomposable and put $p = \sum_{s \in S} p_s$. Such a projection
is steering since each of $p_s$'s is quasi-abelian, e.g. by \cite[Corollary~6.3.5]{k-r2}.
\end{proof}

\textbf{Type II$_{\pmb{\infty}}$.}
Finally, assume $\MMm$ is a type II$_{\infty}$ $\WWw^*$-algebra. Since the unit of $\MMm$ may be written in the form
$\sum_{n=1}^{\infty} p_n$ with $p_n \sim 1$ for each $n \geqsl 1$, for every projection $q \in E(\MMm)$ there is
a countable infinite family of mutually orthogonal projections each of which is Murray-von Neumann equivalent to $q$. For
each $n \in \{1,2,3,\ldots\} \cup \{\omega\}$ we shall write $n \odot q$ to denote any projection (or, a unique member
of $\EEe(\MMm)$) in $\MMm$ which is the sum of (exactly) $n$ copies of $q$. (Here by a \textit{copy} we mean any projection
which is Murray-von Neumann equivalent to $q$; $\omega \odot q$ is the sum of $\aleph_0$ copies of $q$.)\par
We begin with

\begin{lem}{fin}
For $p \in E(\MMm)$ \tfcae
\begin{enumerate}[\upshape(i)]
\item $p$ is finite,
\item whenever $p \leqsl c_q$ for $q \in E(\MMm)$, there is a sequence $(z_n)_{n=1}^{\infty}$ of central projections
   in $\MMm$ such that $\sum_{n=1}^{\infty} z_n = 1$ and $z_n p \preccurlyeq n \odot q$ for any $n \geqsl 1$.
\end{enumerate}
\end{lem}
\begin{proof}
Let $q_0 \in E(\MMm)$ be a finite projection such that $c_{q_0} = 1$. If (iii) is satisfied, then $z_n p \preccurlyeq
n \odot q_0$ for a suitable sequence $(z_n)_{n=1}^{\infty}$ of central projections. Then $z_n p$ is finite and thus so is
$(\bigvee_{n \geqsl 1} z_n) p = p$.\par
Conversely, if $p$ is finite and $p \leqsl c_q$, there is a family $\{q_s\}_{s \in S}$ of mutually orthogonal projections
such that $p = \sum_{s \in S} q_s$ and $q_s \preccurlyeq q$ for all $s \in S$. Let $\tr\dd p \MMm p \to \ZZz(p \MMm p) =
\ZZz(\MMm) p$ be the trace on $p \MMm p$. There are central (in $\MMm$) projections $z^{(s)}_{n,k}$ with $1 \leqsl k \leqsl
2^n$ and $n \geqsl 1$ such that
$$
\tr(q_s) = \sum_{n=1}^{\infty} (\sum_{k=1}^{2^n} \frac{k}{2^n} z^{(s)}_{n,k} p).
$$
Since $\tr(z^{(s)}_{n,k} p) \leqsl 2^n \tr(q_s)$, $z^{(s)}_{n,k} p \preccurlyeq 2^n \odot q_s \preccurlyeq 2^n \odot q$.
Moreover, we infer from the relation $p = \tr(p) = \sum_{s \in S} \tr(q_s)$ that $\bigvee_{s,n,k} z^{(s)}_{n,k} \geqsl p$.
Reindexing the family $\{z^{(s)}_{n,k}\}_{s,n,k}$ we obtain a collection $\{w_t\}_{t \in T} \subset E(\ZZz(\MMm))$ such
that
$$
w_t p \preccurlyeq m(t) \odot q \qquad \textup{and} \qquad w := \bigvee_{t \in T} w_t \geqsl p
$$
where $m(t)$ is some positive integer. Now let $\{v_t\}_{t \in T}$ be a family of mutually orthogonal central projections
such that $v_t \leqsl w_t\ (t \in T)$ and $\sum_{t \in T}  v_t = w$. Let $* \notin T$, $v_* = 1 - w$ and $m(*) = 1$.
Observe that $v_t p \preccurlyeq m(t) \odot q$ for every $t \in T_* := T \cup \{*\}$, and $\sum_{t \in T_*} v_t = 1$.
To this end, define $z_n$ for $n > 0$ by $z_n = \sum \{v_t\dd\ t \in T_*,\ m(t) = n\}$.
\end{proof}

Let $$E_{\omega}(\MMm) = \{q \in E(\MMm)\dd\ q \sim \omega \odot p \textup{ for some finite projection } p\}.$$

\begin{lem}{omega}
\begin{enumerate}[\upshape(a)]
\item For every $p \in E_{\omega}(\MMm)$ and a properly infinite projection $q \in E(\MMm)$, $p \preccurlyeq q \iff
   p \leqsl c_q$.
\item If $p \in E_{\omega}(\MMm)$ is such that $c_p = 1$, then $q \sim c_q p$ for every $q \in E_{\omega}(\MMm)$.
\item If $p \in E_{\omega}(\MMm)$ and $z \in E(\ZZz(\MMm))$, then $z p \in E_{\omega}(\MMm)$.
\end{enumerate}
\end{lem}
\begin{proof}
Point (c) is immediate and (b) follows from (a) and (c). So, it suffices to check (a). Assume $p$ and $q$ are as there
and $p \leqsl c_q$. Take a finite projection $p_0$ such that $p \sim \omega \odot p_0$. By \LEM{fin}, $z_n p_0 \preccurlyeq
n \odot q$ for a suitable sequence $(z_n)_{n=1}^{\infty}$ of central projections. Since $q$ is properly infinite, $q \sim
\omega \odot q$ and hence $z_n p_0 \preccurlyeq z_n q$ which gives $p_0 \preccurlyeq q$. Consequently, $p \sim \omega \odot
p_0 \preccurlyeq \omega \odot q \sim q$ and we are done.
\end{proof}

A \textit{steering} projection in (a type II$_{\infty}$ $\WWw^*$-algebra) $\MMm$ is a projection
$p \in E_{\omega}(\MMm)$ with $c_p = 1$. Since $E_{\omega}(\MMm)$ consists of properly infinite projections, \LEM{omega}
ensures that any two steering projections in $\MMm$ are Murray-von Neumann equivalent.
\vspace{0.3cm}

Now if $\MMm$ is an arbitrary $\WWw^*$-algebra, the steering projection of $\MMm$ is defined as the sum of the steering
projections of type I, II$_1$, II$_{\infty}$ and III parts of $\MMm$. It is clear that any two steering projections
in $\MMm$ are Murray-von Neumann equivalent. The reader should also verify that if $p \in E(\MMm)$ is a steering
projection, then $c_p = 1$ and $z p$ is a steering projection of $\MMm z$ for every central projection $z$ in $\MMm$.

\SECT{Decomposition relative to steering projection}

Let us first generalize the idea of the previous section. Whenever $\alpha$ is an (arbitrary) cardinal number and $p$
and $q$ are projections in a $\WWw^*$-algebra $\MMm$, $p$ is said to be a \textit{copy} of $q$ provided $p \sim q$;
and $p \sim \alpha \odot q$ iff $p$ is a sum of $\alpha$ copies of $q$. In particular, $p \sim 0 \odot q$ is equivalent to
$p = 0$. When $\MMm$ contains $\alpha$ mutually orthogonal copies of $q$, we shall also write $p \preccurlyeq \alpha \odot
q$ with obvious meaning. Similarly, we shall say that $p$ contains $\alpha$ orthogonal copies of $q$ iff $q' \sim \alpha
\odot q$ for some projection $q' \leqsl p$.\par
Using standard methods (such as Lemma~6.3.9 and Theorem~6.3.11 of \cite{k-r2}; cf. \cite[Proposition~III.1.7.1]{bl}),
similar to those in \cite{gr1,gr2}, \cite{tom} or \cite{sh}, one shows the next result (we skip its proof). To simplify
its statement, let us define the classes $\Lambda_I$, $\Lambda_{\tII}$ and $\Lambda_{\tIII}$ as follows. $\Lambda_I =
\Card$ (the class of all cardinals), $\Lambda_{\tII} = \Card_{\infty} \cup \{0,1\}$ and $\Lambda_{\tIII} = \Card_{\infty}
\cup \{0\}$ where $\Card_{\infty}$ is the class of all infinite cardinals. For any cardinal $\alpha$, $\alpha^+$
is the direct successor of $\alpha$, that is, $\alpha^+ = \min \{\beta \in \Card\dd\ \beta > \alpha\}$. Below `$\sim$'
refers to the Murray-von Neumann equivalence in $\MMm$.

\begin{thm}{deco}
Let $\MMm$ be a properly infinite $\WWw^*$-algebra, $p$ a steering projection of $\MMm$ and let $\AAa = p \MMm p$. Let
$z^I, z^{\tII}, z^{\tIII} \in \ZZz(\AAa)$ be projections such that $z^I + z^{\tII} + z^{\tIII} = p$ and $\AAa z^i$ is
of type i for $i = I,\tII,\tIII$. For every $q \in E(\MMm)$ there is a unique system
$\{z^I_{\alpha}(q)\}_{\alpha \in \Lambda_I} \cup \{z^{\tII}_{\alpha}(q)\}_{\alpha \in \Lambda_{\tII}} \cup
\{z^{\tIII}_{\alpha}(q)\}_{\alpha \in \Lambda_{\tIII}} \subset \ZZz(\AAa)$ of mutually orthogonal projections such that
for $i = I,\tII,\tIII$, $\sum_{\alpha \in \Lambda_i} z^i_{\alpha}(q) = z^i$ and $c_{z^i_{\alpha}(q)} q \sim \alpha \odot
z^i_{\alpha}(q)$ if only $\alpha \in \Lambda_i$ and $(i,\alpha) \neq (\tII,1)$, while $c_{z^{\tII}_1(q)} q$ is finite and
$z^{\tII}_1(q) \sim \omega \odot c_{z^{\tII}_1(q)} q$.\par
What is more, $z^i_{\alpha}(q)$ may be characterized as follows:
\begin{multline*}
z^{\tII}_1(q) = \bigvee \{w \in E(\AAa)|\ w \leqsl z^{\tII},\ \forall v \in E(\AAa),\ 0 \neq v \leqsl w\dd\\
c_v q \neq 0 \textup{ and $q$ contains no copy of } \omega \odot v\}
\end{multline*}
and when $(i,\alpha) \neq (\tII,1)$,
\begin{multline*}
z^i_{\alpha}(q) = \bigvee \{w \in E(\AAa)|\ w \leqsl z^i,\ c_w q \sim \alpha \odot w,\ \forall v \in E(\AAa)\dd\\
0 \neq v \leqsl w \implies q \textup{ does not contain $\alpha^+$ orthogonal copies of $v$}\}.
\end{multline*}
\end{thm}

The statement of the above theorem is complicated. We have formulated it in this way for further applications to the class
$\CDD_N$.\par
For purpose of this paper, let us introduce the following

\begin{dfn}{type}
Let $i \in \{I,\tII,\tIII\}$ and $\alpha \in \Card_{\infty}$. A $\WWw^*$-algebra $\MMm$ is said to be
of \textit{(pure) type i$_{\alpha}$} iff $\MMm$ is of pure type $i$ and $1 \sim \alpha \odot p$ where $p$ is the steering
projection of $\MMm$.
\end{dfn}

Recall that the above definition of type I$_{\alpha}$ $\WWw^*$-algebras is equivalent to the classical definition
of this type, and that below types I$_n$ for finite $n$ and II$_1$ are understood in the usual sense.

\begin{pro}{types}
For every $\WWw^*$-algebra $\MMm$ there is a unique system $\{z^i_{\alpha}\dd\ i \in \{I,\tII,\tIII\},\ \alpha \in
\Lambda_i \setminus \{0\}\} \subset E(\ZZz(\MMm))$ such that $1 = \sum_{i,\alpha} z^i_{\alpha}$ and for each $i$ and
$\alpha$ either $z^i_{\alpha} = 0$ or $\MMm z^i_{\alpha}$ is of pure type $i_{\alpha}$.
\end{pro}

To simplify statements of next results, we fix $i \in \{I,\tII,\tIII\}$, $\gamma \in \Card_{\infty}$, a type $i_{\gamma}$
$\WWw^*$-algebra $\MMm$ and a steering projection $p$ of $\MMm$. Additionally, we put $\AAa = p \MMm p$ and $\Lambda =
\{\alpha \in \Lambda_i\dd\ \alpha \leqsl \gamma\}$. For every $q \in E(\MMm)$ let $z_{\alpha}(q) = z^i_{\alpha}(q)$
where $z^i_{\alpha}(q)$ is as in \THM{deco}. It is easily seen that $z_{\alpha}(q) = 0$ for $\alpha > \gamma$ and
$\sum_{\alpha \in \Lambda} z_{\alpha}(q) = p$. Therefore for every $q \in E(\MMm)$ we shall deal with a \underline{set}
$\{z_{\alpha}(q)\}_{\alpha \in \Lambda}$ of projections.\par
We skip the proof of the next result (cf. \cite{sh}).

\begin{pro}{ord}
For $q, q' \in E(\MMm)$ \tfcae
\begin{enumerate}[\upshape(i)]
\item $q \preccurlyeq q'$,
\item $z_{\beta}(q) z_{\alpha}(q') = 0$ whenever $\alpha, \beta \in \Lambda$ and $\beta > \alpha$;
   and $c_{z_1(q)} c_{z_1(q')} q \preccurlyeq c_{z_1(q)} c_{z_1(q')} q'$ provided $i = \tII$.
\end{enumerate}
\end{pro}

The following result explains the terminology proposed by us.

\begin{pro}{puretype}
Let $q \in E(\MMm)$ be nonzero. Then $c_q \sim \gamma \odot q$ and $\MMm$ does not contain $\gamma^+$ orthogonal copies
of $q$.
\end{pro}
\begin{proof}
The second claim is left for the reader. For every positive cardinal $\beta \in \Lambda$ let $S_{\beta}$ be a set such that
$\card(S_{\beta}) = \beta$ and let $\kappa_{\beta}\dd S_{\gamma} \times S_{\beta} \to S_{\gamma}$ be a bijection. Since
$\MMm$ is of type $i_{\gamma}$, there is a collection $\{p_s\}_{s \in S_{\gamma}}$ of mutually orthogonal projections
Murray-von Neumann equivalent to $p$ which sum up to $1$. For $s \in S_{\gamma}$ let
$$
q_s = \sum_{\beta \in \Lambda \setminus \{0\}} c_{z_{\beta}(q)} \sum_{t \in S_{\beta}} p_{\kappa_{\beta}(s,t)}.
$$
Since $c_{z_{\beta}(q)} p_s \sim z_{\beta}(q)$ and $\sum_{\beta \in \Lambda \setminus \{0\}} c_{z_{\beta}(q)} = c_q$,
$q_s \sim q$ for $s \in S_{\gamma}$. To this end, observe that
$$
\sum_{s \in S_{\gamma}} q_s = \sum_{\beta \in \Lambda \setminus \{0\}} c_{z_{\beta}(q)}
\sum_{(s,t) \in S_{\gamma} \times S_{\beta}} p_{\kappa_{\beta}(s,t)} = \sum_{\beta \in \Lambda \setminus \{0\}}
c_{z_{\beta}(q)} = c_q.
$$
\end{proof}

\begin{pro}{complement}
For every $q \in E(\MMm)$ there are projections $q_{\#}, q^{\#} \in E(\MMm)$ such that $1 - q_{\#} \sim q \sim 1 - q^{\#}$
and $q_{\#} \preccurlyeq q' \preccurlyeq q^{\#}$ for every $q' \in E(\MMm)$ with $1 - q' \sim q$. Moreover, $q^{\#} \sim
1$ and $q_{\#} \sim 1 - c_{z_{\gamma}(q)}$.
\end{pro}
\begin{proof}
Since for $i = I,\tIII$ arguments are similar, we shall only sketch the proof for $i = \tII$ (which is most complicated).
Since $1 \sim 2 \odot 1$, it is clear that there is $q^{\#} \in E(\MMm)$ such that $q^{\#} \sim 1$ and $1 - q^{\#} \sim q$.
Thus we only need to find $q_{\#}$. For each $\beta \in \Lambda$ let $S_{\beta}$ be a set of cardinality $\beta$
and $\{p_s\}_{s \in S_{\gamma}}$ be a collection of mutually orthogonal projections which are Murray-von Neumann equivalent
to $p$ and sum up to $1$. We assume that $S_{\beta} \subset S_{\gamma}$ for each $\beta \in \Lambda$. Let $s_1 \in S_1$.
Take $v \in E(\MMm)$ with $v \leqsl p_{s_1}$ and $v \sim c_{z_1(q)} q$, and put
$$
q_{\#} = c_{z_1(q)} (p_{s_1} - v) + \sum_{\beta \in \Lambda} c_{z_{\beta}(q)} \sum_{s \in S_{\gamma} \setminus S_{\beta}}
p_s.
$$
Since $\sum_{\beta \in \Lambda} c_{z_{\beta}(q)} = 1$ and $\card(S_{\gamma} \setminus S_{\beta}) = \gamma$ if only $\beta <
\gamma$, we infer from these that $q_{\#} \sim 1 - c_{z_{\gamma}(q)}$. This implies that $z_{\gamma}(q_{\#}) =
(1 - c_{z_{\gamma}(q)}) p = p - z_{\gamma}(q) = \sum_{\beta \in \Lambda \setminus \{\gamma\}} z_{\beta}(q)$, $z_0(q_{\#}) =
c_{z_{\gamma}(q)} p = z_{\gamma}(q)$ and $z_{\beta}(q_{\#}) = 0$ for each $\beta \in \Lambda \setminus \{0,\gamma\}$
(in particular, $z_1(q_{\#}) = 0$). Further, observe that $c_{z_1(q)} v = v$ and thus
$$
1 - q_{\#} = v + \sum_{\beta \in \Lambda \setminus \{1\}} c_{z_{\beta}(q)} \sum_{s \in S_{\beta}} p_s
$$
which yields that $1 - q_{\#} \sim q$. Now let $q' \in E(\MMm)$ be such that $1 - q' \sim q$. Thanks to \PRO{ord}, $q_{\#}
\preccurlyeq q'$ iff $z_{\beta}(q_{\#}) z_{\alpha}(q') = 0$ whenever $\alpha, \beta \in \Lambda$ and $\alpha < \beta$
(because $z_1(q_{\#}) = 0$). In our situation the latter is equivalent to $z_{\beta}(q) z_{\alpha}(q') = 0$ for every
$\alpha, \beta \in \Lambda \setminus \{\gamma\}$. For such $\alpha$ and $\beta$ we have
$$
w := c_{z_{\beta}(q)} c_{z_{\alpha}(q')} = w q' + w (1 - q')
$$
and $w (1 - q') \sim w q$. But
$$\begin{cases}
w q' \sim \alpha \odot (w p) & \textup{if } \alpha \neq 1,\\
w q' \quad \textup{is finite} & \textup{if } \alpha = 1,
\end{cases}
\quad \textup{and} \quad
\begin{cases}
w q \sim \beta \odot (w p) & \textup{if } \beta \neq 1,\\
w q \quad \textup{is finite} & \textup{if } \beta = 1.
\end{cases}$$
We conclude from this that either $w$ is finite (and hence $w = 0$) or $w \sim \max(\alpha,\beta) \odot wp$. At the same
time, thanks to e.g. \PRO{puretype}, $w \sim \gamma \odot wp$ which implies that $w = 0$ and we are done.
\end{proof}

Since in every finite $\WWw^*$-algebra $\WWw$, $1 - q' \sim q$ iff $q' \sim 1 - q$ for any $q,q' \in E(\WWw)$,
\PRO{complement} gives

\begin{thm}{complement}
Let $\WWw$ be a $\WWw^*$-algebra and $q \in E(\WWw)$. There are projections $q^{\#}$ and $q_{\#}$ such that $1 - q_{\#}
\sim q \sim 1 - q^{\#}$ and $q_{\#} \preccurlyeq q' \preccurlyeq q^{\#}$ whenever $q' \in E(\WWw)$ is such that $1 - q'
\sim q$. What is more, if $\WWw$ is properly infinite, $q^{\#} \sim 1$ and $q_{\#}$ is Murray-von Neumann equivalent to
a central projection.
\end{thm}

Our last aim of this section is

\begin{pro}{limit}
Let $S$ be an (infinite) set whose power is a limit cardinal. Let $\{p_s\}_{s \in S}$ be a family of mutually orthogonal
projections in a $\WWw^*$-algebra $\WWw$ which sum up to $1$. For nonempty set $A \subset S$ put $q_A = \sum_{s \in A}
p_s$. Then $1$ is the l.u.b. of the family $\{q_A\dd\ A \subset S,\ 0 < \card(A) < \card(S)\}$ with respect
to the Murray-von Neumann order.
\end{pro}
\begin{proof}
Thanks to \PRO{types}, we may and do assume that $\WWw$ is of pure type $i_{\gamma}$. Since the assertion of the theorem
is known to be true for finite algebras $\WWw$, we assume in addition that $\WWw$ is properly infinite---that is, that
$\gamma$ is infinite. Finally, we reduce our considerations to the case when the steering projection $p$ of $\WWw$ is
countably decomposable.\par
Let $q \in E(\WWw)$ be such that $q_A \preccurlyeq q$ for each $A \in \SsS := \{A \subset S\dd\ 0 < \card(A) < \card(S)\}$.
We need to show that $q \sim 1$. Equivalently, we have to prove that $z^i_{\alpha}(q) = 0$ provided $\alpha < \gamma$. When
$i = \tII$, $c_{z^{\tII}_1(q)} q$ is finite and $c_{z^{\tII}_1(q)} \sum_{s \in A} p_s \preccurlyeq c_{z^{\tII}_1(q)} q$ for
each $A \in \SsS$ which implies that $c_{z^{\tII}_1(q)} \preccurlyeq c_{z^{\tII}_1(q)} q$. Consequently,
$c_{z^{\tII}_1(q)}$ is finite and thus $z^{\tII}_1(q) = 0$. Also when $i = I$ and $\alpha$ is finite, $z^i_{\alpha}(q) =
0$, beacuse then $\alpha \odot p$ is finite.\par
Now we are in position when $\alpha$ is infinite. Then $c_{z^i_{\alpha}(q)} q \sim \alpha \odot z^i_{\alpha}(q) \sim
\alpha \odot (c_{z^i_{\alpha}(q)} p)$ and $c_{z^i_{\alpha}(q)} q_A \preccurlyeq c_{z^i_{\alpha}(q)} q$ for any $A \in
\SsS$. We argue by contradiction. Assume $z^i_{\alpha}(q) \neq 0$. Replacing $\WWw$ by $\WWw c_{z^i_{\alpha}(q)}$, we may
assume $c_{z^i_{\alpha}(q)} = 1$, that is, $z^i_{\alpha}(q) = p$. We then have $q \sim \alpha \odot p$, $1 \sim \gamma
\odot p$ and $q_A \preccurlyeq q$ ($A \in \SsS$). We distinguish between two cases. When $\card(S) \leqsl \alpha$,
we easily get $p_s \preccurlyeq q$ and thus $1 = \sum_{s \in S} p_s \preccurlyeq \alpha \odot q \sim \alpha^2 \odot p$
which denies the facts that $\alpha^2 < \gamma$ and $1 \sim \gamma \odot p$.\par
Finally, assume that $\card(S) > \alpha$. Since $p$ is countably decomposable and $q \sim \alpha \odot p$,
$\card(\{s \in A\dd\ p_s \neq 0\}) \leqsl \alpha$ for any $A \in \SsS$ (because $q_A \preccurlyeq q$). We conclude from
this that $A := \{s \in S\dd\ p_s \neq 0\} \in \SsS$ (because $\card(S) > \alpha^+$). But then $1 = q_A \preccurlyeq q$
and we are done.
\end{proof}

\begin{exm}{lim}
As the following example shows (compare with \cite[Example~3]{tom}), the assumption in \PRO{limit} that the power
of $S$ is a limit cardinal is essential. Let $\HHh$ be a Hilbert space of dimension $\aleph_1$, $S$ be a set of power
$\aleph_1$ and let $\{e_s\}_{s \in S}$ be an orthonormal basis of $\HHh$. Further, let $\MMm = \BBb(\HHh)$ and for
$s \in S$ let $p_s \in E(\MMm)$ be the orthogonal rank-one projection onto the linear span of $e_s$. Now if $q_A$'s are
defined as in \PRO{limit}, then $q_A \preccurlyeq q_J$ for every nonempty set $A \subset S$ of power less than $\aleph_1$
where $J$ is a countable infinite subset of $S$ and hence $1$ is nonequivalent to the l.u.b. (which is $q_J$).
\end{exm}

\SECT{Minimal and semiminimal $N$-tuples}

The idea of steering projections will now be adapted to the class $\CDD_N$. Following Ernest \cite{e}, we say a nontrivial
$N$-tuple $\AAA \in \CDD_N$ is (\textit{of}) \textit{type} I, II, III iff such is $\WWw'(\aAA)$. Additionally, we let
the trivial $N$-tuple be of each of these types.\par
We begin with a result which will find many applications in the sequel.

\begin{lem}{continuum}
Every collection of mutually unitarily disjoint nontrivial members of $\CDD_N$ has power no greater than $2^{\aleph_0}$.
\end{lem}
\begin{proof}
Suppose
\begin{equation}\label{eqn:aux11}
\AAA^{(s)} \disj \AAA^{(s')}
\end{equation}
(and $\AAA^{(s)} \neq \zero$) for distinct $s, s' \in S$. For $n \in J = \{1,2,3,\ldots\} \cup \{\aleph_0\}$ let $\HHh_n$
be a fixed Hilbert space of dimension $n$. By \THMp{common}, for each $s \in S$ there is $n(s) \in J$ and $\bBB^{(s)} \in
\CDDc_N(\HHh_{n(s)})$ such that $\BBB^{(s)} \leqsl \AAA^{(s)}$. We infer from \eqref{eqn:aux11} that $\bBB^{(s)} \neq
\bBB^{(s')}$ for distinct $s, s' \in S$. Now the assertion easily follows from the fact that $\card(\CDDc_N(\HHh_n)) \leqsl
2^{\aleph_0}$ for every $n \in J$.
\end{proof}

\begin{dfn}{minimal}
$\AAA \in \CDD_N$ is said to be \textit{minimal} iff for every $\BBB \in \CDD_N$,
$$
\AAA \ll \BBB \implies \AAA \leqsl \BBB.
$$
$\AAA$ is said to be \textit{multiplicity free} ($\AAA \in \MmM\FfF_N$) iff there is no nontrivial $\BBB \in \CDD_N$
for which $2 \odot \BBB \leqsl \AAA$. $\AAA$ is a \textit{hereditary idempotent} ($\AAA \in \HhH\IiI_N$) iff $\BBB = 2
\odot \BBB$ for every $\BBB \leqsl \AAA$. We shall write $\AAA \in \HhH\IiI\MmM_N$ to express that $\AAA$ is both
a hereditary idempotent and minimal.
\end{dfn}

Minimal members of $\CDD_N$ correspond to quasi-abelian projections.

\begin{rem}{e}
The work of Ernest \cite{e} deals with (single) bounded operators. In this context, our definition of a multiplicity free
operator is equivalent to Ernest's one (Definition~1.21 in \cite{e}).
\end{rem}

\begin{thm}{minimal}
\begin{enumerate}[\upshape(I)]
\item For every $\AAA \in \CDD_N$, $$\AAA = 2 \odot \AAA \iff \AAA = \aleph_0 \odot \AAA.$$
\item For $\AAA \in \CDD_N$ \tfcae
   \begin{enumerate}[\upshape(i)]
   \item $\AAA$ is minimal,
   \item for each $\BBB \in \CDD_N$, $\BBB \leqsl \AAA \implies \BBB \leqsl^s \AAA$.
   \end{enumerate}
   If $\AAA$ is minimal and $\BBB \leqsl \AAA$, then $\BBB$ is minimal as well.
\item For $\AAA \in \CDD_N$ \tfcae
   \begin{enumerate}[\upshape(i)]
   \item $\AAA \in \MmM\FfF_N$,
   \item $\AAA = \zero$ or $\WWw'(\aAA)$ is commutative.
   \end{enumerate}
   In particular, if $\AAA \in\MmM\FfF_N$ and $\BBB \leqsl \AAA$, then $\BBB \in \MmM\FfF_N$ as well.
\item Every multiplicity free $N$-tuple is minimal and unitarily disjoint from any hereditary idempotent.
\item If $\AAA \in \HhH\IiI_N$ and $\BBB \ll \AAA$, then $\BBB \in \HhH\IiI_N$ as well.
\item There exist unique $\JJJ_I, \JJJ_{\tIII} \in \CDD_N$ such that $\JJJ_I \in \MmM\FfF_N$, $\JJJ_{\tIII} \in
   \HhH\IiI\MmM_N$, $\JJJ_I \sqplus \JJJ_{\tIII}$ is minimal and for every $\AAA \in \CDD_N$:
   \begin{enumerate}[\upshape(a)]
   \item $\AAA \in \MmM\FfF_N$ iff $\AAA \leqsl \JJJ_I$,
   \item $\AAA \ll \JJJ_I$ iff $\AAA = \zero$ or $\WWw'(\aAA)$ is type I,
   \item $\AAA \in \HhH\IiI_N$ iff $\AAA \ll \JJJ_{\tIII}$, iff $\AAA = \zero$ or $\WWw'(\aAA)$ is type III,
   \item $\AAA \in \HhH\IiI\MmM_N$ iff $\AAA \leqsl \JJJ_{\tIII}$,
   \item $\AAA$ is minimal iff $\AAA \leqsl \JJJ_I \sqplus \JJJ_{\tIII}$.
   \end{enumerate}
   What is more, $\dim(\JJJ_I) + \dim(\JJJ_{\tIII}) \leqsl 2^{\aleph_0}$.
\end{enumerate}
\end{thm}
\begin{proof}
In all points of the theorem we make use of \PROp{transl}. Counterparts of points (I) and (V) are well known for
projections in $\WWw^*$-algebras, (II) follows from \LEMp{q-a}, (III) is immediate, (IV) is implied by (III)
and the definitions of suitable notions. We shall describe how to prove (VI). Take a maximal collection
(cf. \LEM{continuum}) of nontrivial mutually unitarily disjoint multiplicity free $N$-tuples (respectively hereditary
idempotents) whose representatives act in separable spaces and define $\JJJ_I$ ($\JJJ_{\tIII}$) as the direct sum
of this family. One may check that obtained in this way $N$-tuple belongs to $\MmM\FfF_N$ ($\HhH\IiI\MmM_N$) and---since
$\JJJ_I$ and $\JJJ_{\tIII}$ are unitarily disjoint---that $\JJJ_I \sqplus \JJJ_{\tIII}$ is minimal. That $\JJJ_I$
and $\JJJ_{\tIII}$ are greatest members of $\MmM\FfF_N$ and $\HhH\IiI\MmM_N$ it follows from the maximality of the taken
families and \THMp{common}. The details are left for the reader (cf. Propositions~2.12, 1.27 and 1.29 and Corollary~1.37
in \cite{e}). (For the proof of (b) and (c) see also \THMP{decomp}.)
\end{proof}

\THM{minimal} shows that there is a greatest minimal $N$-tuple in $\CDD_N$, namely $\JJJ_I \sqplus \JJJ_{\tIII}$,
and that it covers all type I and III $N$-tuples. Since there are also type II ones, we need to introduce one more notion.

\begin{dfn}{semiminimal}
$\AAA \in \CDD_N$ is said to be \textit{semiminimal} ($\AAA \in \SsS\MmM_N$) iff $\AAA$ is unitarily disjoint from every
minimal $N$-tuple and $\AAA$ satisfies the following condition. Whenever $\BBB \in \CDD_N$ is such that $\AAA \ll \BBB$,
$\AAA$ may be written in the form $\AAA = \bigsqplus_{n=1}^{\infty} \AAA_n$ where $\AAA_n \leqsl n \odot \BBB$ for each
$n \geqsl 1$.
\end{dfn}

Before stating the next result, we underline that there is no greatest semiminimal member of $\CDD_N$.

\begin{thm}{semiminimal}
\begin{enumerate}[\upshape(I)]
\item For $\AAA \in \CDD_N$, $\AAA \in \SsS\MmM_N$ iff $\AAA = \zero$ or $\WWw'(\aAA)$ is type II$_1$. In particular,
   if $\AAA \in \SsS\MmM_N$ and $\BBB \leqsl \AAA$, then $\BBB \in \SsS\MmM_N$ as well; the direct sum of finitely many
   semiminimal $N$-tuples belongs to $\SsS\MmM_N$.
\item There is unique $\JJJ_{\tII} \in \CDD_N$ such that for every $\AAA \in \SsS\MmM_N$ there is $\BBB \in \SsS\MmM_N$
   for which $\JJJ_{\tII} = \aleph_0 \odot (\AAA \sqplus \BBB)$. Moreover, $\dim(\JJJ_{\tII}) \leqsl
   2^{\aleph_0}$ and
   \begin{enumerate}[\upshape(a)]
   \item for $\EEE, \FFF \in \CDD_N$ with $\EEE \leqsl \FFF \leqsl \JJJ_{\tII}$,
      \begin{equation}\label{eqn:J2}
      \EEE \leqsl^s \FFF \leqsl^s \JJJ_{\tII} \iff \EEE = 2 \odot \EEE \quad \textup{and} \quad \FFF = 2 \odot \FFF,
      \end{equation}
   \item $\AAA \ll \JJJ_{\tII}$ iff $\AAA = \zero$ or $\WWw'(\aAA)$ is type II.
   \end{enumerate}
\end{enumerate}
\end{thm}
\begin{proof}
Point (I) follows from \LEMp{fin} and \THM{minimal} from which we infer that $\WWw'(\aAA)$ is type II for every $\AAA \in
\SsS\MmM_N$ (because every semiminimal $N$-tuple is unitarily disjoint from $\JJJ_I \sqplus \JJJ_{\tIII}$). To prove (II),
proceed similarly as in the proof of \THM{minimal}. Take a maximal family $\AaA$ of mutually unitarily disjoint nontrivial
members of $\SsS\MmM_N$ whose representatives act in separable spaces and denote by $\SSS(\AaA)$ its direct sum. Next put
$\JJJ_{\tII} = \aleph_0 \odot \SSS(\AaA)$. Check that $\SSS(\AaA) \in \SsS\MmM_N$ for every such $\AaA$. Further, show that
for two maximal families $\AaA$ and $\AaA'$ one has $\SSS(\AaA) \ll \SSS(\AaA') \ll \SSS(\AaA)$ and consequently,
by the definition of semiminimality, $\aleph_0 \odot \SSS(\AaA') = \aleph_0 \odot \SSS(\AaA)$. Having this, one easily
shows the uniqueness of $\JJJ_{\tII}$ and all suitable properties of it. (For example, if $\EEE = 2 \odot \EEE$, then $\EEE
= \aleph_0 \odot \EEE$ and it suffices to apply \LEMP{omega}.)
\end{proof}

The reader should notice that $\JJJ_{\tII}$ corresponds to the steering projection of a type II$_\infty$ $\WWw^*$-algebra.

\begin{rem}{J2}
Point (II) of \THM{semiminimal} implies that $\JJJ_{\tII}$ is the greatest element of the class $\SsS\MmM_N^{\infty} =
\{\aleph_0 \odot \AAA\dd\ \AAA \in \SsS\MmM_N\}$ (and hence $\SsS\MmM_N^{\infty}$ is a set) and that for any $\AAA, \BBB
\in \SsS\MmM_N^{\infty}$, $\AAA \leqsl \BBB \iff \AAA \leqsl^s \BBB$.
\end{rem}

Let us denote by $\JJJ$ the $N$-tuple $\JJJ_I \sqplus \JJJ_{\tII} \sqplus \JJJ_{\tIII}$. We call $\JJJ$ the \textit{unity}
of $\CDD_N$. Since every $\WWw^*$-algebra admits a decomposition into type I, II and III parts, this yields

\begin{pro}{unity}
For every $\AAA \in \CDD_N$, $\AAA \ll \JJJ$.
\end{pro}

\begin{rem}{continuum}
It may be worthwhile to note that $\dim(\JJJ_i) = 2^{\aleph_0}$ for $i = I,\tII,\tIII$. We shall prove this
later (see \CORP{continuum}). We conclude from this and \PRO{unity} that for an infinite cardinal $\alpha$ there exists
$\AAA \in \CDD_N$ such that $\dim(\AAA) = \alpha$ and $\XXX \leqsl \AAA$ whenever $\dim(\XXX) \leqsl \alpha$ iff $\alpha
\geqsl 2^{\aleph_0}$. If the latter happens, such $\AAA$ is of course unique and one may check that $\AAA = \alpha \odot
\JJJ$.
\end{rem}

In the sequel we shall also need the following

\begin{pro}{count}
For every nontrivial $\AAA \leqsl \JJJ$ there is $\BBB \leqsl^s \AAA$ such that $0 < \dim(\BBB) \leqsl \aleph_0$.
\end{pro}
\begin{proof}
By \THMp{common}, there is nontrivial $\BBB_0 \leqsl \AAA$ such that $\bBB_0$ acts in a separable Hilbert space. We may
assume that $\BBB_0 \leqsl \JJJ_i$ for some $i \in \{I,\tII,\tIII\}$. If $i \neq \tII$, we automatically have $\BBB_0
\leqsl^s \AAA$; while when $i = \tII$, it suffices to notice that $\aleph_0 \odot \BBB_0 \leqsl^s \aleph_0 \odot \AAA$
(by \eqref{eqn:J2}) and to apply (PR6) \pREF{PR6} to find $\BBB \leqsl^s \AAA$ with $\aleph_0 \odot \BBB = \aleph_0 \odot
\BBB_0$.
\end{proof}

\begin{exm}{MF}
When $N=1$, one may check that a bounded normal operator on a separable Hilbert space is multiplicity free iff it is
$*$-cyclic (an operator $T \in \BBb(\HHh)$ is $*$-cyclic iff there is $x \in \HHh$ for which the linear span of $\{x\} \cup
\{S_1 \ldots S_m x\dd\ m \geqsl 1,\ S_1,\ldots,S_m \in \{T,T^*\}\}$ is dense in $\HHh$). Taking this into account,
one may ask whether every $*$-cyclic type I operator is multiplicity free. As this simple example shows, it is untrue.
Let $T = \begin{pmatrix}0 & 0\\1 & 0\end{pmatrix}$ and $S = T \oplus T$. Of course, $S \notin \MmM\FfF_1$. However, $S$ is
$*$-cyclic. (For $u = (1,0,0,1)$, $S u = (0,1,0,0)$, $S^* u = (0,0,1,0)$ and $S^*S u = (1,0,0,0)$.)
\end{exm}

\SECT{Unities of ideals in $\CDD_N$}

Reformulating conditions (ID1)--(ID4) \pREF{ID} into the terms of $\CDD_N$, we obtain the notion of an ideal in $\CDD_N$.
Equivalently, a nonempty class $\AaA \subset \CDD_N$ is an ideal provided $\AaA$ is order-complete (i.e. $\bigvee \FfF \in
\AaA$ for every nonempty set $\FfF \subset \AaA$) and $\mM \odot \AAA \in \AaA$ whenever $\mM$ is a cardinal and $\AAA
\leqsl \BBB$ for some $\BBB \in \AaA$.\par
\THMp{main} asserts that for every ideal $\AaA \subset \CDD_N$ and $\XXX \in \CDD_N$ there is unique $\YYY \in \AaA$ such
that $\YYY \leqsl^s \XXX$ and $\XXX \sqminus \YYY \in \AaA^{\perp}$. We shall denote this unique $\YYY$
by $\EEE(\XXX|\AaA)$. Similarly, if $\AAA$ is any member of $\CDD_N$, $\EEE(\XXX | \AAA) := \EEE(\XXX | \{\BBB\dd\
\BBB \ll \AAA\})$. $\EEE(\XXX | \AAA)$ is Ernest's $\AAA$-shadow of $\XXX$ (see \cite[Definition~2.13]{e}).\par
One may easily verify that
$$
\EEE\bigl(\bigoplus_{s \in S} \XXX^{(s)} | \AaA\bigr) = \bigoplus_{s \in S} \EEE(\XXX^{(s)} | \AaA)
$$
for every ideal $\AaA \subset \CDD_N$ and any family $\{\XXX^{(s)}\}_{s \in S} \subset \CDD_N$. We shall use the above
property several times in the sequel.\par
Let $\AaA$ be an ideal in $\CDD_N$. The $N$-tuple $\JJJ(\AaA) := \EEE(\JJJ | \AaA)$ is uniquely determined by $\AaA$
and is called the \textit{unity} of $\AaA$. \PRO{unity} implies that

\begin{pro}{unitid}
For every ideal $\AaA$ in $\CDD_N$, $$\AaA = \{\XXX \in \CDD_N\dd\ \XXX \ll \JJJ(\AaA)\},$$ $\JJJ(\AaA) \leqsl^s \JJJ$ and
$\JJJ(\AaA) = \bigvee \{\AAA \leqsl \JJJ\dd\ \AAA \in \AaA\}$.
\end{pro}

\begin{cor}{1id-1}
There is a one-to-one correspondence between ideals in $\CDD_N$ and members $\AAA$ of $\CDD_N$ such that $\AAA \leqsl^s
\JJJ$. The correspondence is established by assignments $\AaA \mapsto \JJJ(\AaA)$ and $\AAA \mapsto \{\BBB\dd\ \BBB \ll
\AAA\}$. In particular, there is at most $2^{2^{\aleph_0}}$ ideals in $\CDD_N$.
\end{cor}

\begin{exm}{normal}
Let $\NnN_N \subset \CDD_N$ be the ideal of all normal $N$-tuples (see \EXS{id}--(E), \PREF{E}). Since $\WWw''(\mMM)$
is commutative for every $\MMM \in \NnN_N$, $\WWw'(\mMM)$ is type I and hence $\MMM \ll \JJJ_I$. Here we shall give
a description for $\JJJ(\NnN_N)$. First of all, $\MMM \leqsl \JJJ$ iff $\WWw'(\mMM)$ is commutative (provided $\MMM \neq
\zero$ and $\MMM \in \NnN_N$). When $\mMM$ acts in a separable Hilbert space, the latter is equivalent to the fact that
$\mMM$ is $*$-cyclic. That is, there has to exist $x \in \overline{\DdD}(\mMM)$ such that the smallest reducing subspace
for $\mMM$ which contains $x$ coincides with $\overline{\DdD}(\mMM)$. (Indeed, if $\MMM \in \MmM\FfF_N \cap \NnN_N$
is such that $0< \dim(\MMM) \leqsl \aleph_0$, then both $\WWw'(\mMM)$ and $\WWw''(\mMM)$ are commutative
which means that $\WWw''(\mMM)$ is a MASA and consequenty $\WWw''(\mMM)$ is cyclic or, equivalently, $\mMM$ is $*$-cyclic.
Conversely, if $\MMM \in \NnN_N$ and $\mMM$ is $*$-cyclic, then $\mMM$ is unitarily equivalent to $\mMM_{\mu}$ for some
probabilistic Borel measure $\mu$ on $\CCC^N$ where $\mMM_{\mu} = (M_{z_1},\ldots,M_{z_N})$ and $M_{z_j}$ is
the multiplication operator by $z_j$ in $L^2(\mu)$. One may show that $\WWw'(\mMM_{\mu})$ coincides with the algebra
of all multiplication operators by members of $L^{\infty}(\mu)$ and hence $\MMM \in \MmM\FfF_N$.) Having this, one shows
that $\JJJ(\NnN_N)$ may be represented as follows. Take a maximal family $\{\mu_s\}_{s \in S}$ of mutually orthogonal
probabilistic Borel measures on $\CCC^N$. For each $s \in S$ let $\mMM^{(s)} = \mMM_{\mu_s}$ (defined as before). One may
check that $\JJJ(\NnN_N) = \bigsqplus_{s \in S} \MMM^{(s)}$. Moreover, for two probabilistic Borel measures $\mu$ and
$\lambda$ on $\CCC^N$: (a) $\mMM_{\mu} \leqsl \mMM_{\lambda} \iff \mu \ll \lambda$; (b) $\mMM_{\mu} \equiv \mMM_{\lambda}
\iff \mu \ll \lambda \ll \mu$; (c) $\mMM_{\mu} \disj \mMM_{\lambda} \iff \mu \perp \lambda$.
Similar (and more detailed) construction will appear in Section~22.
\end{exm}

Theorems \ref{thm:minimal} and \ref{thm:semiminimal} show that for $i \in \{I,\tII,\tIII\}$ the ideal $\IiI_i = \{\XXX \in
\CDD_N\dd\ \XXX \ll \JJJ_i\}$ consists of all $N$-tuples of type $i$.

\SECT{Decomposition relative to $\JJJ$}

Recall that $\Lambda_I = \Card$, $\Lambda_{\tII} = \Card_{\infty} \cup \{0,1\}$ and $\Lambda_{\tIII} = \Card_{\infty} \cup
\{0\}$. For simplicity, let $\Upsilon = \{(i,\alpha)\dd\ i \in \{I,\tII,\tIII\},\ \alpha \in \Lambda_i\}$ and $\Upsilon_*
= \Upsilon \setminus \{(\tII,1)\}$.

\begin{thm}{decomp}
For every $\AAA \in \CDD_N$ there are a unique regular collection $\{\EEE^i_{\alpha}(\AAA)\dd\ (i,\alpha) \in \Upsilon\}$
and unique $\EEE_{sm}(\AAA) \in \CDD_N$ such that for $i \in \{I,\tII,\tIII\}$,
$$
\JJJ_i = \Bigsqplus_{\alpha \in \Lambda_i} \EEE^i_{\alpha}(\AAA),
$$
$\EEE_{sm}(\AAA)$ is semiminimal and $\EEE^{\tII}_1(\AAA) = \aleph_0 \odot \EEE_{sm}(\AAA)$, and
\begin{equation}\label{eqn:main}
\AAA = \EEE_{sm}(\AAA) \sqplus \Bigsqplus_{(i,\alpha) \in \Upsilon_*} \alpha \odot \EEE^i_{\alpha}(\AAA).
\end{equation}
What is more, $\EEE_{sm}(\AAA) = \AAA \wedge \EEE^{\tII}_1(\AAA)$ and $\EEE^i_{\alpha}(\AAA)$'s may be characterized
as follows:
\begin{equation}\label{eqn:Eii1}
\EEE^{\tII}_1(\AAA) = \bigvee \{\EEE \leqsl \JJJ_{\tII}|\ \EEE \ll \AAA,\ \forall\, \FFF \leqsl \EEE,\ \FFF \neq
\zero\dd\ \aleph_0 \odot \FFF \not\leqsl \AAA\}
\end{equation}
and for $(i,\alpha) \in \Upsilon_*$,
$$
\EEE^i_{\alpha}(\AAA) = \bigvee \{\EEE \leqsl \JJJ_i|\ \alpha \odot \EEE \leqsl \AAA,\ \forall\, \FFF \leqsl \EEE,\
\FFF \neq \zero\dd\ \alpha^+ \odot \FFF \not\leqsl \AAA\}.
$$
\end{thm}
\begin{proof}
Thanks to \PRO{unity}, there is an infinite cardinal $\gamma$ such that $\AAA \leqsl \gamma \odot \JJJ =: \BBB$. Put
$\MMm = \WWw'(\bBB)$, observe that $\jJJ$ corresponds (by \PROP{transl}) to a steering projection of $\MMm$ and apply
\THMp{deco}. (Use \THM{semiminimal} to deduce that suitable $\EEE_{sm}(\AAA)$ is semiminimal. Note that if $\XXX$
and $\YYY$ correspond, by \PRO{transl}, to projections $p$ and $q$, then $p \sim \alpha \odot q$ is equivalent to $\XXX =
\alpha \odot \YYY$.)
\end{proof}

The system $\{\EEE^i_{\alpha}(\AAA)\dd\ (i,\alpha) \in \Upsilon\}$ appearing in \THM{decomp} is said to be
the \textit{partition of unity induced by $\AAA$}. (In general, a \textit{partition of unity} is any regular collection
$\{\EEE^{(j)}\}_{j \in I}$ such that $\JJJ = \bigsqplus_{j \in I} \EEE^{(j)}$. Note that in that case $\EEE^{(j)} \leqsl^s
\JJJ$ for each $j \in I$.)

\begin{rem}{Boole}
\THM{decomp} may be formulated in an equivalent manner that after fixing a representative $\jJJ$ for $\JJJ$ for every
$\aAA \in \CDDc_N$ there are unique systems $\{H^i_{\alpha}\dd\ (i,\alpha) \in \Upsilon\} \subset \RED(\jJJ)$ and
$\{K^i_{\alpha}\dd\ (i,\alpha) \in \Upsilon\} \subset \RED(\aAA)$ such that $\overline{\DdD}(\jJJ_i) =
\bigoplus_{\alpha \in \Lambda_i} H^i_{\alpha}$ for $i \in \{I,\tII,\tIII\}$; $\overline{\DdD}(\aAA) =
\bigoplus_{(i,\alpha) \in \Upsilon} K^i_{\alpha}$; $\WWw'(\aAA\bigr|_{K^{\tII}_1})$ is type II$_1$, $\aleph_0 \odot
\aAA\bigr|_{K^{\tII}_1} \equiv \jJJ\bigr|_{H^{\tII}_1}$ and for every $(i,\alpha) \in \Upsilon_*$,
$$
\aAA\bigr|_{K^i_{\alpha}} \equiv \alpha \odot \jJJ\bigr|_{H^i_{\alpha}}.
$$
(In particular, $K^I_0$, $K^{\tII}_0$ and $K^{\tIII}_0$ are trivial.)
\end{rem}

As an immediate consequence of \PROp{ord} we obtain

\begin{pro}{leqsl}
For any $\AAA, \BBB \in \CDD_N$, $\AAA \leqsl \BBB$ iff $\EEE^i_{\alpha}(\AAA) \disj \EEE^i_{\beta}(\BBB)$ whenever
$(i,\alpha),(i,\beta) \in \Upsilon$ and $\alpha > \beta$; and $\EEE_{sm}(\AAA) \wedge \EEE^{\tII}_1(\BBB) \leqsl
\EEE_{sm}(\BBB)$.
\end{pro}

One may also show that

\begin{pro}{leqsls}
For any $\AAA, \BBB \in \CDD_N$, $\AAA \leqsl^s \BBB$ iff $\EEE^i_{\alpha}(\AAA) \leqsl \EEE^i_{\alpha}(\BBB)$ whenever
$(i,\alpha) \in \Upsilon$ is such that $\alpha \neq 0$, and $\EEE_{sm}(\AAA) \leqsl^s \EEE_{sm}(\BBB)$.
\end{pro}

The proofs of Propositions \ref{pro:leqsl} and \ref{pro:leqsls} are skipped.\par
Another interesting consequences of \THM{decomp} are stated below.

\begin{cor}{ll}
Let $\AAA, \BBB \in \CDD_N$ and let $\alpha$ be an arbitrary infinite cardinal number such that $\alpha \geqsl
\max(\dim(\AAA),\dim(\BBB))$.
\begin{enumerate}[\upshape(I)]
\item $\AAA \ll \BBB \iff \alpha \odot \AAA \leqsl^s \alpha \odot \BBB$.
\item $\AAA \ll \BBB \ll \AAA \iff \alpha \odot \AAA = \alpha \odot \BBB$.
\end{enumerate}
\end{cor}
\begin{proof}
In both the points implication `$\impliedby$' is immediate. Conversely, observe that for each $\XXX \in \CDD_N$
and $(i,\beta) \in \Upsilon$, $\EEE^i_{\beta}(\XXX) = \zero$ provided $\beta > \dim(\XXX)$. This implies
that if $\beta \geqsl \max(\aleph_0,\dim(\XXX))$, then $\beta \odot \XXX = \beta \odot \EEE$ for some
$\EEE \leqsl^s \JJJ$. This notice yields the inverse implication in both points (I) and (II). (Observe that if $\EEE'
\leqsl^s \EEE''$, then $\gamma \odot \EEE' \leqsl^s \gamma \odot \EEE''$ for every cardinal $\gamma$.)
\end{proof}

\begin{cor}{ideal}
A nonempty class $\AaA$ is an ideal iff $\AaA$ satisfies the following three conditions:
\begin{enumerate}[\upshape(a)]
\item for every $\AAA \in \CDD_N$ and $\alpha \in \Card_{\infty}$, $\AAA \in \AaA \iff \alpha \odot \AAA \in \AaA$,
\item whenever $\{\AAA^{(s)}\}_{s \in S} \subset \AaA$ is a regular family of $N$-tuples such that $0 < \dim(\AAA) \leqsl
   \aleph_0$, $\bigsqplus_{s \in S} \AAA^{(s)} \in \AaA$,
\item $\AAA \leqsl^s \BBB$ and $\BBB \in \AaA$ imply $\AAA \in \AaA$.
\end{enumerate}
\end{cor}
\begin{proof}
The necessity is clear. The sufficiency is in fact a consequence of \COR{ll}. Indeed, if $\AAA \leqsl \BBB$ and $\BBB \in
\AaA$, then $\alpha \odot \AAA \leqsl^s \alpha \odot \BBB$ for some infinite cardinal $\alpha$ (by \COR{ll}). It follows
from (a) that $\alpha \odot \BBB \in \AaA$ and consequently $\alpha \odot \AAA \in \AaA$ (by (c)) and $\AAA \in \AaA$,
again by (a). Finally, if $\{\AAA^{(j)}\}_{j \in I} \subset \AaA$ and $\AAA = \bigoplus_{j \in I} \AAA^{(j)}$, then
for huge enough $\alpha \in \Card_{\infty}$ one has $\alpha \odot \AAA^{(j)} = \alpha \odot \EEE^{(j)}$ with $\EEE^{(j)}
\leqsl^s \JJJ$ ($j \in I$) and $\alpha \odot \AAA = \alpha \odot \EEE$ for some $\EEE \leqsl^s \JJJ$ (see the proof
of \COR{ll}). Thanks to (a), $\EEE^{(j)} \in \AaA$ and it is enough to show that $\EEE \in \AaA$. We see that $\EEE^{(j)}
\leqsl^s \EEE$ and $\EEE = \bigvee_{j \in I} \EEE^{(j)}$. These imply (compare with the proof of \THMP{ords}) that there
is a regular family $\{\BBB^{(j)}\}_{j \in I}$ such that $\bigsqplus_{j \in I} \BBB^{(j)} = \EEE$ and $\BBB^{(j)} \leqsl^s
\EEE^{(j)}$ ($j \in I$). We infer from (c) that $\BBB^{(j)} \in \AaA$ for all $j \in I$. Now thanks to \PRO{count}, each
of $\BBB^{(j)}$'s may be written in the form $\bigsqplus_{s \in S_j} \AAA^{(s,j)}$ with $0 < \dim(\AAA^{(s,j)}) \leqsl
\aleph_0$. Consequently, (c) yields $\AAA^{(s,j)} \in \AaA$ and hence $\EEE \in \AaA$ as well, by (b).
\end{proof}

\begin{exm}{commonpart}
Sometimes it may be useful to have the \textit{common} partition of unity for several members of $\CDD_N$ (in particular,
to find the partition of unity induced by their direct sum). It may be understood as follows. For simplicity, we shall
describe this idea only for two $N$-tuples. Below we involve \PROp{leqsl-leqsls} several times, with no comment.\par
Let $\AAA, \BBB \in \CDD_N$. Let $\Upsilon^2 = \{(i,\alpha,\beta)\dd\ (i,\alpha), (i,\beta) \in \Upsilon\}$ and
$\Upsilon^2_* = \{(i,\alpha,\beta)\dd\ (i,\alpha), (i,\beta) \in \Upsilon_*\}$. For $(i,\alpha,\beta) \in \Upsilon^2$ let
$\EEE^i_{\alpha,\beta} = \EEE^i_{\alpha}(\AAA) \wedge \EEE^i_{\beta}(\BBB)$. Additionally, we put $\EEE_{sm,\alpha} =
\EEE_{sm}(\AAA) \wedge \EEE^{\tII}_{\alpha}(\BBB)$ and $\EEE_{\alpha,sm} = \EEE^{\tII}_{\alpha}(\AAA) \wedge
\EEE_{sm}(\BBB)$ for $\alpha \in \Lambda_{\tII}$. One may check that then $\JJJ_i = \bigsqplus_{\alpha,\beta \in \Lambda_i}
\EEE^i_{\alpha,\beta}$ for $i \in \{I,\tII,\tIII\}$; $\EEE_{\alpha,sm}$ and $\EEE_{sm,\alpha}$ are semiminimal and
\begin{equation}\label{eqn:aux20}
\EEE^{\tII}_{1,\alpha} = \aleph_0 \odot \EEE_{sm,\alpha}, \qquad
\EEE^{\tII}_{\alpha,1} = \aleph_0 \odot \EEE_{\alpha,sm}
\end{equation}
for each $\alpha \in \Lambda_{\tII}$. Further,
\begin{multline}\label{eqn:A-AB}
\AAA = \bigl(\Bigsqplus_{\alpha \in \Card_{\infty}} \EEE_{sm,\alpha}\bigr) \sqplus
\bigl(\Bigsqplus_{\alpha \in \Card_{\infty}} \alpha \odot \EEE^{\tII}_{\alpha,1}\bigr)\\
\sqplus \bigl(\Bigsqplus_{(i,\alpha,\beta) \in \Upsilon^2_*} \alpha \odot \EEE^i_{\alpha,\beta}\bigr)
\sqplus (\EEE_{sm,1} \sqplus \EEE_{sm,0})
\end{multline}
and correspondingly
\begin{multline}\label{eqn:B-AB}
\BBB = \bigl(\Bigsqplus_{\alpha \in \Card_{\infty}} \alpha \odot \EEE^{\tII}_{1,\alpha}\bigr) \sqplus
\bigl(\Bigsqplus_{\alpha \in \Card_{\infty}} \EEE_{\alpha,sm}\bigr)\\
\sqplus \bigl(\Bigsqplus_{(i,\alpha,\beta) \in \Upsilon^2_*} \beta \odot \EEE^i_{\alpha,\beta}\bigr)
\sqplus (\EEE_{1,sm} \sqplus \EEE_{0,sm}).
\end{multline}
In particular, thanks to \eqref{eqn:aux20},
$$
\AAA \oplus \BBB = [\EEE_{sm,0} \sqplus \EEE_{0,sm} \sqplus (\EEE_{sm,1} \oplus \EEE_{1,sm})] \sqplus
\Bigsqplus_{(i,\alpha,\beta) \in \Upsilon^2_{\#}} (\alpha + \beta) \odot \EEE^i_{\alpha,\beta}
$$
where $\Upsilon^2_{\#} = \Upsilon^2 \setminus \{(\tII,\alpha,\beta)\dd\ (\alpha,\beta) = (0,1), (1,0), (1,1)\}$. So
(below $(i,\gamma) \in \Upsilon_*$),
\begin{equation}\label{eqn:part2}\begin{cases}
\EEE_{sm}(\AAA \oplus \BBB) = \EEE_{sm,0} \sqplus \EEE_{0,sm} \sqplus [\EEE_{sm,1} \oplus \EEE_{1,sm}],&\\
\EEE^{\tII}_1(\AAA \oplus \BBB) = \EEE^{\tII}_{0,1} \sqplus \EEE^{\tII}_{1,0} \sqplus \EEE^{\tII}_{1,1},&\\
\EEE^i_{\gamma}(\AAA \oplus \BBB) = \Bigsqplus \{\EEE^i_{\alpha,\beta}\dd\ (i,\alpha,\beta) \in \Upsilon^2_{\#},\ \alpha
+ \beta = \gamma\}.&
\end{cases}\end{equation}
In a similar manner one may find the formulas for $\AAA \vee \BBB$ and $\AAA \wedge \BBB$ and the partitions of unity
induced by them.
\end{exm}

\SECT{Algebraic and order properties of $\CDD_N$}

The following is a kind of folklore (see e.g. \cite[Exercise~6.9.14]{k-r2}): if $p$ and $q$ are two projections in a von
Neumann algebra $\MMm$ such that $n \odot p \sim n \odot q$ for some $n \geqsl 1$, then $p \sim q$. This has an interesting
consequence for the class $\CDD_N$:
\begin{enumerate}[({A}O1)]
\item\label{AO1} $n \odot \AAA = n \odot \BBB \implies \AAA = \BBB$
\end{enumerate}
provided $n$ is positive and finite. Further properties in this style are listed below.
\begin{enumerate}[({A}O1)]\addtocounter{enumi}{1}
\item For finite positive $n$ and $m$: $n \odot \AAA = m \odot \BBB \iff \AAA = k \odot \XXX$ and $\BBB = l \odot \XXX$
   for some $\XXX \in \CDD_N$ with $k = m / \GCD(n,m)$ and $l = n / \GCD(n,m)$ (`$\GCD$' is the abbreviation for
   the greatest common divisor). If $n \neq m$, then $n \odot \AAA = m \odot \AAA \iff \AAA = \aleph_0 \odot \AAA$.
\item If $\alpha$ and $\beta$ are cardinals such that $\alpha < \beta$ and $\beta$ is infinite, then
   $$\alpha \odot \AAA = \beta \odot \BBB \iff \AAA = \beta \odot \BBB.$$
   ((AO2) and (AO3) follow from \eqref{eqn:main}; cf. also the beginning of Section~14.)
\item\label{AO4} For a nontrivial $\AAA \in \CDD_N$ \tfcae
   \begin{enumerate}[(i)]
   \item for any $\XXX, \YYY \in \CDD_N$, $\AAA \oplus \XXX = \AAA \oplus \YYY \iff \XXX = \YYY$,
   \item $\BBB \leqsl^s \AAA$ and $\AAA \oplus \BBB = \AAA$ imply $\BBB = \zero$,
   \item $\WWw'(\aAA)$ is finite,
   \item $\EEE^i_{\alpha}(\AAA) = \zero$ for each $i \in \{I,\tII,\tIII\}$ and infinite $\alpha$.
   \end{enumerate}
   All $N$-tuples $\AAA$ satisfying (i) form a \textbf{set}, denoted by $\fIN_N$. $(\fIN_N,\oplus)$ is a semigroup which
   may be enlarged to an Abelian group (by (i)).
\item For every $\AAA \in \fIN_N$ and $\BBB \geqsl \AAA$ there is a unique $\XXX$ such that $\AAA \oplus \XXX = \BBB$.
   Thus, $\BBB \ominus \AAA$ is well defined in that case.
\item\label{AO6} Let $S$ be an infinite set whose power is a limit cardinal. For every collection $\{\AAA^{(s)}\}_{s \in S}
   \subset \CDD_N$,
   \begin{equation}\label{eqn:limcard}
   \bigoplus_{s \in S} \AAA^{(s)} = \bigvee \bigl\{\bigoplus_{s \in S'} \AAA^{(s)}\dd\ S' \subset S,\ 0 < \card(S')
   < \card(S)\bigr\}.
   \end{equation}
   In particular, for every sequence $(\BBB^{(n)})_{n=1}^{\infty} \subset \CDD_N$,
   $$\bigoplus_{n=1}^{\infty} \BBB^{(n)} = \bigvee_{n=1}^{\infty} \BBB^{(1)} \oplus \ldots \oplus \BBB^{(n)},$$
   and for each $\AAA \in \CDD_N$ and an infinite limit cardinal $\gamma$,
   $$\gamma \odot \AAA = \bigvee_{\alpha < \gamma} \alpha \odot \AAA.$$
   (By \PROP{limit}.)
\item Whenever $\AAA \leqsl \BBB$, there are $(\BBB \ominus \AAA)^{\nabla}, (\BBB \ominus \AAA)_{\Delta} \in \CDD_N$
   such that $\AAA \oplus \XXX = \BBB$ iff $(\BBB \ominus \AAA)_{\Delta} \leqsl \XXX \leqsl (\BBB \ominus \AAA)^{\nabla}$.
   Moreover, if $\BBB = 2 \odot \BBB$, then $(\BBB \ominus \AAA)_{\Delta} \leqsl^s \BBB = (\BBB \ominus \AAA)^{\nabla}$.
   $\BBB \ominus \AAA$ is well defined iff $(\BBB \ominus \AAA)^{\nabla} = (\BBB \ominus \AAA)_{\Delta}$.
   (See \THMP{complement}.)
\item If $\AAA \leqsl^s \BBB$, then $(\BBB \ominus \AAA)_{\Delta} = \BBB \sqminus \AAA$. (Thanks to (PR1), \PREF{PR1}.)
\item $(\BBB \ominus \AAA)_{\Delta} \leqsl (\BBB \ominus \XXX)_{\Delta} \oplus (\XXX \ominus \AAA)_{\Delta} \leqsl
   (\BBB \ominus \XXX)^{\nabla} \oplus (\XXX \ominus \AAA)^{\nabla} \leqsl (\BBB \ominus \AAA)^{\nabla}$ whenever
   $\AAA \leqsl \XXX \leqsl \BBB$.
\item $(\BBB \ominus \AAA)_{\Delta} \leqsl^s (\BBB \ominus \AAA)^{\nabla}$ provided $\AAA \leqsl \BBB$.
\end{enumerate}
Let us prove (AO10). We use the notation of \EXM{commonpart}. We infer from \eqref{eqn:A-AB} and \eqref{eqn:B-AB} that
$\AAA \leqsl \BBB$ iff $\EEE_{sm,1} \leqsl \EEE_{1,sm}$, $\EEE_{sm,0} = \zero$ and for every $\gamma \in \Card_{\infty}$
and $(i,\alpha,\beta) \in \Upsilon^2_*$ with $\alpha > \beta$,
$$
\EEE^{\tII}_{\gamma,1} = \EEE^i_{\alpha,\beta} = \zero.
$$
In that case \eqref{eqn:A-AB} reduces to
$$
\AAA = \bigl(\Bigsqplus_{\alpha \in \Card_{\infty}} \EEE_{sm,\alpha}\bigr) \sqplus
\bigl(\Bigsqplus_{\substack{(i,\alpha,\beta) \in \Upsilon^2_*\\\alpha \leqsl \beta}} \alpha \odot
\EEE^i_{\alpha,\beta}\bigr) \sqplus \EEE_{sm,1},
$$
while \eqref{eqn:B-AB} is equivalent to
$$
\BBB = \bigl(\Bigsqplus_{\alpha \in \Card_{\infty}} \alpha \odot \EEE^{\tII}_{1,\alpha}\bigr) \sqplus
\bigl(\Bigsqplus_{\substack{(i,\alpha,\beta) \in \Upsilon^2_*\\\alpha \leqsl \beta}} \beta \odot
\EEE^i_{\alpha,\beta}\bigr) \sqplus \EEE_{sm}(\BBB).
$$
Now we infer from the above formulas and (AO5) that
\begin{multline*}
(\BBB \ominus \AAA)_{\Delta} = [\EEE_{sm}(\BBB) \ominus \EEE_{sm,1}] \sqplus
\bigl(\Bigsqplus_{\alpha \in \Card_{\infty}} \alpha \odot \EEE^{\tII}_{1,\alpha}\bigr)\\
\sqplus \bigl(\Bigsqplus \{(\beta - \alpha) \odot \EEE^i_{\alpha,\beta}\dd\ (i,\alpha,\beta) \in \Upsilon^2_*,\
\alpha < \beta\}\bigr)
\end{multline*}
where $\beta - \alpha = \beta$ provided $\beta$ is infinite (and $\beta > \alpha$). The above formula may be written
in the following self-contained form:
\begin{multline}\label{eqn:B-Al}
(\BBB \ominus \AAA)_{\Delta} = [\EEE_{sm}(\BBB) \ominus (\EEE_{sm}(\AAA) \wedge \EEE^{\tII}_1(\BBB))]\\
\sqplus \Bigl[\Bigsqplus_{(i,\alpha,\beta) \in \Upsilon^2_+} (\beta - \alpha) \odot (\EEE^i_{\alpha}(\AAA) \wedge
\EEE^i_{\beta}(\BBB))\Bigr]
\end{multline}
where $\Upsilon^2_+ = \{(i,\alpha,\beta) \in \Upsilon^2\dd\ \alpha < \beta,\ (i,\alpha,\beta) \neq (\tII,0,1)\}$.
It is also easy to verify that $(\BBB \ominus \AAA)^{\nabla} = (\BBB \ominus \AAA)_{\Delta} \oplus \XXX$ where
$\XXX = \bigsqplus_{\alpha \in \Card_{\infty}} \alpha \odot [\EEE^I_{\alpha,\alpha} \sqplus \EEE^{\tII}_{\alpha,\alpha}
\sqplus \EEE^{\tIII}_{\alpha,\alpha}]$. Since $\XXX \disj (\BBB \ominus \AAA)_{\Delta}$, the proof of (AO10) is finished.
Recall that we have shown that
\begin{equation}\label{eqn:(B-A)-(B-A)}
(\BBB \ominus \AAA)^{\nabla} \sqminus (\BBB \ominus \AAA)_{\Delta} =
\Bigsqplus_{\substack{\alpha \in \Card_{\infty}\\i \in \{I,\tII,\tIII\}}} \alpha \odot
(\EEE^i_{\alpha}(\AAA) \wedge \EEE^i_{\alpha}(\BBB)).
\end{equation}
In particular, $(\BBB \ominus \AAA)^{\nabla} = (\BBB \ominus \AAA)_{\Delta}$ \iaoi{} $\EEE^i_{\alpha}(\AAA) \disj
\EEE^i_{\alpha}(\BBB)$ for every infinite $\alpha$. This proves
\begin{enumerate}[({A}O1)]\addtocounter{enumi}{10}
\item Whenever $\AAA \leqsl \BBB$, $B \ominus \AAA$ is well defined iff $\EEE^i_{\alpha}(\AAA) \disj \EEE^i_{\alpha}(\BBB)$
   for every $\alpha \in \Card_{\infty}$ and $i \in \{I,\tII,\tIII\}$.
\item $(\BBB \ominus \XXX)_{\Delta} \vee (\XXX \ominus \AAA)_{\Delta} \leqsl (\BBB \ominus \AAA)_{\Delta}$ whenever $\AAA
   \leqsl \XXX \leqsl \BBB$.
\item\label{AO13} For any nonempty set $\{\AAA^{(s)}\}_{s \in S} \subset \CDD_N$ and $\BBB \in \CDD_N$, $\BBB \vee
   (\bigwedge_{s \in S} \AAA^{(s)}) = \bigwedge_{s \in S} (\BBB \vee \AAA^{(s)})$ and $\BBB \wedge (\bigvee_{s \in S}
   \AAA^{(s)}) = \bigvee_{s \in S} (\BBB \wedge \AAA^{(s)})$.
\item\label{AO14} For any nonempty set $\{\AAA^{(s)}\}_{s \in S}$ of $N$-tuples, any $\AAA, \BBB \in \CDD_N$ and each
   $\alpha \in \Card$,
   \begin{equation*}\begin{array}{l c l}
   \alpha \odot (\AAA \wedge \BBB) = (\alpha \odot \AAA) \wedge (\alpha \odot \BBB),&&\\
   \alpha \odot (\bigwedge_{s \in S} \AAA^{(s)}) = \bigwedge_{s \in S} (\alpha \odot \AAA^{(s)}) & \textup{if} &
   \begin{array}{l}\alpha \textup{ is finite \ or}\\
   \forall\, s \in S\dd\ \EEE_{sm}(\AAA^{(s)}) = \zero,\end{array}\\
   \alpha \odot (\bigvee_{s \in S} \AAA^{(s)}) = \bigvee_{s \in S} (\alpha \odot \AAA^{(s)}).&&
   \end{array}\end{equation*}
\end{enumerate}
For proofs of (AO12)--(AO14) see \CORp{minus}, \THMp{Boole} and \PROp{multiply}.

\begin{exm}{inf}
Taking into account (AO14), it seems to be surprising that in general $\alpha \odot (\bigwedge_{s \in S} \AAA^{(s)})$
differs from $\bigwedge_{s \in S} (\alpha \odot \AAA^{(s)})$ for infinite cardinal $\alpha$, even if $S$ is countable. Let
us give a short counterexample for this. Let $\alpha \geqsl \aleph_0$ and $\XXX \in \SsS\MmM_N$ be nontrivial. There is
a sequence $(\AAA^{(n)})_{n=1}^{\infty}$ such that $n \odot \AAA^{(n)} = \XXX$ (see the beginning of Section~14). Then
$\alpha \odot \AAA^{(n)} = \alpha \odot \XXX \neq \zero$, while $\bigwedge_{n=1}^{\infty} \AAA^{(n)} = \zero$.
\end{exm}

(AO12) has an interesting consequence.

\begin{pro}{luboplus}
Let $\AaA, \BbB \subset \CDD_N$ be nonempty sets. Then $\bigvee (\AaA \oplus \BbB) = (\bigvee \AaA) \oplus (\bigvee \BbB)$
and $\bigwedge (\AaA \oplus \BbB) = (\bigwedge \AaA) \oplus (\bigwedge \BbB)$ where $\AaA \oplus \BbB = \{\AAA \oplus
\BBB\dd\ \AAA \in \AaA,\ \BBB \in \BbB\}$.
\end{pro}
\begin{proof}
Since the proof for the l.u.b.'s is much simpler, we shall show only the g.l.b.'s part. It is clear that $(\bigwedge \AaA)
\oplus (\bigwedge \BbB) \leqsl \bigwedge (\AaA \oplus \BbB)$. To see the converse inequality, assume that $\XXX \leqsl
\AAA \oplus \BBB$ for any $\AAA \in \AaA$ and $\BBB \in \BbB$. Fix $\BBB \in \BbB$ and put $\EEE = \bigwedge (\AaA \oplus
\{\BBB\})$. For each $\AAA \in \AaA$ we clearly have $\AAA \oplus \BBB \geqsl \EEE \geqsl \BBB$ and consequently, thanks
to (AO12), $(\EEE \ominus \BBB)_{\Delta} \leqsl [(\AAA \oplus \BBB) \ominus \BBB]_{\Delta} \leqsl \AAA$ where the last
inequality follows from the definition of $[\ldots]_{\Delta}$. So, $(\EEE \ominus \BBB)_{\Delta} \leqsl \bigwedge \AaA$
and therefore $\EEE = (\EEE \ominus \BBB)_{\Delta} \oplus \BBB \leqsl (\bigwedge \AaA) \oplus \BBB$. This shows that
$$
\bigwedge(\AaA \oplus \{\BBB\}) \leqsl \bigl(\bigwedge \AaA\bigr) \oplus \BBB
$$
which yields
\begin{multline*}
\bigwedge(\AaA \oplus \BbB) = \bigwedge_{\BBB \in \BbB} \bigl[\bigwedge(\AaA \oplus \{\BBB\})\bigr] \leqsl
\bigwedge_{\BBB \in \BbB} \bigl[(\bigwedge \AaA) \oplus \BBB\bigr] =\\= \bigwedge \Bigl(\BbB \oplus \bigl\{\bigwedge
\AaA\bigr\}\Bigr) \leqsl \bigl(\bigwedge \BbB\bigr) \oplus \bigl(\bigwedge \AaA\bigr)
\end{multline*}
and we are done.
\end{proof}

\begin{cor}{countable}
Let $\AaA_1,\AaA_2,\AaA_3,\ldots$ be nonempty sets of members of $\CDD_N$ and let $\AaA = \{\bigoplus_{n=1}^{\infty}
\AAA^{(n)}\dd\ \AAA^{(n)} \in \AaA_n\ (n \geqsl 1)\}$. Then $\bigvee \AaA = \bigoplus_{n=1}^{\infty} (\bigvee \AaA_n)$.
\end{cor}
\begin{proof}
It is clear that $\bigvee \AaA \leqsl \bigoplus_{n=1}^{\infty} (\bigvee \AaA_n)$. Conversely, by (AO6),
$\bigoplus_{n=1}^{\infty} (\bigvee \AaA_n) = \bigvee_{n\geqsl1} \bigl[(\bigvee \AaA_1) \oplus \ldots \oplus (\bigvee
\AaA_n)\bigr]$. Now by induction and \PRO{luboplus}, $(\bigvee \AaA_1) \oplus \ldots \oplus (\bigvee \AaA_n) = \bigvee
(\AaA_1 \oplus \ldots \oplus \AaA_n) \leqsl \bigvee \AaA$ and we are done.
\end{proof}

In the next section we shall prove a counterpart of \COR{countable} for uncountable collections of sets of $N$-tuples
(see \THMP{sup-sum}).

\begin{exm}{countsum}
It may be seen surprising that the counterpart of \COR{countable} for infima fails to be true, even if each
of $\AaA_n$'s is a finite collection of minimal normal $N$-tuples. That is, in general $\bigwedge
(\bigoplus_{n=1}^{\infty} \AaA_n)$ differs from $\bigoplus_{n=1}^{\infty} (\bigwedge \AaA_n)$ where
$\bigoplus_{n=1}^{\infty} \AaA_n = \{\bigoplus_{n=1}^{\infty} \AAA^{(n)}\dd\ \AAA^{(n)} \in \AaA_n\}$. Let us
justify this claim.\par
For every $u \in L^{\infty}([0,1])$ we shall write, for simplicity, $\xXX_{u}$ to denote the $N$-tuple
$(M_u,\ldots,M_u)$ where $M_u$ is the multiplication operator by $u$ on $L^2([0,1])$. For each pair $(n,m)$ of naturals
with $1 \leqsl m \leqsl n$ let $j_{n,m}$ be the characteristic function of $[0,1] \setminus [(m-1)/n,m/n]$.
Additionally, let $\id \in L^{\infty}([0,1])$ be the identity map on $[0,1]$. Put $\AAA_{n,m} = \XXX_{j_{n,m}\id}$ and
$\AaA_n = \{\AAA_{n,j}\dd\ j=1,\ldots,n\}$. Then $\AaA_n \subset \MmM\FfF_N$ (because $\WWw'(\xXX_{\id}) = \{M_u\dd\
u \in L^{\infty}([0,1])\}$) and $\bigwedge \AaA_n = \zero$ for every $n \geqsl 1$. However, if $(m_n)_{n=1}^{\infty}$
is any sequence of natural numbers such that $1 \leqsl m_n \leqsl n$, then $\bigoplus_{n=1}^{\infty} \AAA_{n,m_n}
\geqsl \bigvee_{n\geqsl1} \AAA_{n,m_n} = \XXX_{\id}$ (the latter equality holds true since $\bigcup_{n=1}^{\infty}
([0,1] \setminus [(m_n-1)/n,m_n/n])$ is of full Lebesgue measure in $[0,1]$). Consequently, $\bigwedge
(\bigoplus_{n=1}^{\infty} \AaA_n) \geqsl \XXX_{\id} \neq \zero = \bigoplus_{n=1}^{\infty} (\bigwedge \AaA_n)$.
\end{exm}

One may deduce from \EXMp{lim} that the assumption in (AO6) that the power of $S$ is a limit cardinal is essential
(in the next section we shall discuss in details \eqref{eqn:limcard} for sets $S$ whose power is not limit). However,
for semiminimal parts of $N$-tuples a stronger property (than in (AO6)) holds true in general (see below). For simplicity,
for every set $S$ let us denote by $\PPp_f(S)$ and $\PPp_{\omega}(S)$ the families of all finite and, respectively,
countable (finite or infinite) subsets of $S$.

\begin{pro}{SM}
Let $S$ be an infinite set and $\{\AAA^{(s)}\}_{s \in S}$ be an arbitrary collection of $N$-tuples, $\AAA =
\bigoplus_{s \in S} \AAA^{(s)}$ and
$$\AAA' = \bigvee \bigl\{\bigoplus_{s \in S_0} \AAA^{(s)}\dd\ S_0 \in \PPp_f(S)\bigr\}.$$
Then $\EEE_{sm}(\AAA) = \EEE_{sm}(\AAA')$ and $\EEE^i_{\alpha}(\AAA) = \EEE^i_{\alpha}(\AAA')$ for each $(i,\alpha) \in
\Upsilon$ with finite $\alpha$.
\end{pro}
\begin{proof}
It is clear that $\EEE^i_0(\AAA) = \EEE^i_0(\AAA')$ for $i \in \{I,\tII,\tIII\}$. Further, let us prove that
\begin{equation}\label{eqn:aux100}
\EEE^{\tII}_1(\AAA) = \EEE^{\tII}_1(\AAA').
\end{equation}
Since $\AAA' \leqsl \AAA \ll \AAA'$, \EQp{Eii1} shows that $\EEE^{\tII}_1(\AAA) \leqsl \EEE^{\tII}_1(\AAA')$.
Conversely, if $\XXX_a = \EEE(\XXX | \EEE^{\tII}_1(\AAA'))$ for every $\XXX \in \CDD_N$, then $(\AAA')_a = \bigvee
\bigl\{\bigoplus_{s \in S_0} (\AAA^{(s)})_a\dd\ S_0 \in \PPp_f(S)\bigr\}$ and $\AAA_a = \bigoplus_{s \in S}
(\AAA^{(s)})_a$. But $(\AAA')_a = \EEE_{sm}(\AAA') \in \SsS\MmM_N$ and hence $(\AAA')_a = \AAA_a$, thanks to \PRO{fin}
(see below). So, $\AAA_a \in \SsS\MmM_N$ and consequently, again by \eqref{eqn:Eii1}, $\EEE^{\tII}_1(\AAA') \leqsl
\EEE^{\tII}_1(\AAA)$. This proves \eqref{eqn:aux100}.\par
Now we have $\EEE_{sm}(\AAA) = \EEE(\AAA | \EEE^{\tII}_1(\AAA)) = \AAA_a = (\AAA')_a = \EEE_{sm}(\AAA')$.\par
It remains to check that $\EEE^I_n(\AAA) = \EEE^I_n(\AAA')$ for natural $n$. Let $\FFF = \bigsqplus_{n=1}^{\infty}
\EEE^I_n(\AAA)$ and $\FFF' = \bigsqplus_{n=1}^{\infty} \EEE^I_n(\AAA')$. It is enough to show that $\FFF = \FFF'$ which
we leave for the reader (as it is similar to the proof of \eqref{eqn:aux100}).
\end{proof}

The following result is in the same spirit.

\begin{pro}{fin}
Let $S$ be an infinite set, $\{\AAA^{(s)}\}_{s \in S} \subset \CDD_N$ and let
$$\AAA = \bigvee \bigl\{\bigoplus_{s \in S_0} \AAA^{(s)}\dd\ S_0 \in \PPp_f(S)\bigr\}.$$
If $\AAA \in \fIN_N$, then $\AAA = \bigoplus_{s \in S} \AAA^{(s)}$.
\end{pro}
\begin{proof}
Let $\MMm = \WWw'(\aAA)$ and $p_s \in E(\MMm)$ ($s \in S$) correspond to $\AAA^{(s)}$ (by \PROP{transl}). Further, let
$\tr\dd \MMm \to \ZZz(\MMm)$ be the trace on $\MMm$. For every $s \in S$ put $w_s = \tr(p_s)$. Since $\bigoplus_{s \in S_0}
\AAA^{(s)} \leqsl \AAA$ where $S_0 \in \PPp_f(S)$, $\sum_{s \in S_0} w_s \leqsl 1$ and consequently $\sum_{s \in S} w_s$ is
convergent and the sum is no greater than $1$. Recall that for any $q, q' \in E(\MMm)$, $q \preccurlyeq q' \iff \tr(q)
\leqsl \tr(q')$ (see e.g. \cite[Corollary~5.2.8]{tk1} or \cite[Theorem~8.4.3]{k-r2}). This implies that it is possible,
well ordering the set $S$ and using transfinite induction, to construct a family $\{q_s\}_{s \in S}$ of mutually orthogonal
projections in $\MMm$ such that $p_s \sim q_s$ for any $s \in S$. Hence $\sum_{s \in S} q_s \leqsl 1$ which yields
$\bigoplus_{s \in S} \AAA^{(s)} \leqsl \AAA$ and we are done.
\end{proof}

\SECT{Reconstructing infinite operations}

Classical algebraic structures deal with operations on pairs (such as the action of a semigroup). However, some operations
naturally make sense also for infinitely (possibly uncountably) many arguments (e.g. unions of sets) and sometimes it is
necessary to use these extended `infinite' operations in order to understand, formulate or prove some statements. Unless
infinite operations can be `defined' (or characterized) in terms of their finite versions, every such a theorem may be seen
as a result from outside the theory. The most typical example of an infinite operation is the union of a family of sets.
However, it may be characterized by means of the union of two sets. Namely, for any family $\AAa$ put $\AAa^{\Delta} =
\{B\dd\ A \cup B = B \textup{ for each } A \in \AAa\}$ and then $\bigcup \AAa$ is the \textbf{unique} set $B \in
\AAa^{\Delta}$ such that $B \cup C = C$ for any $C \in \AAa^{\Delta}$. This characterization is possible for a one simple
reason: the union coincides with the l.u.b. of the family with respect to the inclusion order which may be defined in terms
of the union of a pair. When we pass to the class $\CDD_N$, the direct sum operation cannot be characterized in a similar
manner, because $\bigoplus_{s \in S} \AAA^{(s)}$ differs, in general, from $\bigvee \{\bigoplus_{s \in S_0} \AAA^{(s)}\dd\
S_0 \textup{ a finite subset of } S\}$. Nevertheless, infinite direct sums may be reconstructed from the finite ones,
and this is the subject of this section. Thus, every result of the paper concerning unitary equivalence classes
of $N$-tuples is a part of the theory which starts with the class $\CDD_N$ and the operation $\CDD_N \times \CDD_N \ni
(\AAA,\BBB) \mapsto \AAA \oplus \BBB \in \CDD_N$. (This refers to the material of Sections~1--17, but does \textbf{not}
to the rest.)\par
Our aim is to show that $\bigoplus_{s \in S} \AAA^{(s)}$ may be `recognized' if the only admissible `tool' in the class
$\CDD_N$ is the direct sum operation of a pair. Below we draw step by step how to do this. Each of the listed steps begins
with the tool which may be defined.\par
\begin{enumerate}[(ST1)]\label{ST*}
\item\label{ST1} `$\zero$': It is the unique member $\AAA$ of $\CDD_N$ such that $\AAA \oplus \XXX = \XXX$ for every $\XXX
   \in \CDD_N$.
\item `$\leqsl$': $\AAA \leqsl \BBB$ iff $\BBB = \AAA \oplus \XXX$ for some $\XXX \in \CDD_N$. Accordingly, the l.u.b.'s
   and g.l.b.'s are well defined.
\item `$\disj$': $\AAA \disj \BBB \iff \AAA \wedge \BBB = \zero$.
\item `$\leqsl^s$': $\AAA \leqsl^s \BBB$ iff $\BBB = \AAA \oplus \XXX$ for some $\XXX$ such that $\XXX \disj \AAA$.
\item `$\bigsqplus$' and `$\sqminus$': $\AAA = \bigsqplus_{s \in S} \AAA^{(s)}$ ($S$ any set) iff $\AAA^{(s)} \disj
   \AAA^{(s')}$ for distinct $s, s' \in S$ and $\AAA = \bigvee_{s \in S} \AAA^{(s)}$; if $\AAA \leqsl^s \BBB$, $\BBB
   \sqminus \AAA$ is the unique $\XXX$ such that $\XXX \disj \AAA$ and $\BBB = \XXX \oplus \AAA$.
\item `$\bigoplus_{n=1}^{\infty}$': $\bigoplus_{n=1}^{\infty} \AAA^{(n)} = \bigvee_{n \geqsl 1} \AAA^{(1)} \oplus \ldots
   \oplus \AAA^{(n)}$. In particular, $\aleph_0 \odot \AAA$ is well defined for each $\AAA$.
\item `$\ll$': $\AAA \ll \BBB$ iff there is no $\XXX \neq \zero$ such that $\XXX \leqsl \AAA$ and $\XXX \disj \BBB$.
\item `$\EEE(\AAA | \BBB)$': $\XXX = \EEE(\AAA | \BBB) \iff \exists\, \YYY\dd\ \AAA = \XXX \oplus \YYY,\ \XXX \ll \BBB$
   and $\YYY \disj \BBB$.
\item `Multiplicity free $N$-tuples': $\AAA \in \MmM\FfF_N$ \iaoi{} there is no $\XXX \neq \zero$ such that $\XXX \oplus
   \XXX \leqsl \AAA$. Accordingly, $\JJJ_I$ is well defined.
\item `Minimal $N$-tuples': $\AAA$ is minimal iff $\AAA \leqsl \XXX$ whenever $\AAA \ll \XXX$.
\item `Hereditary idempotents': $\AAA \in \HhH\IiI_N \iff \BBB = \BBB \oplus \BBB$ for each $\BBB \leqsl \AAA$.
   Accordingly, the class $\HhH\IiI\MmM_N$ and $\JJJ_{\tIII}$ are well defined, thanks to (ST10).
\item `Semiminimal $N$-tuples': Use (ST5), (ST7) and the definition of semiminimality.
\item `$\JJJ_{\tII}$': $\JJJ_{\tII} = \bigvee \{\aleph_0 \odot \XXX\dd\ \XXX \in \SsS\MmM_N\}$ (use (ST6) and (ST12)).
\item `$\alpha \odot \AAA$ for $\AAA \leqsl \JJJ_i$' with $i \in \{I,\tII,\tIII\}$: Thanks to (ST6), we may assume that
   $\alpha > \aleph_0$. If $i = \tII$, $\alpha \odot \AAA = \alpha \odot (\aleph_0 \odot \AAA)$ and $\aleph_0 \odot \AAA
   \leqsl^s \JJJ_{\tII}$; while for $i \neq \tII$ one has $\AAA \leqsl^s \JJJ_i$. These notices show that we may assume
   that $\AAA \leqsl^s \JJJ_i$. Under this assumption, $\alpha \odot \AAA$ may be characterized by means of transfinite
   induction with respect to $\alpha$. When $\alpha$ is limit, it suffices to apply (AO6) \pREF{AO6}. On the other hand,
   if $\alpha = \beta^+$ and $\AAA \neq \zero$,
   \begin{equation*}
   \alpha \odot \AAA = \bigwedge \{\XXX|\quad \forall\, \BBB \leqsl^s \AAA,\ \BBB \neq \zero\dd\
   \beta \odot \BBB \lneqq \EEE(\XXX | \BBB)\}.
   \end{equation*}
   (This formula may be deduced from \THMP{decomp}.)
\item `$\EEE^i_{\alpha}(\AAA)$ and $\EEE_{sm}(\AAA)$': Thanks to \THM{decomp} and previous steps.
\item `$\alpha \odot \AAA$' (arbitrary $\AAA$): Use (ST15), \THM{decomp}, (ST14) and (ST5).
\item\label{ST17} `$\aleph_0 \cdot \dim(\AAA)$': Since
   \begin{multline*}
   \aleph_0 \cdot \dim(\AAA) = \sum_{n=1}^{\infty} \aleph_0 \cdot \dim(\EEE^I_n(\AAA))
   + \aleph_0 \cdot \dim(\EEE^{\tII}_1(\AAA))\\
   + \sum_{\substack{\alpha \in \Card_{\infty}\\i \in \{I,\tII,\tIII\}}} \alpha [\aleph_0 \cdot
   \dim(\EEE^i_{\alpha}(\AAA))]
   \end{multline*}
   (cf. \THM{decomp}; $\dim(\EEE_{sm}(\AAA)) = \dim(\EEE^{\tII}_1(\AAA))$), it suffices
   to characterize the cardinal number $\aleph_0 \cdot \dim(\AAA)$ for $\AAA \leqsl^s \JJJ$. But in that
   case it is quite easy, thanks to \PROp{count}: $\aleph_0 \cdot \dim(\AAA)$ (provided $\AAA \leqsl^s \JJJ$
   is nontrivial) is the least infinite cardinal number $\alpha$ such that each regular subfamily of $\{\XXX\dd\ \zero
   \neq \XXX \leqsl^s \AAA\}$ has power no greater than $\alpha$.
\label{ST+}\end{enumerate}
Now we are ready to characterize infinite direct sums. For simplicity, let us put $\Dim(\FFF) = \aleph_0 \cdot \dim(\FFF)$,
$\EEE_f(\FFF) = \EEE^{\tII}_1(\FFF) \sqplus \bigsqplus_{n=1}^{\infty} \EEE^I_n(\FFF)$ and
$\EEE_{\alpha}(\FFF) = \EEE^I_{\alpha}(\FFF) \sqplus \EEE^{\tII}_{\alpha}(\FFF) \sqplus \EEE^{\tIII}_{\alpha}(\FFF)$ for
$\alpha \in \{0\} \cup \Card_{\infty}$ and any $\FFF \in \CDD_N$. By (ST17) and (ST15), `$\Dim$', `$\EEE_f$' and
`$\EEE_{\alpha}$' are well defined. Fix a collection $\{\AAA^{(s)}\}_{s \in S}$ of $N$-tuples and put $\AAA =
\bigoplus_{s \in S} \AAA^{(s)}$ and $\AAA_f = \bigvee \{\bigoplus_{s \in S'} \AAA^{(s)}\dd\ S' \in \PPp_f(S)\}$. $\AAA_f$
is `known' by (ST2). Thanks to (ST16) and (ST5), it suffices to find $\EEE_{sm}(\AAA)$ and $\EEE^i_{\alpha}(\AAA)$.
By \PRO{SM}, $\EEE^i_{\alpha}(\AAA) = \EEE^i_{\alpha}(\AAA')$ for $(i,\alpha) \in \Upsilon$ with finite $\alpha$ and
$\EEE_{sm}(\AAA) = \EEE_{sm}(\AAA')$. Since $\EEE^i_{\alpha}(\AAA) = \EEE_{\alpha}(\AAA) \wedge \JJJ_i$, we see that
it remains to find $\EEE_{\alpha}(\AAA)$ for infinite $\alpha$ (thanks to (ST9), (ST11) and (ST13)). To this end, let us
show that
\begin{equation}\label{eqn:Ealpha}
\EEE_{\alpha}(\AAA) = \bigvee \XxX
\end{equation}
where
\begin{multline*}
\XxX = \{\XXX \leqsl^s \JJJ \sqminus (\EEE_0(\AAA) \sqplus \EEE_f(\AAA))|\\
\forall\,\YYY \leqsl^s \XXX,\ \Dim(\YYY) = \aleph_0\dd\ \sum_{s \in S} \Dim \EEE(\AAA^{(s)} | \YYY) = \alpha\}.
\end{multline*}
It is clear that $\EEE_{\alpha}(\AAA) \leqsl^s \JJJ \sqminus (\EEE_0(\AAA) \sqplus \EEE_f(\AAA))$. Furthermore, if $\YYY
\leqsl^s \EEE_{\alpha}(\AAA)$, then $\bigoplus_{s \in S} \EEE(\AAA^{(s)} | \YYY) = \EEE(\AAA | \YYY) = \alpha \odot \YYY$
and thus $$\aleph_0 \cdot \sum_{s \in S} \dim(\EEE(\AAA^{(s)} | \YYY)) = \alpha$$ provided $\Dim(\YYY) = \aleph_0$. This
yields $\EEE_{\alpha}(\AAA) \in \XxX$. Consequently, the proof of \eqref{eqn:Ealpha} will be finished if we show that
$\XXX \leqsl \EEE_{\alpha}(\AAA)$ for every $\XXX \in \XxX$. To get this inequality, it is enough to check that $\XXX \disj
\EEE_{\beta}(\AAA)$ for any infinite $\beta \neq \alpha$. Suppose $\YYY' = \XXX \wedge \EEE_{\beta}(\AAA)$ is nontrivial.
Since $\YYY' \leqsl^s \JJJ$, we infer from \PROp{count} that there is $\YYY \leqsl^s \YYY'$ with $\Dim(\YYY) = \aleph_0$.
But then $\YYY \leqsl^s \XXX$ and $\EEE(\AAA | \YYY) = \beta \odot \YYY$ (because $\YYY \leqsl^s \EEE_{\beta}(\AAA)$).
Consequently, $\aleph_0 \cdot \sum_{s \in S} \dim(\EEE(\AAA^{(s)} | \YYY)) = \beta$ which denies the fact that $\XXX \in
\XxX$ and finishes the proof of \eqref{eqn:Ealpha}.\par
The arguments of this section prove

\begin{pro}{fin-infin}
If $\Phi\dd \CDD_N \to \CDD_N$ is a bijective assignment such that $\Phi(\AAA \oplus \BBB) = \Phi(\AAA) \oplus \Phi(\BBB)$
for any $\AAA, \BBB \in \CDD_N$, then $\Phi$ preserves all notions, features and operations appearing
in \textup{(ST1)--(ST17)} and
$$
\Phi\bigl(\bigoplus_{s \in S} \AAA^{(s)}\bigr) = \bigoplus_{s \in S} \Phi(\AAA^{(s)})
$$
for any set $\{\AAA^{(s)}\}_{s \in S} \subset \CDD_N$.
\end{pro}

Let us now discuss the relation between both sides of \eqref{eqn:limcard} for a set $S$ whose cardinality is not limit,
i.e. $\card(S) = \gamma^+$ for some infinite cardinal $\gamma$. Let $\AAA' = \bigvee \{\bigoplus_{s \in S'} \AAA^{(s)}\dd\
S' \subset S,\ 0 < \card(S') \leqsl \gamma\}$. By \PRO{SM}, $\EEE_{sm}(\AAA) = \EEE_{sm}(\AAA')$ and $\EEE^i_{\alpha}(\AAA)
= \EEE^i_{\alpha}(\AAA')$ for every $(i,\alpha) \in \Upsilon$ with finite $\alpha$. These equalities are more general:
we claim that
\begin{equation}\label{eqn:alpha}
\EEE^i_{\alpha}(\AAA) = \EEE^i_{\alpha}(\AAA')
\end{equation}
provided $\alpha \neq \gamma,\gamma^+$. To show this, it suffices to check that $\EEE^i_{\alpha}(\AAA') \leqsl
\EEE^i_{\alpha}(\AAA)$ for $\alpha \neq \gamma,\gamma^+$ and $\EEE^i_{\gamma}(\AAA') \sqplus \EEE^i_{\gamma^+}(\AAA')
\leqsl \EEE^i_{\gamma}(\AAA) \sqplus \EEE^i_{\gamma^+}(\AAA)$ (since we deal with partitions of unity). We need to check
only infinite $\alpha$'s.\par
First assume $\aleph_0 \leqsl \alpha < \gamma$. Fix $\XXX \leqsl^s \EEE^i_{\alpha}(\AAA')$ with $\Dim(\XXX) =
\aleph_0$. Let $$S' = \{s \in S\dd\ \EEE(\AAA^{(s)} | \XXX) \neq \zero\}.$$
If the power of $S'$ was greater than $\alpha$, there would exist a subset $S''$ of $S'$ with $\card(S'') = \alpha^+ \leqsl
\gamma$. Then we would have $\bigoplus_{s \in S''} \AAA^{(s)} \leqsl \AAA'$ and $\alpha \odot \XXX = \EEE(\AAA' | \XXX)
\geqsl \bigoplus_{s \in S''} \EEE(\AAA^{(s)} | \XXX)$ which denies the fact that $\Dim(\alpha \odot \XXX) = \alpha <
\card(S'') \leqsl \Dim(\bigoplus_{s \in S''} \EEE(\AAA^{(s)} | \XXX))$. We infer from this that $\card(S') \leqsl \alpha$.
So, $\EEE(\AAA | \XXX) = \EEE(\AAA' | \XXX)$. Consequently, $\EEE(\AAA | \EEE^i_{\alpha}(\AAA')) = \EEE(\AAA' |
\EEE^i_{\alpha}(\AAA')) = \alpha \odot \EEE^i_{\alpha}(\AAA')$ (thanks to \PROP{count}). This means that
$\EEE^i_{\alpha}(\AAA') \leqsl \EEE^i_{\alpha}(\AAA)$.\par
Now assume that $\alpha \geqsl \gamma^+ = \card(S)$. As before, take $\XXX \leqsl^s \EEE^i_{\alpha}(\AAA')$
with $\Dim(\XXX) = \aleph_0$. Then $\alpha \odot \XXX = \EEE(\AAA' | \XXX) \geqsl \EEE(\AAA^{(s)}
| \XXX)$ for every $s \in S$ and thus $\EEE(\AAA | \XXX) = \bigoplus_{s \in S} \EEE(\AAA^{(s)} | \XXX) \leqsl \card(S)
\odot (\alpha \odot \XXX) = \alpha \odot \XXX = \EEE(\AAA' | \XXX) \leqsl \EEE( \AAA | \XXX)$. So, $\EEE(\AAA | \XXX) =
\EEE(\AAA' | \XXX)$ and consequently $\EEE^i_{\alpha}(\AAA') \leqsl \EEE^i_{\alpha}(\AAA)$.\par
To finish the proof of \eqref{eqn:alpha}, it remains to consider $\alpha = \gamma$. For $\XXX = \EEE^i_{\gamma}(\AAA')$
we readily have $\gamma \odot \XXX = \EEE(\AAA' | \XXX) \leqsl \EEE(\AAA | \XXX) = \bigoplus_{s \in S} \EEE(\AAA^{(s)} |
\XXX) \leqsl \card(S) \odot \EEE(\AAA' | \XXX) = \gamma^+ \odot \XXX$. These inequalities imply that $\XXX \leqsl
\EEE^i_{\gamma}(\AAA) \sqplus \EEE^i_{\gamma^+}(\AAA)$ and we are done.\par
Having \eqref{eqn:alpha}, we obtain that $\EEE^i_{\gamma}(\AAA') \sqplus \EEE^i_{\gamma^+}(\AAA') = \EEE^i_{\gamma}(\AAA)
\sqplus \EEE^i_{\gamma^+}(\AAA)$. We also know that $\EEE^i_{\gamma^+}(\AAA') \leqsl^s \EEE^i_{\gamma^+}(\AAA)$
(by the above argument and \THMP{ords}). So, $\EEE^i_{\gamma^+}(\AAA) \sqminus \EEE^i_{\gamma^+}(\AAA')$ gives
full information about the difference between $\AAA$ and $\AAA'$. We have
\begin{multline}\label{eqn:aux102}
\EEE^i_{\gamma^+}(\AAA) \sqminus \EEE^i_{\gamma^+}(\AAA') = \bigvee \{\YYY \leqsl^s \EEE^i_{\gamma}(\AAA')|\quad
\forall\, \XXX \leqsl^s \YYY,\ \XXX \neq \zero\dd\\
\card(\{s \in S\dd\ \EEE(\AAA^{(s)} | \XXX) \neq \zero\}) = \gamma^+\}.
\end{multline}
Indeed, $\EEE^i_{\gamma^+}(\AAA) \sqminus \EEE^i_{\gamma^+}(\AAA') \leqsl^s \EEE^i_{\gamma}(\AAA')$ and if $\XXX \leqsl^s
\EEE^i_{\gamma^+}(\AAA) \sqminus \EEE^i_{\gamma^+}(\AAA')$ is nontrivial, then there is $\XXX' \leqsl^s \XXX$
with $\Dim(\XXX') = \aleph_0$ (by \PRO{count}). Then $\bigoplus_{s \in S} \EEE(\AAA^{(s)} | \XXX') = \EEE(\AAA | \XXX') =
\gamma^+ \odot \XXX'$ and $$\Dim(\EEE(\AAA^{(s)} | \XXX')) \leqsl \Dim(\EEE(\AAA' | \XXX')) = \Dim(\gamma \odot \XXX')
= \gamma$$ which yields that the set $\{s \in S\dd\ \EEE(\AAA^{(s)} | \XXX') \neq \zero\}$ has power $\gamma^+$.
Conversely, if $\YYY$ is a member of the right-hand side set appearing in \eqref{eqn:aux102}, then necessarily $\YYY \disj
\EEE^i_{\gamma^+}(\AAA')$ and $\YYY \leqsl \EEE^i_{\gamma}(\AAA) \sqplus \EEE^i_{\gamma^+}(\AAA)$. So, we only need to show
that $\YYY \disj \EEE^i_{\gamma}(\AAA)$. Suppose, for the contrary, that $\YYY' = \YYY \wedge \EEE^i_{\gamma}(\AAA)
(\leqsl^s \YYY)$ is nontrivial. Then there is $\XXX \leqsl^s \YYY'$ such that $\Dim(\XXX) = \aleph_0$. Observe
that $\bigoplus_{s \in S} \EEE(\AAA^{(s)} | \XXX) = \EEE(\AAA | \XXX) = \gamma \odot \XXX$ which denies the fact that
$\card(\{s \in S\dd\ \EEE(\AAA^{(s)} | \XXX) \neq \zero\}) = \gamma^+$.\par
One may deduce from \eqref{eqn:alpha} and \eqref{eqn:aux102} that (below we use the notation introduced in this section)
\begin{equation}\label{eqn:g-g+}
\AAA = \AAA' \vee [\gamma^+ \odot (\EEE_{\gamma^+}(\AAA) \sqminus \EEE_{\gamma^+}(\AAA'))].
\end{equation}
The above notices show that \EXMp{lim} demonstrates all reasons for which it may happen that \eqref{eqn:limcard} is false.
We end the section with the announced

\begin{thm}{sup-sum}
Let $S$ be a nonempty set and $\{\AaA_s\}_{s \in S}$ be a collection of nonempty subsets of $\CDD_N$. Then
\begin{equation}\label{eqn:sup-sum}
\bigvee \bigl(\bigoplus_{s \in S} \AaA_s\bigr) = \bigoplus_{s \in S} \bigl(\bigvee \AaA_s\bigr)
\end{equation}
where $\bigoplus_{s \in S} \AaA_s = \{\bigoplus_{s \in S} \XXX^{(s)}\dd\ \XXX^{(s)} \in \AaA_s\ (s \in S)\}$.
\end{thm}
\begin{proof}
The inequality `$\leqsl$' in \eqref{eqn:sup-sum} is clear. We shall prove the converse one by transfinite induction
on $\card(S)$. The cases when $\card(S) < \aleph_0$ or $\card(S) = \aleph_0$ are included in \PROp{luboplus}
and \CORp{countable}, respectively. Assume $\beta$ is an uncountable cardinal such that \eqref{eqn:sup-sum} is fulfilled
provided $\card(S) < \beta$. Now suppose $\card(S) = \beta$. If $\beta$ is limit, the assertion (i.e. the inequality
`$\geqsl$' in \eqref{eqn:sup-sum}) follows from (AO6) \pREF{AO6} and the transfinite induction hypothesis. Thus we may
assume that $\beta = \gamma^+$. Put $\AaA = \bigoplus_{s \in S} \AaA_s$, $\AAA^{(s)} = \bigvee \AaA_s$ ($s \in S$), $\AAA =
\bigoplus_{s \in S} \AAA^{(s)}$ and $\AAA' = \bigvee \{\bigoplus_{s \in S'} \AAA^{(s)}\dd\ S' \subset S,\ 0 < \card(S')
\leqsl \gamma\}$. We infer from the transfinite induction hypothesis that $\AAA' \leqsl \bigvee \AaA$. Hence, according
to \eqref{eqn:g-g+}, we only need to show that $\gamma^+ \odot (\EEE_{\gamma^+}(\AAA) \sqminus \EEE_{\gamma^+}(\AAA'))
\leqsl \bigvee \AaA$. Having in mind the partition of unity induced by $\bigvee \AaA$, we see that the latter inequality
will be fulfilled if only
\begin{equation}\label{eqn:aux39}
\EEE_{\gamma^+}(\AAA) \sqminus \EEE_{\gamma^+}(\AAA') \disj \EEE^i_{\alpha}\bigl(\bigvee \AaA\bigr)
\end{equation}
for every $(i,\alpha) \in \Upsilon$ with $\alpha \leqsl \gamma$. Suppose \eqref{eqn:aux39} is false for some $\alpha \leqsl
\gamma$. Then \PROp{count} implies that there is $\XXX \in \CDD_N$ such that $0 < \dim(\XXX) \leqsl
\aleph_0$,
\begin{equation}\label{eqn:aux37}
\XXX \leqsl^s \EEE^i_{\alpha}\bigl(\bigvee \AaA\bigr) \qquad \textup{and} \qquad \XXX \leqsl^s \EEE_{\gamma^+}(\AAA)
\sqminus \EEE_{\gamma^+}(\AAA').
\end{equation}
The first relation of \eqref{eqn:aux37} yields that $\EEE(\bigvee \AaA | \XXX) \leqsl \alpha \odot \XXX$
(by the characterization of $\EEE^i_{\alpha}(\bigvee \AaA)$ given in \THMP{decomp}). Consequently,
\begin{equation}\label{eqn:aux31}
\dim\Bigl(\EEE\bigl(\bigoplus_{s \in S} \BBB^{(s)} | \XXX\bigr)\Bigr) \leqsl \alpha
\end{equation}
whenever $\BBB^{(s)} \in \AaA_s$ ($s \in S$). But $\EEE(\bigoplus_{s \in S} \BBB^{(s)} | \XXX) = \bigoplus_{s \in S}
\EEE(\BBB^{(s)} | \XXX)$ and thus \eqref{eqn:aux31} changes into $\sum_{s \in S} \dim(\EEE(\BBB^{(s)} |
\XXX)) \leqsl \alpha$. So, we have obtained that whatever $\BBB^{(s)} \in \AaA_s$ we choose,
\begin{equation}\label{eqn:aux30}
\card (\{s \in S\dd\ \BBB^{(s)} \not\disj \XXX\}) \leqsl \gamma.
\end{equation}
However, the second relation of \eqref{eqn:aux37} combined with \eqref{eqn:aux102} yields that the set $S' = \{s \in S\dd\
\AAA^{(s)} \not\disj \XXX\}$ has power $\gamma^+$. Observe that for $s \in S$, if $\YYY \disj \XXX$ for every $\YYY \in
\AaA_s$, then necessarily $\AAA^{(s)} = \bigvee \AaA_s \disj \XXX$ and hence $s \notin S'$. We conclude from this that
for every $s \in S'$ there is $\BBB^{(s)} \in \AaA_s$ such that $\BBB^{(s)} \not\disj \XXX$. Now \eqref{eqn:aux30} denies
the fact that $\card(S') = \gamma^+$. Consequently, \eqref{eqn:aux39} is satisfied and we are done.
\end{proof}

\SECT{Semigroup of semiminimal $N$-tuples}

This section is devoted to a deeper study of $\SsS\MmM_N$. Thanks to (AO4) \pREF{AO4}, $\SsS\MmM_N$ is a \textbf{set} and
$(\SsS\MmM_N,\oplus)$ is a semigroup which may be enlarged to an Abelian group.\par
A similar construction to the following may be found in \cite{e} (see Proposition~1.41 there). Fix nontrivial $\AAA \in
\SsS\MmM_N$. Since $\WWw'(\aAA)$ is type II$_1$, for every $n \geqsl 1$ there is a unique (by (AO1), \PREF{AO1})
$\AAA^{(n)} \in \SsS\MmM_N$ such that $\AAA = n \odot \AAA^{(n)}$. Notation: $\frac1n \odot \AAA = \AAA^{(n)}$. Now if $w$
is a positive rational number and $w = \frac{p}{q}$ with natural coprime $p$ and $q$, we define $w \odot \AAA = p \odot
(\frac1q \odot \AAA)$. Finally, for a positive real number $t$ let
$$
t \odot \AAA = \bigvee \{w \odot \AAA\dd\ w \in \QQQ,\ w \leqsl t\}.
$$
Additionally, put $t \odot \zero = \zero$ for each $t \in \RRR_+$. Using traces on $*$-commutants of semiminimal $N$-tuples
(i.e. on $\WWw'(\aAA)$ for $\AAA \in \SsS\MmM_N$), one shows that for every $s,t \in \RRR_+$ and any $\AAA, \BBB \in
\SsS\MmM_N$,
\begin{enumerate}[(VS1)]
\item $0 \odot \AAA = \zero$; $1 \odot \AAA = \AAA$,
\item $s \odot \AAA \leqsl t \odot \AAA$ provided $s \leqsl t$,
\item\label{VS3} $t \odot \AAA = \bigwedge \{x \odot \AAA\dd\ x > t\}$ and for $t > 0$, $t \odot \AAA = \bigvee \{x \odot
   \AAA\dd\ 0 \leqsl x < t\}$,
\item $(st) \odot \AAA = s \odot (t \odot \AAA)$; $(s + t) \odot \AAA = (s \odot \AAA) \oplus (t \odot \AAA)$,
\item $t \odot (\AAA \oplus \BBB) = (t \odot \AAA) \oplus (t \odot \BBB)$,
\item if `$\sim$' denotes one of `$\leqsl$', `$\leqsl^s$', `$\ll$', `$\disj$' and $t > 0$, then $t \odot \AAA \sim t \odot
   \BBB \iff \AAA \sim \BBB$,
\item if $\AAA = \bigoplus_{s \in S} \AAA^{(s)}$, then $t \odot \AAA = \bigoplus_{s \in S} (t \odot \AAA^{(s)})$
   (this follows from (VS5), (VS6) and \PROP{fin}),
\item $\bB(\AAA) \in \SsS\MmM_N$ and $\bB(t \odot \AAA) = t \odot \bB(\AAA)$,
\item for every sequence $(t_n)_{n=1}^{\infty}$ of nonnegative reals, $(\sum_{n=1}^{\infty} t_n) \odot \AAA =
   \bigoplus_{n=1}^{\infty} t_n \odot \AAA$ (where $\infty$ is identified with $\aleph_0$, if needed).
\end{enumerate}
Now by (VS1), (VS4), (VS5) and (AO4) \pREF{AO4}, there is a real vector space
$$(\EeE_N,+,\cdot) \supset (\SsS\MmM_N,\oplus,\odot).$$
The above inclusion means that the addition and multiplication by reals in $\EeE_N$ extend `$\oplus$' and the just
defined `$\odot$'. $\SsS\MmM_N$ as a subset of $\EeE_N$ is a cone ($\SsS\MmM_N + \SsS\MmM_N \subset \SsS\MmM_N$, $\RRR_+
\cdot \SsS\MmM_N \subset \SsS\MmM_N$ and $\SsS\MmM_N \cap (-\SsS\MmM_N) = \{0\} = \{\zero\}$). We may assume that $\EeE_N
= \SsS\MmM_N - \SsS\MmM_N$. Under such an assumption, we may consider the partial order on $\EeE_N$ induced
by $\SsS\MmM_N$: $\xi_1 \leqsl_{\EeE} \xi_2 \iff \xi_2 - \xi_1 \in \SsS\MmM_N$ ($\xi_1, \xi_2 \in \EeE_N$). It may be
checked that for $\AAA, \BBB \in \SsS\MmM_N$, $\AAA \leqsl_{\EeE} \BBB \iff \AAA \leqsl \BBB$. So, `$\leqsl_{\EeE}$'
extends `$\leqsl$' and therefore we shall omit the subscript `$\EeE$' in `$\leqsl_{\EeE}$' and we shall simply write
`$\leqsl$' instead of `$\leqsl_{\EeE}$'. Since every nonempty subset of $\SsS\MmM_N$ which is upper bounded in $\SsS\MmM_N$
has the l.u.b. (in $\SsS\MmM_N$), $\EeE_N$ is a conditionally complete lattice (which means that every nonempty upper
bounded subset of $\EeE_N$ has the l.u.b. in $\EeE_N$). Our aim is to find a `model' for the lattice $\EeE_N$.\par
From now to the end of the section we fix a representative $\jJJ_{\tII}$ of $\JJJ_{\tII}$, a compact Hausdorff space
$\Omega_{\tII}$ homeomorphic to the Gelfand spectrum of $\ZZz(\WWw'(\jJJ_{\tII}))$ and an isomorphism $\Psi\dd
\ZZz(\WWw'(\jJJ_{\tII})) \to \CCc(\Omega_{\tII})$ of $*$-algebras where $\CCc(\Omega_{\tII})$ is the algebra of all
continuous complex-valued functions on $\Omega_{\tII}$. Every $\AAA \leqsl^s \JJJ_{\tII}$ corresponds to a unique
central projection $z_{\AAA}$ in $\WWw'(\jJJ_{\tII})$. Let $U_{\AAA}$ be a clopen (i.e. simultaneously closed and open)
subset of $\Omega_{\tII}$ whose characteristic function coincides with $\Psi(z_{\AAA})$. $\Omega_{\tII}$ is extremely
disconnected (that is, the closure of every open subset of $\Omega_{\tII}$ is open as well; see \cite[Theorem~5.2.1]{k-r1})
and the assignment $\AAA \mapsto U_{\AAA}$ establishes a one-to-one correspondence between all $N$-tuples $\XXX \in
\SsS\MmM_N^{\infty}$ (where $\SsS\MmM_N^{\infty} = \{\XXX \in \CDD_N\dd\ \XXX \leqsl^s \JJJ_{\tII}\} = \{\aleph_0 \odot
\AAA\dd\ \AAA \in \SsS\MmM_N\}$) and all clopen subsets of $\Omega_{\tII}$. Moreover, for $\AAA, \BBB \in
\SsS\MmM_N^{\infty}$, $\AAA \leqsl^s \BBB \iff U_{\AAA} \subset U_{\BBB}$.\par
For every $\AAA \in \SsS\MmM_N$, $\widetilde{\AAA} = \aleph_0 \odot \AAA \leqsl^s \JJJ_{\tII}$ and thus
$U_{\widetilde{\AAA}}$ makes sense. The latter set is said to be the \textit{support} of $\AAA$ and denoted
by $\supp_{\Omega_{\tII}} \AAA$. There is no difficulty in verifying that $\supp_{\Omega_{\tII}} \AAA \subset
\supp_{\Omega_{\tII}} \BBB$ (respectively $\supp_{\Omega_{\tII}} \AAA \cap \supp_{\Omega_{\tII}} \BBB = \varempty$)
provided $\AAA \ll \BBB$ (respectively $\AAA \disj \BBB$) and $\AAA, \BBB \in \SsS\MmM_N$.\par
The following idea comes from theory of $\WWw^*$-algebras (compare with \cite[Definition~5.6.5]{k-r1}) especially when
working with the so-called extended center valued traces (see the notes on page~329 of \cite{tk1} and Definition~V.2.33
there). We consider the set
\begin{equation*}
\MmM(\Omega_{\tII}) = \{f \in \CCc(\Omega_{\tII},[-\infty,+\infty])\dd\
f^{-1}(\RRR) \textup{ is dense in } \Omega_{\tII}\}.
\end{equation*}
To make the space $\MmM(\Omega_{\tII})$ a real vector space, we need the following well-known result (it follows from
\cite[Corollary~5.2.11]{k-r1} or \cite[Corollary~III.1.8]{tk1}; see also \cite{g-j} for more general results in this
direction).

\begin{lem}{C-S}
If $\XxX$ and $K$ are compact Hausdorff spaces and $\XxX$ is extremely disconnected, then every continuous function
of an arbitrary open dense subset of $\XxX$ into $K$ is extendable to a continuous function of $\XxX$ into $K$.
\end{lem}

Now if $f,g \in \MmM(\Omega_{\tII})$, the set $D = f^{-1}(\RRR) \cap g^{-1}(\RRR)$ is open and dense in $\Omega_{\tII}$
and the function $f\bigr|_D + g\bigr|_D$ is well defined and continuous. Consequently, thanks to \LEM{C-S}, there is
a unique member of $\MmM(\Omega_{\tII})$, which we shall denote by $f + g$, which coincides with the usual sum on $D$.
Similarly one defines $f \cdot g$ and $t \cdot f$ for $t \in \RRR$. We leave this as an exercise that $\MmM(\Omega_{\tII})$
is a real vector space with these operations. Further, we equip $\MmM(\Omega_{\tII})$ with the pointwise order. One may
easily check that $\MmM(\Omega_{\tII})$ is a lattice (i.e. every finite nonempty subset of $\MmM(\Omega_{\tII})$ has
the l.u.b. and g.l.b.). What is more, $\MmM(\Omega_{\tII})$ is conditionally complete, since $\Omega_{\tII}$ is
extremely disconnected (it follows from \cite[Proposition~III.1.7]{tk1}). In the sequel we shall show that $\EeE_N$ and
$\MmM(\Omega_{\tII})$ are lattice-isomorphic. For every $f \in \MmM(\Omega_{\tII})$ let $\supp f$ be the closure of the set
$\{x \in \Omega_{\tII}\dd\ f(x) \neq 0\}$. Since $\Omega_{\tII}$ is extremely disconnected, $\supp f$ is clopen.\par
When $\XxX$ is a clopen subset of $\Omega_{\tII}$, let $\MmM(\Omega_{\tII} | \XxX)$ be the set of all $f \in
\MmM(\Omega_{\tII})$ for which $\supp f \subset \XxX$. $\MmM(\Omega_{\tII} | \XxX)$ is a sublattice
of $\MmM(\Omega_{\tII})$. By $\MmM_+(\Omega_{\tII})$ and $\MmM_+(\Omega_{\tII} | \XxX)$ we denote the cones of nonnegative
elements of the suitable lattices.\par
For the next step of our considerations we need

\begin{lem}{W*}
Let $\Omega$ be the Gelfand spectrum of a commutative $\WWw^*$-algebra. Every dense $\GGg_{\delta}$ subset of $\Omega$
has dense interior. What is more, for each Borel function $f\dd \Omega \to \RRR$ there is an open dense set $D \subset
\Omega$ such that $f\bigr|_D$ is continuous.
\end{lem}

\LEM{W*} follows from \cite[Lemma~5.2.10]{k-r1} combined with Proposition~III.1.15 and Theorem~III.1.17 of \cite{tk1}
(see also the statement preceding Corollary~III.1.16 there).\par
Let $\{f_n\}_{n=1}^{\infty} \subset \MmM_+(\Omega_{\tII})$ be such that $\sum_{k=1}^n f_k \leqsl g$ for some $g \in
\MmM_+(\Omega_{\tII})$ and each $n$. We define $\sum_{n=1}^{\infty} f_n \in \MmM_+(\Omega_{\tII})$ as follows. Let $f\dd
\Omega \to \RRR$ be given by $f(x) = \sum_{n=1}^{\infty} f_n(x)$ provided the latter series is convergent and $f(x) = 0$
otherwise. By \LEM{W*}, there is an open dense subset $D$ of $\Omega_{\tII}$ such that $f\bigr|_D$ is continuous. We define
$\sum_{n=1}^{\infty} f_n \in \MmM_+(\Omega_{\tII})$ as the unique continuous extension of $f\bigr|_D$. Since $f \leqsl g$,
$(\sum_{n=1}^{\infty} f_n)(x) = \sum_{n=1}^{\infty} f_n(x)$ for $x$ belonging to an open dense subset of $\Omega_{\tII}$.
One may check that $\sum_{n=1}^{\infty} f_n = \sup{}_{\MmM(\Omega_{\tII})} \{\sum_{k=1}^n f_k\}_{n \geqsl 1}$.\par
Fix nontrivial $\XXX \in \SsS\MmM_N$. Let $\LlL[\XXX] = \{\FFF \in \SsS\MmM_N\dd\ \FFF \ll \XXX\}$ and $\XxX =
\supp_{\Omega_{\tII}} \XXX$. Then we have

\begin{thm}{TR}
There is a unique operator
$$
\LlL[\XXX] \ni \FFF \mapsto \DiNT{\FFF}{\XXX} \in \MmM_+(\Omega_{\tII} | \XxX)
$$
such that for any $\FFF, \FFF^{(n)} \in \SsS\MmM_N$ ($n = 1,2,3,\ldots$):
\begin{enumerate}[\upshape(TR1)]\addtocounter{enumi}{-1}
\item $\DiNT{\XXX}{\XXX} =$ the characteristic function of $\XxX$,
\item $\supp \DiNT{\FFF}{\XXX} \subset \supp_{\Omega_{\tII}} \FFF$,
\item $\DiNT{\bigl(\bigoplus_{n=1}^{\infty} \FFF^{(n)}\bigr)}{\XXX} = \sum_{n=1}^{\infty} \DiNT{\FFF^{(n)}}{\XXX}$
   if $\bigoplus_{n=1}^{\infty} \FFF^{(n)} \in \SsS\MmM_N$ (see the note preceding the theorem).
\end{enumerate}
Moreover, the above operator has further properties:
\begin{enumerate}[\upshape(TR1)]\addtocounter{enumi}{3}
\item[\upshape(TR1')] $\supp \DiNT{\FFF}{\XXX} = \supp_{\Omega_{\tII}} \FFF$ for every $\FFF \in \LlL[\XXX]$,
\item[\upshape(TR2')] whenever $\AAA \in \LlL[\XXX]$ is of the form $\AAA = \bigoplus_{s \in S} \AAA^{(s)}$,
   $$
   \DiNT{\AAA}{\XXX} = \sum_{s \in S} \DiNT{\AAA^{(s)}}{\XXX} := \sup{}_{\MmM(\Omega_{\tII} | \XxX)}
   \bigl\{\sum_{s \in S_0} \DiNT{\AAA^{(s)}}{\XXX}\dd\ S_0 \in \PPp_f(S)\bigr\},
   $$
\item $\DiNT{(t \odot \FFF)}{\XXX} = t \DiNT{\FFF}{\XXX}$ for each $\FFF \in \LlL[\XXX]$,
\item for any $\AAA, \BBB \in \LlL[\XXX]$, $\AAA \leqsl \BBB \iff \DiNT{\AAA}{\XXX} \leqsl \DiNT{\BBB}{\XXX}$,
\item for every $f \in \MmM_+(\Omega_{\tII} | \XxX)$ there is unique $\FFF \in \LlL[\XXX]$ with $\DiNT{\FFF}{\XXX} = f$.
\end{enumerate}
\end{thm}
\begin{proof}
The existence of the operator may be deduced from the result on faithful normal extended center valued traces for
semifinite $\WWw^*$-algebras (\cite[Theorem~V.2.34]{tk1}) applied to $\WWw'(\yYY)$ with $\YYY = \aleph_0 \odot \XXX$.
The operator may also be constructed as follows. By \THMp{semiminimal}, there is $\XXX' \in \SsS\MmM_N$ such that
$\JJJ_{\tII} = \aleph_0 \odot (\XXX \sqplus \XXX')$. Put $\widetilde{\XXX} = \XXX \sqplus \XXX'$ and let $\MMm =
\WWw''(\widetilde{\xXX})$. Then $\MMm' = \WWw'(\widetilde{\xXX})$. Since $\widetilde{\XXX} \in \SsS\MmM_N$, $\MMm'$ is type
II$_1$ and hence there is a trace $\tr\dd \MMm' \to \ZZz(\MMm')$ (with $\tr(1) = 1$). Since $\ZZz(\MMm') = \ZZz(\MMm)$,
$\ZZz(\WWw''(\jJJ_{\tII})) = \ZZz(\WWw'(\jJJ_{\tII}))$ and a function $\MMm \ni T \mapsto \aleph_0 \odot T
\in \WWw''(\jJJ_{\tII})$ is an isomorphism of $*$-algebras, a function $\kappa\dd \ZZz(\MMm') \ni T \mapsto \aleph_0 \odot
T \in \ZZz(\WWw'(\jJJ_{\tII}))$ is a well defined $*$-isomorphism. Define $\Tr\dd \MMm' \to \CCc(\Omega_{\tII})$
by $\Tr = \Psi \circ \kappa \circ \tr$. Now if $\AAA \ll \XXX$, by \DEFp{semiminimal}, $\AAA$ may be written in the form
$\AAA = \bigsqplus_{n=1}^{\infty} \AAA^{(n)}$ with $\AAA^{(n)} \leqsl n \odot \XXX$. The latter implies that $\frac1n \odot
\AAA^{(n)} \leqsl \widetilde{\XXX}$ and thus there is a projection $p_n$ in $\MMm'(\widetilde{\xXX})$ which corresponds
(by \PROP{transl}) to $\frac1n \odot \AAA^{(n)}$. We put
\begin{equation}\label{eqn:aux105}
\DiNT{\AAA}{\XXX} = \sum_{n=1}^{\infty} n \Tr(p_n).
\end{equation}
Since $\AAA^{(n)} \disj \AAA^{(m)}$ for $n \neq m$, $\supp \Tr(p_n) \cap \supp \Tr(p_m) = \varempty$ and thus
\eqref{eqn:aux105} is well understood, by \LEM{C-S}. We leave this as an exercise that such a definition is independent
of the choice of $(\AAA^{(n)})_{n=1}^{\infty}$ and that all conditions of the theorem are fulfilled (observe that $\XXX$
corresponds to a central projection in $\MMm'(\widetilde{\xXX})$). Here we focus on the uniqueness of the operator.\par
When $\AAA \leqsl^s \XXX$, then $\XXX = \AAA \sqplus \BBB$ with $\BBB = \XXX \sqminus \AAA$ and $\supp_{\Omega_{\tII}} \BBB
= \supp_{\Omega_{\tII}} \XXX \setminus \supp_{\Omega_{\tII}} \AAA$. Consequently, by (TR0)--(TR2), $\DiNT{\AAA}{\XXX}$ is
the characteristic function of $\supp_{\Omega_{\tII}} \AAA$. This shows that $\DiNT{\FFF}{\XXX}$ is uniquely determined
by (TR0)--(TR2) for $\FFF \in \FfF_0 := \{w \odot \AAA\dd\ w \in \QQQ_+,\ \AAA \leqsl^s \XXX\}$.\par
Further, if $\AAA \leqsl \XXX$, then $\AAA$ may be written in the form $\AAA = \bigoplus_{n=1}^{\infty} \FFF^{(n)}$ with
$\FFF^{(n)} \in \FfF_0$ (this may be deduced, by means of the trace, from the representation of a continuous function
on an extremely disconnected compact Hausdorff space as a series of continuous functions with finite ranges). So, according
to (TR2), $\DiNT{\BBB}{\XXX}$ is uniquely determined by (TR0)--(TR2) for $\BBB = w \odot \AAA$ with rational $w$ and $\AAA
\leqsl \XXX$. To this end, it suffices to recall that if $\AAA \in \LlL[\XXX]$, then $\AAA = \bigsqplus_{n=1}^{\infty}
\AAA^{(n)}$ with $\frac1n \odot \AAA^{(n)} \leqsl \XXX$.
\end{proof}

\begin{cor}{lattiso}
$\EeE_N$ and $\MmM(\Omega_{\tII})$ are isomorphic as ordered vector spaces.
\end{cor}
\begin{proof}
Take $\XXX \in \SsS\MmM_N$ such that $\aleph_0 \odot \XXX = \JJJ_{\tII}$ and define $\Phi_+\dd \SsS\MmM_N \ni \FFF \mapsto
\DiNT{\FFF}{\XXX} \in \MmM_+(\Omega_{\tII})$. By \THM{TR}, $\Phi_+$ is an additive bijection preserving orders. Now it
suffices to extend $\Phi_+$ in a standard way: $\Phi(\xi) = \Phi(\xi_+) - \Phi(\xi_-)$.
\end{proof}

\begin{pro}{dint}
If $\AAA, \XXX, \YYY \in \SsS\MmM_N$ are such that $\AAA \ll \XXX \ll \YYY$, then
\begin{equation}\label{eqn:var}
\DiNT{\AAA}{\YYY} = \DiNT{\AAA}{\XXX} \cdot \DiNT{\XXX}{\YYY}.
\end{equation}
\end{pro}
\begin{proof}
Arguing as in the uniqueness part of \THM{TR}, we only need to check that \eqref{eqn:var} is fulfilled for $\AAA \leqsl^s
\XXX$. When $\AAA = \XXX$, \eqref{eqn:var} is clear. So, for arbitrary $\AAA \leqsl^s \XXX$, \eqref{eqn:var} follows from
(TR1) and (TR2).
\end{proof}

We end the section with the following two remarks.

\begin{rem}{integr}
The notation `$\DiNT{\AAA}{\XXX}$' suggests to denote the inverse operator, of $\MmM_+(\Omega_{\tII} | \XxX)$ onto
$\LlL[\XXX]$, by $\int f \dint{\XXX}$. Thus, for $f \in \MmM_+(\Omega_{\tII} | \XxX)$, $\int f \dint{\XXX} = \BBB$ iff
$\BBB \in \LlL[\XXX]$ is such that $\DiNT{\BBB}{\XXX} = f$. Arguing similarly as in the proof of \THM{TR}, one may show
that the operator $\MmM_+(\Omega_{\tII} | \XxX) \ni f \mapsto \int f \dint{\XXX} \in \LlL[\XXX]$ is uniquely determined
by the following three conditions:
\begin{enumerate}[({A}D1)]
\item $\int j_{\XxX} \dint{\XXX} = \XXX$ where $j_{\XxX}$ is the characteristic function of $\XxX$,
\item $\supp_{\Omega_{\tII}} (\int f \dint{\XXX}) \subset \supp f$ for each $f \in \MmM_+(\Omega_{\tII} | \XxX)$,
\item whenever $f \in \MmM_+(\Omega_{\tII} | \XxX)$ has the form $f = \sum_{n=1}^{\infty} f_n$ (with $f_n \in
   \MmM_+(\Omega_{\tII} | \XxX)$),
   $$\int f \dint{\XXX} = \bigoplus_{n=1}^{\infty} \int f_n \dint{\XXX}.$$
\end{enumerate}
It seems to be interesting that (AD3) reminds classical Lebesgue's monotone convergence theorem.
\end{rem}

\begin{rem}{selfadj}
Specialists in Hilbert space operators probably would prefer the version of `$\DiNT{\YYY}{\XXX}$' whose values are
operators rather than functions. This is possible and may be provide as follows. Since every bounded member
of $\MmM_+(\Omega_{\tII})$ corresponds, by $\Psi$, to a nonnegative element of $\ZZz(\WWw'(\jJJ_{\tII}))$, each member
of $\MmM_+(\Omega_{\tII})$ corresponds to a (possibly unbounded) nonnegative selfadjoint operator $A$ such that $\bB(A) \in
\ZZz(\WWw'(\jJJ_{\tII}))$ (in von Neumann algebra theory such an operator $A$ is said to be \textit{affiliated} with
$\ZZz(\WWw'(\jJJ_{\tII}))$; see e.g. \cite[Definition~5.6.2]{k-r1}). Thus, if we let $\lLL[\xXX]$
and $\widehat{\ZZz}_+(\WWw'(\jJJ_{\tII}))$ denote, respectively, the classes of all $\yYY \in \CDDc_N$ whose unitary
equivalence class is semiminimal and which are covered by $\xXX$ (i.e. $\yYY \ll \xXX$) and of all last mentioned operators
$A$, then \THM{TR} may be adapted to these settings in such a way that $\DiNT{\yYY}{\xXX} \in
\widehat{\ZZz}_+(\WWw'(\xXX))$ for any $\yYY \in \lLL[\xXX]$ and (here we list only these properties which do not need
additional explanations): (a) $\DiNT{\xXX}{\xXX}$ is the unit of $\ZZz(\WWw'(\xXX))$ (so, $\DiNT{\xXX}{\xXX}$ is a central
projection in $\WWw'(\jJJ_{\tII})$); (b) $\DiNT{\yYY'}{\xXX} = \DiNT{\yYY''}{\xXX}$ iff $\yYY'$ and $\yYY''$ are unitarily
equivalent; (c) if $\aAA \leqsl m \odot \xXX$ and $\bBB \leqsl n \odot \xXX$ for some natural numbers $m$ and $n$, then
both $\DiNT{\aAA}{\xXX}$ and $\DiNT{\bBB}{\xXX}$ are bounded and $\DiNT{(\aAA \oplus \bBB)}{\xXX} = \DiNT{\aAA}{\xXX} +
\DiNT{\bBB}{\xXX}$; (d) if $\yYY_t$ is such that $\YYY_t = t \odot \YYY$ (for some $\yYY \in \lLL[\xXX]$ and $t > 0$), then
$\DiNT{\yYY_t}{\xXX} = t \DiNT{\yYY}{\xXX}$. The reader interested in this approach should consult
\cite[Theorem~5.6.15]{k-r1}.
\end{rem}

\SECT{Model for $\CDD_N$}

Now we shall develop the idea of the previous part. This will also be an adaptation of the dimension theory for
$\WWw^*$-algebras. Let $\jJJ$ be a representative of $\JJJ$, $\Omega$ be a compact Hausdorff space homeomorphic
to the Gelfand spectrum of $\ZZz(\WWw'(\jJJ))$ and let $\Psi\dd \ZZz(\WWw'(\jJJ)) \to \CCc(\Omega)$ be an isomorphism
of $*$-algebras. When the triple $(\jJJ,\Omega,\Psi)$ is fixed, $\JJJ_i$ for $i = I,\tII,\tIII$ corresponds to a clopen
subset $\Omega_i$ of $\Omega$. In what follows, we assume that $\Card \cap \RRR_+ = \ZZZ \cap \RRR_+$. We add and multiply
two reals and two infinite cardinals in the usual way and additionally, we put $0 \cdot \alpha = \alpha \cdot 0 = 0$ and
$t + \alpha = \alpha + t = \alpha + 0 = 0 + \alpha = \alpha = t \cdot \alpha = \alpha \cdot t$ for $t \in \RRR_+ \setminus
\{0\}$ and $\alpha \in \Card_{\infty}$. We also extend the natural total orders on $\RRR_+$ and $\Card_{\infty}$ assuming
that $t < \alpha$ for every real $t$ and an infinite cardinal $\alpha$. In this way the order on $\RRR_+ \cup \Card$ is
total and complete. We equip every set $Y \subset \RRR_+ \cup \Card$ with the topology inherited from the one
of the linearly ordered space $I_{\alpha} := \{\xi \in \RRR_+ \cup \Card\dd\ \xi \leqsl \alpha\}$ where $\alpha = \sup
(Y \cup \{\aleph_0\})$ (cf. \cite[Problem~1.7.4]{eng}). Since the topology of the linearly ordered space $I_{\alpha}$
coincides with the one inherited from the topology of the linearly ordered space $I_{\beta}$ whenever $\aleph_0 \leqsl
\alpha < \beta$, such a definition of the topology on $Y$ makes no confusion. For every topological space $X$, we call
a function $f\dd X \to \RRR_+ \cup \Card$ continuous if $f$ is continuous as a function of $X$ into $f(X)$. One may check
that for every $\alpha \in \Card_{\infty}$, $I_{\alpha}$ is compact, the order is a closed subset of $I_{\alpha} \times
I_{\alpha}$ and functions $I_{\alpha} \times I_{\alpha} \ni (\xi,\xi') \mapsto \xi + \xi' \in I_{\alpha}$ and $I_{\alpha}
\times I_{\alpha} \ni (\xi,\xi') \mapsto \xi \cdot \xi' \in I_{\alpha}$ are continuous.\par
Let $\Lambda(\Omega)$ be the class of all continuous functions $u\dd \Omega \to \RRR_+ \cup \Card$ such that $u(\Omega_I)
\subset \Card$ and $u(\Omega_{\tIII}) \subset \{0\} \cup \Card_{\infty}$. We add and multiply members of $\Lambda(\Omega)$
pointwisely. We shall also multiply elements of $\Lambda(\Omega)$ by cardinal numbers pointwisely and we equip
$\Lambda(\Omega)$ with the pointwise order. For each $f \in \Lambda(\Omega)$, $\supp f$ is the closure of the (open) set
$\{x \in \Omega\dd\ f(x) \neq 0\}$. Observe that $\supp f$ is clopen.\par
Suppose\label{disjsum} $\{f_s\}_{s \in S} \subset \Lambda(\Omega)$ is any family such that $\supp f_s \cap \supp f_{s'}
= \varempty$ for distinct $s, s' \in S$. We define $\sum_{s \in S} f_s \in \Lambda(\Omega)$ in the following manner. Let
$D_0 = \bigcup_{s \in S} \supp f_s$, $D = D_0 \cup \intt (\Omega \setminus D_0)$ (`$\intt$' denotes the interior of a set)
and $u\dd D \to \RRR_+ \cup \Card$ be given by $u(x) = f_s(x)$ for $x \in \supp f_s$ ($s \in S$) and $u(x) = 0$ for $x \in
\intt (\Omega \setminus D_0)$. It is clear that $D$ is open and dense in $\Omega$ and $u$ is continuous. Now by \LEM{C-S},
$u$ may be (uniquely) continuously extended to a member of $\Lambda(\Omega)$, denoted by $\sum_{s \in S} f_s$. One may
check that in that case $\sum_{s \in S} f_s = \sup{}_{\Lambda(\Omega)} \{\sum_{s \in S_0} f_s\dd\ S_0 \in \PPp_f(S)\}$.

\begin{lem}{sup-inf}
Let $\{f_n\}_{n=1}^{\infty} \subset \Lambda(\Omega)$ and $u, v\dd \Omega \to \RRR_+ \cup \Card$ be given by $u(x) =
\inf_{n\geqsl1} f_n(x)$ and $v(x) = \sup_{n\geqsl1} f_n(x)$ ($x \in \Omega$). There are open dense subsets $U$ and $V$
of $\Omega$ such that $u\bigr|_U$ and $v\bigr|_V$ are continuous.
\end{lem}
\begin{proof}
Since the proofs for $u$ and $v$ differ, we shall present both of them. We start with $u$ for which the proof is simpler.
Let $U_0 = u^{-1}(\RRR_+)$ and for $\alpha \in \Card_{\infty}$ let $U_{\alpha} = \intt u^{-1}(\{\alpha\})$. Since $U_0 =
\bigcup_{n=1}^{\infty} f_n^{-1}(\RRR_+)$, $U_0$ is open. Now a function $u'\dd \Omega \to \RRR_+$ given by $u'(x) = u(x)$
for $x \in U_0$ and $u'(x) = 0$ otherwise is Borel (because on $U_0$ it coincides with the infimum of a sequence
of continuous functions taking values in $[0,\infty]$, after suitable change of $f_n$'s). Thus, according to \LEM{W*},
there is a dense open subset $U'$ of $\Omega$ such that $u'\bigr|_{U'}$ is continuous. Consequently, $u\bigr|_{U_1}$ is
continuous where $U_1 = U_0 \cap U'$ is open and dense in $U_0$. We see that $U = U_1 \cup
\bigcup_{\alpha\in\Card_{\infty}} U_{\alpha}$ is open and $u\bigr|_U$ is continuous. To show that $U$ is dense in $\Omega$,
it remains to check that the set $G = \intt [\Omega \setminus (U_0 \cup \bigcup_{\alpha\in\Card_{\infty}} U_{\alpha})]$
is empty. Suppose, for the contrary, that $G \neq \varempty$. Note that $G$ is clopen and $u(G) \subset \Card_{\infty}$.
Let $\alpha = \min u(G) \geqsl \aleph_0$. We conclude from the definition of $u$ that $f_n(x) \geqsl \alpha$ for all $x \in
G$ and $n \geqsl 1$. What is more, there is $x_0 \in G$ such that $u(x_0) = \alpha$ and there exists $m \geqsl 1$ with
$u(x_0) = f_m(x_0)$. Since $\alpha$ is an isolated point of $f_m(G)$, the set $G_0 = f_m^{-1}(\{\alpha\}) \cap G$ is clopen
(and nonempty). We see that then $u(x) = \alpha$ for each $x \in G_0$ and hence $G_0 \subset U_{\alpha}$ which denies
the definition of $G$. This finishes the proof for $u$.\par
To show the assertion for $v$, we begin similarly: let $F = v^{-1}(I_{\aleph_0})$ and $V_{\infty} =
\bigcup_{\alpha\in\Card_{\infty}} \intt v^{-1}(\{\alpha\})$. The set $F$ is closed since $F = \bigcap_{n=1}^{\infty}
f_n^{-1}(I_{\aleph_0})$. We claim that
\begin{equation}\label{eqn:aux90}
F \cup \cll V_{\infty} = \Omega
\end{equation}
(`$\cll$' stands for the closure of a set). Again, we argue by contradiction. Suppose that the set $D = \Omega \setminus
(F \cup \cll V_{\infty})$ is nonempty. Since $D$ is open, there is a clopen set $G \neq \varempty$ such that $G \subset
D$. Notice that $v(G) \subset \Card_{\infty} \setminus \{\aleph_0\}$. Let $\gamma$ be the first infinite cardinal number
such that
\begin{equation}\label{eqn:aux91}
\intt[G \cap v^{-1}(I_{\gamma})] \neq \varempty.
\end{equation}
Let $W$ be any nonempty clopen subset of $G \cap v^{-1}(I_{\gamma})$. Let us show that
\begin{equation}\label{eqn:aux92}
\gamma = \sup \{\sup f_n(W)\dd\ n \geqsl 1\} = \sup v(W) > \aleph_0.
\end{equation}
Put $\gamma' = \sup \{\sup f_n(W)\dd\ n \geqsl 1\}$. It is clear that $\gamma' \leqsl \sup v(W) \leqsl \gamma$ (because
$v(W) \subset I_{\gamma}$). On the other hand, by the definition of $v$, $v(x) \leqsl \gamma'$ for each $x \in W$ which
yields that $\gamma' > \aleph_0$ (since $W \subset G$) and $W \subset v^{-1}(I_{\gamma'}) \cap G$. We infer from this and
the definition of $\gamma$ that $\gamma \leqsl \gamma'$. This proves \eqref{eqn:aux92}.\par
Now let $W_0$ be an arbitrary nonempty clopen subset of $G \cap v^{-1}(I_{\gamma})$ (cf. \eqref{eqn:aux91}). Put $Z = W_0
\cap \bigcup_{n=1}^{\infty} f_n^{-1}(\{\gamma\})$. $Z$ is $\FFf_{\sigma}$ and, by Baire's theorem, $\intt Z = \varempty$
(because, thanks to \eqref{eqn:aux92}, $\intt [W_0 \cap f_n^{-1}(\{\gamma\})] \subset \intt v^{-1}(\{\gamma\}) \subset
V_{\infty}$ and $W_0 \cap V_{\infty} = \varempty$). An application of \LEM{W*} shows that $\intt (\cll Z) = \varempty$.
This implies that there is a nonempty clopen set $W \subset W_0 \setminus Z$. We conclude from the definition of $Z$ that
$f_n(x) < \gamma$ for any $x \in W$ and $n \geqsl 1$. But since $W$ is compact, $f_n$ assumes its maximum on $W$ and
consequently $\gamma_n := \max(\aleph_0,\sup f_n(W)) < \gamma$. Now by \eqref{eqn:aux92},
\begin{equation}\label{eqn:aux93}
\sup_{n\geqsl1} \gamma_n = \gamma.
\end{equation}
Further, by the minimality of $\gamma$, each of the sets $G_n = G \cap v^{-1}(I_{\gamma_n})$ has empty interior.
Moreover, $G_n$'s are closed ($G_n = G \cap \bigcap_{k=1}^{\infty} f_k^{-1}(I_{\gamma_n})$). Consequently, another
applications of Baire's theorem and \LEM{W*} give $\intt[\cll(G_{\infty})] = \varempty$ where $G_{\infty} =
\bigcup_{n=1}^{\infty} G_n$. But $G_{\infty} = G \cap v^{-1}(I_{\gamma} \setminus \{\gamma\})$ (by \eqref{eqn:aux93}).
Finally, by \eqref{eqn:aux91}, we obtain
$$
\intt[G \cap v^{-1}(\{\gamma\})] = \intt(G \cap v^{-1}(I_{\gamma}) \setminus G_{\infty}) \supset
\intt[G \cap v^{-1}(I_{\gamma})] \setminus \cll G_{\infty} \neq \varempty
$$
which denies the fact that $G \cap V_{\infty} = \varempty$. This finishes the proof of \eqref{eqn:aux90}.\par
Relation \eqref{eqn:aux90} means that the set $E = \Omega \setminus \cll V_{\infty}$ is contained in $F$ and consequently
$v(E) \subset I_{\aleph_0}$. Observe that $E$ is clopen and $I_{\aleph_0}$ is both homeomorphic and order-isomorphic to
$[0,1]$. Therefore $v\bigr|_E$ is Borel and by \LEM{W*} there is an open dense subset $V_0$ of $E$ such that
$v\bigr|_{V_0}$ is continuous. To end the proof, put $V = V_0 \cup V_{\infty}$.
\end{proof}

Now assume\label{countsum} $(f_n)_{n=1}^{\infty}$ is a sequence of members of $\Lambda(\Omega)$. Let $v\dd \Omega \ni x
\mapsto \sum_{n=1}^{\infty} f_n(x) \in \RRR_+ \cup \Card$. (The series $\sum_{n=1}^{\infty} f_n(x)$ is understood
as the supremum of its partial sums.) It is clear that $v(\Omega_I) \subset \Card$ and $v(\Omega_{\tIII}) \subset \{0\}
\cup \Card_{\infty}$. By \LEM{sup-inf}, there is an open dense subset $D$ of $\Omega$ such that $v\bigr|_D$ is continuous.
Consequently, thanks to \LEMp{C-S}, there is a unique $\tilde{v} \in \Lambda(\Omega)$ which extends $v\bigr|_D$. This
unique extension $\tilde{v}$ will be denoted by $\sum_{n=1}^{\infty} f_n$. One may check that $\sum_{n=1}^{\infty} f_n =
\sup{}_{\Lambda(\Omega)} \bigl\{\sum_{k=1}^n f_k\dd\ n \geqsl 1\bigr\}$.\par
Now let $\AAA \in \CDD_N$. Put
\begin{equation}\label{eqn:s(A)}
s(\AAA) = \JJJ \sqminus (\EEE^I_0(\AAA) \sqplus \EEE^{\tII}_0(\AAA) \sqplus \EEE^{\tIII}_0(\AAA)).
\end{equation}
Since $s(\AAA) \leqsl^s \JJJ$, $s(\AAA)$ corresponds to a unique central projection $z_{\AAA} \in \MMm'(\jJJ)$. There is
a unique clopen set, denoted by $\supp_{\Omega} \AAA$, in $\Omega$ whose characteristic function coincides with
$\Psi(z_{\AAA})$. It is clear that for $\AAA, \BBB \in \CDD_N$, $\AAA \ll \BBB \iff \supp_{\Omega} \AAA \subset
\supp_{\Omega} \BBB$; and $\AAA \disj \BBB \iff \supp_{\Omega} \AAA \cap \supp_{\Omega} \BBB = \varempty$. When $\XXX,
\YYY \in \SsS\MmM_N$ are such that $\XXX \ll \YYY$, $u = \DiNT{\XXX}{\YYY}$ is defined on $\Omega_{\tII}$ and real-valued
on an open dense subset $D$ of $\Omega_{\tII}$. Extending $u\bigr|_D$ to a continuous function of $\Omega$ into
$I_{\aleph_0}$ by putting zero on $\Omega_I \cup \Omega_{\tIII}$ and applying \LEM{C-S}, we may consider
$\DiNT{\XXX}{\YYY}$ as a member of $\Lambda(\Omega)$, as it is done in this section. Under such an understanding,
\begin{multline}\label{eqn:X|Y}
\Bigl\{\DiNT{\XXX}{\YYY}\dd\ \XXX \in \SsS\MmM_N,\ \XXX \ll \YYY\Bigr\} = \{u \in \Lambda(\Omega)\dd\\
\supp u \subset \supp_{\Omega} \YYY,\ u^{-1}(\RRR_+) \textup{ is dense in } \Omega\}
\end{multline}
(by \THM{TR}). Since the addition is continuous on $I_{\aleph_0}$, $\DiNT{(\XXX' \oplus \XXX'')}{\YYY} = \DiNT{\XXX'}{\YYY}
+ \DiNT{\XXX''}{\YYY}$ whenever $\XXX', \XXX'' \ll \YYY$.\par
Everywhere below $j_E$ denotes the characteristic function of a set $E \subset \Omega$.

\begin{thm}{model}
Let $\TTT \in \SsS\MmM_N$ be such that $\aleph_0 \odot \TTT = \JJJ_{\tII}$ (there exists such $\TTT$). There is a unique
assignment $\Phi_{\TTT}\dd \CDD_N \to \Lambda(\Omega)$ such that
\begin{enumerate}[\upshape(D1)]\addtocounter{enumi}{-1}
\item $\Phi_{\TTT}(\TTT) = j_{\Omega_{\tII}}$, $\Phi_{\TTT}(\JJJ_I) = j_{\Omega_I}$
   and $\Phi_{\TTT}(\JJJ_{\tIII}) = \aleph_0 \cdot j_{\Omega_{\tIII}}$,
\item $\supp \Phi_{\TTT}(\AAA) \subset \supp_{\Omega} \AAA$ for each $\AAA \in \CDD_N$,
\item $\Phi_{\TTT}(\alpha \odot \AAA) = \alpha \cdot \Phi_{\TTT}(\AAA)$ for any $\alpha \in \Card$ and $\AAA \in
   \CDD_N$,
\item whenever $\{\AAA^{(s)}\}_{s \in S} \subset \CDD_N$ is a regular family (cf. \textup{(D1)} and notes
   on page~\pageref{disjsum}),
   $$\Phi_{\TTT}\Bigl(\Bigsqplus_{s \in S} \AAA^{(s)}\Bigr) = \sum_{s \in S} \Phi_{\TTT}(\AAA^{(s)}),$$
\item whenever $(\AAA^{(n)})_{n=1}^{\infty} \subset \CDD_N$ is such that $\bigoplus_{n=1}^{\infty} \AAA^{(n)} \in
   \SsS\MmM_N$ (see notes on page~\pageref{countsum}), $$\Phi_{\TTT}\Bigl(\bigoplus_{n=1}^{\infty} \AAA^{(n)}\Bigr) =
   \sum_{n=1}^{\infty} \Phi_{\TTT}(\AAA^{(n)}).$$
\end{enumerate}
What is more, $\Lambda(\Omega)$ is order-complete and $\Phi_{\TTT}$ has further properties (below $\AAA, \BBB \in \CDD_N$):
\begin{enumerate}[\upshape(D1)]\addtocounter{enumi}{4}
\item[\upshape(D1')] $\supp \Phi_{\TTT}(\AAA) = \supp_{\Omega} \AAA$; in particular, $\AAA \ll \BBB$ ($\AAA
   \disj \BBB$) iff $\supp_{\Omega} \Phi_{\TTT}(\AAA) \subset \supp_{\Omega} \Phi_{\TTT}(\BBB)$
   ($\supp_{\Omega} \Phi_{\TTT}(\AAA) \cap \supp_{\Omega} \Phi_{\TTT}(\BBB) = \varempty$),
\item[\upshape(D4')] for any sequence $(\AAA^{(n)})_{n=1}^{\infty} \subset \CDD_N$,
   $$\Phi_{\TTT}\Bigl(\bigoplus_{n=1}^{\infty} \AAA^{(n)}\Bigr) = \sum_{n=1}^{\infty} \Phi_{\TTT}(\AAA^{(n)});$$
   in particular,
   \begin{equation}\label{eqn:add}
   \Phi_{\TTT}(\AAA \oplus \BBB) = \Phi_{\TTT}(\AAA) + \Phi_{\TTT}(\BBB),
   \end{equation}
\item $\AAA \leqsl \BBB \iff \Phi_{\TTT}(\AAA) \leqsl \Phi_{\TTT}(\BBB)$,
\item $\AAA \leqsl^s \BBB \iff \Phi_{\TTT}(\AAA) = \Phi_{\TTT}(\BBB) \cdot j_E$ for some clopen set $E \subset
   \Omega$,
\item for every $\XXX \in \SsS\MmM_N$, $\Phi_{\TTT}(\XXX) = \DiNT{\XXX}{\TTT}$,
\item for every $u \in \Lambda(\Omega)$ there is a unique $\XXX \in \CDD_N$ such that $\Phi_{\TTT}(\XXX) = u$.
\end{enumerate}
\end{thm}
\begin{proof}
Let us start with uniqueness of $\Phi_{\TTT}$. First of all, for $\AAA \in \SsS\MmM_N$, $s(\AAA) = \aleph_0 \odot \AAA$
and hence $\supp_{\Omega} \AAA$ coincides with $\supp_{\Omega_{\tII}} \AAA$ introduced in the previous section. Therefore
(D0), (D1) and (D4) combined with \THMp{TR} yield that $\Phi_{\TTT}(\AAA) = \DiNT{\AAA}{\TTT}$ for $\AAA \in
\SsS\MmM_N$ (notice that $\AAA \ll \TTT$ for every such $\AAA$). Further, we infer from (D0) and (D2) that
$\Phi_{\TTT}(\JJJ_{\tII}) = \aleph_0 \cdot j_{\Omega_{\tII}}$ and consequently, by (D3) and (D0),
\begin{equation}\label{eqn:Phi(J)}
\Phi_{\TTT}(\JJJ) = j_{\Omega_I} + \aleph_0 \cdot j_{\Omega_{\tII} \cup \Omega_{\tIII}}.
\end{equation}
Now if $\XXX \leqsl^s \JJJ$, (D3) implies that $\Phi_{\TTT}(\JJJ) = \Phi_{\TTT}(\XXX) + \Phi_{\TTT}(\YYY)$ with $\YYY
= \JJJ \sqminus \XXX$. What is more, $\supp_{\Omega} \XXX \cap \supp_{\Omega} \YYY = \varempty$ from which we conclude,
thanks to (D1), that $\Phi_{\TTT}(\XXX) = \Phi_{\TTT}(\JJJ) \cdot j_{\supp_{\Omega} \XXX}$. Finally, if $\AAA \in
\CDD_N$ is arbitrary, the above notices combined with (D3) and (D2) give
\begin{equation}\label{eqn:Phi}
\Phi_{\TTT}(\AAA) = \DiNT{\EEE_{sm}(\AAA)}{\TTT} + \sum_{(i,\alpha) \in \Upsilon_*} \alpha \cdot \Phi_{\TTT}(\JJJ)
\cdot j_{\supp_{\Omega} \EEE^i_{\alpha}(\AAA)}.
\end{equation}
To establish the existence of $\Phi_{\TTT}$ together with all suitable properties, define $\Phi_{\TTT}(\AAA)$
by \eqref{eqn:Phi} with $\Phi_{\TTT}(\JJJ)$ given by \eqref{eqn:Phi(J)}. Observe that (D0), (D1'), (D2) and (D7) are
fulfilled. We shall show now \eqref{eqn:add}. We shall apply here calculations presented in \EXMp{commonpart}. Under
the notation of that example, \eqref{eqn:Phi} and \eqref{eqn:part2} give
\begin{multline*}
\Phi_{\TTT}(\AAA \oplus \BBB) = \DiNT{\EEE_{sm,0}}{\TTT} + \DiNT{\EEE_{sm,1}}{\TTT} + \DiNT{\EEE_{0,sm}}{\TTT}
+ \DiNT{\EEE_{1,sm}}{\TTT}\\+ \sum_{(i,\alpha,\beta) \in \Upsilon^2_{\#}} (\alpha + \beta) \cdot (\Phi_{\TTT}(\JJJ)
\cdot j_{\supp_{\Omega} \EEE^i_{\alpha,\beta}}).
\end{multline*}
Further, it follows from \THMp{TR} that
\begin{gather*}
\DiNT{\EEE_{sm}(\AAA)}{\TTT} = \DiNT{\EEE_{sm,0}}{\TTT} + \DiNT{\EEE_{sm,1}}{\TTT} + \sum_{\alpha \in \Card_{\infty}}
\DiNT{\EEE_{sm,\alpha}}{\TTT},\\
\DiNT{\EEE_{sm}(\BBB)}{\TTT} = \DiNT{\EEE_{0,sm}}{\TTT} + \DiNT{\EEE_{1,sm}}{\TTT} + \sum_{\alpha \in \Card_{\infty}}
\DiNT{\EEE_{\alpha,sm}}{\TTT}.
\end{gather*}
On the other hand, $(i,\alpha) \in \Upsilon_*$, $\EEE^i_{\alpha}(\AAA) = \bigsqplus_{\beta \in \Lambda_i}
\EEE^i_{\alpha,\beta}$ and $\EEE^i_{\alpha}(\BBB) = \bigsqplus_{\beta \in \Lambda_i} \EEE^i_{\beta,\alpha}$ which means
that
\begin{equation*}
j_{\supp_{\Omega} \EEE^i_{\alpha}(\AAA)} = \sum_{\beta \in \Lambda_i} j_{\supp_{\Omega} \EEE^i_{\alpha,\beta}} \quad
\textup{and} \quad
j_{\supp_{\Omega} \EEE^i_{\alpha}(\BBB)} = \sum_{\beta \in \Lambda_i} j_{\supp_{\Omega} \EEE^i_{\beta,\alpha}}.
\end{equation*}
Substituting the above formulas in the ones for $\Phi_{\TTT}(\AAA)$ and $\Phi_{\TTT}(\BBB)$, we see that \eqref{eqn:add}
is satisfied.\par
Now let $g$ be an arbitrary member of $\Lambda(\Omega)$. For $(i,\alpha) \in \Upsilon_*$ let $U^i_{\alpha} = \Omega_i \cap
\intt g^{-1}(\{\alpha\})$ and let $U^{\tII}_1$ be the closure of $g^{-1}(\RRR_+ \setminus \{0\}) \cap \Omega_{\tII}$. Since
$\Omega$ is extremely disconnected, the sets $U^i_{\alpha}$'s (with $(i,\alpha) \in \Upsilon$) are clopen and pairwise
disjoint. The arguments used in the proof of \LEM{sup-inf} show that their union is dense in $\Omega$. This implies that
there is a partition of unity $\{\EEE^i_{\alpha}\}_{(i,\alpha) \in \Upsilon} \subset \CDD_N$ such that $\supp_{\Omega}
\EEE^i_{\alpha} = U^i_{\alpha}$ for every $(i,\alpha) \in \Upsilon$. Moreover, thanks to \eqref{eqn:X|Y}, there is
$\EEE_{sm} \in \SsS\MmM_N$ such that $\DiNT{\EEE_{sm}}{\TTT} = g \cdot j_{U^{\tII}_1}$. The latter implies that
$\supp_{\Omega} \EEE_{sm} = \supp_{\Omega} \EEE^{\tII}_1$ and hence $\EEE^{\tII}_1 = \aleph_0 \odot \EEE_{sm}$. Now
the formulas $\EEE^i_{\alpha}(\AAA) := \EEE^i_{\alpha}$ and $\EEE_{sm}(\AAA) := \EEE_{sm}$ well defines $\AAA \in \CDD_N$
such that $\Phi_{\TTT}(\AAA) = g$. Further, if $\Phi_{\TTT}(\BBB) = g$ and $V^i_{\alpha} = \supp_{\Omega}
\EEE^i_{\alpha}(\BBB)$ ($(i,\alpha) \in \Upsilon$), then $V^i_{\alpha} \subset U^i_{\alpha}$ for $(i,\alpha) \in
\Upsilon$, by \eqref{eqn:Phi}. But the union of all $V^i_{\alpha}$'s is dense in $\Omega$ and $U^i_{\alpha} \setminus
V^i_{\alpha}$ is open. We infer from this that $V^i_{\alpha} = U^i_{\alpha}$ and consequently $\DiNT{\BBB}{\TTT} =
\DiNT{\AAA}{\TTT}$ and $\BBB = \AAA$. This shows (D8).\par
We are now able to prove easily (D5) and (D8). Indeed, if $\AAA \leqsl \BBB$, then $\BBB = \AAA \oplus \XXX$ for some
$\XXX$ and then, by \eqref{eqn:add}, $\Phi_{\TTT}(\BBB) = \Phi_{\TTT}(\AAA) + \Phi_{\TTT}(\XXX) \geqsl \Phi_{\TTT}(\AAA)$.
Conversely, if $\Phi_{\TTT}(\AAA) \leqsl \Phi_{\TTT}(\BBB)$, there is $g \in \Lambda(\Omega)$ (see \COR{minus} below) for
which $\Phi_{\TTT}(\BBB) = \Phi_{\TTT}(\AAA) + g$. We know from the previous argument that $g = \Phi_{\TTT}(\XXX)$ for some
$\XXX \in \CDD_N$. Consequently, $\Phi_{\TTT}(\BBB) = \Phi_{\TTT}(\AAA \oplus \XXX)$ and by (D8), $\BBB = \AAA \oplus \XXX$
and we are done.\par
We have shown that $\Phi_{\TTT}$ is a bijective order isomorphism. This implies that $\Lambda(\Omega)$ is order-complete
(by \THMP{complete}) and for every nonempty set $\{\AAA^{(s)}\}_{s \in S} \subset \CDD_N$,
\begin{multline*}
\Phi_{\TTT}\Bigl(\bigvee \bigl\{\bigoplus_{s \in S_0} \AAA^{(s)}\dd\ S_0 \in \PPp_f(S)\bigr\}\Bigr) =\\
= \sup{}_{\Lambda(\Omega)} \bigl\{\sum_{s \in S_0} \Phi_{\TTT}(\AAA^{(s)})\dd\ S_0 \in \PPp_f(S)\bigr\}.
\end{multline*}
But this and (AO6) \pREF{AO6} imply (D3), (D4) and (D4'). Point (D6) is left for the reader.
\end{proof}

Every topological space homeomorphic to $\Omega$ is called by us an \textit{underlying model space for $\CDD_N$}.
In the sequel we shall show that underlying model spaces for $\CDD_N$ and $\CDD_{N'}$ are homeomorphic for any natural
numbers $N$ and $N'$. We shall also propose a simplified form of them.\par
Let us now list a few basic consequences of \THM{model}. Some of them were announced in Section~12. For simplicity, we fix
$\TTT \in \SsS\MmM_N$ such that $\aleph_0 \odot \TTT = \JJJ_{\tII}$ and for each $\AAA \in \CDD_N$, $\widehat{\AAA}$ will
denote $\Phi_{\TTT}(\AAA)$. Since $\Lambda(\Omega)$ is order-complete, for every nonempty set $\{f_s\}_{s \in S} \subset
\Lambda(\Omega)$, $\bigvee_{s \in S} f_s$ and $\bigwedge_{s \in S} f_s$ will stand for, respectively,
$\sup{}_{\Lambda(\Omega)} \{f_s\dd\ s \in S\}$ and $\inf{}_{\Lambda(\Omega)} \{f_s\dd\ s \in S\}$.

\begin{cor}{minus}
$(\BBB \ominus \XXX)_{\Delta} \vee (\XXX \ominus \AAA)_{\Delta} \leqsl (\BBB \ominus \AAA)_{\Delta}$ provided $\AAA \leqsl
\XXX \leqsl \BBB$.
\end{cor}
\begin{proof}
It suffices to prove a counterpart of the corollary in the class $\Lambda(\Omega)$. Let $f, g \in \Lambda(\Omega)$
be such that $f \leqsl g$. The set $D_0 = \{x \in \Omega\dd\ f(x) < f(y) \textup{ or } f(y) \in \RRR_+\}$ is open
in $\Omega$ and there is a unique function $u_0\dd D_0 \to \RRR_+ \cup \Card$ such that $g(x) = u_0(x) + f(x)$ for every
$x \in D_0$. It may be easily seen that $u_0$ is continuous. Let $D(f,g) = D_0 \cup \intt (\Omega \setminus D_0)$ and $u
\in \Lambda(\Omega)$ be a unique continuous function (guaranteed by \LEMP{C-S}) such that $u(x) = u_0(x)$ for $x \in D_0$
and $u(x) = 0$ for $x \in D(f,g) \setminus D_0$. We see that $g = f + u$ on $D(f,g)$ and hence $g = f + u$ on $\Omega$.
It is easily seen that $u$ is the least member of $(\Lambda(\Omega),\leqsl)$ with the latter property. We shall denote this
$u$ by $(g - f)_{\Delta}$. It is clear that
$$
\widehat{(\BBB \ominus \AAA)_{\Delta}} = (\widehat{\BBB} - \widehat{\AAA})_{\Delta}
$$
whenever $\AAA \leqsl \BBB$. Thus, we need to check that $(h - g)_{\Delta} \vee (g - f)_{\Delta} \leqsl (h - f)_{\Delta}$
if only $f \leqsl g \leqsl h$. It suffices to check suitable inequality on a dense subset of $\Omega$. We leave this
as a simple exercise that for $x \in D(f,g) \cap D(g,h) \cap D(f,h)$ it is fulfilled.
\end{proof}

\begin{rem}{more}
Using the same idea as in the proof of \COR{minus}, one may show that whenever $\AAA, \BBB \in \CDD_N$ are such that
$\AAA \leqsl \BBB$, then
\begin{multline*}
[\BBB \ominus (\BBB \ominus \AAA)^{\nabla}]_{\Delta} \leqsl^s [\BBB \ominus (\BBB \ominus \AAA)_{\Delta}]_{\Delta} \leqsl^s
\AAA \leqsl\\
\leqsl [\BBB \ominus (\BBB \ominus \AAA)_{\Delta}]^{\nabla} = [\BBB \ominus (\BBB \ominus \AAA)^{\nabla}]^{\nabla}.
\end{multline*}
\end{rem}

Recall that the \textit{Souslin number} of a topological space $X$, denoted by $c(X)$ (\cite[Problem~1.7.12]{eng}),
is the least infinite cardinal number $\alpha$ such that every family of mutually disjoint nonempty open subsets of $X$
has power no greater than $\alpha$. Let us modify this by putting $c_*(\varempty) = 0$ and $c_*(X) = c(X)$ for nonempty
topological spaces $X$. It turns out that the modified Souslin numbers of certain clopen subsets of $\Omega$ may be used
to give the formula for $\dim(\AAA)$ if only this dimension is infinite. Namely,

\begin{pro}{Dim}
Let $\AAA \in \CDD_N$ and $f = \widehat{\AAA}$. Let $U^{\tII}_1$ be the closure of $f^{-1}(\RRR_+ \setminus \{0\}) \cap
\Omega_{\tII}$ and for $(i,\alpha) \in \Upsilon_*$ let $U^i_{\alpha} = \Omega_i \cap \intt f^{-1}(\{\alpha\})$. Then
$$
\aleph_0 \cdot \dim(\AAA) = \sum_{(i,\alpha) \in \Upsilon} \alpha \cdot c_*(U^i_{\alpha}).
$$
\end{pro}
\begin{proof}
In extremely disconnected spaces, the closures of two disjoint open sets are disjoint as well. Consequently, whenever $E$
is a clopen subset of $\Omega$, $c(E)$ is the least infinite cardinal $\alpha$ such that every family of pairwise disjoint
nonempty clopen sets has power no greater than $\alpha$. Since clopen sets correspond to $N$-tuples $\AAA$ such that $\AAA
\leqsl^s \JJJ$, the assertion follows from the argument used in (ST17) \pREF{ST17}. The details are left for the reader
(cf. the proof of point (D8) in \THM{model}).
\end{proof}

\begin{rem}{fin!}
It is worthwhile to mention that it is impossible to recognize $N$-tuples whose representatives act on finite-dimensional
spaces by means of corresponding to them members of $\Lambda(\Omega)$, unless we distinguish some special subsets
of $\Omega$, as it is done in the next section. To convince of that, it suffices to note that $\widehat{\AAA}$ is
the characteristic function of a one-point subset of $\Omega_I$ if e.g. $\aAA = (T,\ldots,T) \in \CDDc_N$ where $T$ is
either the identity operator on $\CCC$ or a unilateral shift on $l^2$.
\end{rem}

We shall now prove a useful

\begin{lem}{count}
\begin{enumerate}[\upshape(A)]
\item For every clopen nonempty set $E \subset \Omega$ there is a family $\{E_s\}_{s \in S}$ of pairwise disjoint clopen
   nonempty sets such that $c(E_s) = \aleph_0$ for every $s \in S$ and $\bigcup_{s \in S} E_s$ is a dense subset of $E$.
\item Let $\{f_s\}_{s \in S}$ be a nonempty set of members of $\Lambda(\Omega)$ and let $u = \bigwedge_{s \in S} f_s$
   and $v = \bigvee_{s \in S} f_s$. For every clopen nonempty set $E \subset \Omega$ with $c(E) = \aleph_0$ there are
   a nonempty set $S(E) \in \PPp_{\omega}(S)$ and an open dense subset $D(E)$ of $E$ with the following property. Whenever
   $S' \supset S(E)$ ($S' \subset S$) and $x \in D(E)$,
   \begin{equation}\label{eqn:inf}
   u(x) = \inf_{s \in S'} f_s(x)
   \end{equation}
   and if, in addition, $v(E) \subset I_{\aleph_0}$, then also
   $$
   v(x) = \sup_{s \in S'} f_s(x).
   $$
\end{enumerate}
\end{lem}
\begin{proof}
(A): Let $\EeE = \{E_s\}_{s \in S}$ be a maximal family of pairwise disjoint nonempty clopen sets such that $c(E_s) =
\aleph_0$ and $E_s \subset E$ for every $s \in S$. Let $D = E \setminus \cll (\bigcup_{s \in S} E_s)$. We have to show
that $D$ is empty. But this follows from \PROp{count}. Indeed, we infer from this result that every nonempty clopen subset
of $\Omega$ contains a nonempty clopen set $G$ with $c(G) = \aleph_0$. Consequently, since $D$ is clopen and $\EeE$
is maximal, $D = \varempty$.\par
(B): Let $U_1 = \cll u^{-1}(\RRR_+) \cap E$ and $U_{\alpha} = \intt u^{-1}(\{\alpha\}) \cap E$ for $\alpha \in
\Card_{\infty}$. We know (cf. the proof of \LEMP{sup-inf}) that the collection $\UUu = \{U_{\alpha}\dd\ \alpha \in
\Card_{\infty} \cup \{1\}\}$ consists of pairwise disjoint clopen sets whose union is dense in $E$. Further, for each
$\alpha \in \Card_{\infty}$ and $s \in S$ put $U_{\alpha,s} = U_{\alpha} \cap f_s^{-1}(\{\alpha\})$. Since $f_s \geqsl
\alpha$ on $U_{\alpha}$ and $\alpha$ is an isolated point of $\Card \setminus \{\beta \in \Card\dd\ \beta < \alpha\}$,
$U_{\alpha,s}$ is clopen. It is clear that $\bigcup_{s \in S} U_{\alpha,s}$ is dense in $U_{\alpha}$. (Indeed, the set
$G = U_{\alpha} \setminus \cll (\bigcup_{s \in S} U_{\alpha,s})$ is clopen and $f_s(x) \geqsl \alpha^+$ for any $x \in G$
and $s \in S$ and thus $u' \in \Lambda(\Omega)$ given by $u'|_G \equiv \alpha^+$ and $u' = u$ on $\Omega \setminus G$
is such that $u' \leqsl f_s$ ($s \in S$) which gives $u' \leqsl u$ and consequently $G = \varempty$.) Let `$<$' be a well
order on $S$ with the first element $s_*$. We define clopen sets $V_{\alpha,s}$ by transfinite induction as follows. Let
$V_{\alpha,s_*} = U_{\alpha,s_*}$ and for any $s \in S \setminus \{s_*\}$,
$$
V_{\alpha,s} = U_{\alpha,s} \setminus \cll \bigl(\bigcup_{s' < s} V_{\alpha,s'}\bigr).
$$
We see that $V_{\alpha,s} \subset U_{\alpha,s}$ and hence
\begin{equation}\label{eqn:fs}
u\bigr|_{V_{\alpha,s}} = f_s\bigr|_{V_{\alpha,s}}.
\end{equation}
Further, the sets $V_{\alpha,s}$ ($s \in S$) are pairwise disjoint. Using transfinite induction one may check that
$\cll(\bigcup_{s' < s} V_{\alpha,s'}) = \cll(\bigcup_{s' < s} U_{\alpha,s'})$ for each $s \in S$ and thus
\begin{equation}\label{eqn:alpha,s}
\cll \bigl(\bigcup_{s \in S} V_{\alpha,s}\bigr) = U_{\alpha}.
\end{equation}
Now we pass to the set $U_1$. By the definition, $U_1$ is clopen and $u(U_1) \subset I_{\aleph_0}$. In what follows, we
assume $U_1$ is nonempty. Let $g_s = f_s \wedge \aleph_0$. We naturally identify $I_{\aleph_0}$ with $[0,\infty]$. Let
$\tau\dd [0,\infty] \ni x \mapsto \frac{x}{x+1} \in [0,1]$ (with convention that $\frac{\infty}{\infty + 1} = 1$). Put
$u' = \tau \circ u\bigr|_{U_1} \in \CCc(U_1,[0,1])$ and $g_s' = \tau \circ g_s\bigr|_{U_1} \in \CCc(U_1,[0,1])$. Note that
\begin{equation}\label{eqn:infim}
u' = \bigwedge_{s \in S} g_s'.
\end{equation}
Since $U_1$ is clopen in $\Omega$ and $\CCc(\Omega)$ is a $\WWw^*$-algebra, so is $\CCc(U_1)$. Further, we conclude
from the fact that $c(U_1) = \aleph_0$ that $\CCc(U_1)$ is countably decomposable. Thus, it may be infered from
\cite[Theorem~III.1.18]{tk1} or \cite[Proposition~1.18.1]{sak} that $\CCc(U_1)$ is isomorphic to $L^{\infty}(\mu)$
for some probabilistic space $(X,\Mm,\mu)$. Under the isomorphism, $g_s'$ and $u'$ correspond to, respectively, $\xi_s \in
L^{\infty}(\mu)$ and $w \in L^{\infty}(\mu)$. Consequently, $w = \inf{}_{L^{\infty}(\mu)} \{\xi_s\dd\ s \in S\}$
(by \eqref{eqn:infim}). For a nonempty set $S_0 \in \PPp_{\omega}(S)$ let $w_{S_0}\dd X \ni x \mapsto \inf_{s \in S_0}
\xi_s(x) \in [0,1]$. Since $S_0$ is countable, $w_{S_0}$ is measurable and hence $w_{S_0} \in L^{\infty}(\mu)$. Let
$$
c = \inf \Bigl\{\int_X w_{S_0} \dint{\mu}\dd\ S_0 \in \PPp_{\omega}(S)\Bigr\}.
$$
It is easily seen that there is $S_1 \in \PPp_{\omega}(S)$ for which $c = \int_X w_{S_1} \dint{\mu}$. Now if $s$ is
an arbitrary element of $S$, then $w_{S_1 \cup \{s\}} \leqsl w_{S_1}$ and $\int_X w_{S_1 \cup \{s\}} \dint{\mu} \geqsl c =
\int_X w_{S_1} \dint{\mu}$. These imply that $w_{S_1 \cup \{s\}} = w_{S_1}$ ($\mu$-almost everywhere) and consequently
$\xi_s \geqsl w_{S_1}$ in $L^{\infty}(\mu)$. The latter gives $w \geqsl w_{S_1} = \inf{}_{L^{\infty}(\mu)} \{\xi_s\dd\
s \in S_1\}$ and therefore $w = w_{S_1}$ (in $L^{\infty}(\mu)$). In $\CCc(U_1)$ this is interpreted as
$u' = \bigwedge_{s \in S_1} g_s'$ which is equivalent to $u\bigr|_{U_1} = \bigwedge_{s \in S_1} g_s\bigr|_{U_1}$.
Now by \LEMp{sup-inf}, $u(x) = \inf_{s \in S_1} g_s(x)$ for $x \in D_1$ where $D_1$ is an open dense subset of $U_1$. This
implies that for each $x \in D_1(E) := D_1 \cap u^{-1}(\RRR_+) \cap E$ there is $s_x \in S_1$ such that $g_{s_x}(x) \in
\RRR_+$. Consequently, $g_{s_x}(x) = f_{s_x}(x)$ and hence
\begin{equation}\label{eqn:S1}
u(x) = \inf_{s \in S_1} f_s(x)
\end{equation}
for $x \in D_1(E)$. Notice that $D_1(E)$ is dense in $U_1$.\par
Further, observe that the family $\{U_1\} \cup \{V_{\alpha,s}\dd\ s \in S,\ \alpha \in \Card_{\infty}\}$ constists
of pairwise disjoint clopen subsets of $E$. Since $c(E) = \aleph_0$, the set $J := \{(\alpha,s)\dd\ s \in S,\ \alpha \in
\Card_{\infty},\ U_{\alpha,s} \neq \varempty\}$ is countable (finite or not). Put $S(E) = S_1 \cup \{s\dd\ (\alpha,s) \in
J\}$ and $D(E) = D_1(E) \cup \bigcup_{(\alpha,s) \in J} V_{\alpha,s}$. We see that $S(E) \in \PPp_{\omega}(S)$ and $D(E)$
is open and dense in $E$ (by \eqref{eqn:alpha,s} and the density of $D_1(E)$ in $U_1$). Take an arbitrary set $S'$ such
that $S(E) \subset S' \subset S$. For each $x \in \Omega$ one has $\inf_{s \in S'} f_s(x) \geqsl u(x)$. On the other hand,
if $x \in D(E)$, then either $x \in D_1(E)$ or $x \in V_{\alpha,s}$ for some $(\alpha,s) \in J$. In the first case
the inequality $\inf_{s \in S'} f_s(x) \leqsl u(x)$ follows from \eqref{eqn:S1}, while in the second one from
\eqref{eqn:fs}.\par
If we additionally assume that $v(E) \subset I_{\aleph_0}$, we have to enlarge the above defined set $S(E)$ and decrease
$D(E)$. Arguing as in the paragraph for $U_1$ (that is, representing $\CCc(E)$ as $L^{\infty}(\mu)$ for some probabilistic
measure $\mu$), we see that there is $S_2 \in \PPp_{\omega}(S)$ such that $v\bigr|_E = \bigvee_{s \in S_2} f_s$.
By \LEM{sup-inf}, there is an open dense subset $D_2$ of $E$ such that $v(x) = \sup_{s \in S_2} f_s(x)$. Now it suffices
to replace $S(E)$ by $S(E) \cup S_2$ and $D(E)$ by $D(E) \cap D_2$. (The details are left for the reader.)
\end{proof}

Both points of \LEM{count} yield

\begin{cor}{inf-sup}
Let $\{f_s\}_{s \in S}$ be a nonempty subset of $\Lambda(\Omega)$.
\begin{enumerate}[\upshape(A)]
\item There is an open dense subset $D$ of $\Omega$ such that for all $x \in D$,
   $$
   (\bigwedge_{s \in S} f_s)(x) = \inf_{s \in S} f_s(x).
   $$
\item If $E$ is a clopen subset of $\Omega$ such that $(\bigvee_{s \in S} f_s)(E) \subset I_{\aleph_0}$, then there exists
   an open dense subset $G$ of $E$ such that for any $x \in G$, $(\bigvee_{s \in S} f_s)(x) = \sup_{s \in S} f_s(x)$.
\end{enumerate}
\end{cor}

\begin{rem}{sup}
We suspect that the counterpart of point (A) of \COR{inf-sup} for suprema fails to be true in general. However, partial
results in this direction may be shown. Let $u = \bigvee_{s \in S} f_s$. Put $U_1 = u^{-1}(I_{\aleph_0})$ and $U_{\alpha} =
\intt f^{-1}(\{\alpha\})$ for $\alpha \in \Card_{\infty} \setminus \{\aleph_0\}$. The argument used in the proof
of \LEMp{sup-inf} shows that $U_1 \cup \bigcup_{\alpha > \aleph_0} U_{\alpha}$ is dense in $\Omega$. By \COR{inf-sup},
there is an open dense subset of $U_1$ such that
\begin{equation}\label{eqn:aux70}
u(x) = \sup_{s \in S} f_s(x)
\end{equation}
for $x \in D_1$. We ask for which $\alpha \in \Card_{\infty} \setminus \{\aleph_0\}$ there is an open dense subset
$D_{\alpha}$ of $U_{\alpha}$ such that \eqref{eqn:aux70} is satisfied for all $x \in D_{\alpha}$. It is quite easy to show
that this is true when $\alpha = \beta^+$ for some $\beta \geqsl \aleph_0$ (indeed, it suffices to put $D_{\alpha} =
U_{\alpha} \cap \bigcup_{s \in S} f_s^{-1}(\{\alpha\})$; since $f_s \leqsl \alpha$ on $U_{\alpha}$ and $\alpha$ is
an isolated point of $I_{\alpha}$, the set $D_{\alpha}$ is open; that $\cll D_{\alpha} = U_{\alpha}$ may be proved
by a standard argument on the difference of the latter sets). A little bit more difficult is to prove that $D_{\alpha}$
exists for every limit cardinal $\alpha$ which has countable cofinality. The latter means that there is a sequence
$(\beta_n)_{n=1}^{\infty}$ of cardinals such that $\beta_n < \alpha$ for every $n$, and $\alpha = \sup_{n\geqsl1} \beta_n$.
In that case we put $G = U_{\alpha} \cap \bigcap_{n=1}^{\infty} \bigcup_{s \in S} f_s^{-1}(\Card \setminus I_{\beta_n})$
and $D = U_{\alpha} \setminus \cll G$. Our first claim is that $D$ is empty. For if not, there would exist a nonempty
clopen set $E \subset D$. Then put $E_n = E \cap \bigcap_{s \in S} f_s^{-1}(I_{\beta_n})$. Noticing that $E =
\bigcup_{n=1}^{\infty} E_n$ (since $E \cap G = \varempty$) and $E_n$'s are closed, infer from Baire's theorem that
$W = \intt E_n$ is nonempty for some $n$ and thus $\bigvee_{s \in S} (f_s\bigr|_W) \leqsl \beta_n$ ($W$ is clopen),
contradictory to the fact that $[\bigvee_{s \in S} (f_s\bigr|_W)](x) = u(x) = \alpha$ for $x \in W$. So, $D$ is indeed
empty and hence $G$ is a dense $\GGg_{\delta}$ subset of $U_{\alpha}$. Now an application of \LEMp{W*} yields that
$D_{\alpha} = \intt G$ is dense in $U_{\alpha}$ as well.\par
The above arguments show that if $(\bigvee_{s \in S} f_s)(\Omega) \cap \Card_{\infty}$ consists only of cardinals which
are non-limit or have countable cofinality, then $\bigvee_{s \in S} f_s$ may be computed pointwisely on an open dense set.
\end{rem}

\begin{thm}{Boole}
For every nonempty set $\{\AAA^{(s)}\}_{s \in S} \subset \CDD_N$ and each $\BBB \in \CDD_N$,
\begin{gather}
\BBB \wedge \bigl(\bigvee_{s \in S} \AAA^{(s)}\bigr) = \bigvee_{s \in S} (\BBB \wedge \AAA^{(s)}), \label{eqn:aux71}\\
\BBB \vee \bigl(\bigwedge_{s \in S} \AAA^{(s)}\bigr) = \bigwedge_{s \in S} (\BBB \vee \AAA^{(s)}). \label{eqn:aux72}
\end{gather}
\end{thm}
\begin{proof}
As usual, we pass to $\Lambda(\Omega)$. Put $f_s = \widehat{\AAA^{(s)}}$ and $g = \widehat{\BBB}$. Let $u =
\bigwedge_{s \in S} f_s$ and $u' = \bigwedge_{s \in S} (g \vee f_s)$. By \COR{inf-sup}, there are open dense sets $D$ and
$D'$ such that $u(x) = \inf_{s \in S} f_s(x)$ for $x \in D$ and $u'(x) = \inf_{s \in S} (g \vee f_s)(x)$ for $x \in D'$.
Then for $x \in D \cap D'$,
$$
(g \vee u)(x) = \max(g(x),\inf_{s \in S} f_s(x)) = \inf_{s \in S} (\max(g(x),f_s(x)) = u'(x)
$$
which gives \eqref{eqn:aux72}. Now we pass to \eqref{eqn:aux71}.\par
Let $v = \bigvee_{s \in S} f_s$ and $v' = \bigvee_{s \in S} (g \wedge f_s)$. We only need to show that $v' \geqsl g \wedge
v$. As usual, put $U_0 = g^{-1}(I_{\aleph_0}) \cap v^{-1}(I_{\aleph_0})$, $U_1 = g^{-1}(I_{\aleph_0}) \setminus
v^{-1}(I_{\aleph_0})$ and $U_{\alpha} = \intt g^{-1}(\{\alpha\})$ for $\alpha \in \Card_{\infty} \setminus \{\aleph_0\}$.
We know that each of the just defined sets is clopen and their union is dense in $\Omega$. Hence it suffices to show that
$g \wedge v \leqsl v'$ on a dense subset of $U_{\alpha}$ for any $\alpha \in \{0,1\} \cup \Card_{\infty} \setminus
\{\aleph_0\}$.\par
On $U_0$ it suffices to apply \COR{inf-sup}: if $v'(x) = \sup_{s \in S} (g \wedge f_s)(x)$ for $x \in D'$ and $v(x) =
\sup_{s \in S} f_s(x)$ for $x \in D$, then $v' = v \wedge g$ on $D \cap D'$. Further, since $v > \aleph_0$ on $U_1$,
the set $D_1 = U_1 \cap \bigcup_{s \in S} f_s^{-1}(\Card \setminus I_{\aleph_0})$ is dense in $U_1$. What is more,
for every $x \in D_1$ there is $s \in S$ with $f_s(x) > \aleph_0$ and therefore $v'(x) \geqsl (f_s \wedge g)(x) = g(x)$.
Consequently, $v' \geqsl g \wedge v$ on $D_1$ and we are done.\par
Now fix $\alpha \in \Card_{\infty} \setminus \{\aleph_0\}$. We divide $U_{\alpha}$ into two clopen parts: $V_1 = U_{\alpha}
\cap v^{-1}(I_{\alpha})$ and $V_2 = U_{\alpha} \setminus v^{-1}(I_{\alpha})$. Let $D_{\alpha} = V_1 \cup \bigcup_{s \in S}
[U_{\alpha} \setminus f_s^{-1}(I_{\alpha})]$. Notice that $f_s \leqsl \alpha$ on $V_1$ (hence $v' = v$ on $V_1$) and for
every $x \in D_{\alpha} \setminus V_1$ there is $s \in S$ such that $f_s(x) > \alpha$ (so, $v' = g$ on $D_{\alpha}
\setminus V_1$). This proves that $v' \geqsl v \wedge g$ on $D_{\alpha}$. Finally, standard argument shows that $D_{\alpha}
\cap V_2$ is dense in $V_2$ and this finishes the proof.
\end{proof}

\begin{pro}{multiply}
The assertion of \textup{(AO14) \pREF{AO14}} is satisfied.
\end{pro}
\begin{proof}
Again, it suffices to prove the counterpart of (AO14) in the realm $\Lambda(\Omega)$. It is clear that $\alpha \cdot
(f \vee g) = (\alpha \cdot f) \vee (\alpha \cdot g)$ and $\alpha \cdot (f \wedge g) = (\alpha \cdot f) \wedge
(\alpha \cdot g)$ for all $f, g \in \Lambda(\Omega)$ and each $\alpha \in \Card$. Now let $\alpha = k$ be a positive finite
cardinal. In order to show that $k \cdot (\bigvee_{s \in S} f_s) = \bigvee_{s \in S} (k \cdot f_s)$ and $k \cdot
(\bigwedge_{s \in S} f_s) = \bigwedge_{s \in S} (k \cdot f_s)$, let us consider an `extended' version
of $\Lambda(\Omega)$, namely $\widetilde{\Lambda}(\Omega)$ which is defined in the same way as $\Lambda(\Omega)$ with
the only difference that members of $\widetilde{\Lambda}(\Omega)$ send $\Omega_I$ into $\RRR_+ \cup \Card$. We shall prove
in \CORp{123} that $\Omega_I$ is homeomorphic to $\Omega_{\tII}$. Consequently, $\widetilde{\Lambda}(\Omega)$ is
order-complete. It is immediate that the assignment $\widetilde{\Lambda}(\Omega) \ni f \mapsto k \cdot f \in
\widetilde{\Lambda}(\Omega)$ is a bijective order isomorphism. Hence it preserves g.l.b.'s and l.u.b.'s computed
in the space $\widetilde{\Lambda}(\Omega)$. So, we only need to check that $u := \sup{}_{\widetilde{\Lambda}(\Omega)} F$
and $v := \inf{}_{\widetilde{\Lambda}(\Omega)} F$ are members of $\Lambda(\Omega)$ for every nonempty set $F \subset
\Lambda(\Omega)$. Since the proof for $u$ is similar, we shall only show that $v \in \Lambda(\Omega)$. Let $D_0 = \Omega_I
\cap \intt v^{-1}(\{0\})$, $B_0 = v^{-1}(\{0\}) \cap \Omega_I \setminus D_0$ and for positive integer $m$ let $D_m =
\Omega_I \cap \intt v^{-1}((m-1,m])$ and $B_m = v^{-1}((m-1,m]) \cap \Omega_I \setminus D_m$. We claim that $D =
(\Omega_I \cap v^{-1}(\Card_{\infty})) \cup \bigcup_{m=0}^{\infty} D_m$ is dense in $\Omega_I$ ($D$ is of course open).
Indeed, $\Omega_I \setminus D = \bigcup_{m=0}^{\infty} B_m$. Since each of $B_m$'s is nowhere dense (by \LEMP{W*}),
Baire's theorem yields our assertion. Now let $v' \in \Lambda(\Omega)$ be a function such that $v' = v$ on $(\Omega_I \cap
v^{-1}(\Card_{\infty})) \cup \Omega_{\tII} \cup \Omega_{\tIII}$ and $v(D_m) \subset \{m\}$ for every integer $m \geqsl 0$
(see \LEMP{C-S}). We see that $v(x) \leqsl v'(x)$ for $x \in D \cup \Omega_{\tII} \cup \Omega_{\tIII}$ and consequently $v
\leqsl v'$. Moreover, since $v \leqsl f \in \Lambda(\Omega)$ for any $f \in F$, $v' \leqsl f$ as well ($f \in F$) and hence
$v = v' \in \Lambda(\Omega)$.\par
In the second part of the second claim of (AO14) one assumes that $\EEE_{sm}(\AAA^{(s)}) = \zero$ which corresponds
to $f_s(\Omega_{\tII}) \subset \{0\} \cup \Card_{\infty}$. Here we shall weaken this, assuming that $f_s(\Omega_{\tII})
\subset \Card$ for each $s \in S$. It follows from \COR{inf-sup} that there is an open dense subset $D$ of $\Omega$ such
that for all $x \in D$, $(\bigwedge_{s \in S} f_s)(x) = \inf_{s \in S} f_s(x)$ as well as $[(\bigwedge_{s \in S} (\alpha
\cdot f_s)](x) = \inf_{s \in S} (\alpha \cdot f_s)(x)$. Since all values of (all) $f_s$'s are cardinals, we see that
in both the latter formulas `$\inf$' may be replaced by `$\min$'. But $\alpha \cdot \min_{s \in S} f_s(x) = \min_{s \in S}
(\alpha \cdot f_s(x))$ and thus $\alpha \cdot (\bigwedge_{s \in S} f_s)(x) = [\bigwedge_{s \in S} (\alpha \cdot f_s)](x)$
for $x \in D$ and we are done.\par
We now pass to the last claim: that $\alpha \cdot \bigvee_{f \in F} f = \bigvee_{f \in F} (\alpha \cdot f)$ for every
nonempty set $F \subset \Lambda(\Omega)$ and $\alpha \in \Card_{\infty}$. The inequality `$\geqsl$' is clear. To prove
the converse, put $u = \bigvee_{f \in F} (\alpha \cdot f)$. It is enough to show that $\alpha \cdot u = u$. Equivalently,
we have to check that for each $x \in \Omega$, $u(x) \geqsl \alpha$ or $u(x) = 0$. Suppose, for the contrary, that $0 <
u(x_0) < \alpha$ for some $x_0 \in \Omega$. Take a closed set $B \subset I_{\alpha} \setminus \{\alpha\}$ such that
$u(x_0) \in \intt B$ and put $D = \intt u^{-1}(B)$. $D$ is clopen and $x_0 \in D$. Now let $u' \in \Lambda(\Omega)$ be
given by $u' = u$ on $\Omega \setminus D$ and $u' = 0$ on $D$. We see that $u'(x_0) < u(x_0)$. However, $\alpha \cdot f
\leqsl u'$ for every $f \in F$. Indeed, if $x \in D$, then $\alpha > u(x) \geqsl \alpha \cdot f(x)$ which implies that
$f(x) = 0$. Thus, $u$ is not the l.u.b. of $\alpha \cdot F$ and this finishes the proof.
\end{proof}

\begin{rem}{inf-sum}
It is natural to ask which element of $\Lambda(\Omega)$ corresponds to $\AAA = \bigoplus_{s \in S} \AAA^{(s)}$
for an uncountable set $S$. In other words, how to express $\sum_{s \in S} f_s := \widehat{\AAA}$ by means of $f_s =
\widehat{\AAA^{(s)}}$ ($s \in S$). \LEMp{sup-inf} and \THMp{model} show that for countable $S$, $\sum_{s \in S} f_s$ may be
computed pointwisely on an open dense subset of $\Omega$. Let us demonstrate how to find $\sum_{s \in S} f_s$ when $S$
is uncountable. We shall use here the arguments of Section~13. First of all, let $g = \bigvee \{\sum_{s \in S'} f_s\dd\
S' \in \PPp_f(S)\}$ and $U_f = \cll g^{-1}(\RRR_+)$. It may be deduced from the arguments of Section~13 that
$\sum_{s \in S} f_s = g$ on $U_f$ and the function $f := \sum_{s \in S} f_s$ takes infinite values on $\Omega \setminus
U_f$. So, we only need to characterize $U_{\alpha} = \intt f^{-1}(\{\alpha\})$ for $\alpha \in \Card_{\infty}$ (since
we know that $U_f \cup \bigcup_{\alpha \in \Card_{\infty}} U_{\alpha}$ is dense in $\Omega$). This is possible thanks
to \EQp{Ealpha}. For this purpose, we define $\dim_E u$ for $u \in \Lambda(\Omega)$ and a nonempty clopen set $E \subset
\Omega$ with $c(E) = \aleph_0$ as follows:
\begin{multline*}
\dim_E u = \sum \{\alpha\in\Card_{\infty}\dd\ E \cap \intt u^{-1}(\{\alpha\}) \neq \varempty\}\\
+ c_*(E \cap \cll u^{-1}(\RRR_+ \setminus \{0\}))
\end{multline*}
(notice that the latter summand is either $0$ or $\aleph_0$). Now one may conclude from \eqref{eqn:Ealpha} that
$U_{\alpha}$ is the closure of the union of all clopen sets $V \subset \Omega \setminus U_f$ such that $\sum_{s \in S}
\dim_E f_s = \alpha$ for every nonempty clopen set $E \subset V$ with $c(E) = \aleph_0$ (of course, $U_{\alpha}$ may be
empty). We leave the details for the interested reader.
\end{rem}

\begin{rem}{changeT}
It is clear that the formula for $\Phi_{\TTT}$ essentially depends on $\TTT$. However, there is a quite simple
connection between $\Phi_{\TTT}$ and $\Phi_{\SSS}$ for any two semiminimal $N$-tuples $\TTT$ and $\SSS$ such that
$\aleph_0 \odot \TTT = \aleph_0 \odot \SSS = \JJJ_{\tII}$. Put $u = j_{\Omega_I \cup \Omega_{\tIII}} + \DiNT{\SSS}{\TTT}$
and $D := u^{-1}(\RRR_+ \setminus \{0\})$. We leave this as an easy exercise that $D$ is dense in $\Omega$ and for every
$\XXX \in \CDD_N$, $\Phi_{\SSS}(\XXX)$ is the unique continuous extension of $(\frac1u \Phi_{\TTT}(\XXX))\bigr|_D$.
\end{rem}

\SECT{Types of $N$-tuples}

As in the previous section, $\widehat{\AAA} = \Phi_{\TTT}(\AAA)$ for each $\AAA \in \CDD_N$ where $\Phi_{\TTT}$
is as in \THMp{model}. This notation is obligatory to the end of the paper.\par
The following result is an immediate consequence of \PROp{unitid}.

\begin{pro}{id-clopen}
For every clopen set $E \subset \Omega$ the class $\IiI[E] := \{\AAA \in \CDD_N\dd\ \supp \widehat{\AAA} \subset E\}$
is an ideal in $\CDD_N$. Conversely, for every ideal $\AaA \subset \CDD_N$ there is a (unique) clopen set $K \subset
\Omega$ such that $\AaA = \IiI[K]$. What is more, $K = \supp \widehat{\JJJ(\AaA)}$.
\end{pro}

For every ideal $\AaA$, the unique clopen set $K$ such that $\AaA = \IiI[K]$ will be denoted by $\supp_{\Omega} \AaA$.
Below we give some illustrative examples related to this subject.

\begin{exs}{id-clopen}
\begin{enumerate}[(A)]
\item Fix nonnegative real number $r$ and let $\IiI(r)$ be the class of all $N$-tuples $\XXX$ for which $\|\XXX\|
   \leqsl r$. It is clear that $\IiI(r)$ is an ideal. Put $\Omega(r) := \supp_{\Omega} \IiI(r)$
   and $\Omega(\operatorname{bd}) := \bigcup_{r\geqsl0} \Omega(r)$. The set $\Omega(\operatorname{bd})$ is open
   in $\Omega$ and for every $\XXX \in \CDD_N$,
   $$
   \|\XXX\| < \infty \iff \supp \widehat{\XXX} \subset \Omega(\operatorname{bd})
   $$
   (indeed, use the fact that $\supp \widehat{\XXX}$ is compact). What is more, if $\|\XXX\| < \infty$, then $\|\XXX\|
   = \min \{r \geqsl 0\dd\ \supp \widehat{\XXX} \subset \Omega(r)\}$. The ideal $\IiI[\cll \Omega(\operatorname{bd})]$
   consists of all $N$-tuples which are direct sums of bounded $N$-tuples. Further, whenever $0 \leqsl r < s$, the ideal
   $\IiI[\Omega(r) \setminus \Omega(s)]$ consists of all $N$-tuples whose every nontrivial reduced part has norm greater
   than $s$ but no greater than $r$. We conclude from this that $\Omega(s) = \intt(\bigcap_{r>s} \Omega(r))$ for any
   $s \geqsl 0$. For positive $r$ put $\Omega\{r\} = \Omega(r) \setminus \cll(\bigcup_{s<r} \Omega(s))$ and $\IiI\{r\} =
   \IiI[\Omega\{r\}]$. The ideal $\IiI\{r\}$ constists of all $N$-tuples whose every nontrivial reduced part has norm
   equal to $r$.
\item Now let $\IiI(\bB) := \{\bB(\AAA)\dd\ \AAA \in \CDD_N\}$. It follows from suitable properties of the $\bB$-transform
   that $\IiI(\bB)$ is an ideal. Let $\Omega(\bB) = \supp_{\Omega} \IiI(\bB)$. Notice that $\IiI(\bB)$ consists of all
   $N$-tuples $\XXX$ such that either $\|\XXX\| < 1$ or $\|\XXX\| = 1$ and $\XXX$ does not assume its norm. Consequently,
   $\Omega(\bB) \varsubsetneq \Omega(1)$.  The ideal $\IiI[E]$ with $E = \Omega(1) \setminus \Omega(\bB)$ constists of all
   $N$-tuples whose each nontrivial reduced part has norm $1$ and assumes its norm. In particular, $E \subset \Omega\{1\}$
   and the ideal $\IiI[\Omega\{1\} \setminus E] = \IiI[\Omega\{1\} \cap \Omega(\bB)]$ coincides with the class of all
   $N$-tuples whose every nontrivial reduced part has norm equal to $1$ and does not assume its norm.
\end{enumerate}
\end{exs}

As a consequence of \THMp{main} and \EXS{id-clopen} we obtain

\begin{cor}{contraction}
Every contraction $T$ acting on a Hilbert space $\HHh$ induces a unique decomposition $\HHh = \HHh_0 \oplus \HHh_1 \oplus
\HHh_2$ such that $\HHh_0,\HHh_1,\HHh_2 \in \red(T)$ and
\begin{enumerate}[\upshape(a)]
\item every nontrivial reduced part of $T\bigr|_{\HHh_0}$ admits a nontrivial reduced part of norm less than $1$,
\item $T\bigr|_{\HHh_1}$ does not assume its norm (unless $\HHh_1$ is trivial) and every its nontrivial reduced part has
   norm equal to $1$,
\item every nontrivial reduced part of $T\bigr|_{\HHh_2}$ has norm $1$ and assumes its norm.
\end{enumerate}
What is more, $\HHh_0, \HHh_1, \HHh_2 \in \RED(T)$.
\end{cor}

As it was done by Ernest \cite{e}, the types of $\WWw''(\xXX)$ and $\WWw'(\xXX)$ may be assigned to $\xXX$. It is easily
seen (and in fact, this was used by us in \THMP{decomp}) that for every nontrivial $\xXX \in \CDDc_N$:
\begin{itemize}
\item $\WWw'(\xXX)$ is type I$_{\alpha}$ ($\alpha \in \Card \setminus \{0\}$) iff $\XXX = \alpha \odot \EEE$ for some
   unique $\EEE \leqsl^s \JJJ_I$,
\item $\WWw'(\xXX)$ is type III$_{\alpha}$ ($\alpha \in \Card_{\infty}$) iff $\XXX = \alpha \odot \EEE$ for some unique
   $\EEE \leqsl^s \JJJ_{\tIII}$,
\item $\WWw'(\xXX)$ is type II$_1$ iff $\XXX$ is semiminimal,
\item $\WWw'(\xXX)$ is type II$_{\alpha}$ ($\alpha \in \Card_{\infty}$) iff $\XXX = \alpha \odot \EEE$ for some unique
   $\EEE \leqsl^s \JJJ_{\tII}$.
\end{itemize}
Ernest calls a bounded operator $T$ of type $i_{\alpha}$ provided $\WWw'(T)$ is of this type
(cf. \cite[Definition~1.28]{e}). We call a nontrivial $N$-tuple $\XXX \in \CDD_N$ (\textit{of}) \textit{type} I$^n$ (with
$n=1,2,\ldots,\infty$), II$^1$, II$^{\infty}$ or III$^{(\infty)}$ iff $\WWw''(\xXX)$ is of this type. Additionally,
we agree that the trivial $N$-tuple is of each of these types.\par
Since a von Neumann algebra is type I, II, III iff so is its commutant, we see that for nontrivial $\XXX$, $\WWw''(\xXX)$
is type III iff so is $\WWw'(\xXX)$ and thus the above definition makes no confusion. Later we shall see that if nontrivial
$\XXX$ is type $i^{\infty}$ ($i \in \{I,\tII,\tIII\}$), then $\WWw''(\xXX)$ is type $i_{\aleph_0}$ and thus there is
no need to use uncountable cardinals here.\par
Fix $i^n \in \{I^1,I^2,\ldots,I^{\infty},\tII^1,\tII^{\infty},\tIII^{\infty}\}$ and let $\IiI_{i_n}$ be the class of all
$N$-tuples of type $i^n$. Our first goal is

\begin{pro}{typei}
$\IiI_{i_n}$ is an ideal in $\CDD_N$.
\end{pro}
\begin{proof}
It suffices to verify all points of \CORp{ideal}. The point (a) is fulfilled since for every $\alpha \in \Card_{\infty}$
and nontrivial $\XXX$, the von Neumann algebras $\WWw''(\xXX)$ and $\WWw''(\yYY)$ are isomorphic where $\YYY = \alpha
\odot \XXX$. Point (b) follows from the following result on $\WWw^*$-algebras: if $\MMm$ is a $\WWw^*$-algebra
and $\{z_s\}_{s \in S}$ is a family of mutually orthogonal central projections in $\MMm$ which sum up to $1$ and $\MMm z_s$
is type $i_n$ for each $s \in S$, then $\MMm$ itself is type $i_n$. Finally, point (c) is a consequence of a similar
result: if $\MMm$ is a type $i_n$ $\WWw^*$-algebra and $z$ is a (nonzero) central projection in $\MMm$, then $\MMm z$ is
type $i_n$ as well.
\end{proof}

Now put $\Omega_{i_n} = \supp_{\Omega} \IiI_{i_n}$. It is clear that the sets $\Omega_{I_1}$, $\Omega_{I_2}$, \ldots,
$\Omega_{I_{\infty}}$, $\Omega_{\tII_1}$, $\Omega_{\tII_{\infty}}$ and $\Omega_{\tIII_{\infty}}$ are pairwise disjoint
and their union is dense in $\Omega$. It is obvious that $\Omega_{\tIII_{\infty}} = \Omega_{\tIII}$. Let us now check that
if $\XXX$ is type $i^{\infty}$ (and nontrivial), then $\WWw''(\xXX)$ is type $i_{\aleph_0}$. Indeed, there is
$\EEE \leqsl^s \JJJ$ (namely, $\EEE = s(\XXX)$, cf. \EQP{s(A)}) and an infinite cardinal $\alpha$ such that $\alpha \odot
\XXX = \alpha \odot \EEE$. The latter implies that $\WWw''(\xXX)$ and $\WWw''(\eEE)$ are isomorphic as $\WWw^*$-algebras
and thus $\WWw''(\eEE)$ is type $i_{\infty}$. Further, we conclude from \PROp{count} that $\EEE = \bigsqplus_{s \in S}
\EEE^{(s)}$ for suitable family such that $0 < \dim(\EEE^{(s)}) \leqsl \aleph_0$. Consequently, $\WWw''(\eEE^{(s)})$ is
type $i_{\infty}$ for each $s \in S$ and therefore (since $\eEE^{(s)}$ act in a separable Hilbert space)
$\WWw''(\eEE^{(s)})$ is type $i_{\aleph_0}$. This yields that $\WWw''(\eEE)$ (and hence $\WWw''(\xXX)$) is type
$i_{\aleph_0}$ as well.\par
One may easily check that $\IiI_{I_1}$ coincides with the ideal $\NnN_N$ introduced in \EXS{id}--(E) \pREF{E} and studied
in \EXMp{normal}. Thus $\Omega_{I_1}$ corresponds to normal $N$-tuples.\par
The sets $\Omega_{I_n}$ may be used to compute $\dim(\XXX)$ for every $\XXX \in \CDD_N$ by means
of $\widehat{\XXX}$. For this, let us introduce the \textit{strict Souslin number}, $c_f(X)$, of a topological space $X$.
Namely, $c_f(X) = c(X)$ iff $X$ is an infinite set and $c_f(X) = \card(X)$ otherwise.

\begin{pro}{dim}
Let $\XXX \in \CDD_N$, $f = \widehat{\XXX}$, $U^i_{\alpha} = \Omega_i \cap \intt f^{-1}(\{\alpha\})$ for $(i,\alpha) \in
\Upsilon_*$ and $U^{\tII}_1 = \Omega_{\tII} \cap \cll f^{-1}(\RRR_+ \setminus \{0\})$. Then
\begin{multline}\label{eqn:dim}
\dim(\XXX) = \sum_{n,m=1}^{\infty} nm \cdot c_f(U^I_n \cap \Omega_{I_m}) + \aleph_0\sum_{n=1}^{\infty}
c_f(U^I_n \cap \Omega_{I_{\infty}})\\ + \aleph_0 \cdot c_f(U^{\tII}_1) + \sum_{\alpha \in \Card_{\infty}} \alpha
[c_f(U^I_{\alpha}) + c_f(U^{\tII}_{\alpha}) + c_f(U^{\tIII}_{\alpha})].
\end{multline}
\end{pro}
\begin{proof}
As in the proof of \PROp{Dim}, we see that $\dim(\XXX) = \sum_{(i,\alpha) \in \Upsilon} \alpha \cdot
\dim(\EEE^i_{\alpha}(\XXX))$ and $\aleph_0 \cdot \dim(\EEE^i_{\alpha}(\XXX)) = c_*(U^i_{\alpha})
= \aleph_0 \cdot c_f(U^i_{\alpha})$. Moreover, $\dim(\EEE^{\tII}_1(\XXX)) \in \Card_{\infty} \cup \{0\}$.
So, to show \eqref{eqn:dim}, it suffices to check that $\dim(\EEE^I_n(\XXX)) = \aleph_0 \cdot c_f(U^I_n
\cap \Omega_{I_{\infty}}) + \sum_{m=1}^{\infty} m \cdot c_f(U^I_n \cap \Omega_{I_m})$. Write $\EEE^I_n(\XXX) =
\bigsqplus_{m=1}^{m=\infty} \EEE_{n,m}$ with $\EEE_{n,m} \in \IiI_{I_m}$ and observe that $\supp_{\Omega} \EEE_{n,m} =
U^I_n \cap \Omega_{I_m} =: V_{n,m}$. So, it is enough to show that
\begin{equation}\label{eqn:aux88}
\dim(\EEE_{n,m}) = m \cdot c_f(V_{n,m})
\end{equation}
(for $m = \infty$ the above means that $\dim(\EEE_{n,\infty}) = \aleph_0 \cdot c_f(V_{n,\infty})$).
If the set $V_{n,m}$ is infinite, then we may `divide' it into arbitrarily (finitely) many pairwise disjoint nonempty
clopen sets which yields that representatives of $\EEE_{n,m}$ act in infinite-dimensional Hilbert spaces and hence
\eqref{eqn:aux88} is fulfilled in that case (e.g. by \PRO{Dim}). On the other hand, if $V_{n,m}$ is finite, $\EEE_{n,m}$
may be decomposed into $\card(V_{n,m})$ irreducible $N$-tuples of type $I^m$. Now \eqref{eqn:aux88} easily follows since
an irreducible $N$-tuple of type $I^m$ acts in an $m$-dimensional Hilbert space.
\end{proof}

\SECT{Primes, semiprimes, atoms and fractals}

Prime numbers may be defined in two ways (below $n$, $k$ and $l$ are positive integers):
\begin{itemize}
\item $n$ is prime iff $n \neq 1$, and $n = k l$ implies $k = 1$ or $l = 1$,
\item $n$ is prime iff $n \neq 1$, and $n = k l$ implies $k,l \in \{1,n\}$.
\end{itemize}
These two conditions may naturally be adapted in more general algebraic structures (especially in monoids, i.e. semigroups
with neutral elements). However, in some structures they may be nonequivalent. We will see in the sequel that this occurs
in $\CDD_N$. Therefore we distinguish the following two classes of $N$-tuples.

\begin{dfn}{prime}
Let $\AAA \in \CDD_N$ be nontrivial. We say $\AAA$ is a \textit{prime} iff $\AAA = \XXX \oplus \YYY$ implies $\XXX, \YYY
\in \{\zero, \AAA\}$. $\AAA$ is an \textit{atom} iff $\AAA = \XXX \oplus \YYY$ implies $\XXX = \zero$ or $\YYY = \zero$.
\end{dfn}

In case of a single bounded operator, our definition of an atom is equivalent to Ernest's one of an irreducible operator
(\cite{e}). It is clear that every atom is a prime. But not conversely. To convince of this,
let us first shortly prove

\begin{pro}{factor}
For a nontrivial $\AAA \in \CDD_N$ \tfcae
\begin{enumerate}[\upshape(i)]
\item $\WWw'(\aAA)$ is a factor,
\item $\WWw''(\aAA)$ is a factor,
\item $\{\XXX \in \CDD_N\dd\ \XXX \leqsl^s \AAA\} = \{\zero, \AAA\}$,
\item exactly one of the following three conditions is fulfilled:
   \begin{enumerate}[\upshape(a)]
   \item there are unique $\XXX \in \MmM\FfF_N$ and a unique positive cardinal $\alpha$ such that $\AAA = \alpha \odot
      \XXX$ and $\WWw'(\xXX)$ constists precisely of scalar multiples of the identity operator; what is more,
      $0 < \dim(\XXX) \leqsl \aleph_0$,
   \item there are unique $\XXX \in \HhH\IiI\MmM_N$ and a unique infinite cardinal $\alpha$ such that $\AAA = \alpha \odot
      \XXX$ and $\WWw'(\xXX)$ is a (type III) factor; what is more, $\dim(\XXX) = \aleph_0$,
   \item there are (nonunique) $\XXX \in \SsS\MmM_N$ and a unique cardinal $\alpha \in \{1\} \cup \Card_{\infty}$ such that
      $\AAA = \alpha \odot \XXX$ and $\WWw'(\xXX)$ is a (type II$_1$) factor; what is more, $\dim(\XXX) =
      \aleph_0$.
   \end{enumerate}
\end{enumerate}
\end{pro}
\begin{proof}
The points (i) and (ii) are clearly equivalent. Further, it follows from (PR3) \pREF{PR3} that (i) is equivalent to (iii).
Consequently, we infer from \THMp{decomp} that if $\WWw'(\aAA)$ is a factor, then either $\AAA = \EEE_{sm}(\AAA)$ or $\AAA
= \beta \odot \EEE^i_{\beta}(\AAA)$ for some $(i,\beta) \in \Upsilon_*$. In the first situation put $\XXX = \EEE_{sm}$ and
$\alpha = 1$; in the second one we distinguish between two cases: if $i \neq \tII$, put $\XXX = \EEE^i_{\beta}(\AAA)$,
otherwise take $\XXX \in \SsS\MmM_N$ such that $\aleph_0 \odot \XXX = \EEE^{\tII}_{\beta}$; in both cases we put $\alpha =
\beta$. Note that $\AAA = \alpha \odot \XXX$. Further, we conclude from (PR6) \pREF{PR6} that $\WWw'(\aAA)$ is a factor iff
so is $\WWw'(\xXX)$. Now \PROp{count} implies that $\dim(\XXX) \leqsl \aleph_0$ provided $\{\YYY \in \CDD_N\dd\
\YYY \leqsl^s \XXX\} = \{\zero,\XXX\}$. All these notices show that (i) is equivalent to (iv).
\end{proof}

We now have

\begin{pro}{fractal}
Let $\AAA \in \CDD_N$ be nontrivial.
\begin{enumerate}[\upshape(A)]
\item $\AAA$ is an atom iff $\WWw'(\aAA)$ consists precisely of scalar multiples of the identity operator. If $\AAA$ is
   an atom, then $\AAA \leqsl \JJJ_I$ and $0 < \dim(\AAA) \leqsl \aleph_0$.
\item Suppose $\AAA \in \CDD_N$ is not an atom. Then $\AAA$ is a prime iff $\dim(\AAA) = \aleph_0$
   and $\WWw'(\aAA)$ is a type III factor.
\end{enumerate}
\end{pro}
\begin{proof}
Point (A) is left for the reader. We pass to (B).\par
First note that if $\AAA$ is type III, then $\AAA \ll \JJJ_{\tIII}$. Consequently, if in addition $\dim \overline{\AAA} =
\aleph_0$, $\AAA = \EEE^{\tIII}_{\aleph_0}(\AAA)$ and thus $\AAA$ is minimal. But then $\{\XXX \in \CDD_N\dd\ \XXX \leqsl
\AAA\} = \{\XXX \in \CDD_N\dd\ \XXX \leqsl^s \AAA\}$. So, the sufficiency of the conditions formulated in proposition for
$\AAA$ to be a prime follows from \PRO{factor}. Conversely, if $\AAA$ is a prime but not an atom, an application
of \PRO{factor} gives us that $\AAA = \alpha \odot \XXX$ for suitable $\alpha$ and $\XXX$. Since $\XXX \leqsl \AAA$,
we infer that $\AAA = \XXX$. So, $\XXX \notin \MmM\FfF_N$ (because $\AAA$ is not an atom) and $\XXX$ is not semiminimal
since $\zero \neq \frac12 \odot \YYY \lneqq \YYY$ for every nontrivial $\YYY \in \SsS\MmM_N$. We infer from this that $\XXX
\in \HhH\IiI\MmM_N$. Thus, $\WWw'(\aAA)$ is type III and, of course, it is a factor.
\end{proof}

Let $\AAA$ be a prime which is not an atom. It follows from \PRO{fractal} that $\AAA = \aleph_0 \odot \AAA$. Consequently,
$\red(\aAA)$ is an infinite set. However, for every $E \in \red(\aAA)$, $\aAA\bigr|_E \equiv \aAA$ (because $\AAA$ is
prime). Conversely, if $\bBB \in \CDDc_N$ is such that $\card(\red(\bBB)) > 2$ and $\bBB\bigr|_E \equiv \bBB$ for any
$E \in \red(\bBB)$, then $\BBB$ is a prime and not an atom. This observation leads us to

\begin{dfn}{fractal}
A \textit{fractal} is a prime which is not an atom.
\end{dfn}

We see that\label{note} every prime $\AAA$ is either an atom (if $\AAA \neq 2 \odot \AAA$) or a fractal (if $\AAA = 2 \odot
\AAA$) and that $\AAA$ is type I or type III. It is immediate that two different primes are unitarily disjoint.\par
A counterpart of primes for type II $N$-tuples are semiprimes.

\begin{dfn}{semiprime}
A nontrivial $N$-tuple $\AAA$ is said to be a \textit{semiprime} iff $\AAA$ is not of the form $n \odot \BBB$ where
$n$ is a natural number and $\BBB$ is a prime and the following condition is fulfilled: whenever $\zero \neq \XXX \leqsl
\AAA$, there is a natural number $m$ such that $\AAA \leqsl m \odot \XXX$.
\end{dfn}

Semiprimes may be characterized as follows.

\begin{pro}{semiprime}
\begin{enumerate}[\upshape(I)]
\item A nontrivial $N$-tuple $\AAA$ is a semiprime iff $\WWw'(\aAA)$ is a type II$_1$ factor.
\item Let $\AAA$ be a semiprime. Then $\AAA$ is semiminimal and $\dim(\AAA) = \aleph_0$. If $\BBB \ll
   \AAA$, then $\BBB$ is a semiprime iff $\BBB = t \odot \AAA$ for some $t \in \RRR_+ \setminus \{0\}$.
\end{enumerate}
\end{pro}
\begin{proof}
First assume that $\WWw'(\aAA)$ is a type II$_1$ factor. Then necessarily $\AAA \neq n \odot \BBB$ for any prime $\BBB$,
and $\AAA \in \SsS\MmM_N$. Moreover, $\WWw'(\aleph_0 \odot \aAA)$ is a factor as well. We conclude from this that
$\supp_{\Omega_{\tII}} \AAA$ consists of a single point (see Section~14). This implies that if $\zero \neq \XXX \leqsl
\AAA$, then $\DiNT{\XXX}{\AAA} = \lambda \cdot \DiNT{\AAA}{\AAA}$ for some real number $\lambda > 0$. But $\lambda \cdot
\DiNT{\AAA}{\AAA} = \DiNT{(\lambda \odot \AAA)}{\AAA}$ and therefore $\XXX = \lambda \odot \AAA$. Now it suffices to take
a natural number $m$ such that $m \lambda \geqsl 1$ to ensure that $\AAA \leqsl m \odot \XXX$. Consequently, $\AAA$ is
a semiprime.\par
We now assume that $\AAA$ is a semiprime. Observe that then $\AAA = \XXX \sqplus \YYY$ implies $\XXX = \zero$ or $\YYY =
\zero$. We infer from this that $\WWw'(\aAA)$ is a factor. So, according to \PRO{factor}, $\AAA = \alpha \odot \XXX$ for
suitable $\alpha$ and $\XXX$. Since $\AAA$ is a semiprime and $\zero \neq \XXX \leqsl \AAA$, $\alpha \odot \XXX \leqsl m
\odot \XXX$ for some natural number $m$. This implies that either $\AAA = \XXX \in \HhH\IiI\MmM_N$ or $\alpha \leqsl m$.
Again taking into account that $\AAA$ is a semiprime, we see that $\AAA = \XXX \in \SsS\MmM_N$ and hence $\WWw'(\aAA)$
is type II$_1$ and $\dim(\AAA) = \aleph_0$. Further, if $\BBB = t \odot \AAA$, then $\BBB$ is semiminimal
(hence $\WWw'(\bBB)$ is type II$_1$) and the $\WWw^*$-algebras $\ZZz(\WWw'(\bBB))$, $\ZZz(\WWw'(\aleph_0 \odot \bBB))$,
$\ZZz(\WWw'(\aleph_0 \odot \aAA))$ and $\ZZz(\WWw'(\aAA))$ are isomorphic (since $\aleph_0 \odot \BBB = \aleph_0 \odot
\AAA$) which implies that $\WWw'(\bBB)$ is a factor. Consequently, $\BBB$ is a semiprime. Finally, if $\BBB$ is a semiprime
such that $\BBB \ll \AAA$, then from semiminimality of $\BBB$ it follows that $\frac1n \odot \BBB \leqsl \AAA$ for some
natural number $n$. Now the first paragraph of the proof shows that then $\frac1n \odot \BBB = \lambda \odot \AAA$ for
some $\lambda > 0$ and we are done.
\end{proof}

The reader will now easily check that if $\AAA$ is a prime or a semiprime and $\XXX \in \CDD_N$ is arbitrary, then either
$\AAA \leqsl n \odot \XXX$ for some natural number $n$ or $\AAA \disj \XXX$. It turns out that much stronger property
may be established, similar to a suitable property of prime numbers. Namely:

\begin{pro}{divide}
Let $\{\XXX^{(s)}\}_{s \in S} \in \CDD_N$ be a nonempty set and let $\AAA \leqsl \bigoplus_{s \in S} \XXX^{(s)}$.
\begin{enumerate}[\upshape(I)]
\item If $\AAA$ is a prime, there is $s \in S$ such that $\AAA \leqsl \XXX^{(s)}$.
\item Suppose $\AAA$ is a semiprime. For each $s \in S$ let $\lambda_s = \sup \{t \in \RRR_+\dd\ t \odot \AAA \leqsl
   \XXX^{(s)}\} \in \RRR_+ \cup \{\aleph_0\}$. Then $\lambda_s \odot \AAA \leqsl \XXX^{(s)}\ (s \in S)$
   and $\sum_{s \in S} \lambda_s \geqsl 1$.
\end{enumerate}
\end{pro}
\begin{proof}
To prove (I), observe that there is $s \in S$ such that $\AAA$ and $\XXX^{(s)}$ are not unitarily disjoint. Since $\AAA$
is a prime, the latter yields that $\AAA \leqsl \XXX^{(s)}$.\par
We now pass to (II). By (VS3) \pREF{VS3}, $\AAA^{(s)} := \lambda_s \odot \AAA \leqsl \XXX^{(s)}$. Assume that $\lambda_s <
1$ for every $s \in S$ and $\lambda = \sum_{s \in S} \lambda_s < \infty$. By the maximality of $\lambda_s$,
$((1 - \lambda_s) \odot \AAA =) \AAA \ominus \AAA^{(s)} \disj \XXX^{(s)} \ominus \AAA^{(s)} =: \YYY^{(s)}$ and consequently
$\AAA \disj \YYY^{(s)}$. Thus, $\AAA \disj \bigoplus_{s \in S} \YYY^{(s)}$. Now since $\bigoplus_{s \in S} \XXX^{(s)} =
(\bigoplus_{s \in S} \AAA^{(s)}) \oplus (\bigoplus_{s \in S} \YYY^{(s)})$, we infer from (PR1) \pREF{PR1} that $\AAA \leqsl
\bigoplus_{s \in S} \AAA^{(s)}$. Further, we see that $\bigvee \{\bigoplus_{s \in S'} \AAA^{(s)}\dd\ S' \in \PPp_f(S)\} =
\lambda \odot \AAA$. This combined with \PROp{fin} yields that $\lambda \odot \AAA = \bigoplus_{s \in S} \AAA^{(s)}$. So,
$\AAA \leqsl \lambda \odot \AAA$ and hence $\lambda \geqsl 1$.
\end{proof}

Denote by $\aA_N$, $\fF_N$ and $\sS_N$ the sets of all, respectively, atoms, fractals and semiprimes in $\CDD_N$. Further,
for $n=1,2,\ldots,\infty$ let $\aA_N(n)$ be the set of all atoms of type I$^n$. Similarly, we denote by $\sS_N(1)$
and $\sS_N(\infty)$ the sets of all semiprimes of type II$^1$ and II$^{\infty}$, respectively. The reader should notice
that an atom $\AAA$ belongs to $\aA_N(n)$ for some finite $n$ iff $\dim(\AAA) = n$ (and $\AAA \in
\aA_N(\infty)$ iff $\dim(\AAA) = \aleph_0$). Finally, we put $\pP_N = \aA_N \cup \fF_N \cup \sS_N$.

\begin{pro}{a-f-s}
The sets $\aA_N(n)$ ($n=1,2,\ldots,\infty$), $\fF_N$, $\sS_N(1)$ and $\sS_N(\infty)$ have power $2^{\aleph_0}$. Each
of these sets contains a subset of power $2^{\aleph_0}$ consisting of mutually unitarily disjoint $N$-tuples.
\end{pro}
\begin{proof}
Let us first justify that each of the sets $\aA_1(n)$, $\fF_1$, $\sS_1(1)$ and $\sS_1(\infty)$ contains at least one
bounded nonzero operator. For $\aA_1(n)$ this is clear, while for $\fF_1$, $\sS_1(1)$ and $\sS_1(\infty)$ this follows
from the existence of factors of each type and the results on generators of such factors \cite{wog}, \cite{g-s} (the same
was in fact observed by Ernest, cf. \cite[Proposition~1.30]{e}).\par
Now let $T$ be a bounded nonzero operator of a suitable type (here by a \textit{type} we mean an atom of type I$^n$,
a fractal or a semiprime of type II$^n$). Notice that then $\{(rT,\ldots,rT) \in \CDDc_N\dd\ r \in (0,\infty)\}$ is
a family of mutually unitarily disjoint $N$-tuples of the same type as $T$ (indeed, if $\XXX$ is a bounded semiprime,
then $\|t \odot \XXX\| = \|\XXX\|$ for each $t \in \RRR_+ \setminus \{0\}$ and thus $r \xXX \disj s \xXX$ for distinct
$r$ and $s$). This proves the second claim of the proposition. To show the first one, it suffices to apply \LEMp{continuum}
and observe that if $\XXX$ is a semiprime, then $\card (\{\YYY \in \sS_N\dd\ \YYY \not\disj \XXX\}) = \card (\{t \odot
\XXX\dd\ t \in \RRR_+ \setminus \{0\}\}) = 2^{\aleph_0}$.
\end{proof}

As an immediate consequence of \PRO{a-f-s} we obtain the following result, announced in \REM{continuum}
\pREF{rem:continuum}.

\begin{cor}{continuum}
For $i = \{I,\tII,\tIII\}$, $\dim(\JJJ_i) = 2^{\aleph_0}$.
\end{cor}

Denote by $\IiI^d$ the ideal generated by $\pP_N$ and let $\IiI^c = (\IiI^d)^{\perp}$. In other words, $\AAA \in \IiI^d$
if $\AAA = \bigoplus_{\XXX \in \pP_N} \beta_{\XXX} \odot \XXX$ for some family $\{\beta_{\XXX}\}_{\XXX \in \pP_N} \subset
\Card$; and $\AAA \in \IiI^c$ if $\BBB \notin \pP_N$ for every $\BBB \leqsl \AAA$. Similarly, whenever $\AaA$ is an ideal
in $\CDD_N$, $\AaA^d$ and $\AaA^c$ denote, respectively, the ideals $\AaA \cap \IiI^d$ and $\AaA \cap \IiI^c$. $\AaA^d$
and $\AaA^c$ are called the \textit{discrete} and \textit{continuous} parts of $\AaA$. For example, we shall write
$\IiI_{\tIII}^c$, $\IiI_{I_1}^d$, etc. We also define the discrete and continuous parts of every member of $\CDD_N$
and each clopen set in $\Omega$: $\XXX^d = \EEE(\XXX | \IiI^d)$ and $\XXX^c = \EEE(\XXX | \IiI^c)$ for $\XXX \in \CDD_N$;
$\Omega^d = \supp_{\Omega} \IiI^d$ and $\Omega^c = \supp_{\Omega} \IiI^c$; and $E^d = E \cap \Omega^d$ and $E^c = E \cap
\Omega^c$ for a clopen set $E \subset \Omega$. We should underline that classically the terms \textit{discrete}
and \textit{continuous} as kinds of operators mean \textit{type I} and \textit{without type I parts}, respectively
(as it is practised e.g. by Ernest, see \cite[Definition~1.22]{e}).\par
It may be easily checked that $\AAA \in \pP_N$ iff $\widehat{\AAA}$ has the form $\widehat{\AAA} = c \cdot j_{\{x\}}$
where either $c = 1$ and $x \in \Omega_I$ or $c \in \RRR_+ \setminus \{0\}$ and $x \in \Omega_{\tII}$, or $c = \aleph_0$
and $x \in \Omega_{\tIII}$. Therefore $\Omega^d$ is the closure of the set of all isolated points of $\Omega$.
Consequently, we infer from \LEMp{C-S} and \PRO{a-f-s} that

\begin{pro}{C-S}
Each of the spaces $\Omega_{I_n}^d$ ($n=1,2,\ldots,\infty$), $\Omega_{\tII_1}^d$, $\Omega_{\tII_{\infty}}^d$
and $\Omega_{\tIII}^d$ is the \v{C}ech-Stone compactification of the discrete space of cardinality $2^{\aleph_0}$.
\end{pro}

\PRO{C-S} and the next two results will be used later to classify ideals in $\CDD_N$ up to isomorphism (see Section~23
for definition and details).

\begin{pro}{Linfty}
Every nonempty clopen set $E \subset \Omega^c$ with $c(E) = \aleph_0$ is homeomorphic to the Gelfand spectrum
of $L^{\infty}([0,1])$.
\end{pro}
\begin{proof}
There is (unique) nontrivial $\AAA \in \CDD_N$ such that $\AAA \leqsl^s \JJJ$ and $\supp \widehat{\AAA} = E$. Since $E
\subset \Omega^c$ and $c(E) = \aleph_0$,
\begin{equation}\label{eqn:aux77}
\AAA \in \IiI^c \qquad \textup{and} \qquad \dim(\AAA) = \aleph_0.
\end{equation}
Further, since $\ZZz(\WWw'(\jJJ))$ is isomorphic to $\CCc(\Omega)$, $\ZZz(\WWw'(\aAA))$ is isomorphic to $\CCc(E)$
(because $\AAA \leqsl^s \JJJ$). The latter means that $E$ is the Gelfand spectrum of $\ZZz(\WWw'(\aAA))$. Now the assertion
easily follows from \eqref{eqn:aux77} and Theorem~III.1.22 of \cite{tk1}. (The last mentioned result asserts that every
commutative von Neumann algebra acting on a separable Hilbert space which has no nonzero minimal projections is isomorphic
to $L^{\infty}([0,1])$.)
\end{proof}

Now for a clopen set $E \subset \Omega$ let $\kappa_d(E)$ be the power of the set of all isolated points of $E$ and let
$\kappa_c(E) = c_*(E^c)$. Additionally, let us denote by $D(\mM)$ the discrete space of cardinality $\mM$ and by $\Xx$
the Gelfand spectrum of $L^{\infty}([0,1])$. Recall that for every completely regular topological space $X$, $\beta X$
stands for the \v{C}ech-Stone compactification of $X$.

\begin{thm}{homeo}
Any clopen set $E \subset \Omega$ is homeomorphic to the topological disjoint union of $\beta D(\kappa_d(E))$ and $\beta
[D(\kappa_c(E)) \times \Xx]$.
\end{thm}
\begin{proof}
Thanks to \LEMp{count} and \PRO{Linfty}, $E^c$ contains an open dense subset homeomorphic to $D(\kappa_c(E)) \times \Xx$.
Now it suffices to apply \LEMp{C-S} to infer that $E^d$ and $E^c$ are homeomorphic to, respectively, $\beta D(\kappa_d(E))$
and $\beta[D(\kappa_c(E)) \times \Xx]$.
\end{proof}

\begin{exm}{I1-I2}
It is clear that $\aA_N(1)$ is the collection of all $N$-tuples acting on a one-dimensional Hilbert space. So, $\aA_N(1)$
may naturally be identified with $\CCC^N$.\par
One may also easily check that $\aA_N(2)$ consists of all $N$-tuples acting on a two-dimensional Hilbert space which
are not of type $\IiI_1$. In other words, if $\aAA = (A_1,\ldots,A_N)$ where $A_1,\ldots,A_N$ are $2$ by $2$ matrices,
then $\AAA \in \aA_N(2)$ iff $A_j A_k^* \neq A_k^* A_j$ for some $j,k \in \{1,\ldots,N\}$.\par
For $n \geqsl 3$ the characterization of members of $\aA_N(n)$ is much more complicated.
\end{exm}

\SECT{Strongly unitarily disjoint families}

Thanks to (BT3) \pREF{BT3} and suitable characterizations of the kinds of $N$-tuples appearing below, we see that for every
$\XXX \in \CDD_N$ the following equivalences hold true:
\begin{quote}
\textit{$\XXX$ is type $I$, $I^n$, $\tII$, $\tII^1$, $\tII^{\infty}$, $\tIII$, minimal, multiplicity free, a hereditary
idempotent, semiminimal, a prime, an atom, a fractal or a semiprime iff so is $\bB(\XXX)$.}
\end{quote}
However, so far there was no need to use the $\bB$-transform, beside \THMp{common}. From now on, this transform will
intensively be involved and without it the presentation would be much more complicated.\par
We say that two classes $\AAa, \BBb \subset \CDDc_N$ are unitarily disjoint iff $\AAa \disj \BBb$, that is, if $\aAA \disj
\bBB$ for any $\aAA \in \AAa$ and $\bBB \in \BBb$. We begin with a classical

\begin{pro}{disjoint}
Let $\aAA, \bBB \in \CDDc_N$ be nontrivial $N$-tuples and let $\xXX = \aAA \oplus \bBB$. \TFCAE
\begin{enumerate}[\upshape(i)]
\item $\aAA \disj \bBB$,
\item $\WWw'(\xXX) = \{S \oplus T\dd\ S \in \WWw'(\aAA),\ T \in \WWw'(\bBB)\} =: \WWw'(\aAA) \oplus \WWw'(\bBB)$,
\item $I \oplus 0 \in \WWw''(\xXX)$ (where $I$ is the identity operator on $\overline{\DdD}(\aAA)$ and $0$ is the zero
   operator on $\overline{\DdD}(\bBB)$).
\end{enumerate}
\end{pro}
\begin{proof}
Using $\bB$-transform and taking into account properties (BT3)--(BT5) \pREF{BT3}, we may assume that $\aAA$ and $\bBB$ are
bounded. In that case the equivalence of (i) and (ii) follows from Schur's lemma (cf. Theorem~1.5 in \cite{e}; see also
Corollary~1.8 there). Further, (ii) easily implies (iii), since $I \oplus 0$ commutes with every member of $\WWw'(\aAA)
\oplus \WWw'(\bBB)$. Finally, if (iii) is satisfied, then all elements of $\WWw'(\xXX)$ commute with $I \oplus 0$ and thus
are of the form $S \oplus T$. It is now easily verified that $S \oplus T$ commutes with each entry of $\xXX$ \iaoi{}
$S \in \WWw'(\aAA)$ and $T \in \WWw'(\bBB)$.
\end{proof}

We are mainly interested in the equivalence of (i) and (iii) in \PRO{disjoint}.\par
Adapting the concept due to Ernest \cite{e} (see Definition~1.31 and \S5.7.f there, especially notes on page~187 there),
let us consider the free complex algebra $F = F(z_1,\ldots,z_N;w_1,\ldots,w_N)$ on $2N$ non-commuting variables
$z_1,\ldots,z_N,w_1,\ldots,w_N$. Each member of $F$ may naturally be identified with a polynomial in $2N$ non-commuting
variables. Let $*$ be a unique involution on the algebra $F$ such that $z_j^* = w_j$ for $j=1,\ldots,N$. We denote
by $\PpP(N)$ obtained in this way $*$-algebra equipped with a norm gived by
$$
\|p(z_1,\ldots,z_N;z_1^*,\ldots,z_N^*)\| = \sup_{\|T_j\| \leqsl 1} \|p(T_1,\ldots,T_N;T_1^*,\ldots,T_N^*)\|
$$
where the supremum is taken over all $N$-tuples of contractions acting on a (common, arbitrary) Hilbert space. It follows
from the definition that for every $p \in \PpP(N)$ and $\xXX \in \CDDc_N$ with $\|\xXX\| \leqsl 1$, $\|p(\xXX,\xXX^*)\|
\leqsl \|p\|$. The following is left as an easy exercise (use the separability of $\PpP(N)$).

\begin{lem}{P2N}
There is a sequence $\{\mMM_n\}_{n=1}^{\infty}$ of atoms in $\CDDc_N$ acting on finite-dimensional Hilbert spaces such that
$\|\mMM_n\| \leqsl 1\ (n\geqsl1)$ and for every $p \in \PpP(N)$,
$$
\|p\| = \sup_{n\geqsl1} \|p(\mMM_n,\mMM_n^*)\|.
$$
\end{lem}

With use of the above result and Kaplansky's density theorem \cite{kap} (cf. \cite[Theorem~5.3.5]{k-r1},
\cite[Theorem~II.4.8]{tk1}, \cite[Theorem~1.9.1]{sak}) we shall now prove a result which is a starting point for our
further investigations. By $\PpP_1(N)$ we denote the closed unit ball of $\PpP(N)$. Everywhere below $I$ and $0$ denote
the identity and zero operators on suitable Hilbert spaces. Recall that a net $(T_{\sigma})_{\sigma \in \Sigma}$
of bounded operators acting on a Hilbert space $\HHh$ converges $*$-strongly to an operator $T \in \BBb(\HHh)$ iff for any
$x \in \HHh$, $T_{\sigma} x \to Tx\ (\sigma \in \Sigma)$ and $T_{\sigma}^* x \to T^* x\ (\sigma \in \Sigma)$. We shall
denote this by $T_{\sigma} \stSt T$.

\begin{pro}{wdisj}
\begin{enumerate}[\upshape(I)]
\item Let $\AAa$ and $\BBb$ be arbitrary subsets of $\CDDc_N$. \TFCAE
   \begin{enumerate}[\upshape(i)]
   \item $\AAa$ and $\BBb$ are unitarily disjoint,
   \item there is a net $(p_{\sigma})_{\sigma \in \Sigma} \subset \PpP_1(N)$ such that for any $\aAA \in \AAa$ and $\bBB
      \in \BBb$, $p_{\sigma}(\bB(\aAA),\bB(\aAA)^*) \stSt I$ and $p_{\sigma}(\bB(\bBB),\bB(\bBB)^*) \stSt 0$.
   \end{enumerate}
\item If $\aAA$ and $\bBB$ are two $N$-tuples acting in separable Hilbert spaces, then $\aAA \disj \bBB$ iff there is
   a sequence $(p_n)_{n=1}^{\infty} \subset \PpP_1(N)$ such that $p_n(\bB(\aAA),\bB(\aAA)^*) \stSt I$
   and $p_n(\bB(\bBB),\bB(\bBB)^*) \stSt 0$.
\end{enumerate}
\end{pro}
\begin{proof}
(I): By (BT5) \pREF{BT5}, (i) is implied by (ii). To prove the converse, assume $\AAa \disj \BBb$. Let $\aAA = \bigoplus
\{\xXX\dd\ \xXX \in \AAa\}$ and $\bBB = \bigoplus \{\yYY\dd\ \yYY \in \BBb\}$. Thanks to (PR2) \pREF{PR2}, $\aAA \disj
\bBB$. Further, let $\{\mMM_n\}_{n=1}^{\infty}$ be as in \LEM{P2N}. Let $\mMM$ be the direct sum of all $\mMM_n$'s which
are unitarily disjoint from $\bB(\bBB)$ ($\mMM$ is trivial provided $\mMM_n \leqsl \bB(\bBB)$ for each $n$). Again by (PR2)
and (BT5), $\mMM \oplus \bB(\aAA) \disj \bB(\bBB)$. Put $\xXX = (\mMM \oplus \bB(\aAA)) \sqplus \bB(\bBB)$, $\HHh_1 =
\overline{\DdD}(\mMM \oplus \bB(\aAA))$ and $\HHh_2 = \overline{\DdD}(\bB(\bBB))$. It follows from our construction that
for each $p \in \PpP(N)$,
\begin{equation}\label{eqn:normP2N}
\|p\| = \|p(\xXX,\xXX^*)\|.
\end{equation}
Let $\MmM = \{p(\xXX,\xXX^*)\dd\ p \in \PpP(N)\}$. $\MmM$ is a unital selfadjoint subalgebra of $\BBb(\HHh_1 \oplus
\HHh_2)$. We infer from von Neumann's double commutant theorem \cite{vN1} (\cite[Theorem~5.3.1]{k-r1},
\cite[Theorem~II.3.9]{tk1}, \cite[Theorem~1.20.3]{sak}) that the closure of $\MmM$ in the strong operator topology
coincides with $\WWw''(\xXX)$. Further, \eqref{eqn:normP2N} yields that the closed unit ball in $\MmM$ coincides with
$\{p(\xXX,\xXX^*)\dd\ p \in \PpP_1(N)\}$. An application of \PRO{disjoint} shows that $I \oplus 0 \in \WWw''(\xXX)$ where
$I \in \BBb(\HHh_1)$ and $0 \in \BBb(\HHh_2)$. Finally, Kaplansky's density theorem asserts that there is a net
$(p_{\sigma})_{\sigma \in \Sigma} \in \PpP_1(N)$ such that $p_{\sigma}(\xXX,\xXX^*) \stSt I \oplus 0$. Since every member
of $\AAa$ and $\BBb$ is a summand of $\aAA$ and $\bBB$, respectively, the assertion of (ii) is fulfilled.\par
To prove (II), repeat the above argument and observe that in that case both $\HHh_1$ and $\HHh_2$ are separable and hence
Kaplansky's density theorem asserts the existence of a suitable sequence, since the closed unit ball in $\BBb(\HHh)$ for
separable $\HHh$ is metrizable in the $*$-strong topology (see e.g. \cite[Proposition~2.2]{e}).
\end{proof}

Let us now introduce the following

\begin{dfn}{sdisj}
Let $\AAa$ and $\BBb$ be arbitrary collections (sets or classes) of $N$-tuples. We say that $\AAa$ and $\BBb$ are
\textit{strongly unitarily disjoint}, in symbol $\AAa \sdisj \BBb$, if there is a sequence $(p_n)_{n=1}^{\infty} \subset
\PpP_1(N)$ such that $p_n(\bB(\aAA),\bB(\aAA)^*) \oplus p_n(\bB(\bBB),\bB(\bBB)^*) \stSt I \oplus 0$ for any $\aAA \in
\AAa$ and $\bBB \in \BBb$. Two $N$-tuples $\xXX, \yYY \in \CDDc_N$ are strongly unitarily disjoint ($\xXX \sdisj \yYY$)
provided so are the sets $\{\xXX\}$ and $\{\yYY\}$.
\end{dfn}

The reader should easily notice that for two sets $\AAa$ and $\BBb$ of $N$-tuples, $\AAa \sdisj \BBb$ iff $(\bigoplus \AAa)
\sdisj (\bigoplus \BBb)$. It is also clear that if $\AAa$ and $\BBb$ are strongly unitarily disjoint, then $\AAa \disj
\BBb$.

\begin{rem}{sdisj}
Let $\aAA$ and $\aAA'$ be two unitarily equivalent $N$-tuples. Observe that then $p(\bB(\aAA),\bB(\aAA)^*) \equiv
p(\bB(\aAA'),\bB(\aAA')^*)$ for every $p \in \PpP(N)$. What is more, for every complex number $\lambda$ and a net
$(p_{\sigma})_{\sigma \in \Sigma} \subset \PpP(N)$, $p_{\sigma}(\bB(\aAA),\bB(\aAA)^*) \to \lambda I$ $*$-strongly
(strongly, weakly, etc.) iff $p_{\sigma}(\bB(\aAA'),\bB(\aAA')^*) \to \lambda I$ in the same topology. This means that for
any $\AAA \in \CDD_N$ and $p \in \PpP(N)$, $p(\bB(\AAA),\bB(\AAA)^*)$ is a well defined member of $\CDD$ and
\begin{equation}\label{eqn:conver}
p_{\sigma}(\bB(\AAA),\bB(\AAA)^*) \stSt \lambda I
\end{equation}
is well understood. (We do not write in \eqref{eqn:conver} `$\III$' instead of `$I$' because `$I$' represents here
the identity operator on a Hilbert space of (suitable) arbitrary dimension. The usage of $\III$ may lead
to misunderstandings. In fact, \eqref{eqn:conver} expresses only a property of the net
$\{p_{\sigma}(\bB(\AAA),\bB(\AAA)^*)\}_{\sigma\in\Sigma}$.) Consequently, in the samy way as in \DEF{sdisj} we may define
strongly unitarily disjoint subclasses of $\CDD_N$. We follow this concept in next sections.
\end{rem}

Surely the main problem concerning the strong unitary disjointness is the question of when two unitarily disjoint families
of $N$-tuples acting in separable Hilbert spaces are strongly unitarily disjoint. We will not answer this question in this
treatise. However, the reader should remember that strong unitary disjointness and unitary disjointness are nonequivalent
even for families of $N$-tuples acting on a one-dimensional Hilbert space. Indeed, such $N$-tuples may naturally be
identified with points of $\CCC^N$. If $p_1,p_2,\ldots$ is an arbitrary sequence of members of $\PpP(N)$ and $\lambda \in
\CCC$, the set $\{z \in \CCC^N\dd\ p_n(\bB(z),\bB(z)^*) \to \lambda\}$ is $\FFf_{\sigma\delta}$ in $\CCC^N$. Thus,
if $A \subset \CCC^N$ is not $\FFf_{\sigma\delta}$, $A \disj \CCC^N \setminus A$ but $A$ and $\CCC^N \setminus A$ are not
strongly unitarily disjoint.\par
The next result is a consequence of \PRO{wdisj}. We omit its proof.

\begin{pro}{sdisj-wdisj}
Let $\AAa$ and $\BBb$ be two \textbf{countable} families of $N$-tuples acting in separable Hilbert spaces. Then $\AAa \disj
\BBb$ \iaoi{} $\AAa \sdisj \BBb$.
\end{pro}

In the sequel we shall also need the following simple

\begin{lem}{separ}
Let $\aAA$ be a bounded $N$-tuple acting on a separable Hilbert space such that $\|\aAA\| \leqsl 1$. For every
$T \in \WWw(\aAA)$ with $\|T\| \leqsl 1$ there is a sequence $(p_n)_{n=1}^{\infty} \subset \PpP_1(N)$ such that
$p_n(\aAA,\aAA^*) \stSt T$.
\end{lem}
\begin{proof}
We mimic the proof of \PRO{wdisj}. As there, we see that there is a sequence $\{\mMM_n\}_{n=1}^{\infty}$ of $N$-tuples
of contraction matrices such that $\mMM_n \disj \aAA$ for each $n$ and $\|p\| = \|p(\mMM \oplus \aAA,\mMM^* \oplus
\aAA^*)\|$ for every $p \in \PpP(N)$ with $\mMM = \bigoplus_{n=1}^{\infty} \mMM_n$. Since $\mMM \disj \aAA$, $\WWw'(\mMM
\oplus \aAA) = \WWw'(\mMM) \oplus \WWw'(\aAA)$ (by \PRO{disjoint}). Consequently, $\WWw''(\mMM) \oplus \WWw''(\aAA) \subset
\WWw''(\mMM \oplus \aAA)$ and thus $0 \oplus T \in \WWw(\mMM \oplus \aAA)$. Finally, since $\mMM \oplus \aAA$ acts
on a separable Hilbert space, Kaplansky's density theorem finishes the proof (cf. the proof of \PRO{wdisj}).
\end{proof}

\begin{rem}{sdisj2}
Let $\Pp = \{p_{\sigma}\}_{\sigma\in\Sigma} \subset \PpP_1(N)$ be any net and let $\lambda \in \CCC$. Denote
by $\IiI_{\Pp}(\lambda)$ the class of all $\XXX \in \CDD_N$ for which
$$p_{\sigma}(\bB(\XXX),\bB(\XXX)^*) \stSt \lambda I.$$
One easily checks that $\IiI_{\Pp}(\lambda)$ is an ideal and $\IiI_{\Pp}(\lambda) \disj \IiI_{\Pp}(\lambda')$ whenever
$\lambda' \neq \lambda$. For every subclass $\AaA$ of $\CDD_N$ let $J(\AaA)$ denote the smallest ideal in $\CDD_N$ which
contains $\AaA$. The above notice shows that for arbitrary two subclasses $\AaA$ and $\BbB$ of $\CDD_N$, $\AaA \sdisj \BbB$
iff $J(\AaA) \sdisj J(\BbB)$, iff $\JJJ(J(\AaA)) \sdisj \JJJ(J(\BbB))$. In particular, strong unitary disjointness of sets
or classes may always be reduced to the issue of strong unitary disjointness of suitable $N$-tuples $\XXX$ and $\YYY$
such that $\XXX \leqsl^s \JJJ$ and $\YYY \leqsl^s \JJJ$.
\end{rem}

\SECT{Measure-theoretic preliminaries}

Our next objective is to propose the prime decomposition of $N$-tuples (\THMP{prime}). Essentially this will be based
on the same idea (that is, on central decompositions of von Neumann algebras) as Ernest's central decomposition
of a bounded operator (\cite[Chapter~3]{e}). The difference between his and our approaches (beside bigger generality) is
the following. Ernest has focused on a single operator $T$ and studied its (nonscalar) spectrum $\widehat{T}$
and quasi-spectrum $\widetilde{T}$. Central decomposition of the operator $T$ `takes place' in $\widetilde{T}$. Further
the author compares operators (and their central decompositions) which have the same quasi-spectra. It seems to us that
Ernest's work was inspired by the spectral theorem for a normal operator. Our treatise is inspired by the prime
decomposition of natural numbers. Our interpretation is therefore in a more algebraic fashion. Also comparing Ernest's work
and our, we may say that his approach is local, while our is global.\par
The road to Prime Decomposition Theorem is long because of mea\-sure-theoretic technicalities. First we shall define
the Borel structure on $\CDD_N$ (this is done in this part), next we shall generalize the notion of a direct integral
to the context of $N$-tuples (Section~20) to define `continuous' direct sums (Section~21) among which we shall distinguish
regular ones (which deal with unitary disjointness) and finally we shall show that every member of $\CDD_N$ admits
an (in a sense) unique regular prime decomposition (Section~22).\par
The concept of direct integrals (of Hilbert spaces, operators, von Neumann algebras, etc.) is essentially due to
von Neumann and is widely discussed in many classical textbooks on von Neumann algebras. Here we shall focus on main ideas
and many proofs will be omitted. The reader interested in details should consult e.g. Chapters~2 and~3 of \cite{e};
\cite{ef1,ef2}; \cite[Chapter~14]{k-r2}; \S{}IV.8, \S{}V.6 and Appendix in \cite{tk1}; \cite[Chapter~3]{sak};
\cite[Chapter~I]{sch}; or the original paper by von Neumann \cite{vN2}. It is also assumed that the reader is well oriented
in basics of measure theory as well as in aspects of reduction theory of von Neumann algebras.\par
Measurable sets (i.e. elements of a given $\sigma$-algebra) will also be called \textit{Borel}. Everywhere below
by a \textit{measurable} or \textit{Borel} function of a measurable space $(X,\Mm)$ into a measurable space $(Y,\Nn)$
we mean any function $f\dd X \to Y$ such that $f^{-1}(B) \in \Mm$ for any $B \in \Nn$. The function $f$ is a \textit{Borel
isomorphism} if $f$ is a bijection and $f$ and $f^{-1}$ are measurable. For two measures $\mu$ and $\nu$ defined
on a common $\sigma$-algebra $\Mm$ we shall write $\mu \ll \nu$ iff $\mu$ is absolutely continuous with respect to $\nu$
and we call $\mu$ and $\nu$ (mutually) \textit{singular} iff $\mu \perp \nu$, i.e. if $\mu$ and $\nu$ are concetrated
on disjoint measurable sets. If $A \in \Mm$, $\mu\bigr|_A$ denotes the measure on $\Mm$ given by $\mu\bigr|_A(B) =
\mu(A \cap B)$. For a topological space $X$, $\Bb(X)$ stands for the smallest $\sigma$-algebra containing all open subsets
of $X$. Following Takesaki \cite[Appendix]{tk1}, we call a measurable space $(X,\Mm)$ a \textit{standard Borel} space iff
$(X,\Mm)$ is Borel isomorphic to $(Y,\Bb(Y))$ where $Y$ is a Borel subset of a separable complete metric space.
Equivalently, $(X,\Mm)$ is standard iff $(X,\Mm)$ is Borel isomorphic to $(A,\Bb(A))$ where $A$ is a countable (finite
or not) subset of $[0,1]$ or $A = [0,1]$ (cf. \cite[Corollary~A.11]{tk1}). If $(X,\Mm)$ and $(Y,\Nn)$ are standard Borel
spaces and $f\dd X \to Y$ is measurable, then $(X \times Y, \Mm \otimes \Nn)$ is a standard Borel space as well
and $\Gamma(f) \in \Mm \otimes \Nn$ where $\Gamma(f) = \{(x,f(x))\dd\ x \in X\}$ is the graph of $f$. The space $(X,\Mm)$
is \textit{Souslin-Borel} iff it is the image of a standard Borel space under a Borel function and $X$ is \textit{countably
separated} (the latter means that there are sets $E_1,E_2,\ldots \in \Mm$ such that for any two distinct points $x$ and $y$
of $X$ there is $n$ with $\card(\{x,y\} \cap E_n) = 1$). In what follows, we shall often identify $I_{\aleph_0}$ with
$[0,\infty]$.\par
Let $(X,\Mm,\mu)$ be a measure space ($\mu$ need not be $\sigma$-finite nor complete). We denote by $\NnN(\mu)$ the null
$\sigma$-ideal in $\Mm$ induced by $\mu$, that is, $\NnN(\mu) = \{A \in \Mm\dd\ \mu(A) = 0\}$. $(X,\Mm,\mu)$ is said to be
a \textit{standard measure} space (or, equivalently, $\mu$ is \textit{standard}) iff $\mu$ is nonzero $\sigma$-finite
and $X \setminus Z$ is a standard Borel space for some $Z \in \NnN(\mu)$. By \cite[Corollary~A.14]{tk1}, every
$\sigma$-finite measure on a Souslin-Borel space is standard.\par
For $n=1,2,\ldots$ let $\HHh_n$ be a fixed Hilbert space of dimension $n$ and let $\HHh_{\infty}$ be a fixed separable
infinite-dimensional Hilbert space (these spaces are fixed for this and the next two sections). Further, let $\HHh$ denote
one of the spaces $\HHh_1,\HHh_2,\ldots,\HHh_{\infty}$. The norm and the weak topologies of $\HHh$ induces the same
$\sigma$-algebra on $\HHh$ which is for us the default Borel structure of $\HHh$. Similarly, the $*$-strong, strong
and weak operator topologies induces the same Borel structures on $\BBb(\HHh)$. In other words, the $\sigma$-algebra
$\Ww_{\HHh}$ generated by all open sets with respect to any of these topologies is independent of the topology we choose.
Moreover, $(\BBb(\HHh),\Ww_{\HHh})$ is a standard Borel space, which means that $(\BBb(\HHh),\Ww_{\HHh})$ is isomorphic
as a measurable space to $([0,1],\Bb([0,1]))$. The addition and multiplication are measurable as functions of $(\BBb(\HHh)
\times \BBb(\HHh), \Ww_{\HHh} \otimes \Ww_{\HHh})$ into $(\BBb(\HHh),\Ww_{\HHh})$ and the functions $T \mapsto T^*$,
$T \mapsto |T|$, $T \mapsto Q_T$ and $T \mapsto T^{-1}$ are measurable as well (the last mentioned function comes from
the set of all invertible operators which is measurable).\par
The following result will enable us to define a Borel structure on the set $\CDDc(\HHh)$.

\begin{lem}{b-meas}
The open unit ball $B$ of $\BBb(\HHh)$ and the set $\bB(\HHh)$ of all $T \in \BBb(\HHh)$ such that $\|Tx\| < \|x\|$ for any
nonzero $x \in \HHh$ are measurable. The $\bB$-transform is an isomorphism between measurable spaces $\BBb(\HHh)$ and $B$.
\end{lem}
\begin{proof}
We shall only explain why $\bB(\HHh)$ is measurable. Notice that $T \in \bB(\HHh)$ iff $\NnN(I - T^* T)$ is trivial. Now
if $P_T$ denotes the orthogonal projection onto $\NnN(I - T^* T)$, then the function $T \mapsto P_T$ is measurable, thanks
to \cite[Proposition~2.4]{e}, and we are done.
\end{proof}

Since the $\bB$-transform establishes a one-to-one correspondence between members of $\CDDc(\HHh)$ and $\bB(\HHh)$,
we may introduce

\begin{dfn}{b-Borel}
The \textit{Borel structure} of $\CDDc(\HHh)$ is the unique Borel structure which makes the $\bB$-transform an isomorphism.
In other words, a set $F \subset \CDDc(\HHh)$ is measurable, in symbol $F \in \Bb(\CDDc(\HHh))$, iff $\{\bB(X)\dd\
X \in F\} \in \Ww_{\HHh}$.
\end{dfn}

\LEM{b-meas} implies that $\CDDc(\HHh)$ is a standard Borel space, that $\BBb(\HHh)$ is a measurable subset
of $\CDDc(\HHh)$ and that the original Borel structure of $\BBb(\HHh)$ coincides with the one inherited from the Borel
structure of $\CDDc(\HHh)$.\par
Recall that $\CDDc_N(\HHh) = \CDDc(\HHh)^N$. We equip $\CDDc_N(\HHh)$ with the product $\sigma$-algebra
$\Bb(\CDDc_N(\HHh)) = \Bb(\CDDc(\HHh)) \otimes \ldots \otimes \Bb(\CDDc(\HHh))$. Observe that $\CDDc_N(\HHh)$ is a standard
Borel space and the $\bB$-transform is an isomorphism of the measurable space $\CDDc_N(\HHh)$ onto a measurable set
$\bB(\HHh)^N$. Moreover, it follows from suitable properties of the $\bB$-transform that each of the functions
$\xXX \mapsto \xXX^*$, $\xXX \mapsto |\xXX|$ and $\xXX \mapsto \qQQ_{\xXX}$ (of $\CDDc_N(\HHh)$ into itself) are
measurable.\par
Now let $\SsS\EeE\PpP_N$ be the \textbf{set} of all $\AAA \in \CDD_N$ such that $0 < \dim(\AAA) \leqsl
\aleph_0$. Observe that the function $\Phi\dd \bigcup_{n=1}^{n=\infty} \CDDc_N(\HHh_n) \ni \xXX \mapsto \XXX \in
\SsS\EeE\PpP_N$ is a surjection. We define a $\sigma$-algebra $\Bb_N$ on $\SsS\EeE\PpP_N$ by the rule: $\FfF \in \Bb_N$
iff for every $n \in \{1,2,3,\ldots,\infty\}$, $\Phi^{-1}(\FfF) \cap \CDDc_N(\HHh_n) \in \Bb(\CDDc_N(\HHh_n))$. It is
obviously seen that the definition of $\Bb_N$ is independent of the choice of $\HHh_n$'s. For every $\AaA \in \Bb_N$
we shall denote by $\Bb(\AaA)$ the $\sigma$-algebra of all sets $\BbB \in \Bb_N$ contained in $\AaA$.\par
As it was shown by Ernest (see \cite[Corollary~2.33]{e}), $\SsS\EeE\PpP_N$ is not countably separated. This causes that
investigating of the Borel structure of $\SsS\EeE\PpP_N$ is difficult and complicated. The rest of this section is devoted
to establish measurability of some (important for us) sets and functions. For $n = 1,2,\ldots,\infty$ let
$\SsS\EeE\PpP_N(n)$ consist of all $\AAA \in \SsS\EeE\PpP_N$ with $\dim(\AAA) = n$. It follows from the definition
of $\Bb_N$ that $\SsS\EeE\PpP_N(n) \in \Bb_N$ for every $n$. When $n$ is finite, much more can be said about
$\SsS\EeE\PpP_N(n)$ (cf. Proposition~2.46 and Corollary~2.47 in \cite{e}). Namely,

\begin{pro}{findim}
For every finite $n$, $\SsS\EeE\PpP_N(n)$ is a standard Borel space and there are a Borel set $S_n \subset \CDDc_N(\HHh_n)$
and a Borel isomorphism $\chi_n\dd \SsS\EeE\PpP_N(n) \ni \AAA \mapsto \tTT_{\AAA} \in S_n$ such that $\tTT_{\AAA}$ is
a representative of $\AAA$ for every $\AAA$.
\end{pro}
\begin{proof}
It is clear that $\CDD_N(\HHh_n)$ coincides with the space $M_n^N$ of all $N$-tuples of $n \times n$ matrices.
Let $\pi\dd M_n^N \to \SsS\EeE\PpP_N(n)$ be the quotient map (i.e. $\pi(\xXX) = \XXX$). Equip $\SsS\EeE\PpP_N(n)$ with
the quotient topology (induced by $\pi$). Since the unitary group of $n \times n$ matrices is compact, $\SsS\EeE\PpP_N(n)$
is locally compact and $\pi$ is a proper continuous mapping. What is more, $\SsS\EeE\PpP_N(n)$ is separable and metrizable.
It is now clear that the $\sigma$-algebra generated by all open sets coincides with the one inherited from $\Bb_N$. This
yields that $\SsS\EeE\PpP_N(n)$ is a standard Borel space. The existence of $S_n$ and $\chi_n$ may easily be deduced e.g.
from \cite[Corollary~XIV.2.1]{k-m} applied to the partition $\{\pi^{-1}(\{\XXX\})\dd\ \XXX \in \SsS\EeE\PpP_N(n)\}$,
or from \cite[Corollary~XIV.1.1]{k-m} (see also \cite{cas}) applied to the multifunction $\SsS\EeE\PpP_N(n) \ni \XXX
\mapsto \pi^{-1}(\{\XXX\}) \subset M_n$.
\end{proof}

We are now mainly interested in the Borel structure of $\SsS\EeE\PpP_N(\infty)$. However, in some arguments we shall need
to work also with $N$-tuples acting on finite-dimensional Hilbert spaces and therefore below we explore
$\CDDc_N(\HHh_{\infty})$ as well as $\CDDc_N(\HHh_n)$ with finite $n$. Since our main interest are primes and semiprimes,
we may restrict our considerations to factor $N$-tuples defined below. Similar results to those presented below the reader
may find in Chapter~2 of \cite{e}.\par
As before, $\HHh$ denotes one of the spaces $\HHh_1,\HHh_2,\ldots,\HHh_{\infty}$. The functions $\CDDc_N(\HHh) \ni \xXX
\mapsto \WWw''(\xXX) \in \wwW(\HHh)$ and $\CDDc_N(\HHh) \ni \xXX \mapsto \WWw'(\xXX) \in \wwW(\HHh)$ are measurable when
$\wwW(\HHh)$ denotes the collection of all von Neumann subalgebras of $\BBb(\HHh)$ and is equipped with the Effros Borel
structure \cite{ef1,ef2} (cf. \cite[page~54]{e} combined with Theorem~IV.8.4 and Corollary~IV.8.6 in \cite{tk1}).
Consequently, the following sets are measurable subsets of $\CDDc_N(\HHh)$ (compare with notes on page~55 of \cite{e};
\cite[Theorem~V.6.6]{tk1} and \cite{nie}):
\begin{itemize}
\item the set of all atoms $\aA_N(\HHh) = \{\aAA \in \CDDc_N(\HHh)\dd\ \AAA \in \aA_N\}$,
\item the set of all fractals $\fF_N(\HHh) = \{\aAA \in \CDDc_N(\HHh)\dd\ \AAA \in \fF_N\}$,
\item the set of all semiprimes $\sS_N(\HHh) = \{\aAA \in \CDDc_N(\HHh)\dd\ \AAA \in \sS_N\}$,
\item the set of all factor $N$-tuples $$\Ff_N(\HHh) = \{\aAA \in \CDDc_N(\HHh)\dd\ \WWw''(\aAA) \textup{ is a factor}\},$$
\item the sets of all factor $N$-tuples of type I, I$^n$, II, II$^1$, II$^{\infty}$ and III.
\end{itemize}
(The above properties imply that $\aA_N, \fF_N, \sS_N$ as well as $\Ff_N := \{\FFF \in \SsS\EeE\PpP_N\dd\ \WWw''(\fFF)
\textup{ is a factor}\}$ are members of $\Bb_N$. When $\HHh$ is finite-dimensional, $\sS_N(\HHh)$ and $\fF_N(\HHh)$ are
of course empty.) We infer from \PROp{factor} that for every $\fFF \in \Ff_N(\HHh) \setminus (\aA_N(\HHh) \cup \fF_N(\HHh)
\cup \sS_N(\HHh))$ either there exist a unique $n \in \{2,3,4,\ldots,\aleph_0\}$ and a unique $\AAA \in \aA_N$ such that
$\FFF = n \odot \AAA$ or there is (nonunique) $\AAA \in \sS_N$ for which $\FFF = \aleph_0 \odot \AAA$.\par
Everywhere below $n$ and $m$ represent positive integer as well as $\infty$.\par
The following result appears in \cite{e} as Corollary~2.11. Below we give a shorter proof.

\begin{lem}{1}
The set $$\Dd_N(n,m) = \{(\aAA,\bBB) \in \CDDc_N(\HHh_n) \times \CDDc_N(\HHh_m)\dd\ \aAA \disj \bBB\}$$ is measurable
(i.e. $\Dd_N(n,m) \in \Bb(\CDDc_N(\HHh_n)) \otimes \Bb(\CDDc_N(\HHh_m))$).
\end{lem}
\begin{proof}
Let $\KKk = \HHh_{\infty}$ and let $U_j\dd \aleph_0 \odot \HHh_j \to \KKk$ be unitary ($\aleph_0 \odot \HHh_j$ symbolizes
the Hilbert space in which act $N$-tuples of the form $\aleph_0 \odot \xXX$ with $\xXX \in \CDDc_N(\HHh_j)$). Let $\QqQ$ be
the set of all $p \in \PpP$ such that $\|p\| \leqsl 2$ and all coefficients of $p$ belong to $\QQQ + i\QQQ$. It may be
deduced from \PROp{disjoint} and \LEMp{separ} that $\aAA \disj \bBB$ where $\aAA \in \CDDc_N(\HHh_n)$ and $\bBB \in
\CDDc_N(\HHh_m)$ iff there is a sequence $(p_k)_{k=1}^{\infty} \subset \QqQ$ such that $U_n p_k(\bB(\aleph_0 \odot \aAA),
\bB(\aleph_0 \odot \aAA)^*) U_n^{-1} \to I$ and $U_m p_k(\bB(\aleph_0 \odot \bBB),\bB(\aleph_0 \odot \bBB)^*) U_m^{-1} \to
0$ strongly as $k \to \infty$. Now if $d$ is a metric on $D = \{T \in \BBb(\KKk)\dd\ \|T\| \leqsl 2\}$ which induces
the strong operator topology of $D$, then for every $p \in \QqQ$ the function $\psi_p^j\dd \CDDc_N(\HHh_j) \ni \xXX \mapsto
U_j p(\bB(\aleph_0 \odot \xXX),\bB(\aleph_0 \odot \xXX)^*) U_j^{-1} \in D$ is measurable and thus so is $\theta_p\dd
\CDDc_N(\HHh_n) \times \CDDc_N(\HHh_m) \ni (\xXX,\yYY) \mapsto d(\psi_p^j(\xXX),I) + d(\psi_p^m(\yYY),0) \in \RRR_+$.
Finally, since $\QqQ$ is countable, also the function $u\dd \CDDc_N(\HHh_n) \times \CDDc_N(\HHh_m) \ni (\xXX,\yYY) \mapsto
\inf_{p \in \QqQ} \theta_p(\xXX,\yYY) \in \RRR_+$ is measurable. The observation that $\Dd_N(n,m) = u^{-1}(\{0\})$
finishes the proof.
\end{proof}

\begin{thm}{ue-factor}
The sets $$\Delta_N(n,m) = \{(\aAA,\bBB) \in \Ff_N(\HHh_n) \times \Ff_N(\HHh_m)\dd\ \aAA \equiv \bBB\}$$
and $\trianglelefteq_N\!\!(n,m) = \{(\aAA,\bBB) \in \Ff_N(\HHh_n) \times \Ff_N(\HHh_m)\dd\ \aAA \leqsl \bBB\}$ are
measurable.
\end{thm}
\begin{proof}
First of all, note that for $(\aAA,\bBB) \in \Ff_N(\HHh_n) \times \Ff_N(\HHh_m)$, $\aAA \not\disj \bBB \iff \aleph_0 \odot
\AAA = \aleph_0 \odot \BBB$. So, \LEM{1} implies that the set $C(n,m) = \{(\aAA,\bBB) \in \Ff_N(\HHh_n) \times
\Ff_N(\HHh_m)\dd\ \aleph_0 \odot \AAA = \aleph_0 \odot \BBB\}$ is measurable. Put $L_N(n,m) = \{(\aAA,\bBB) \in
\Ff_N(\HHh_n) \times \Ff_N(\HHh_m)\dd\ \aAA \lneqq \bBB\}$ and $R_N(n,m) = \{(\aAA,\bBB)\dd\ (\bBB,\aAA) \in L_N(n,m)\}$.
Observe that $$\trianglelefteq_N\!\!(n,m) = \Delta_N(n,m) \cup L_N(n,m),$$
$C(n,m) = \Delta_N(n,m) \cup L_N(n,m) \cup R_N(n,m)$ and the sets $\Delta_N(n,m)$, $L_N(n,m)$ and $R_N(n,m)$ are pairwise
disjoint. Since $C(n,m)$ is a standard Borel space, it suffices therefore to show that each of the latter sets is Souslin
(cf. \cite[Theorem~A.3]{tk1}). We see that $\Delta(n,m) = \varempty$ if $n \neq m$ and
$\Delta_N(n,n) = \{(\aAA,U \aAA U^{-1})\dd\ U \in \UUu(\HHh_n),\ \aAA \in \Ff_N(\HHh_n)\}$
(where $U (A_1,\ldots,A_N) U^{-1} = (U A_1 U^{-1},\ldots,U A_N U^{-1})$) is the image of a standard Borel space
$\UUu(\HHh_n) \times \Ff_N(\HHh_n)$ under a Borel function and thus $\Delta_N(n,n)$ is Souslin. Finally, the set
$\Ff_N^{fin}(\HHh_n)$ of all $N$-tuples $\xXX \in \Ff_N(\HHh_n)$ such that $\WWw'(\xXX)$ is finite is Borel and therefore
$L_N(n,m)$ is Souslin, since $L(n,m) = \varempty$ for $n > m$ or $n=m < \infty$, for $n < m$:
\begin{multline*}
L_N(n,m) = \{(\aAA,U (\aAA \oplus \gGG) U^{-1})\dd\ U \in \UUu(\HHh_n \oplus \HHh_{m-n},\HHh_m),\\
\aAA \in \Ff_N(\HHh_n),\ (\aAA,\gGG) \in C(n,m-n)\},
\end{multline*}
and
\begin{multline*}
L_N(\infty,\infty) = \bigcup_{k=1}^{k=\infty} \Bigl\{(\aAA,U (\aAA \oplus \gGG) U^{-1})\dd\ U \in \UUu(\HHh_{\infty} \oplus
\HHh_k,\HHh_{\infty}),\\
\aAA \in \Ff_N^{fin}(\HHh_{\infty}),\ (\aAA,\gGG) \in C(\infty,k)\Bigr\}.
\end{multline*}
The note that $R_N(n,m)$ is the Borel image of $L_N(n,m)$ finishes the proof.
\end{proof}

\begin{cor}{selector}
Let $\FFf$ be such a Borel subset of $\Ff_N(\HHh_n)$ that the function $\Phi\dd \FFf \ni \xXX \mapsto \XXX \in \CDD_N$
is one-to-one. Then $\widehat{\FFf} = \{\yYY \in \CDDc_N(\HHh_n)\dd\ \yYY \equiv \xXX \textup{ for some } \xXX \in \FFf\}$
is a Borel subset of $\CDDc_N(\HHh_n)$ and $\FfF = \{\XXX\dd\ \xXX \in \FFf\} \subset \CDD_N$ is measurable and it is
a standard Borel space.
\end{cor}
\begin{proof}
By \THM{ue-factor}, the set $\DdD = \Delta_N(n,n) \cap (\CDDc_N(\HHh_n) \times \FFf)$ is Borel. What is more, it follows
from the assumptions that the function $\DdD \ni (\aAA,\bBB) \mapsto \aAA \in \widehat{\FFf}$ is a bijection. It is also
Borel and thus $\widehat{\FFf} \in \Bb(\CDDc_N(\HHh_n))$, by \cite[Corollary~A.7]{tk1}. Since $\{\xXX \in
\CDDc_N(\HHh_n)\dd\ \XXX \in \FfF\} = \widehat{\FFf}$, we obtain that $\FfF \in \Bb_N$.\par
It is clear that $\Phi$ is a Borel bijection of $\FFf$ onto $\FfF$. However, if $\BBb$ is a Borel subset of $\FFf$, then
the above argument shows that $\{\XXX\dd\ \xXX \in \BBb\} \in \Bb_N$ and hence $\Phi$ is a Borel isomorphism
and the assertion follows.
\end{proof}

A variation of \THM{ue-factor} is contained in

\begin{lem}{2}
For each $t \in (0,\infty)$ the sets $\Delta_N^t = \{(\aAA,\bBB) \in \sS_N(\HHh_{\infty}) \times \sS_N(\HHh_{\infty})\dd\
\AAA = t \odot \BBB\}$ and $\trianglelefteq_N^t\, = \{(\aAA,\bBB) \in \sS_N(\HHh_{\infty}) \times \sS_N(\HHh_{\infty})\dd\
\AAA \leqsl t \odot \BBB\}$ are measurable.
\end{lem}
\begin{proof}
Since $\Delta_N^t =\, \trianglelefteq_N^t \cap \trianglerighteq_N^t$ where $\trianglerighteq_N^t = \{(\aAA,\bBB)\dd\
(\bBB,\aAA) \in\, \trianglelefteq_N^t\}$, it is enough to prove that $\trianglelefteq_N^t$ is measurable. It is clear that
for every $n \geqsl 1$ the function $\sS_N(\HHh_{\infty}) \ni \aAA \mapsto n \odot \aAA \in \sS_N(n \odot \HHh_{\infty})$
is measurable. Consequently, thanks to \THM{ue-factor}, the set $D(n,m) = \{(\aAA,\bBB) \in \sS_N(\HHh_{\infty}) \times
\sS_N(\HHh_{\infty})\dd\ n \odot \AAA \leqsl m \odot \BBB\}$ is measurable as well. Now if $w_k = \frac{m_k}{n_k}$ are
rationals which decrease to $t$ (as $k$ increases to $\infty$), then $\trianglelefteq_N^t\, = \bigcap_{k=1}^{\infty}
D(n_k,m_k)$ and we are done.
\end{proof}

Whenever $\AAA, \BBB \in \sS_N$ are such that $\AAA \ll \BBB$, there is a unique positive real number denoted
by $\AAA : \BBB$ such that
\begin{equation}\label{eqn:div}
\AAA = (\AAA : \BBB) \odot \BBB.
\end{equation}
Further, we put $\zero : \XXX = 0$ and $(\alpha \odot \XXX) : \XXX = \alpha$ for any $\XXX \in \Ff_N$ and $\alpha \in
\Card_{\infty}$, and $(n \odot \AAA) : (m \odot \AAA) = n/m$ for every $\AAA \in \aA_N$ and positive integers $n$ and $m$.
It is clear that \eqref{eqn:div} is fulfilled whenever $\BBB \in \Ff_N^{fin} := \Ff_N \cap \fIN_N$ and $\AAA \in \CDD_N$
are such that $\AAA \ll \BBB$.\par
For $n, m \in \{1,2,\ldots,\infty\}$ put $\nabla_N(m,n) = \{(\aAA,\bBB) \in \Ff_N(\HHh_m) \times \Ff_N^{fin}(\HHh_n)\dd\
\AAA \ll \BBB\}$. It follows from \LEM{1} that $\nabla_N(m,n)$ is a Borel subset of $\CDDc_N(\HHh_m) \times
\CDDc_N(\HHh_n)$. What we want is to show measurability of the function
$$\label{Div}\Div\dd \nabla_N(m,n) \ni (\xXX,\yYY) \mapsto \XXX : \YYY \in I_{\aleph_0}.$$
It may be easily shown that $\Div$ is measurable on $\nabla_N(m,n)$ for finite $n$ and on $\nabla_N(\infty,\infty)
\setminus (\sS_N(\HHh_{\infty}) \times \sS_N(\HHh_{\infty}))$ ($\nabla(n,\infty)$ is empty if $n$ is finite). On the other
hand, $\Div^{-1}((0,t]) \cap (\sS_N(\HHh_{\infty}) \times \sS_N(\HHh_{\infty})) =\, \trianglelefteq_N^t$ and therefore
$\Div$ is measurable on $\nabla_N(\infty,\infty) \cap (\sS_N(\HHh_{\infty}) \times \sS_N(\HHh_{\infty}))$ as well.
As a corollary of this we get that the sets $\{(t,\aAA,\bBB) \in (0,\infty) \times \sS_N(\HHh_{\infty}) \times
\sS_N(\HHh_{\infty})\dd\ \AAA \leqsl t \odot \BBB\}$ and
\begin{multline}\label{eqn:rrr}
\rrR_N(n,m) = \{(\AAA : \BBB,\bBB,\aAA) \in I_{\aleph_0} \times \Ff_N^{fin}(\HHh_n) \times \Ff_N(\HHh_m)|\\
\AAA \ll \BBB\} \cup \{(\aleph_0,\bBB,\bBB)\dd\ \bBB \in \Ff_N(\HHh_n) \setminus \Ff_N^{fin}(\HHh_n)\}
\end{multline}
are measurable. This fact will be used in the proof of

\begin{thm}{arb-meas}
Let $(X,\Mm,\mu)$ be a standard measure space, $\FfF \subset \Ff_N$ be a countably separated measurable set and $\Phi\dd
X \ni x \mapsto \AAA^{(x)} \in \FfF$ be a measurable function. Further, let $f\dd X \to I_{\aleph_0} \setminus \{0\}$ be
a Borel function such that $f(X \setminus \Phi^{-1}(\sS_N)) \subset \Card$. Then there are measurable sets
$X_1,X_2,\ldots,X_{\infty} \subset X$ and Borel functions $\Phi_n\dd X_n \ni x \mapsto \bBB^{(x)} \in \CDDc_N(\HHh_n)\
(n=1,2,\ldots,\infty)$ such that $\BBB^{(x)} = f(x) \odot \AAA^{(x)}$ for each $x \in X' := \bigcup_{n=1}^{n=\infty} X_n$
and $\mu(X \setminus X') = 0$. If, in addition,
\begin{equation}\label{eqn:disj-iso}
\AAA^{(x)} \disj \AAA^{(y)}
\end{equation}
for distinct $x, y \in X$, then $\Phi_n(X_n) \in \Bb(\CDDc_N(\HHh_n))$ and $\Phi_n$ is a Borel isomorphism of $X_n$ onto
its range.
\end{thm}
\begin{proof}
Since $\Phi^{-1}(\Ff_N \setminus \Ff_N^{fin})$ is measurable, we may change the function $f$ (with no change
of $f(x) \odot \AAA^{(x)}$) so that $f(x) = \aleph_0$ whenever $\Phi(x) \notin \Ff_N^{fin}$. But then for every $x \in X$,
\begin{equation}\label{eqn:aux69}
(f(x) \odot \AAA^{(x)}) : \AAA^{(x)} = f(x).
\end{equation}
Let $\nu\dd \Bb(\FfF) \ni A \mapsto \mu(\Phi^{-1}(A)) \in [0,\infty]$. Since $\FfF$ is the Borel image of a standard Borel
space $\bigcup_{n=1}^{n=\infty} \{\xXX \in \CDDc_N(\HHh_n)\dd\ \XXX \in \FfF\}$ (and $\FfF$ is countably separated),
$\FfF$ is a Souslin-Borel space and therefore $\nu$ is a standard measure on $\FfF$ (cf. \cite[Corollary~A.14]{tk1}).
So, we may assume (reducing $\FfF$ and $X$) that $\FfF$ and $X$ are standard Borel spaces. For each $n=1,2,\ldots,\infty$
let $\GgG_n$ be the set of all $N$-tuples $\xXX \in \CDDc_N(\HHh_n)$ such that $\XXX \in \FfF(n) := \FfF \cap
\SsS\EeE\PpP_N(n)$. Note that $\GgG_n \in \Bb(\CDDc_N(\HHh_n))$. Since $\FfF$ is a standard Borel space, it follows from
\cite[Theorem~A.16]{tk1} that there are a set $\FfF_n \in \Bb(\FfF(n))$ and a measurable function $\FfF_n \ni \XXX \mapsto
\gGG^{(\XXX)} \in \GgG_n$ such that $\nu(\FfF(n) \setminus \FfF_n) = 0$ and $\GGG^{(\XXX)} = \XXX$ for each $\XXX \in
\FfF_n$. Again, we may assume that $\FfF = \bigcup_{n=1}^{n=\infty} \FfF_n$ (since $\nu(\FfF \setminus
\bigcup_{n=1}^{n=\infty} \FfF_n) = 0$). Put $X(n) = \{x \in X\dd\ \AAA^{(x)} \in \FfF_n\}$ and $\tTT^{(x)} =
\gGG^{(\AAA^{(x)})}$ for $x \in X(n)$. Note that the function $X(n) \ni x \mapsto \tTT^{(x)} \in \CDDc_N(\HHh_n)$ is
measurable. This implies that the set $\Gamma_n = \{(x,f(x),\tTT^{(x)})\dd\ x \in X(n)\}$ is Borel in $X(n) \times
I_{\aleph_0} \times \CDDc_N(\HHh_n)$ (as the graph of a Borel function) and consequently for each $m=1,2,\ldots,\infty$
the set $\bbB_{n,nm} = \{(x,\yYY)\dd\ \YYY = f(x) \odot \AAA^{(x)},\ x \in X(n),\ \yYY \in \Ff_N(\HHh_{nm})\}$,
as the image of $(\Gamma_n \times \Ff_N(\HHh_{nm})) \cap (X(n) \times \rrR_N(n,nm))$ under the projection map
(cf. \eqref{eqn:rrr} and \eqref{eqn:aux69}) which is one-to-one on the latter set, is Borel as well. Now put $X(n,nm) =
\{x \in X(n)\dd\ f(x) \cdot \dim(\AAA^{(x)}) = nm\}$ and note that $X(n,nm)$'s are measurable sets such that $X(n) =
\bigcup_{m=1}^{m=\infty} X(n,nm)$. Since the function $p_{n,nm}\dd \bbB_{n,nm} \ni (x,\yYY) \mapsto x \in X(n,nm)$ is
a Borel surjection, we deduce from \cite[Theorem~A.16]{tk1} that there is a Borel function $w_{n,nm}\dd X(n,nm) \to
\bbB_{n,nm}$ such that $(p_{n,nm} \circ w_{n,nm})(x) = x$ for $\mu$-almost all $x \in X(n,nm)$. For $x \in X(n,nm)$ let
$\bBB^{(x)} \in \Ff_N(\HHh_{nm})$ be the second coordinate of $w_{n,nm}(x)$. Then the function $\Phi_{n,nm}\dd X(n,nm) \ni
x \mapsto \bBB^{(x)} \in \CDDc_N(\HHh_{nm})$ is measurable and for $\mu$-almost all $x \in X$,
\begin{equation}\label{eqn:aux64}
\BBB^{(x)} = f(x) \odot \AAA^{(x)}.
\end{equation}
Again, by reducing $X$, we may assume that \eqref{eqn:aux64} is fulfilled for all $x \in X$. To this end, put $X_k =
\bigcup \{X(n,nm)\dd\ nm = k\}$ and let $\Phi_k\dd X_k \to \CDDc_N(\HHh_k)$ be given by $\Phi_k(x) = \Phi_{n,nm}(x)$
provided $nm = k$ and $x \in X(n,nm)$. Since the sets $X(n,nm)$ are pairwise disjoint, $\Phi_k$ is well defined and Borel.
Finally, if \eqref{eqn:disj-iso} is fulfilled, \eqref{eqn:aux64} implies that $\Phi_k$ is one-to-one and thus the assertion
follows.
\end{proof}

\SECT{Direct integrals and measurable domains\\of strong unitary disjointness}

In this part we establish only the most relevant (for our further investigations) properties of direct integrals.
The `continuous' operation in $\CDD_N$ is defined and main results on it are placed in the next two sections.\par
We now fix a standard measure space $(X,\Mm,\mu)$. For a separable Hilbert space $\HHh$ the Hilbert space $L^2(X,\HHh) =
L^2(\mu,\HHh)$ consists of all (equivalence classes of) measurable functions $\xi\dd X \to \HHh$ such that $\|\xi\|_2^2 =
\int_X \|\xi(x)\|^2 \dint{\mu(x)} < \infty$ ($L^2(\mu,\HHh)$ is separable). Let $X \ni x \mapsto T_x \in \CDDc(\HHh)$ be
a measurable function. We define an operator $T := \int^{\oplus}_X T_x \dint{\mu(x)}$ in $L^2(\mu,\HHh)$ by
\begin{multline*}
\DdD(T) = \{\xi \in L^2(\mu,\HHh)\dd\ \xi(x) \in \DdD(T_x) \textup{ for $\mu$-almost all } x \in X\\
\textup{and } \int_X \|T_x \xi(x)\|^2 \dint{\mu(x)} < \infty\}
\end{multline*}
and $(T \xi)(x) = T_x \xi(x)$ for $\xi \in \DdD(T)$ and ($\mu$-almost all) $x \in X$. It is not obvious that $T \xi$
is measurable (for $\xi \in \DdD(T)$) and that $T \in \CDDc(\HHh)$. These are guaranteed by the next result which may be
deduced from \cite[Lemma~VI.3.3]{tk2} (cf. \cite[Definition~VI.3.4]{tk2}).

\begin{pro}{d-int}
For every measurable function $X \ni x \mapsto T_x \in \CDDc(\HHh)$ the operator $\int^{\oplus}_X T_x \dint{\mu(x)}$
is well defined, closed and densely defined. What is more, $\bB(\int^{\oplus}_X T_x \dint{\mu(x)}) = \int^{\oplus}_X
\bB(T_x) \dint{\mu(x)}$.
\end{pro}

Now let $\Phi\dd X' \ni x \mapsto \tTT^{(x)} \in \bigcup_{n=1}^{n=\infty} \CDDc_N(\HHh_n)$ where $X \setminus X'
\in \NnN(\mu)$ be any function and $\tTT^{(x)} = (\tTT^{(x)}_1,\ldots,\tTT^{(x)}_N)$ for each $x \in X'$. If there are
measurable sets $X_1,X_2,\ldots,X_{\infty} \subset X'$ such that $\mu(X' \setminus \bigcup_{n=1}^{n=\infty} X_n) = 0$,
$\Phi(X_j) \subset \CDDc_N(\HHh_j)$ (the latter implies that $X_j$'s are pairwise disjoint) and $\Phi\bigr|_{X_j}\dd
X_j \to \CDDc_N(\HHh_j)$ is measurable for each $j$, we call $\Phi$ \textit{integrable} and define the \textit{direct
integral} $\int^{\oplus}_X \tTT^{(x)} \dint{\mu(x)}$ of the field $\{\tTT^{(x)}\}_{x \in X'}$ by
$$
\int^{\oplus}_X \tTT^{(x)} \dint{\mu(x)} = \bigoplus_{n=1}^{n=\infty}
\bigl(\int^{\oplus}_{X_n} T^{(x)}_1 \dint{\mu(x)},\ldots,\int^{\oplus}_{X_n} T^{(x)}_N \dint{\mu(x)}\bigr).
$$
Below we list most important (for our investigations) properties of direct integrals of measurable fields of $N$-tuples.
\begin{enumerate}[(d{i}1)]\addtocounter{enumi}{-1}
\item $\dim \overline{\DdD}(\int^{\oplus}_X \tTT^{(x)} \dint{\mu(x)}) \leqsl \aleph_0$.
\item\label{di1} $\bB(\int^{\oplus}_X \tTT^{(x)} \dint{\mu(x)}) = \int^{\oplus}_X \bB(\tTT^{(x)}) \dint{\mu(x)}$.
\item If $X_1,X_2,X_3,\ldots$ are pairwise disjoint measurable subsets of $X$ such that $\mu(X_j) > 0$ for each $j$
   and $\mu(X \setminus \bigcup_{n=1}^{\infty} X_n) = 0$, then $$\int^{\oplus}_X \aAA^{(x)} \dint{\mu(x)} \equiv
   \bigoplus_{n=1}^{\infty} \int^{\oplus}_{X_n} \aAA^{(x)} \dint{\mu(x)}.$$
\item $\bigoplus_{n=1}^{\infty}(\int^{\oplus}_X \tTT^{(x)}_n \dint{\mu(x)}) \equiv \int^{\oplus}_X
   (\bigoplus_{n=1}^{\infty} \tTT^{(x)}_n) \dint{\mu(x)}$.
\item If $\tTT^{(x)} \equiv \sSS^{(x)}$ for $\mu$-almost all $x \in X$, then $\int^{\oplus}_X \tTT^{(x)} \dint{\mu(x)}
   \equiv \int^{\oplus}_X \sSS^{(x)} \dint{\mu(x)}$. This follows from (BT5) \pREF{BT5}, (di1) and the proof of
   \cite[Theorem~IV.8.28]{tk1}.
\item If $\nu$ is a $\sigma$-finite measure on $(X,\Mm)$ such that $\nu \ll \mu \ll \nu$ (that is, $\NnN(\mu) =
   \NnN(\nu)$), then $\int^{\oplus}_X \tTT^{(x)} \dint{\mu(x)} \equiv \int^{\oplus}_X \tTT^{(x)} \dint{\nu(x)}$.
\item If $(Y,\Nn,\nu)$ is a standard measure space, $X_0 \in \NnN(\mu)$, $Y_0 \in \NnN(\nu)$ and $\psi\dd Y \setminus
   Y_0 \to X \setminus X_0$ is a Borel isomorphism such that $\mu(\psi(A)) = \nu(A)$ for every $A \in \Nn$ disjoint from
   $Y_0$, then $$\int^{\oplus}_X \tTT^{(x)} \dint{\mu(x)} \equiv \int^{\oplus}_Y \tTT^{(\psi(y))} \dint{\nu(y)}.$$
\end{enumerate}

Further, let $X \ni x \mapsto \AAA^{(x)} \in \SsS\EeE\PpP_N$ be any function. If there exist sets $X_1,X_2,\ldots,
X_{\infty} \in \Mm$ and measurable functions
\begin{equation}\label{eqn:meas}
X_n \ni x \mapsto \aAA^{(x)} \in \CDDc_N(\HHh_n)
\end{equation}
($n=1,2,\ldots,\infty$) such that $\mu(X \setminus \bigcup_{n=1}^{n=\infty} X_n) = 0$ and for each $x \in
\bigcup_{n=1}^{n=\infty} X_n$, $\aAA^{(x)}$ is a representative of $\AAA^{(x)}$, we say the field
$\{\AAA^{(x)}\}_{x \in X}$ is \textit{integrable} and define the \textit{direct integral} $\int^{\oplus}_X \AAA^{(x)}
\dint{\mu(x)}$ as the unitary equivalence class of
\begin{equation}\label{eqn:mu-dir-int}
\bigoplus_{n=1}^{n=\infty} \int^{\oplus}_{X_n} \aAA^{(x)} \dint{\mu(x)}.
\end{equation}
Thanks to (di4), $\int^{\oplus}_X \AAA^{(x)} \dint{\mu(x)}$ is well defined, i.e. it does not depend of the choice
of measurable functions \eqref{eqn:meas} of representatives. As it is easily seen, in the above situation the function
$\bigcup_{n=1}^{n=\infty} X_n \ni x \mapsto \AAA^{(x)} \in \SsS\EeE\PpP_N$ is measurable. We call a field $\Psi\dd
X \ni x \mapsto \BBB^{(x)} \in \SsS\EeE\PpP_N$ \textit{almost measurable} (or \textit{almost Borel}) iff
$\Psi\bigr|_{X \setminus X_0}$ is Borel for some $X_0 \in \NnN(\mu)$. Thus, every integrable field is almost
measurable.\par
In our investigations all almost measurable fields are defined on standard measure spaces. Properties (di0)--(di6) may
naturally be translated into the realm of unitary equivalence classes of $N$-tuples:
\begin{enumerate}[(D{I}1)]\addtocounter{enumi}{-1}
\item\label{DI0} $\int^{\oplus}_X \AAA^{(x)} \dint{\mu(x)} \in \SsS\EeE\PpP_N$.
\item $\bB(\int^{\oplus}_X \AAA^{(x)} \dint{\mu(x)}) = \int^{\oplus}_X \bB(\AAA^{(x)}) \dint{\mu(x)}$.
\item If $X_1,X_2,X_3,\ldots$ are pairwise disjoint measurable subsets of $X$ such that $\mu(X_j) > 0$ for each $j$
   and $\mu(X \setminus \bigcup_{n=1}^{\infty} X_n) = 0$, then $$\int^{\oplus}_X \AAA^{(x)} \dint{\mu(x)} =
   \bigoplus_{n=1}^{\infty} \int^{\oplus}_{X_n} \AAA^{(x)} \dint{\mu(x)}.$$
\item\label{DI3} $\bigoplus_{n=1}^{\infty}(\int^{\oplus}_X \TTT^{(x)}_n \dint{\mu(x)}) = \int^{\oplus}_X
   (\bigoplus_{n=1}^{\infty} \TTT^{(x)}_n) \dint{\mu(x)}$.
\item\label{DI4} If $(Y,\Nn,\nu)$ is a standard measure space, $X_0 \in \NnN(\mu)$, $Y_0 \in \NnN(\nu)$, $\psi\dd Y
   \setminus Y_0 \to X \setminus X_0$ is a Borel isomorphism and $\{\psi(B)\dd\ B \in \NnN(\nu),\ B \cap Y_0 = \varempty\}
   = \{A \in \NnN(\nu)\dd\ A \cap X_0 = \varempty\}$, then $$\int^{\oplus}_X \AAA^{(x)} \dint{\mu(x)} = \int^{\oplus}_Y
   \AAA^{(\psi(y))} \dint{\nu(y)}.$$
\end{enumerate}

A counterpart of regular collections and direct sums ((UE4), \PREF{UE4}) for direct integrals are \textit{regular fields}
and \textit{regular direct integrals} `$\int^{\sqplus}$' which we define as follows. Assume $X \ni x \mapsto \AAA^{(x)} \in
\SsS\EeE\PpP_N$ is an integrable field. If for any two disjoint Borel sets $A, B \subset X$ one has
\begin{equation}\label{eqn:disj-int}
\int^{\oplus}_A \AAA^{(x)} \dint{\mu(x)} \disj \int^{\oplus}_B \AAA^{(x)} \dint{\mu(x)},
\end{equation}
we call the field $\{\AAA^{(x)}\}_{x \in X}$ \textit{regular} and write $\int^{\sqplus}_X \AAA^{(x)} \dint{\mu(x)}$
in place of $\int^{\oplus}_X \AAA^{(x)} \dint{\mu(x)}$. (Condition \eqref{eqn:disj-int} naturally corresponds to (PR2),
\PREF{PR2}.) As in case of direct sums, the notation `$\int^{\sqplus}$' includes information that the integrable field
is regular.\par
In practice it is quite difficult to verify whether an almost measurable field is integrable. However, as an immediate
consequence of \PRO{findim} we obtain

\begin{pro}{finint}
Every almost measurable field of a standard measure space into $\SsS\EeE\PpP_N \setminus \SsS\EeE\PpP_N(\infty)$ is
integrable.
\end{pro}
\begin{proof}
Let $\Phi\dd X \to \SsS\EeE\PpP_N \setminus \SsS\EeE\PpP_N(\infty)$ be measurable. The sets $X_n =
\Phi^{-1}(\SsS\EeE\PpP_N(n))$ are Borel and if $\chi_n$'s are as in \PROp{findim}, then $\chi_n \circ \Phi\bigr|_{X_n}$
is a measurable field of representatives for $\Phi$.
\end{proof}

In general we are unable to characterize integrable fields taking values in $\SsS\EeE\PpP_N$. This is in fact not of our
interest. More preferable are regular fields taking values in $\Ff_N$. In that case a characterization is possible
and we formulate it in the next result. For this purpose we introduce

\begin{dfn}{meas-domain}
A set $\FfF \in \Bb_N$ is said to be a \textit{measurable domain of strong unitary disjointness} iff there is a sequence
$(\EeE_n)_{n=1}^{\infty}$ of subsets of $\CDD_N$ which separates points of $\FfF$ and for every $n \geqsl 1$ the families
$\FfF \cap \EeE_n$ and $\FfF \setminus \EeE_n$ are strongly unitarily disjoint (cf. \REM{sdisj}, \PREF{rem:sdisj}).
We shall shorten the name of this and we shall speak briefly of \textit{measurable domains}.
\end{dfn}

It follows from the definition that measurable domains consist of pairwise unitarily disjoint $N$-tuples. It may also be
easily verified that the union of a countable family of measurable domains each two of which are strongly unitarily
disjoint as well as every measurable subset of a measurable domain is again a measurable domain. Another important property
of measurable domains is that they are Souslin-Borel. Indeed, when $\FfF$ is a measurable domain, it is the Borel image
of a standard Borel space (by measurability of $\FfF$) and $\FfF$ is countably separated, since if $\EeE \subset \CDD_N$
is such that $\FfF \cap \EeE \sdisj \FfF \setminus \EeE$, then $\FfF \cap \EeE \in \Bb_N$ (because for every sequence
$(p_n)_{n=1}^{\infty} \subset \PpP_1(N)$ and each complex scalar $\lambda$ the set of all $\tTT \in \CDDc_N(\HHh_k)$ such
that $p_n(\bB(\tTT),\bB(\tTT)^*)$ converges $*$-strongly to $\lambda I$ is Borel and invariant under unitary equivalence),
and thus the assertion follows from \DEF{meas-domain}.\par
Measurable domains are useful to produce regular fields, as it is shown by

\begin{pro}{meas-domain}
Let $(X,\Mm,\mu)$ be a standard measure space and $\Phi\dd X \ni x \mapsto \AAA^{(x)} \in \Ff_N$ be any field. Then \tfcae
\begin{enumerate}[\upshape(i)]
\item $\{\AAA^{(x)}\}_{x \in X}$ is regular,
\item there is a Borel set $X' \subset X$ such that $X \setminus X' \in \NnN(\mu)$, $\Phi(X')$ is a measurable domain
   and $\Phi\bigr|_{X'}$ is a Borel isomorphism of $X'$ onto its range.
\end{enumerate}
\end{pro}
\begin{proof}
First of all, by reducing $X$, we may assume that $X$ is a standard Borel space. Suppose condition (i) is satisfied.
This yields that there is $Z \in \NnN(\mu)$ and an integrable field $\{\aAA^{(x)}\}_{x \in X \setminus Z} \subset
\bigcup_{n=1}^{n=\infty} \CDDc_N(\HHh_n)$ of representatives for $\Phi$. Take a separating sequence $X_1,X_2,\ldots$
of measurable subsets of $X$. We infer from (di0), \eqref{eqn:disj-int} and \PROp{wdisj} that for each $k \geqsl 1$
there is a sequence $(q^{(k)}_n)_{n=1}^{\infty} \subset \PpP_1(N)$ such that
\begin{gather*}
q^{(k)}_n\left(\bB\Bigl(\int^{\oplus}_{X_k} \aAA^{(x)} \dint{\mu(x)}\Bigr),\bB\Bigl(\int^{\oplus}_{X_k} \aAA^{(x)}
\dint{\mu(x)}\Bigr)^*\right) \stSt I,\\
q^{(k)}_n\left(\bB\Bigl(\int^{\oplus}_{X \setminus X_k} \aAA^{(x)} \dint{\mu(x)}\Bigr),
\bB\Bigl(\int^{\oplus}_{X \setminus X_k} \aAA^{(x)} \dint{\mu(x)}\Bigr)^*\right) \stSt 0.
\end{gather*}
Now taking into account that
\begin{multline}\label{eqn:aux48}
p\left(\bB\Bigl(\int^{\oplus}_D \aAA^{(x)} \dint{\mu(x)}\Bigr),\bB\Bigl(\int^{\oplus}_D \aAA^{(x)}
\dint{\mu(x)}\Bigr)^*\right) =\\= \int^{\oplus}_D p(\bB(\aAA^{(x)}),\bB(\aAA^{(x)})^*) \dint{\mu(x)}
\end{multline}
for every measurable set $D \subset X$ and $p \in \PpP(N)$ (cf. (di1)), we infer from \cite[Proposition~3.2.7]{sak} that
there are a subsequence $(p^{(k)}_n)_{n=1}^{\infty}$ of $(q^{(k)}_n)_{n=1}^{\infty}$ and a measurable set $X'_k \subset X
\setminus Z$ such that $X \setminus X'_k \in \NnN(\mu)$ and
\begin{equation}\label{eqn:aux44}
p(\bB(\aAA^{(x)}),\bB(\aAA^{(x)})^*) \stSt j_k(x) I
\end{equation}
for any $x \in X'_k$ where $j_k$ is the characteristic function of $X_k$. Put $X' = \bigcap_{k=1}^{\infty} X'_k$ and note
that $\mu(X \setminus X') = 0$. Since $\{X_k\}_{k\geqsl1}$ is a separating family and thanks to \eqref{eqn:aux44},
$\Phi\bigr|_{X'}$ is one-to-one. It may be also deduced from \CORp{selector} that $\Phi(X')$ is measurable.
Consequently, $\Phi(X')$ is a measurable domain, by \eqref{eqn:aux44}. Now it suffices to apply \cite[Corollary~A.10]{tk1}
to get that $\Phi\bigr|_{X'}$ is a Borel isomorphism.\par
We now pass to the converse implication. It follows from \THMp{arb-meas} that $\Phi$ is integrable. So, let
$$\{\aAA^{(x)}\}_{x \in X''} \subset \bigcup_{n=1}^{n=\infty} \CDDc_N(\HHh_n)$$ be an integrable field of representatives
for $\Phi$ where $X'' \subset X'$ and $X \setminus X'' \in \NnN(\mu)$. Put $\aAA = \int^{\oplus}_X \aAA^{(x)}
\dint{\mu(x)}$. Let $\EeE_1,\EeE_2,\ldots$ be a separating family for $\Phi(X')$ such that
\begin{equation}\label{eqn:aux46}
\Phi(X') \cap \EeE_k \sdisj \Phi(X') \setminus \EeE_k
\end{equation}
for every $k$. It follows from the note preceding the proposition that $\EeE_k \cap \Phi(X')
\in \Bb_N$. Consequently, the sets $X_k = \Phi^{-1}(\EeE_k) \cap X''\ (k=1,2,\ldots)$ are measurable and separate
the points of $X''$ (because $\Phi$ is one-to-one on $X' \supset X''$). We infer from this, thanks
to \cite[Corollary~A.12]{tk1}, that the $\sigma$-algebra of subsets of $X''$ generated by the sets $X_k$'s coincides with
$\Mm'' := \{A \subset X''|\quad A \in \Mm\}$. Further, the space $\overline{\DdD}(\aAA)$ has the form
$\bigoplus_{n=1}^{n=\infty} L^2(X_n'',\HHh_n)$ where $X_1'',X_2'',\ldots$ are pairwise disjoint members of $\Mm''$ whose
union is $X''$. For each $k$ let $M_k$ be the multiplication operator by the characteristic function $j_k$ of $X_k''$
on $\overline{\DdD}(\aAA)$. Fix for a moment $k$. By \eqref{eqn:aux46}, there is a sequence $(p_n)_{n=1}^{\infty} \subset
\PpP_1(N)$ such that $p_n(\bB(\aAA^{(x)}),\bB(\aAA^{(x)})^*)$ converges $*$-strongly to $j_k(x) I$ for every $x \in X''$.
Since in addition $\|p_n(\bB(\aAA^{(x)}),\bB(\aAA^{(x)})^*)\| \leqsl 1$, Proposition~3.2.7 of \cite{sak} implies that
$$
\int^{\oplus}_{X''} p_n(\bB(\aAA^{(x)}),\bB(\aAA^{(x)})^*) \dint{\mu(x)} \stSt \int^{\oplus}_{X''} j_k(x) I \dint{\mu(x)}.
$$
This combined with \eqref{eqn:aux48} gives $p_n(\bB(\aAA),\bB(\aAA)^*) \stSt M_k$ and consequently $M_k \in \WWw''(\aAA)$.
In this way we have shown that $\{X_1,X_2,\ldots\} \subset \Nn$ where $\Nn$ consists of all $B \in \Mm''$ such that
the multiplication operator $M(B)$ by the characteristic function of $B$ belongs to $\WWw''(\aAA)$. Since $\Nn$ is
a $\sigma$-algebra, we finally obtain that $\Nn = \Mm''$.\par
Since $\WWw''(\aAA) = \WWw(\bB(\aAA))$ and each entry of $\bB(\aAA)$ is a decomposable operator, $\WWw''(\aAA)$ consists
of decomposable operators. If $B$ is an arbitrary member of $\Mm$, $M(B \cap X'')$ is a diagonalizable operator and hence
$M(B \cap X'') \in \ZZz(\WWw''(\aAA))$. So, $\int^{\oplus}_B \AAA^{(x)} \dint{\mu(x)} (= \int^{\oplus}_{B \cap X''}
\AAA^{(x)} \dint{\mu(x)})$ and $\int^{\oplus}_{X \setminus B} \AAA^{(x)} \dint{\mu(x)}$ correspond (by \PROP{transl})
to mutually orthogonal central projections in $\WWw''(\aAA)$ from which we conclude that $\int^{\oplus}_B \AAA^{(x)}
\dint{\mu(x)} \disj \int^{\oplus}_{X \setminus B} \AAA^{(x)} \dint{\mu(x)}$. Now \eqref{eqn:disj-int} is implied by (di2).
\end{proof}

\begin{rem}{weak}
Since every Borel injection of a standard Borel space into a Souslin-Borel one has measurable image and is a Borel
isomorphism between its domain and range (cf. Theorem~A.6 and Corollary~A.7 in \cite{tk1}), condition (ii)
of \PRO{meas-domain} may be weakened by replacing the assumption that \textit{$\Phi(X')$ is a measurable domain
and $\Phi\bigr|_{X'}$ is a Borel isomorphism} by the one that \textit{$\Phi\bigr|_{X'}$ is Borel and one-to-one
and $\Phi(X')$ is contained in a measurable domain}.
\end{rem}

For simplicity, let us call a $\sigma$-finite measure $\nu$ on a measurable set $\BbB \subset \Ff_N$ a \textit{regularity}
measure ($\nu \in \rgM(\BbB)$) if $\nu$ is standard and the identity field of $\BbB$ into $\Ff_N$ is regular. Equivalently,
$\nu \in \rgM(\BbB)$ iff $\nu$ is concentrated on a measurable domain (since measurable domains are Souslin-Borel and all
$\sigma$-finite measures on the latter sets are standard). To shorten statements, we shall write $(\mu,\Phi) \in
\rgS(X,\Mm)$ to express that $\mu$ is a standard measure on $(X,\Mm)$ and $\Phi\dd X \to \Ff_N$ is a regular field.\par
Suppose $(\mu,\Phi) \in \rgS(X,\Mm)$ is a regular field. Let $X'$ be as in point (ii) of \PRO{meas-domain}. Define
a measure $\nu = \Phi^*(\mu)\dd \Bb(\Ff_N) \to [0,\infty]$ by $\nu(\BbB) = \mu(\Phi^{-1}(\BbB) \cap X')$. Notice that $\nu
\in \rgM(\Ff_N)$ and that $\int^{\sqplus}_X \Phi(x) \dint{\mu(x)} = \int^{\sqplus}_{\Ff_N} \FFF \dint{\nu(\FFF)}$, thanks
to (DI4). This observation shows that it suffices to consider regularity measures instead of abstract regular fields.\par
The following result is a link between regular fields and central decompositions of von Neumann algebras.

\begin{pro}{pd-transl}
Let $(X,\Mm,\mu)$ be a standard measure space, $\Phi\dd X \ni x \mapsto \aAA^{(x)} \in \bigcup_{n=1}^{n=\infty}
\Ff_N(\HHh_n)$ an integrable field and let $$\aAA = \int^{\oplus}_X \aAA^{(x)} \dint{\mu(x)}.$$ Then \tfcae
\begin{enumerate}[\upshape(i)]
\item $\{\AAA^{(x)}\}_{x \in X}$ is regular,
\item $\{\XXX \in \CDD_N\dd\ \XXX \leqsl^s \AAA\} = \{\int^{\oplus}_B \AAA^{(x)} \dint{\mu(x)}\dd\ B \in \Mm\}$,
\item $\int^{\oplus}_X \WWw''(\aAA^{(x)}) \dint{\mu(x)}$ is the central decomposition of the von Neumann algebra
   $\WWw''(\aAA)$.
\end{enumerate}
\end{pro}
\begin{proof}
First of all, note that the field $\{\WWw''(\aAA^{(x)})\}_{x \in X}$ is measurable according
to \cite[Definition~3.2.9]{sak}, since $\WWw''(\aAA^{(x)}) = \WWw(\bB(\aAA^{(x)}))$. Further, under the assumptions
of the proposition, (iii) is equivalent to
\begin{enumerate}[\upshape(i')]\addtocounter{enumi}{2}
\item the von Neumann algebra $\AAa$ of all diagonalizable operators is contained in $\WWw''(\aAA)$.
\end{enumerate}
It is clear that (iii') is implied by (iii). Conversely, when (iii') is fulfilled, $\WWw'(\aAA)$ consists of (some)
decomposable operators (thanks to \cite[Corollary~IV.8.16]{tk1} or \cite[Theorem~14.1.10]{k-r2}). We see that so does
$\WWw''(\aAA)$ (since $\bB(\aAA)$ is an $N$-tuple of decomposable operators) and hence $\AAa \subset \WWw'(\aAA)$. This
yields that $\AAa \subset \ZZz(\WWw''(\aAA))$. Now using the terminology of Kadison and Ringrose \cite{k-r2}, we conclude
that $\WWw''(\aAA)$ is decomposable (Theorem~14.1.16 and Proposition~14.1.18 in \cite{k-r2}), i.e. $\WWw''(\aAA) =
\int^{\oplus}_X \MmM_x \dint{\mu(x)}$ for some measurable field $\{\MmM_x\}_{x \in X}$ of von Neumann algebras.
By the uniqueness of the decomposition $\bB(\aAA) = \int^{\oplus}_X \bB(\aAA^{(x)}) \dint{\mu(x)}$ (cf. (di1), \PREF{di1}),
$\WWw(\bB(\aAA^{(x)})) \subset \MmM_x$ for $\mu$-almost all $x \in X$ and thus $\int^{\oplus}_X \WWw''(\aAA^{(x)})
\dint{\mu(x)} \subset \WWw''(\aAA)$. Since the converse inclusion is immediate, we get $\WWw''(\aAA) = \int^{\oplus}_X
\WWw''(\aAA^{(x)}) \dint{\mu(x)}$. This proves (iii) because $\WWw''(\aAA^{(x)})$ is a factor for all $x \in X$
and consequently (by \cite[Corollary~IV.8.20]{tk1}) $\ZZz(\WWw''(\aAA)) = \int^{\oplus}_X \ZZz(\WWw''(\aAA^{(x)}))
\dint{\mu(x)} = \AAa$.\par
We leave this as a simple exercise that the whole assertion of the proposition now easily follows.
\end{proof}

As an important for us consequence of \PRO{pd-transl} we now obtain

\begin{cor}{ue}
Let $(\mu,\Phi) \in \rgS(X,\Mm)$, $(\nu,\Psi) \in \rgS(Y,\Nn)$ and let $\widehat{\mu} = \Phi^*(\mu)$ and $\widehat{\nu}
= \Psi^*(\nu)$. For $$\XXX = \int^{\sqplus}_X \Phi(x) \dint{\mu(x)} \qquad \textup{and} \qquad \YYY = \int^{\sqplus}_Y
\Psi(y) \dint{\nu(y)}$$ we have:
\begin{enumerate}[\upshape(a)]
\item $\XXX = \YYY \iff \widehat{\mu} \ll \widehat{\nu} \ll \widehat{\mu}$,
\item $\XXX \leqsl^s \YYY \iff \widehat{\mu} \ll \widehat{\nu}$.
\end{enumerate}
\end{cor}
\begin{proof}
We know that $\XXX = \int^{\sqplus}_{\Ff_N} \FFF \dint{\widehat{\mu}(\FFF)}$ and $\YYY = \int^{\sqplus}_{\Ff_N} \FFF
\dint{\widehat{\nu}(\FFF)}$. Observe that point (b) follows from (a) and \PRO{pd-transl}, and implication `$\impliedby$'
in (a) is a consequence of (DI4). To prove the converse one, assume $\xXX = \int^{\sqplus}_X \aAA^{(x)} \dint{\mu(x)}$ with
$\AAA^{(x)} = \Phi(x)$ for $\mu$-almost all $x \in X$, $\yYY = \int^{\sqplus}_Y \bBB^{(y)} \dint{\nu(y)}$ with $\BBB^{(y)}
= \Psi(y)$ for $\nu$-almost all $y \in Y$, and $U$ is a unitary operator such that $U \cdot \xXX \cdot U^{-1} = \yYY$.
It then follows from \PRO{pd-transl} that $U$ sends the algebra of all diagonalizable operators on $\overline{\DdD}(\xXX)$
onto the algebra of all diagonalizable operators on $\overline{\DdD}(\yYY)$. Thus, according
to \cite[Theorem~IV.8.23]{tk1}, there is a Borel isomorphism $\kappa\dd Y \setminus Y_0 \to X \setminus X_0$ where $X_0 \in
\NnN(\mu)$ and $Y_0 \in \NnN(\nu)$ such that
\begin{equation}\label{eqn:aux55}
\kappa^*(\nu) \ll \mu \ll \kappa^*(\nu)
\end{equation}
and $U$ may be written in the form $U = \int^{\oplus}_X U_x \sqrt{\frac{\dint{\kappa^*(\nu)}}{\dint{\mu}}(x)}
\dint{\mu(x)}$ where $\{U_x\}_{x \in X}$ is a certain measurable field of unitary operators (for the details we refer
to Takesaki's book \cite{tk1}). Since $U \cdot \bB(\xXX) = \bB(\yYY) \cdot U$, we conclude from (di1) \pREF{di1} that
$\int^{\oplus}_X U_x \cdot \bB(\aAA^{(x)}) \sqrt{\frac{\dint{\kappa^*(\nu)}}{\dint{\mu}}(x)} \dint{\mu(x)} =
\int^{\oplus}_X \bB(\bBB^{(\kappa(x))}) \cdot U_x \sqrt{\frac{\dint{\kappa^*(\nu)}}{\dint{\mu}}(x)} \dint{\mu(x)}$. Now
thanks to the uniqueness of the decomposition of a bounded decomposable operator and positivity of the function
$\sqrt{\frac{\dint{\kappa^*(\nu)}}{\dint{\mu}}}$, the latter equation implies that $U_x \cdot \bB(\aAA^{(x)}) =
\bB(\bBB^{(\kappa(x))}) \cdot U_x$ for $\mu$-almost all $x \in X$. Consequently, $\BBB^{(\kappa(x))} = \AAA^{(x)}$
for $\mu$-almost all $x \in X$. We leave this as an exercise that the latter combined with \eqref{eqn:aux55} gives
$\widehat{\mu} \ll \widehat{\nu} \ll \widehat{\mu}$ which finishes the proof.
\end{proof}

A similar result was obtained by Ernest (cf. \cite[Theorem~3.8]{e}). However, he was working (when speaking of the central
decomposition of an operator) with quasi-equivalence classes instead of unitary equivalence ones.\par
To avoid repetitions, let us say a function $f\dd X \to I_{\aleph_0}$ \textit{fits} to $(\mu,\Phi) \in \rgS(X,\Mm)$ iff
$f$ is almost measurable and there are disjoint measurable sets $X_1$ and $X_2$ such that $\mu(X \setminus (X_1 \cup X_2))
= 0$, $f(X_1) \subset \Card$ and $\Phi(X_2) \subset \sS_N$. Note that if the latter happens, the function $f \odot \Phi$
given by $(f \odot \Phi)(x) = f(x) \odot \Phi(x)$ is well defined on $X_1 \cup X_2$.

\begin{lem}{disj-pr}
Let $(\mu,\Phi) \in \rgS(X,\Mm)$ and $f\dd X \to I_{\aleph_0} \setminus \{0\}$ be a function which fits to $(\mu,\Phi)$.
Then $(\mu,f \odot \Phi) \in \rgS(X,\Mm)$ as well.
\end{lem}
\begin{proof}
It follows from \THMp{arb-meas} that $f \odot \Phi$ is integrable. Further, we infer from (DI3) \pREF{DI3} that $\aleph_0
\odot \int^{\oplus}_D f(x) \odot \Phi(x) \dint{\mu(x)} = \int^{\oplus}_D (\aleph_0 \cdot f(x)) \odot \Phi(x) \dint{\mu(x)}
= \aleph_0 \odot \int^{\oplus}_D \Phi(x) \dint{\mu(x)}$ and thus $\int^{\oplus}_D f(x) \odot \Phi(x) \dint{\mu(x)} \disj
\int^{\oplus}_{X \setminus D} f(x) \odot \Phi(x) \dint{\mu(x)}$ since $$\int^{\oplus}_D \Phi(x) \dint{\mu(x)} \disj
\int^{\oplus}_{X \setminus D} \Phi(x) \dint{\mu(x)}.$$
\end{proof}

Whenever a function $f\dd X \to I_{\aleph_0}$ fits to $(\mu,\Phi) \in \rgS(X,\Mm)$, we define $\int^{\sqplus}_X f(x) \odot
\Phi(x) \dint{\mu(x)}$ as follows. Put $s(f) = \{x \in X\dd\ f(x) > 0\}$ and take $X_0 \in \NnN(\mu)$ such that $f$ is
measurable on $X \setminus X_0$. If $\mu(s(f) \setminus X_0) > 0$, $\int^{\sqplus}_X f(x) \odot \Phi(x) \dint{\mu(x)}$
denotes $\int^{\sqplus}_{s(f) \setminus X_0} f(x) \odot \Phi(x) \dint{\mu(x)}$ (see \LEM{disj-pr}). Otherwise let
$\int^{\sqplus}_X f(x) \odot \Phi(x) \dint{\mu(x)} = \zero$. The usage of `$\int^{\sqplus}$' here is justified
by \LEM{disj-pr}.\par
Below we formulate a variation of \cite[Proposition~3.2]{e}. We shall use it in our theorem on prime decomposition.

\begin{lem}{pd-sep}
Let $\AAA \in \SsS\EeE\PpP_N$ be the direct sum of a minimal $N$-tuple and a semiminimal one.
\begin{enumerate}[\upshape(A)]
\item There is $\mu_{\AAA} \in \rgM(\pP_N)$ such that $\AAA = \int^{\sqplus}_{\pP_N} \PPP \dint{\mu_{\AAA}(\PPP)}$.
   For $\mu \in \rgM(\pP_N)$, $\AAA = \int^{\sqplus}_{\pP_N} \PPP \dint{\mu(\PPP)} \iff \mu \ll \mu_{\AAA} \ll \mu$.
\item For $\BBB \in \SsS\EeE\PpP_N$ \tfcae
   \begin{enumerate}[\upshape(i)]
   \item $\BBB \ll \AAA$,
   \item there is an almost measurable function $f\dd \pP_N \to I_{\aleph_0}$ such that $f(\aA_N) \subset \Card$,
      $f(\fF_N) \subset \{0,\aleph_0\}$ and
      \begin{equation}\label{eqn:pd-sep}
      \BBB = \int^{\sqplus}_{\pP_N} f(\PPP) \odot \PPP \dint{\mu_{\AAA}(\PPP)}.
      \end{equation}
   \end{enumerate}
\item Let $(\mu,\Phi) \in \rgS(X,\Mm)$.
   \begin{enumerate}[\upshape(a)]
   \item If $\Phi(X) \subset \aA_N$, $\int^{\sqplus}_X \Phi(x) \dint{\mu(x)} \in \MmM\FfF_N$.
   \item If $\Phi(X) \subset \fF_N$, $\int^{\sqplus}_X \Phi(x) \dint{\mu(x)} \in \HhH\IiI\MmM_N$.
   \item If $\Phi(X) \subset \sS_N$ and $f\dd X \to \RRR_+$ is almost measurable, $\int^{\sqplus}_X f(x) \odot \Phi(x)
      \dint{\mu(x)} \in \SsS\MmM_N$.
   \end{enumerate}
\end{enumerate}
\end{lem}
\begin{proof}
Let $\FFF$ be an arbitrary member of $\SsS\EeE\PpP_N$ and let $\fFF$ be a representative of $\FFF$. It follows from
the reduction theory of von Neumann algebras (see e.g. \cite[Theorem~IV.8.21]{tk1}) that there is a standard Borel space
$(X,\Mm)$ with a probabilistic Borel measure $\lambda$ and a measurable field $\{\MmM_x\}_{x \in X}$ of factors (each
of which acts on some $\HHh_n$) such that the von Neumann algebras  $\MmM := \int^{\oplus}_X \MmM_x \dint{\lambda(x)}$
and $\WWw''(\fFF)$ are spatially isomorphic. Write $\bB(\fFF) = (T_1,\ldots,T_N)$. $T_j$ corresponds (under the spatial
isomorphism) to $T_j' \in \MmM$. Since then $\bB(\fFF) \equiv (T_1',\ldots,T_N')$, we see that there is $\fFF' \in \CDDc_N$
such that $\bB(\fFF') = (T_1',\ldots,T_N')$ and consequently $\fFF' \equiv \fFF$. Thus replacing $\fFF$ by $\fFF'$, we may
assume that $\WWw''(\fFF) = \MmM$. Write $T_j = \int^{\oplus}_X T^{(x)}_j \dint{\lambda(x)}$ where $T^{(x)}_j \in \MmM_x$
for $\lambda$-almost all $x \in X$. Since $\|T_j\| \leqsl 1$, we also have $\|T^{(x)}_j\| \leqsl 1$ $\lambda$-almost
everywhere. Further, the function $x \mapsto \NnN(I - (T^{(x)}_j)^* T^{(x)}_j)$ is measurable (in the target space
we consider the Effros Borel structure separately on each of $\HHh_n$'s) and hence the set $X_0 = \{x \in X\dd\
\NnN(I - (T^{(x)}_j)^* T^{(x)}_j) \neq \{0\}\}$ is measurable. Suppose $\lambda(X_0) > 0$. Then there exists a measurable
vector field $x \mapsto \xi_x$ such that $\xi_x \in \NnN(I - (T^{(x)}_j)^* T^{(x)}_j)$ and $\|\xi_x\| \leqsl 1$
for $\lambda$-almost all $x \in X$, and $\int_X \|\xi_x\|^2 \dint{\lambda(x)} > 0$ (see Corollary after Theorem~2
in \cite{ef1}; or \cite[Corollary~IV.8.3]{tk1}). We infer from this that $\xi = \int^{\oplus}_X \xi_x \dint{\lambda(x)}$
is well defined and nonzero, and $T_j^* T_j \xi = \xi$ which denies the fact that $T_j$ is a value of the $\bB$-transform.
This shows that $\lambda(X_0) = 0$ and hence for $\lambda$-almost all $x \in X$ there is an operator $F^{(x)}_j \in \CDDc$
such that $\bB(F^{(x)}_j) = T^{(x)}_j$. Put $\fFF^{(x)} = (F^{(x)}_1,\ldots,F^{(x)}_N)$ and observe that the function
$x \mapsto \fFF^{(x)}$ is measurable (since the $\bB$-transform is an isomorphism) and $\fFF = \int^{\oplus}_X \fFF^{(x)}
\dint{\lambda(x)}$. Since the field $x \mapsto \WWw''(\fFF^{(x)})$ is measurable and $\WWw''(\fFF^{(x)}) \subset \MmM_x$,
$\int^{\oplus}_X \WWw''(\fFF^{(x)}) \dint{\lambda(x)} \subset \MmM = \WWw''(\fFF)$. At the same time, $T_1,\ldots,T_N \in
\int^{\oplus}_X \WWw''(\fFF^{(x)}) \dint{\lambda(x)}$ and therefore $\WWw''(\fFF) \subset \int^{\oplus}_X
\WWw''(\fFF^{(x)}) \dint{\lambda(x)}$ as well. We conclude from this that $\WWw''(\fFF^{(x)}) = \MmM_x$
for $\lambda$-almost all $x \in X$ and consequently $\int^{\oplus}_X \WWw''(\fFF^{(x)}) \dint{\lambda(x)}$ is the central
decomposition of $\WWw''(\fFF)$. In particular, $\FFF^{(x)} \in \Ff_N$ for $\lambda$-almost all $x \in X$. Now
\PRO{pd-transl} implies that $\FFF = \int^{\sqplus}_X \Phi(x) \dint{\lambda(x)}$ where $\Phi\dd X \ni x \mapsto \FFF^{(x)}
\in \Ff_N$. Let $\mu_{\FFF} = \Phi^*(\lambda) \in \rgM(\Ff_N)$. We know that then
\begin{equation}\label{eqn:centr-dec}
\FFF = \int^{\sqplus}_{\Ff_N} \XXX \dint{\mu_{\FFF}(\XXX)}.
\end{equation}
Further, since central decompositions of von Neumann algebras preserve the types (\cite[Theorem~14.1.21]{k-r1}
or \cite[Corollary~V.6.7]{tk1}), we infer from this that $\FFF$ is type I, I$^n$, II, II$^1$, II$^{\infty}$ or III iff
$\mu_{\FFF}$-almost all $\XXX \in \Ff_N$ are such. In particular, if $\FFF$ is the direct sum of a minimal $N$-tuple
and a semiminimal one, $\WWw''(\fFF)$ decomposes into type I$_1$, II$_1$ and III parts (and no other parts)
and consequently $\mu_{\FFF}$-almost all $\XXX \in \Ff_N$ are type I$^1$ (atoms) or II$^1$ (semiprimes), or III
(fractals)---cf. Propositions~\ref{pro:fractal} \pREF{pro:fractal} and \ref{pro:semiprime} \pREF{pro:semiprime}. This
proves the first claim of (A). The remainder of (A) follows from \COR{ue}.\par
We pass to (B). First of all, note that \eqref{eqn:pd-sep} makes sense thanks to \LEM{disj-pr}. Suppose $\BBB$ is
given by \eqref{eqn:pd-sep}. We may assume that $f$ is measurable. Then $s(f) = \{\PPP \in \pP_N\dd\ f(\PPP) > 0\} \in
\Bb_N$. It follows from (DI3) \pREF{DI3} that $\aleph_0 \odot \BBB = \int^{\oplus}_{\pP_N} (\aleph_0 \cdot f(\PPP)) \odot
\PPP \dint{\mu_{\AAA}(\PPP)} = \aleph_0 \odot \int^{\oplus}_{s(f)} \PPP \dint{\mu_{\AAA}(\PPP)} \leqsl \aleph_0 \odot \AAA$
and thus $\BBB \ll \AAA$.\par
Now assume that $\BBB \ll \AAA$. Let $\mu_{\BBB} \in \rgM(\Ff_N)$ be as in \eqref{eqn:centr-dec} with $\FFF = \BBB$. Since
$\BBB \ll \AAA$ and $\AAA, \BBB \in \SsS\EeE\PpP_N$, $\aleph_0 \odot \BBB \leqsl^s \aleph_0 \odot \AAA$ (cf. \CORP{ll}).
So, (PR6) \pREF{PR6} and \PRO{pd-transl} yield that there is a measurable set $\BbB \subset \pP_N$ such that $\aleph_0
\odot \BBB = \aleph_0 \odot \int^{\sqplus}_{\BbB} \PPP \dint{\mu_{\AAA}(\PPP)}$. Now we infer from (DI3) and \LEM{disj-pr}
that
\begin{equation}\label{eqn:aux110}
\int^{\sqplus}_{\Ff_N} \aleph_0 \odot \FFF \dint{\mu_{\BBB}(\FFF)} = \int^{\sqplus}_{\BbB} \aleph_0 \odot \PPP
\dint{\mu_{\AAA}(\PPP)}.
\end{equation}
An application of \PRO{meas-domain} shows that there are measurable domains $\FfF_0 \subset \BbB$ and $\GgG_0 \subset
\Ff_N$ such that $\mu_{\AAA}(\BbB \setminus \FfF_0) = 0$, $\mu_{\BBB}(\Ff_N \setminus \GgG_0) = 0$, $\FfF_0^* = \{\aleph_0
\odot \PPP\dd\ \PPP\in\FfF_0\} \in \Bb_N$, $\GgG_0^* = \{\aleph_0 \odot \FFF\dd\ \FFF\in\GgG_0\} \in \Bb_N$, the sets
$\FfF_0$, $\GgG_0$, $\FfF_0^*$ and $\GgG_0^*$ are standard Borel spaces and the functions $\Phi\dd \FfF_0 \ni \PPP \mapsto
\aleph_0 \odot \PPP \in \FfF_0^*$ and $\Psi\dd \GgG_0 \ni \FFF \mapsto \aleph_0 \odot \FFF \in \GgG_0^*$ are Borel
isomorphisms. Put $\FfF = \Phi^{-1}(\FfF_0^* \cap \GgG_0^*) \in \Bb_N$ and $\GgG = \Psi^{-1}(\FfF_0^* \cap \GgG_0^*) \in
\Bb_N$. Let $\Theta = \Psi^{-1} \circ \Phi\bigr|_{\FfF}$. Observe that $\Theta$ is a Borel isomorphism of $\FfF$ onto
$\GgG$. One may deduce from \COR{ue} and \eqref{eqn:aux110} that $\mu_{\AAA}(\BbB \setminus \FfF) = 0$
and $\mu_{\BBB}(\Ff_N \setminus \GgG) = 0$, and $\lambda \ll \mu_{\AAA}\bigr|_{\FfF} \ll \lambda$ where $\lambda(\sigma) =
\mu_{\BBB}(\Theta(\sigma \cap \FfF))$ for measurable $\sigma \subset \pP_N$. Consequently (by (DI4), \PREF{DI4}),
\begin{equation}\label{eqn:aux50}
\YYY = \int^{\sqplus}_{\FfF} \Theta(\PPP) \dint{\mu_{\AAA}(\PPP)}.
\end{equation}
Since $\Theta(\PPP) \ll \PPP$ for any $\PPP \in \FfF$, we may define $f\dd \pP_N \to I_{\aleph_0}$ by $f(\PPP) =
\Theta(\PPP) : \PPP$ for $\PPP \in \FfF$ and $f(\PPP) = 0$ for $\PPP \in \pP_N \setminus \FfF$. Thanks
to \eqref{eqn:aux50}, it suffices to show that $f\bigr|_{\FfF}$ is measurable. Since $\FfF$ and $\GgG$ are standard Borel
spaces, the graph $\Gamma = \{(\PPP,\Theta(\PPP))\dd\ \PPP \in \FfF\}$ of $\Theta$ is a Borel subset of $\FfF \times \GgG$
and $u\dd \FfF \ni \PPP \mapsto (\PPP,\Theta(\PPP)) \in \Gamma$ is a Borel isomorphism. Finally, since $\Div$ is Borel (see
Section~19, \PREF{Div}), so is the function $v\dd \Gamma \ni (\AAA,\BBB) \mapsto \BBB : \AAA \in I_{\aleph_0}$ (here is
important that $\FfF$ and $\GgG$ are standard Borel spaces). The notice that $f\bigr|_{\FfF} = v \circ u$ finishes
the proof.\par
Finally, point (C) follows from \PRO{pd-transl} and the previously mentioned fact that central decompositions
of von Neumann algebras preserve the types.
\end{proof}

The formula \eqref{eqn:centr-dec} corresponds to Ernest's central decomposition of a bounded operator
\cite[Chapter~III]{e}. It is however not of our interest. Also a variation of \eqref{eqn:pd-sep} appears
in \cite[Lemma~4.4]{e}.\par
In the sequel we shall need one more result.

\begin{lem}{orth-meas}
For $\mu, \nu \in \rgM(\Ff_N)$ \tfcae
\begin{enumerate}[\upshape(i)]
\item $\int^{\sqplus}_{\Ff_N} \FFF \dint{\mu(\FFF)} \disj \int^{\sqplus}_{\Ff_N} \FFF \dint{\nu(\FFF)}$,
\item there are measurable sets $\AaA, \BbB \subset \Ff_N$ such that $\mu(\Ff_N \setminus \AaA) = 0$, $\nu(\Ff_N \setminus
   \BbB) = 0$ and $\AaA \disj \BbB$,
\item there are measurable sets $\AaA, \BbB \subset \Ff_N$ such that $\mu(\Ff_N \setminus \AaA) = 0$, $\nu(\Ff_N \setminus
   \BbB) = 0$ and $\AaA \sdisj \BbB$,
\item $\mu \perp \nu$ and $\mu + \nu \in \rgM(\Ff_N)$.
\end{enumerate}
\end{lem}
\begin{proof}
(i)$\implies$(iv): Put $\AAA = \int^{\sqplus}_{\Ff_N} \XXX \dint{\mu(\XXX)}$, $\BBB = \int^{\sqplus}_{\Ff_N} \XXX
\dint{\nu(\XXX)}$ and $\FFF = \AAA \sqplus \BBB$, and let $\lambda = \mu_{\FFF}$ where $\mu_{\FFF}$ is
as in \eqref{eqn:centr-dec}. Since $\AAA, \BBB \leqsl^s \FFF$, we infer from \COR{ue} that $\mu, \nu \ll \lambda$. So,
$\mu + \nu \ll \lambda$ and therefore $\mu + \nu \in \rgM(\Ff_N)$. Further, there are measurable sets $\AaA, \BbB \subset
\Ff_N$ such that $\mu \ll \lambda\bigr|_{\AaA} \ll \mu$ and $\nu \ll \lambda\bigr|_{\BbB} \ll \nu$ and consequently, again
by \COR{ue}, $\AAA = \int^{\sqplus}_{\AaA} \XXX \dint{\lambda(\XXX)}$ and $\BBB = \int^{\sqplus}_{\BbB} \XXX
\dint{\lambda(\XXX)}$. Since then $\int^{\sqplus}_{\AaA \cap \BbB} \XXX \dint{\lambda(\XXX)} \leqsl^s \AAA, \BBB$, one has
$\lambda(\AaA \cap \BbB) = 0$ and hence $\mu \perp \nu$.\par
(iv)$\implies$(iii): Put $\lambda = \mu + \nu$ and let $\AaA_0$ and $\BbB_0$ be disjoint measurable subsets of $\Ff_N$
on which (respectively) $\mu$ and $\nu$ are concentrated. Since $\lambda \in \rgM(\Ff_N)$, $\int^{\oplus}_{\AaA_0} \FFF
\dint{\lambda(\FFF)} \disj \int^{\oplus}_{\BbB_0} \FFF \dint{\lambda(\FFF)}$ which yields (cf. \PROP{wdisj}, and the proof
of \PROP{meas-domain}, or \cite[Proposition~3.2.7]{sak}) that there exist a sequence $(p_n)_{n=1}^{\infty} \subset
\PpP_1(N)$ and a set $\ZzZ \in \NnN(\lambda)$ such that $p_n(\bB(\FFF),\bB(\FFF)^*) \stSt j(\FFF) I$ for each $\FFF \in
(\AaA_0 \cup \BbB_0) \setminus \ZzZ$ where $j$ is the characteristic function of $\AaA_0$. Consequently, $\mu$ and $\nu$
are concentrated on, respectively, $\AaA = \AaA_0 \setminus \ZzZ$ and $\BbB = \BbB_0 \setminus \ZzZ$, and $\AaA \sdisj
\BbB$.\par
Since (ii) obviously follows from (iii), it remains to show that (i) is implied by (ii). Suppose (i) is false. This means
that there are nontrivial $N$-tuples $\AAA \leqsl^s \int^{\sqplus}_{\Ff_N} \FFF \dint{\mu(\FFF)}$ and $\BBB \leqsl^s
\int^{\sqplus}_{\Ff_N} \FFF \dint{\nu(\FFF)}$ such that $\aleph_0 \odot \AAA = \aleph_0 \odot \BBB$. By \COR{ue}, there are
measurable sets $\AaA_1, \BbB_1 \subset \Ff_N$ such that $\AAA = \int^{\sqplus}_{\AaA_1} \FFF \dint{\mu(\FFF)}$ and $\BBB =
\int^{\sqplus}_{\BbB_1} \FFF \dint{\nu(\FFF)}$. All these combined with (DI3) \pREF{DI3} and \LEM{disj-pr} give
\begin{equation}\label{eqn:aux111}
\int^{\sqplus}_{\AaA_1 \cap \AaA} \aleph_0 \odot \FFF \dint{\mu(\FFF)} = \int^{\sqplus}_{\BbB_1 \cap \BbB} \aleph_0 \odot
\FFF \dint{\nu(\FFF)}
\end{equation}
where $\AaA$ and $\BbB$ are as in (ii). Thanks to \PROp{meas-domain}, we may assume that $\FfF = \{\aleph_0 \odot \FFF\dd\
\FFF \in \AaA_1 \cap \AaA\}$ and $\GgG = \{\aleph_0 \odot \FFF\dd\ \FFF \in \BbB_1 \cap \BbB\}$ are measurable. We conclude
from unitary disjointness of $\AaA$ and $\BbB$ that
\begin{equation}\label{eqn:aux112}
\Phi^*(\mu)(\GgG) = 0 \qquad \textup{and} \qquad \Phi^*(\nu)(\FfF) = 0
\end{equation}
where $\Phi\dd \Ff_N \ni \FFF \mapsto \aleph_0 \odot \FFF \in \Ff_N$. But \eqref{eqn:aux111} yields, by \COR{ue}, that
$\Phi^*(\mu) \ll \Phi^*(\nu) \ll \Phi^*(\mu)$. Consequently, it follows from \eqref{eqn:aux112} that $\mu(\AaA_1 \cap \AaA)
= 0$ and $\nu(\BbB \cap \BbB_1) = 0$ which denies the fact that $\AAA$ and $\BBB$ were nonzero.
\end{proof}

Taking into account the above result, for arbitrary two measures $\mu,\nu \in \rgM(\Ff_N)$ we shall write $\mu \perp_s \nu$
iff any of the equivalent conditions (i)--(iv) of \LEM{orth-meas} is fulfilled.

\SECT{`Continuous' direct sums}

Property (DI4) \pREF{DI4} suggests replacing standard measures $\mu$ by their null $\sigma$-ideals $\NnN(\mu)$. In this
section we follow this concept. In that way we shall extend the notion of the (standard `discrete') direct sum to more
general context. We begin with

\begin{dfn}{meas-null}
A \textit{measurable space with nullity} is a triple $(\xxX,\Mm,\NnN)$ where $(\xxX,\Mm)$ is a measurable space and $\NnN$
is a $\sigma$-ideal in $\Mm$; that is, $\varempty \in \NnN \subset \Mm$, $\bigcup_{n=1}^{\infty} A_n \in \NnN$ whenever
$\{A_n\}_{n=1}^{\infty} \subset \NnN$ and $\{B \in \Mm\dd\ B \subset A\} \subset \NnN$ for every $A \in \NnN$.\par
Whenever $(\xxX,\Mm,\NnN)$ is a measurable space with nullity, $\overline{\NnN}$ denotes the family of all (possibly
nonmeasurable) sets which are contained in members of $\NnN$. Members of $\overline{\NnN}$ are called \textit{null} sets,
the other subsets of $\xxX$ are called \textit{nonnull}. For $Y \in \Mm$, $(Y,\Mm\bigr|_Y,\NnN\bigr|_Y)$ is the induced
measurable space with nullity, i.e. $\Mm\bigr|_Y = \{B \in \Mm\dd\ B \subset Y\}$ and $\NnN\bigr|_Y = \Mm\bigr|_Y \cap
\NnN$. The space $(\xxX,\Mm,\NnN)$ is \textit{trivial} iff $\xxX \in \NnN$.\par
A function $\Phi\dd \xxX_1 \to \xxX_2$ is a \textit{null-isomorphism} between measurable spaces with nullities
$(\xxX_1,\Mm_1,\NnN_1)$ and $(\xxX_2,\Mm_2,\NnN_2)$ if $\Phi$ is a Borel isomorphism such that $\NnN_2 = \{\Phi(Z)\dd\
Z \in \NnN_1\}$. If $\Psi\dd X_1 \to \xxX_2$ (with $X_1 \subset \xxX_1$) is a function such that there are sets
$Z_1 \in \NnN_1$ and $Z_2 \in \NnN_2$ for which $\xxX_1 \setminus Z_1 \subset X_1$ and $\Psi\bigr|_{\xxX_1 \setminus Z_1}$
is a null-isomorphism of $(\xxX_1 \setminus Z_1,\Mm_1\bigr|_{\xxX_1 \setminus Z_1},\NnN_1\bigr|_{\xxX_1 \setminus Z_1})$
onto $(\xxX_2 \setminus Z_2,\Mm_2\bigr|_{\xxX_2 \setminus Z_2},\NnN_2\bigr|_{\xxX_2 \setminus Z_2})$, $\Psi$ is said to be
an \textit{almost null-isomorphism} and spaces $(\xxX_1,\Mm_1,\NnN_1)$ and $(\xxX_2,\Mm_2,\NnN_2)$ are \textit{almost
isomorphic}. Similarly, a function $u\dd X \to Y$ (where $X \subset \xxX$, $(\xxX,\Mm,\NnN)$ is a measurable space with
nullity and $(Y,\Nn)$ is a measurable space) is said to be \textit{almost measurable} iff there is a set $X' \in \Mm$
contained in $X$ such that $X \setminus X' \in \overline{\NnN}$ and $u\bigr|_{X'}$ is measurable.
\end{dfn}

Our main interest are measurable spaces whose nullities come from certain measures. For this purpose we introduce

\begin{dfn}{standard}
Let $(\xxX,\Mm,\NnN)$ be a measurable space with nullity. A measurable set $A \subset \xxX$ is \textit{standard} iff
$(A,\Mm\bigr|_A,\NnN\bigr|_A)$ is almost isomorphic to $(Y,\Nn,\NnN(\nu))$ for some standard measure space
$(Y,\Nn,\nu)$. Standard sets are nonnull.\par
A family $\BbB$ is said to be a \textit{base} of $(\xxX,\Mm,\NnN)$ iff the following two conditions are fulfilled:
\begin{itemize}
\item $\BbB$ consists of pairwise disjoint measurable sets and $\xxX \setminus \bigcup \BbB \in \NnN$,
\item for any $A \subset \bigcup \BbB$, $A \in \Mm$ (respectively $A \in \NnN$) iff $A \cap B \in \Mm$
   ($A \cap B \in \NnN$) for any $B \in \BbB$.
\end{itemize}
A base is \textit{standard} iff it consists of standard sets. $(\xxX,\Mm,\NnN)$ is called \textit{multi-standard} iff
it admits a standard base.
\end{dfn}

Let $\ffF = \{(\xxX_s,\Mm_s,\NnN_s)\}_{s \in S}$ be a family of measurable spaces with nullities. The direct sum of $\ffF$,
denoted by $\bigoplus_{s \in S} (\xxX_s,\Mm_s,\NnN_s)$, is a measurable space with nullity $(\xxX,\Mm,\NnN)$ defined
as follows: $\xxX = \bigcup_{s \in S} (\xxX_s \times \{s\})$; $\pi\dd \xxX \to \bigcup_{s \in S} \xxX_s$ is given
by $\pi(x,s) = x$; $A \in \Mm$ (respectively $A \in \NnN$) iff $\pi(A \cap (\xxX_s \times \{s\})) \in \Mm_s$
($\pi(A \cap (\xxX_s \times \{s\})) \in \NnN_s$) for every $s \in S$. Note that $\{\xxX_s \times \{s\}\}_{s \in S}$ is
a base of $\bigoplus_{s \in S} (\xxX_s,\Mm_s,\NnN_s)$. We call $\pi$ the \textit{canonical projection}.\par
Let $(\xxX,\Mm,\NnN)$ be a multi-standard measurable space with nullity. Let $\xxX^d$ be the set of all points $x \in \xxX$
such that $\{x\} \notin \overline{\NnN}$. One may show that $\xxX^d \in \Mm$ (since $\xxX$ is multi-standard),
$\Mm\bigr|_{\xxX^d}$ is the power set of $\xxX^d$ and $\NnN\bigr|_{\xxX^d} = \{\varempty\}$. Points of $\xxX^d$ are called
\textit{atoms}, while $\xxX^d$ and its complement $\xxX^c$ are called, respectively, the \textit{discrete}
and \textit{continuous} parts of $\xxX$. Further, if $(Y,\Nn,\mu)$ is a nonatomic standard measure space, then there is
$Z \in \NnN(\mu)$ such that $(Y \setminus Z,\Nn\bigr|_{Y \setminus Z},\NnN(\mu)\bigr|_{Y \setminus Z})$ is isomorphic
to $([0,1],\Bb([0,1]),\LlL_0)$ where $\LlL_0$ is the $\sigma$-ideal of all Borel subsets of $[0,1]$ whose Lebesgue measure
is equal to $0$ (by Theorem~14.3.9 on page~270 in \cite{roy}). Using this fact, one may check that there is a base
of $(\xxX,\Mm,\NnN)$ whose every member either consists of a single point belonging to $\xxX^d$ or is isomorphic
to $([0,1],\Bb([0,1]),\LlL_0)$. Since every base of the last mentioned measurable space with nullity is countable (finite
or not; to convince of that see the proof of \LEM{alm-meas} below), one deduces from this that either $\xxX^c$ is null
or is a standard set, or every standard base of $(\xxX,\Mm,\NnN)$ contains the same, uncountable, number of sets almost
isomorphic to $([0,1],\Bb([0,1]),\LlL_0)$. We define two characteristic cardinal numbers related to $\xxX$ as follows:
$\iota^d(\xxX) = \card(\xxX^d)$ and $\iota^c(\xxX)$ is either $0$ (if $\xxX^c$ is null) or $\aleph_0$ (if $\xxX^c$ is
standard), or is equal to the uncountable number of members of a standard base which are almost isomorphic
to $([0,1],\Bb([0,1]),\LlL_0)$. We see that two multi-standard measurable spaces with nullities $\xxX$ and $\yyY$ are
almost isomorphic iff $\iota^d(\xxX) = \iota^d(\yyY)$ and $\iota^c(\xxX) = \iota^c(\yyY)$. What is more, for any $\alpha
\in \Card$ and $\beta \in \Card_{\infty} \cup \{0\}$ there is a multi-standard measurable space with nullity $\zzZ$
for which $\iota^d(\zzZ) = \alpha$ and $\iota^c(\zzZ) = \beta$. (Indeed, take a set $D$ of cardinality $\alpha$ and a set
$S$ disjoint from $D$ whose cardinality is either $\beta$ if $\beta \neq \aleph_0$ or $1$ if $\beta = \aleph_0$. For each
$s \in S$ let $(I_s,\Mm_s,\NnN_s)$ be a copy of $([0,1],\Bb([0,1]),\LlL_0)$ and for $d \in D$ let $(I_d,\Mm_d,\NnN_d)$ be
a standard one-point measurable space with nullity. Now it suffices to define $\zzZ$ as $\bigoplus_{x \in D \cup S}
(I_x,\Mm_x,\NnN_x)$.)\par
From now on, $(\xxX,\Mm,\NnN)$ and $(\xxX',\Mm',\NnN')$ denote multi-standard measurable spaces with nullities.
Let $\Phi\dd \xxX \ni x \mapsto \BBB^{(x)} \in \SsS\EeE\PpP_N$ be any function. If there exist $Z \in \NnN$
and an integrable field $\xxX \setminus Z \ni x \mapsto \aAA^{(x)} \in \bigcup_{n=1}^{n=\infty} \CDDc_N(\HHh_n)$ such that
$\AAA^{(x)} = \Phi(x)$ for all $x \in \xxX \setminus Z$, we call $\Phi$ a \textit{summable field} and define
$\bigoplus^{\NnN}_{x\in\xxX} \BBB^{(x)}$ as follows. If $\xxX$ is trivial, we put $\bigoplus^{\NnN}_{x\in\xxX} \BBB^{(x)}
= \zero$. Otherwise let $\BbB$ be a standard base of $(\xxX,\Mm,\NnN)$. For every $B \in \BbB$ there is a standard measure
$\mu_B$ on $(B,\Mm\bigr|_B)$ such that $\NnN(\mu_B) = \NnN\bigr|_B$. We put
\begin{equation}\label{eqn:cont-sum}
\contS{x\in\xxX}{\NnN} \BBB^{(x)} = \bigoplus_{B\in\BbB} \int^{\oplus}_B \BBB^{(x)} \dint{\mu_B}(x).
\end{equation}
The next result shows that $\bigoplus^{\NnN}_{x\in\xxX} \BBB^{(x)}$ is well defined.

\begin{pro}{well-def}
Formula \eqref{eqn:cont-sum} well defines $\bigoplus^{\NnN}_{x \in \xxX} \BBB^{(x)}$. That is, the right-hand side
expression of \eqref{eqn:cont-sum} is independent of the choice of a standard base $\BbB$ and standard measures $\mu_B$'s;
and $\{\BBB^{(x)}\}_{x \in B}$ is an integrable (with respect to $\mu_B$) field for each $B \in \BbB$.
\end{pro}
\begin{proof}
Let $\BbB_1$ and $\BbB_2$ be standard bases for $(\xxX,\Mm,\NnN)$ and $\{\mu^{(j)}_B\dd\ B \in \BbB_j\}\ (j=1,2)$
corresponding families of standard measures. For each $D \in \BbB_j$ let $D' \in \Mm\bigr|_D$ be such that $D \setminus D'
\in \NnN$ and $(D',\Mm\bigr|_{D'})$ is a standard Borel space. Then the set
\begin{equation}\label{eqn:aux121}
I(D',\BbB_{3-j}) = \{E \in \BbB_{3-j}\dd\ D' \cap E \neq \varempty\} \quad \textup{is countable}
\end{equation}
(see the last fragment of the proof of \LEM{alm-meas} below). Additionally put $\IIi = \{(D_1,D_2) \in \BbB_1 \times
\BbB_2\dd\ D_1' \cap D_2' \notin \NnN\}$. Thanks to (DI3) \pREF{DI3} and \eqref{eqn:aux121} we obtain
\begin{multline*}
\bigoplus_{A \in \BbB_1} \int^{\oplus}_A \BBB^{(x)} \dint{\mu^{(1)}_A(x)} =\\
= \bigoplus_{A \in \BbB_1} \Bigl(\bigoplus \Bigl\{\int^{\oplus}_{A' \cap B'} \BBB^{(x)} \dint{\mu^{(1)}_A(x)}\dd\
(A,B) \in \IIi\Bigr\}\Bigr) =\\
= \bigoplus \Bigl\{\int^{\oplus}_{A' \cap B'} \BBB^{(x)} \dint{\mu^{(1)}_A(x)}\dd\ (A,B) \in \IIi\Bigr\}
\end{multline*}
and similarly
$$
\bigoplus_{B \in \BbB_2} \int^{\oplus}_B \BBB^{(x)} \dint{\mu^{(2)}_B(x)} =
\bigoplus \Bigl\{\int^{\oplus}_{A' \cap B'} \BBB^{(x)} \dint{\mu^{(2)}_B(x)}\dd\ (A,B) \in \IIi\Bigr\}.
$$
Now the notice that $\NnN(\mu^{(1)}_A\bigr|_{A' \cap B'}) = \NnN(\mu^{(2)}_B\bigr|_{A' \cap B'})$ combined with (DI4)
\pREF{DI4} yields that $$\bigoplus_{A \in \BbB_1} \int^{\oplus}_A \BBB^{(x)} \dint{\mu^{(1)}_A(x)} =
\bigoplus_{B \in \BbB_2} \int^{\oplus}_B \BBB^{(x)} \dint{\mu^{(2)}_B(x)}.$$
The remainder is left for the reader.
\end{proof}

It is easily seen that the restriction of a summable field to a measurable set is summable as well. Thanks
to \PRO{well-def}, we may rewrite \eqref{eqn:cont-sum} in a new form: whenever $\BbB$ is a standard base
of $(\xxX,\Mm,\NnN)$ and $\{\AAA^{(x)}\}_{x\in\xxX}$ is summable,
\begin{equation}\label{eqn:base}
\contS{x\in\xxX}{\NnN} \AAA^{(x)} = \bigoplus_{B\in\BbB} \Bigl(\contS{x \in B}{\NnN} \AAA^{(x)}\Bigr).
\end{equation}
Using this, one may prove that \eqref{eqn:base} is satisfied for a completely arbitrary (unnecessarily standard) base
$\BbB$.\par
Our next goal is to extend the notion of summability to more general context. In what follows, we equip $\RRR_+ \cup \Card$
with the Borel structure induced by the order topology (precisely, each of the sets $I_{\alpha}$ with $\alpha \in
\Card_{\infty}$ is equipped with this Borel structure).

\begin{lem}{alm-meas}
For a function $f\dd \xxX \to \RRR_+ \cup \Card$ \tfcae
\begin{enumerate}[\upshape(i)]
\item $f$ is almost measurable,
\item there is $Z \in \NnN$ with the following properties:
   \begin{enumerate}[\upshape(a)]
   \item $A = f^{-1}(\RRR_+) \setminus Z \in \Mm$ and $f\bigr|_A\dd A \to \RRR_+$ is measurable,
   \item for every $\alpha \in \Card_{\infty}$, $f^{-1}(\{\alpha\}) \setminus Z \in \Mm$,
   \item for each standard set $B \in \Mm$ there exists $Z_B \in \NnN$ such that the set $f(B \setminus Z_B) \cap
      \Card_{\infty}$ is countable (finite or not).
   \end{enumerate}
\end{enumerate}
\end{lem}
\begin{proof}
Suppose all conditions of (ii) are fulfilled. In what follows we preserve the notation of (ii). Let $\BbB$ be a standard
base of $\xxX$. Put $\zzZ = Z \cup \bigcup_{B\in\BbB} (B \cap Z_B)$. Then $\zzZ \in \NnN$ and points (a)--(c) imply that
$f\bigr|_{\xxX \setminus \zzZ}$ is measurable.\par
Now assume that $Z \in \NnN$ is such that $f\bigr|_{\xxX \setminus Z}$ is measurable. It is clear that conditions
(a) and (b) are satisfied. To show (c), it suffices to prove the following claim: if $(Y,\Nn)$ is a standard Borel space
and $u\dd Y \to I_{\gamma}$ is measurable, then $D = u(Y) \cap \Card_{\infty}$ is countable. Since $D$ is well ordered,
$D$ is countable iff so is the subset $D_0$ of $D$ consisting of all elements of $D$ which have the direct precedessor
(relative to $D$) in $D$. Note that if $\alpha \in D_0$, then $\{\alpha\}$ is open in $D$ with respect to the topology
inherited from $I_{\gamma}$. Consequently, every subset of $D_0$ is open in $D$ and hence $Y_0 = u^{-1}(D_0)$ is Borel and
$u\bigr|_{Y_0}$ is a Borel function of $Y_0$ (which is a standard Borel space) onto the discrete space $D_0$. It therefore
follows from theory of Souslin sets that $D_0$ is countable. (Indeed, if $D_0$ was uncountable, there would exist
a continuous mapping of $D_0$ onto a non-Souslin subset of $[0,1]$. It would then follow that a non-Souslin subset
of $[0,1]$ could be the image of a standard Borel space under a Borel function, which is impossible.)
\end{proof}

\LEM{alm-meas} has two important consequences: if $f, g\dd \xxX \to \RRR_+ \cup \Card$ are almost measurable and $\alpha
\in \Card$, the functions $f + g$ and $\alpha \cdot f$ are almost measurable as well. We shall use these facts several
times.\par
In the next two paragraphs $\Phi\dd \ddD \to \SsS\EeE\PpP_N$ is a summable field and $f\dd \ddD \to \RRR_+ \cup \Card$ is
an almost measurable function where $\ddD \in \Mm$ (notice that $\ddD$ is multi-standard).\par
We say that $f$ \textit{fits} to $\Phi$ iff there are two disjoint measurable sets $D_1$ and $D_2$ such that $\ddD
\setminus (D_1 \cup D_2) \in \NnN$, $f(D_1) \subset \Card$ and $\Phi(D_2) \subset \SsS\MmM_N$. (If $f$ fits to $\Phi$,
$f(x) \odot \Phi(x)$ makes sense for almost all $x \in \ddD$.)\par
There is $Z \in \NnN$ such that the sets $A = f^{-1}(I_{\aleph_0} \setminus \{0\}) \setminus Z$ and $A_{\alpha} =
f^{-1}(\{\alpha\}) \setminus Z$ with uncountable $\alpha$'s are measurable and the function $f\bigr|_A\dd A \to
I_{\aleph_0}$ is Borel. We call the pair $(f,\Phi)$ \textit{summable} if $f$ fits to $\Phi$ and the field $A \ni x
\mapsto f(x) \odot \Phi(x) \in \SsS\EeE\PpP_N$ is summable. If this is the case, we define $\bigoplus^{\NnN}_{x\in\ddD}
f(x) \odot \Phi(x)$ by
$$
\contS{x\in\ddD}{\NnN} f(x) \odot \Phi(x) = \Bigl(\contS{x \in A}{\NnN} f(x) \odot \Phi(x)\Bigr) \oplus
\bigoplus_{\alpha>\aleph_0} \Bigl(\alpha \odot \contS{x \in A_{\alpha}}{\NnN} \Phi(x)\Bigr).
$$
It is clear that summability of $(f,\Phi)$ and the formula for $\bigoplus^{\NnN}_{x\in\ddD} f(x) \odot \Phi(x)$ is
independent of the choice of $Z$. Notice that summability of $\Phi$ is equivalent to summability of $(\delta,\Phi)$ where
$\delta\dd \ddD \to \RRR_+ \cup \Card$ is constantly equal to $1$.\par
The following properties are infered from (DI0)--(DI4) \pREF{DI0} and \eqref{eqn:base}. Everywhere below
$(f,\{\AAA^{(x)}\}_{x\in\xxX})$ is a summable pair.
\begin{enumerate}[(CS1)]\addtocounter{enumi}{-1}
\item For each $\ddD \in \Mm$ the pair $(f,\{\AAA^{(x)}\}_{x\in\ddD})$ is summable as well and $\bigoplus^{\NnN}_{x\in\ddD}
   f(x) \odot \AAA^{(x)} = \zero$ iff $s_{\ddD}(f) := \{x\in\ddD\dd\ f(x) \neq 0\} \in \overline{\NnN}$;
   $\bigoplus^{\NnN}_{x\in\ddD} f(x) \odot \AAA^{(x)} \in \SsS\EeE\PpP_N$ iff there is $Z \in \NnN$ such that $s_{\ddD}(f)
   \setminus Z$ is standard.
\item The pair $(f,\{\bB(\AAA^{(x)})\}_{x\in\xxX})$ is summable and $\bB(\bigoplus^{\NnN}_{x\in\xxX} f(x) \odot \AAA^{(x)})
   = \bigoplus^{\NnN}_{x\in\xxX} f(x) \odot \bB(\AAA^{(x)})$.
\item\label{CS3} Whenever $\BbB$ is a base of $(\xxX,\Mm,\NnN)$,
   $$\contS{x\in\xxX}{\NnN} f(x) \odot \AAA^{(x)} = \bigoplus_{B\in\BbB} \Bigl(\contS{x \in B}{\NnN} f(x) \odot
   \AAA^{(x)}\Bigr).$$
\item \begin{enumerate}[(A)]
   \item If $(f,\{\BBB^{(x)}\}_{x\in\xxX})$ is summable, so is $(f,\{\AAA^{(x)} \oplus \BBB^{(x)}\}_{x\in\xxX})$ and
      $$\contS{x\in\xxX}{\NnN} f(x) \odot (\AAA^{(x)} \oplus \BBB^{(x)}) = \Bigl(\contS{x\in\xxX}{\NnN} f(x) \odot
      \AAA^{(x)}\Bigr) \oplus \Bigl(\contS{x\in\xxX}{\NnN} f(x) \odot \BBB^{(x)}\Bigr).$$
   \item For every $\alpha \in \Card$, the pair $(\alpha \cdot f,\{\AAA^{(x)}\}_{x\in\xxX})$ is summable
      and $\bigoplus^{\NnN}_{x\in\xxX} (\alpha \cdot f(x)) \odot \AAA^{(x)} = \alpha \odot (\bigoplus^{\NnN}_{x\in\xxX}
      f(x) \odot \AAA^{(x)})$.
   \item If in addition $(g,\{\AAA^{(x)}\}_{x\in\xxX})$ is summable, so is the pair $(f+g,\{\AAA^{(x)}\}_{x\in\xxX})$ and
      $$\contS{x\in\xxX}{\NnN} (f(x) + g(x)) \odot \AAA^{(x)} = \Bigl(\contS{x\in\xxX}{\NnN} f(x) \odot \AAA^{(x)}\Bigr)
      \oplus \Bigl(\contS{x\in\xxX}{\NnN} g(x) \odot \AAA^{(x)}\Bigr).$$
   \end{enumerate}
\item If $\psi\dd \xxX' \to \xxX$ is an almost null-isomorphism, the pair
   $(f \circ \psi,\{\AAA^{(\psi(x'))}\}_{x'\in\xxX'})$ is summable and $\bigoplus^{\NnN'}_{x'\in\xxX'} f(\psi(x')) \odot
   \AAA^{(\psi(x'))} = \bigoplus^{\NnN}_{x\in\xxX} f(x) \odot \AAA^{(x)}$.
\end{enumerate}
Since properties (CS3)--(B) and (CS3)--(C) are of great importance for us and are not so easy, let us prove them. It is
quite simple that both the pairs appearing in the assertions of these points are summable. Thanks to (CS2), we may assume
that $\xxX$ is standard. It then follows from \LEM{alm-meas} that we may also assume both the sets $f(\xxX) \cap
\Card_{\infty}$ and $g(\xxX) \cap \Card_{\infty}$ are countable and $f$ and $g$ are Borel. We start with (CS3)--(B).
Observe that (DI3) yields the assertion for $\alpha \leqsl \aleph_0$. So, $\aleph_0 \odot (\bigoplus^{\NnN}_{x\in\xxX} f(x)
\odot \AAA^{(x)}) = \bigoplus^{\NnN}_{x\in\xxX} (\aleph_0 \cdot f(x)) \odot \AAA^{(x)}$. This implies that we may further
assume that $f(\xxX) \subset \Card_{\infty}$ (replacing $f$ by $\aleph_0 \cdot f$ and reducing $\xxX$ to $s(f) =
s_{\xxX}(f)$). But then the assertion easily follows from (CS2) and the countability of $f(\xxX)$.\par
We now pass to (CS3)--(C). Put $A_f(\aleph_0) = f^{-1}(I_{\aleph_0})$ and $A_f(\alpha) = f^{-1}(\{\alpha\})$
for uncountable $\alpha$. In the same way define the sets $A_g(\beta)$ (corresponding to $g$) for $\beta \in
\Card_{\infty}$. Notice that the sets $I_f = \{\alpha\in\Card_{\infty}\dd\ A_f(\alpha) \neq \varempty\}$ and $I_g =
\{\alpha\in\Card_{\infty}\dd\ A_g(\alpha) \neq \varempty\}$ are countable and hence the family $\{A_f(\alpha) \cap
A_g(\beta)\dd\ (\alpha,\beta) \in I_f \times I_g\}$ is a base of $(\xxX,\Mm,\NnN)$. Using again (CS2), we may therefore
assume that $I_f$ and $I_g$ consist of single cardinals. The case $I_f = I_g = \{\aleph_0\}$ follows from (DI3), while
the one when $\aleph_0 \notin I_f \cup I_g$ is obvious. Finally, if e.g. $I_f = \{\aleph_0\}$ and $I_g = \{\alpha\}$ for
some $\alpha > \aleph_0$, then (by (CS3)--(B)) $\bigoplus^{\NnN}_{x\in\xxX} (f(x) + g(x)) \odot \AAA^{(x)} =
\bigoplus^{\NnN}_{x\in\xxX} g(x) \odot \AAA^{(x)} = \alpha \odot \bigoplus^{\NnN}_{x\in\xxX} \AAA^{(x)} \geqsl
\bigoplus^{\NnN}_{x\in\xxX} \aleph_0 \odot \AAA^{(x)}$ and (again by (CS2) and (CS3)--(B))
\begin{multline*}
\contS{x\in\xxX}{\NnN} \aleph_0 \odot \AAA^{(x)} = \Bigl(\contS{x\in\xxX}{\NnN} (\aleph_0 \cdot f(x)) \odot
\AAA^{(x)}\Bigr) \oplus \Bigl(\contS[\!\!]{x \notin s(f)}{\NnN} \aleph_0 \odot \AAA^{(x)}\Bigr)\\
\geqsl \aleph_0 \odot \Bigl(\contS{x\in\xxX}{\NnN} f(x) \odot \AAA^{(x)}\Bigr)
\geqsl \contS{x\in\xxX}{\NnN} f(x) \odot \AAA^{(x)}
\end{multline*}
which finishes the proof.\par
We now repeat the idea of the previous section. Let $(f,\{\AAA^{(x)}\}_{x\in\xxX})$ be a summable pair. If
\begin{equation}\label{eqn:disj-sum}
\contS{x\in\ddD'}{\NnN} f(x) \odot \AAA^{(x)} \disj \contS{x\in\ddD''}{\NnN} f(x) \odot \AAA^{(x)}
\end{equation}
for any two disjoint sets $\ddD', \ddD'' \in \Mm$, we call the pair $(f,\{\AAA^{(x)}\}_{x\in\xxX})$ \textit{regular}
and we write $\bigsqplus^{\NnN}_{x\in\xxX} f(x) \odot \AAA^{(x)}$ in place of $\bigoplus^{\NnN}_{x\in\xxX} f(x) \odot
\AAA^{(x)}$.\par
Similarly, a summable field $\{\AAA^{(x)}\}_{x\in\xxX}$ is regular iff \eqref{eqn:disj-sum} if fulfilled with $f$
constantly equal to $1$. As usual, the usage of $\bigsqplus^{\NnN}_{x\in\xxX} f(x) \odot \AAA^{(x)}$ includes information
that $(f,\{\AAA^{(x)}\}_{x\in\xxX})$ is regular. Note that, by definition, regular pairs and fields are summable.\par
The next result collects fundamental facts on the just defined notion.

\begin{thm}{disj-sum}
Let $\Phi\dd \xxX \ni x \to \AAA^{(x)} \in \Ff_N$ be any function.
\begin{enumerate}[\upshape(I)]
\item \TFCAE
   \begin{enumerate}[\upshape(i)]
   \item the field $\{\AAA^{(x)}\}_{x\in\xxX}$ is regular,
   \item for every standard set $A \in \Mm$ there is $Z \in \NnN$ such that $\Phi(A \setminus Z)$ is a measurable domain
      and $\Phi\bigr|_{A \setminus Z}$ is a Borel isomorphism of $A \setminus Z$ onto $\Phi(A \setminus Z)$.
   \end{enumerate}
\item If $\Phi$ satisfies condition \textup{(ii)} of point \textup{(I)} and $f\dd \xxX \to \RRR_+ \cup \Card$ is an almost
   measurable function which fits to $\Phi$, then $(f,\{\AAA^{(x)}\}_{x\in\xxX})$ is regular. Moreover,
   \begin{multline}\label{eqn:cont-ords}
   \Bigl\{\YYY \in \CDD_N\dd\ \YYY \leqsl^s \contSQ{x\in\xxX}{\NnN} f(x) \odot \AAA^{(x)}\Bigr\} =\\=
   \Bigl\{\contSQ{x\in\ddD}{\NnN} f(x) \odot \AAA^{(x)}\dd\ \ddD \in \Mm\Bigr\}.
   \end{multline}
\end{enumerate}
\end{thm}
\begin{proof}
Implication `(i)$\implies$(ii)' in point (I) follows immediately from \PROp{meas-domain}. To prove the converse one, first
note that $\Phi$ is summable because of (ii), the existence of a standard base of $\xxX$ and \PRO{meas-domain}. Further,
take two disjoint nonnull measurable sets $\ddD_1$ and $\ddD_2$. Let $\BbB_j$ be a standard base of $\ddD_j$. Since $B_1
\cup B_2$ is standard for $B_j \in \ddD_j$, we infer from point (ii) and \PRO{meas-domain} that
$\bigoplus^{\NnN}_{x \in B_1} \AAA^{(x)} \disj \bigoplus^{\NnN}_{x \in B_2} \AAA^{(x)}$. Consequently,
$\bigoplus_{B \in \BbB_1} (\bigoplus^{\NnN}_{x \in B} \AAA^{(x)}) \disj \bigoplus_{B \in \BbB_2}
(\bigoplus^{\NnN}_{x \in B} \AAA^{(x)})$ and hence the assertion of (i) follows from (CS2).\par
Now assume $\Phi$ and $f$ are as in (II). We may assume that $f$ is Borel. Define $f_0\dd \xxX \to I_{\aleph_0} \setminus
\{0\}$ by $f_0(x) = f(x)$ if $f(x) \in I_{\aleph_0} \setminus \{0\}$ and $f_0(x) = 1$ otherwise. The function $f_0$ is
Borel and fits to $\Phi$. Let $B \in \Mm$ be a standard set. Then there is a standard measure $\mu$ on $(B,\Mm\bigr|_B)$
such that $\NnN(\mu) = \NnN\bigr|_B$. We infer from the assumptions that $(\mu,\Phi\bigr|_B) \in \rgS(B,\Mm\bigr|_B)$.
Hence, \LEMp{disj-pr} implies that
\begin{equation}\label{eqn:aux122}
(\mu,(f_0 \odot \Phi)\bigr|_B) \in \rgS(B,\Mm\bigr|_B).
\end{equation}
Consequently, if $B_1$ and $B_2$ are two disjoint standard (measurable) subsets of $\xxX$, then
\begin{equation}\label{eqn:aux123}
\contS{x \in B_1}{\NnN} f_0(x) \odot \AAA^{(x)} \disj \contS{x \in B_2}{\NnN} f_0(x) \odot \AAA^{(x)}.
\end{equation}
We also conclude from \eqref{eqn:aux122} that $(f_0,\Phi)$ is summable on every standard subset of $\xxX$. Since $\xxX$
is multi-standard, $(f_0,\Phi)$ is therefore summable. It now follows from the definitions of $f_0$ and of summability that
$(f,\Phi)$ is summable as well.\par
Further, if $B$ is a standard subset of $\xxX$, it follows from \LEM{alm-meas} and the definitions of $f_0$
and of $\bigoplus^{\NnN}_{x \in B} f(x) \odot \Phi(x)$ that $\bigoplus^{\NnN}_{x \in B} f(x) \odot \Phi(x) \ll
\bigoplus^{\NnN}_{x \in B} f_0(x) \odot \AAA^{(x)}$. This combined with \eqref{eqn:aux123} yields that
\begin{equation}\label{eqn:aux125}
\contS{x \in B_1}{\NnN} f(x) \odot \AAA^{(x)} \disj \contS{x \in B_2}{\NnN} f(x) \odot \AAA^{(x)}
\end{equation}
for any two disjoint standard sets $B_1, B_2 \subset \xxX$. Now if $\ddD'$ and $\ddD''$ are two arbitrary disjoint nonnull
Borel subsets of $\xxX$, the fact that they are multi-standard together with (CS2) and \eqref{eqn:aux125} gives
\eqref{eqn:disj-sum}. It therefore suffices to check \eqref{eqn:cont-ords}. We have already shown the inclusion `$\supset$'
in \eqref{eqn:cont-ords} (cf. (CS2)). To this end, fix $\YYY \in \CDD_N$ such that
\begin{equation}\label{eqn:aux130}
\YYY \leqsl^s \contSQ{x\in\xxX}{\NnN} f(x) \odot \AAA^{(x)}.
\end{equation}
Let $\BbB_0$ be a standard base of $\xxX$. Thanks to \LEM{alm-meas}, for every $B \in \BbB_0$ there are pairwise disjoint
measurable subsets $W^B_0, W^B_1, \ldots$ of $B$ such that $B \setminus \bigcup_{n=0}^{\infty} W^B_n \in \NnN$, $f(W^B_0)
\subset \RRR_+ \setminus \{0\}$ and $f\bigr|_{W^B_n}$ is constantly equal to some $\alpha \in \Card_{\infty} \cup \{0\}$.
Notice that then $\BbB = \{W^B_n\dd\ B \in \BbB_0,\ n \geqsl 0\} \setminus \NnN$ is a standard base of $\xxX$ as well.
Denote by $\BbB_f$ the set of all $B \in \BbB$ for which $f(B) \subset \RRR_+ \setminus \{0\}$ and let $\BbB' = \BbB
\setminus \BbB_f$. For each $B \in \BbB'$ there is (unique) $\alpha_B \in \Card_{\infty} \cup \{0\}$ such that $f(B) =
\{\alpha_B\}$. Now (CS2), (CS3) and already proved part of (II) give
\begin{multline}\label{eqn:aux133}
\contSQ{x\in\xxX}{\NnN} f(x) \odot \AAA^{(x)} =\\=
\Bigl[\Bigsqplus_{B \in \BbB_f} \Bigl(\contSQ{x \in B}{\NnN} f(x) \odot \AAA^{(x)}\Bigr)\Bigr] \sqplus
\Bigr[\Bigsqplus_{B \in \BbB'} \alpha_B \odot \Bigl(\contSQ{x \in B}{\NnN} \AAA^{(x)}\Bigr)\Bigr].
\end{multline}
It may be deduced from \eqref{eqn:aux130} and \eqref{eqn:aux133} (using e.g. \PROP{leqsl-leqsls}, and \THMP{ords}) that
$\YYY$ is of the form
$$
\YYY = \Bigl(\Bigsqplus_{B \in \BbB_f} \YYY_B\Bigr) \sqplus \Bigl(\Bigsqplus_{B \in \BbB'} \widetilde{\YYY}_B\Bigr)
$$
where $\YYY_B \leqsl^s \bigsqplus^{\NnN}_{x \in B} f(x) \odot \AAA^{(x)}$ for $B \in \BbB_f$ and $\widetilde{\YYY}_B
\leqsl^s \alpha_B \odot \bigsqplus^{\NnN}_{x \in B} \AAA^{(x)}$ for $B \in \BbB'$. Further, by (PR6) \pREF{PR6}, for each
$B \in \BbB'$ there is $\YYY_B \leqsl^s \bigsqplus^{\NnN}_{x \in B} \AAA^{(x)}$ such that $\widetilde{\YYY}_B = \alpha_B
\odot \YYY_B$. Since $\BbB$ consists of standard sets, we infer from \PROp{pd-transl} that for every $B \in \BbB$ there
exists a measurable set $\ddD_B \subset B$ for which $\YYY_B = \bigsqplus^{\NnN}_{x\in\ddD_B} f(x) \odot \AAA^{(x)}$
provided $B \in \BbB_f$ and $\YYY_B = \bigsqplus^{\NnN}_{x\in\ddD_B} \AAA^{(x)}$ if $B \in \BbB'$. Put $\ddD =
\bigcup_{B \in \BbB} \ddD_B$ and note that $\ddD$ is Borel since $\BbB$ is a base. Finally, the family $\{\ddD_B\dd\
B \in \BbB\}$ is a base of $\ddD$ and hence we obtain from (CS2) and (CS3) that
\begin{multline*}
\contSQ{x\in\ddD}{\NnN} f(x) \odot \AAA^{(x)} = \Bigsqplus_{B \in \BbB} \Bigl(\contSQ{x\in\ddD_B}{\NnN} f(x) \odot
\AAA^{(x)}\Bigr) =\\= \Bigl(\Bigsqplus_{B\in\BbB_f} \YYY_B\Bigr) \sqplus \Bigl(\Bigsqplus_{B\in\BbB'} \alpha_B \odot
\YYY_B\Bigr) = \YYY
\end{multline*}
and we are done.
\end{proof}

Similarly as in the previous section, for a field $\Phi\dd \xxX \to \Ff_N$ we shall write $\Phi \in \rgS_{loc}$
or $\Phi \in \rgS_{loc}(\xxX)$ if $\Phi$ satisfies condition (ii) of \THM{disj-sum}.

\SECT{Prime decomposition}

Semiprimes are those members of $\pP_N$ which make the issue of prime decomposition of $N$-tuples more complicated and
ambiguous. To shape this in a way similar to that in the ring of natural numbers, we have to allow multiplicity functions
to take real values (beside infinite cardinals) instead of (only) integer ones. Such an approach is therefore similar
to Ernest's multiplicity theory (Chapter~4 of \cite{e}) and will enable us to propose the prime decomposition
of an arbitrary $N$-tuple in an (essentially) unique form (see \THMP{prime}). We consider this as a more attractive manner
of `factor decomposing' of $N$-tuples than Ernest's central decomposition \cite{e}.\par
In this section $(\xxX,\Phi)$ is a fixed pair such that $(\xxX,\Mm,\NnN)$ is a multi-standard measurable space with nullity
and $\Phi \in \rgS_{loc}(\xxX)$ is such that $\Phi(\xxX) \subset \pP_N$. After erasing from $\xxX$ a null measurable set,
we may assume $\Phi$ is measurable. Let
$$
\xxX_I = \Phi^{-1}(\aA_N), \qquad \xxX_{\tII} = \Phi^{-1}(\sS_N), \qquad \xxX_{\tIII} = \Phi^{-1}(\fF_N).
$$
Notice that $\xxX_I, \xxX_{\tII}$ and $\xxX_{\tIII}$ are measurable, pairwise disjoint and $\xxX_I \cup \xxX_{\tII} \cup
\xxX_{\tIII} = \xxX$.

\begin{dfn}{admiss}
A function $f\dd \ddD \to \RRR_+ \cup \Card$ where $\ddD \in \Mm$ is \textit{admissible} for $\Phi$ iff $f$ is almost
measurable, $f(\xxX_I \cap \ddD) \subset \Card$ and $f(\xxX_{\tIII} \cap \ddD) \subset \{0\} \cup \Card_{\infty}$.
The class of all admissible functions on $\xxX$ is denoted by $\aaA(\xxX,\Phi)$ or shortly by $\aaA(\xxX)$.\par
For each $f \in \aaA(\xxX)$, $s(f)$ is the \textit{support} of $f$, i.e. $s(f) = \{x\in\xxX\dd\ f(x) \neq 0\}$ ($s(f)$ is
measurable provided so is $f$).
\end{dfn}

Note that each admissible function fits to $\Phi$. Thus, by \THMp{disj-sum}, for every $f \in \aaA(\xxX)$ we may write
$\bigsqplus^{\NnN}_{x\in\xxX} f(x) \odot \Phi(x)$. As it is practised in measure theory, the term \textit{almost
everywhere}, which we shall abbreviate by writing \textit{a.e.}, will mean that suitable property (relation, etc.) holds
true on $\xxX \setminus Z$ for some $Z \in \overline{\NnN}$.\par
As a consequence of \LEMp{alm-meas} we obtain

\begin{cor}{A(X)}
For $f, g \in \aaA(\xxX)$,
\begin{enumerate}[\upshape(a)]
\item $f + g$, $f \cdot g$, $f \vee g$, $f \wedge g \in \aaA(\xxX)$ where $f \vee g = \max(f,g)$ and $f \wedge g =
   \min(f,g)$,
\item $\alpha \cdot f \in \aaA(\xxX)$ for each $\alpha \in \Card$,
\item if $f(\xxX_I \cup \xxX_{\tIII}) \subset \{0\}$, $t \cdot f \in \aaA(\xxX)$ for every $t \in \RRR_+$,
\item if $f \leqsl g$ a.e., there is $u \in \aaA(\xxX)$ such that $g = f + u$ a.e.
\end{enumerate}
\end{cor}

We leave the proof of \COR{A(X)} as an exercise. A part of it may be strengthened:

\begin{lem}{pd-count}
Whenever $f_1,f_2,\ldots$ are admissible functions, so are $\bigwedge_{n\geqsl1} f_n\dd \xxX \ni x \mapsto \inf_{n\geqsl1}
f_n(x) \in \RRR_+ \cup \Card$ and $\bigvee_{n\geqsl1} f_n\dd \xxX \ni x \mapsto \sup_{n\geqsl1} f_n(x) \in \RRR_+ \cup
\Card$. In particular, $\sum_{n=1}^{\infty} f_n \in \aaA(\xxX)$ (where $(\sum_{n=1}^{\infty} f_n)(x) = \sum_{n=1}^{\infty}
f_n(x)$).
\end{lem}
\begin{proof}
We leave this as an exercise that it is enough to show, thanks to \LEMp{alm-meas}, that the closure of any countable subset
$K$ of $\Card_{\infty}$ (in $I_{\gamma} \supset K$ whatever $\gamma \in \Card_{\infty}$ is) is countable as well (recall
that countable compact Hausdorff spaces are metrizable, by \cite[Theorem~3.1.9]{e}). But this is quite simple: for every
non-last element $x$ of $L = (\cll K) \setminus K$ there exists $c_x \in K$ which lies between $x$ and its direct successor
(relative to $L$) in $L$. Since the function $L \ni x \mapsto c_x \in K$ is one-to-one, the assertion follows.
\end{proof}

\begin{pro}{uniq}
For $f,g \in \aaA(\xxX)$,
\begin{equation}\label{eqn:equal}
\contSQ{x\in\xxX}{\NnN} f(x) \odot \Phi(x) = \contSQ{x\in\xxX}{\NnN} g(x) \odot \Phi(x)
\end{equation}
iff $f = g$ a.e.
\end{pro}
\begin{proof}
The `if' part is clear. Suppose \eqref{eqn:equal} holds true. It follows from (CS3) \pREF{CS3} that
$\bigsqplus^{\NnN}_{x\in\bbB} u(x) \odot \Phi(x) \ll \bigsqplus^{\NnN}_{x\in\bbB} \Phi(x)$ for each $\bbB \in \Mm$ and
$u \in \{f,g\}$. Since $\bigsqplus^{\NnN}_{x\in\bbB} \Phi(x) \disj \bigsqplus^{\NnN}_{x\notin\bbB} \Phi(x)$,
\eqref{eqn:equal} and (CS2) imply therefore that
\begin{equation}\label{eqn:eq}
\contSQ{x\in\bbB}{\NnN} f(x) \odot \Phi(x) = \contSQ{x\in\bbB}{\NnN} g(x) \odot \Phi(x)
\end{equation}
for any $\bbB \in \Mm$. Let $\ddD \in \Mm$ be standard. It suffices to check that $f = g$ almost everywhere on $\ddD$.
Thanks to \LEMp{alm-meas} we may assume that $f\bigr|_{\ddD}$ and $g\bigr|_{\ddD}$ are Borel and
\begin{equation}\label{eqn:aux140}
(f(\ddD) \cup g(\ddD)) \cap \Card_{\infty} \textup{ is countable.}
\end{equation}
By \eqref{eqn:aux140}, the sets $\ddD_+ = \{x \in \ddD\dd\ f(x) < g(x)\}$ and $\ddD_- = \{x \in \ddD\dd\ f(x) > g(x)\}$ are
Borel. Suppose, for the contrary, that e.g. $\ddD_+ \notin \NnN$. We distinguish between two cases.\par
Assume there are a nonnull measurable set $\bbB \subset \ddD_+$ and two cardinals $\alpha$ and $\beta$ such that $f(\bbB) =
\{\alpha\}$ and $g(\bbB) = \{\beta\}$. Let $\BBB = \bigsqplus^{\NnN}_{x\in\bbB} \Phi(x)$. We infer from (CS0) that $\BBB
\neq \zero$. Moreover, since $\Phi(\xxX) \subset \pP_N$, $\Phi \in \rgS_{loc}$ and $\bbB$ is standard, \LEMp{pd-sep} yields
that $\BBB$ is the direct sum of a minimal $N$-tuple and a semiminimal one. Consequently, $\alpha \odot \BBB < \beta \odot
\BBB$ (to convince of that use e.g. \THMP{decomp}, and (AO4), \PREF{AO4}, if needed). But this denies \eqref{eqn:eq}
because $\bigsqplus^{\NnN}_{x\in\bbB} f(x) \odot \Phi(x) = \alpha \odot \BBB$ and $\bigsqplus^{\NnN}_{x\in\bbB} g(x) \odot
\Phi(x) = \beta \odot \BBB$.\par
Finally, if there is no set $\bbB$ with all above mentioned properties, it may be deduced from \eqref{eqn:aux140} that
there exists a nonnull measurable set $\bbB \subset \ddD_+ \cap \xxX_{\tII}$ such that $f(\bbB) \subset \RRR_+$. Let $\BBB
= \bigsqplus^{\NnN}_{x\in\bbB} f(x) \odot \Phi(x)$. As before, an application of \LEM{pd-sep} shows that
\begin{equation}\label{eqn:sM}
\BBB \in \SsS\MmM_N.
\end{equation}
On the other hand, there is a measurable function $u\dd \bbB \to (\RRR_+ \cup \Card) \setminus \{0\}$ such that $g(x) =
f(x) + u(x)$ for all $x \in \bbB$. Then (CS3) combined with \eqref{eqn:eq} gives $\BBB = \bigsqplus^{\NnN}_{x\in\bbB} g(x)
\odot \Phi(x) = \BBB \oplus (\bigsqplus^{\NnN}_{x\in\bbB} u(x) \odot \Phi(x))$ which means, thanks to \eqref{eqn:sM}, that
$\bigsqplus^{\NnN}_{x\in\bbB} u(x) \odot \Phi(x) = \zero$ (cf. (AO4)), contradictory to (CS0).
\end{proof}

\begin{thm}{reg-id}
Let $\TTT = \bigsqplus^{\NnN}_{x\in\xxX} \Phi(x)$. Then
$$
\Bigl\{\contSQ{x\in\xxX}{\NnN} f(x) \odot \Phi(x)\dd\ f \in \aaA(\xxX,\Phi)\Bigr\} = \{\XXX \in \CDD_N\dd\ \XXX \ll \TTT\}.
$$
\end{thm}
\begin{proof}
It easily follows from (CS3) \pREF{CS3} that $\bigsqplus^{\NnN}_{x\in\xxX} f(x) \odot \Phi(x) \ll \TTT$ for every $f \in
\aaA(\xxX)$. We fix $\XXX \in \CDD_N$ such that $\XXX \ll \TTT$. Let $\{B_s\}_{s \in S}$ be a standard base of $\xxX$. We
may assume that $\bigcup_{s \in S} B_s = \xxX$. For each $s \in S$ put $\TTT_s = \bigsqplus^{\NnN}_{x \in B_s} \Phi(x)$. We
infer from (CS0) that $\TTT_s \in \SsS\EeE\PpP_N$ and from (CS2) that $\TTT = \bigsqplus_{s \in S} \TTT_s$. Let $\XXX_s =
\EEE(\XXX | \TTT_s)$. Observe that $\XXX = \bigsqplus_{s \in S} \XXX_s$ and $\XXX_s \ll \TTT_s$. Suppose for each $s \in S$
there is an admissible function $f_s\dd B_s \to \RRR_+ \cup \Card$ such that $\XXX_s = \bigsqplus^{\NnN}_{x \in B_s} f_s(x)
\odot \Phi(x)$. Then the union $f\dd \xxX \to \RRR_+ \cup \Card$ of $f_s$'s is admissible as well and it follows from (CS2)
that
$$
\contSQ{x\in\xxX}{\NnN} f(x) \odot \Phi(x) = \Bigsqplus_{s \in S} \Bigl(\contSQ{x \in B_s}{\NnN} f_s(x) \odot \Phi(x)\Bigr)
= \Bigsqplus_{s \in S} \XXX_s = \XXX.
$$
The above argument reduces the problem to the case when $\xxX$ is standard. Then there is a standard measure $\mu$
on $\Mm$ such that $\NnN(\mu) = \NnN$. Consequently,
\begin{equation}\label{eqn:aux127}
\contSQ{x\in\xxX}{\NnN} f(x) \odot \Phi(x) = \int^{\sqplus}_{\xxX} f(x) \odot \Phi(x) \dint{\mu(x)}
\end{equation}
for every Borel function $f\dd \xxX \to I_{\aleph_0}$ which fits to $\Phi$. Recall that for each $\AAA \in \CDD_N$,
$s(\AAA)$ is given by \EQp{s(A)} and $s(\AAA) = \bigwedge \{\EEE \leqsl^s \JJJ\dd\ \AAA \ll \EEE\}$. Since $\TTT \in
\SsS\EeE\PpP_N$ (because $\xxX$ is standard), $s(\TTT) \in \SsS\EeE\PpP_N$ as well. So, if $\XXX \ll \TTT$, then $s(\XXX)
\leqsl^s s(\TTT)$ and consequently the set $J = \{(i,\alpha) \in \Upsilon\dd\ \EEE^i_{\alpha}(\XXX) \neq \zero\}$ is
countable.\par
We infer from \LEMp{pd-sep} that:
\begin{itemize}
\item $\TTT$ is the direct sum of a minimal $N$-tuple and a semiminimal one,
\item there is $\lambda \in \rgM(\pP_N)$ such that $\TTT = \int^{\sqplus}_{\pP_N} \PPP \dint{\lambda(\PPP)}$,
\item for each $(i,\alpha) \in J$ there is a Borel function $u^i_{\alpha}\dd \pP_N \to I_{\aleph_0}$ such that
   $u^i_{\alpha}(\aA_N) \subset \Card$, $u^i_{\alpha}(\fF_N) \subset \{0,\aleph_0\}$ and
   \begin{equation}\label{eqn:aux128}\begin{array}{l @{\quad} l}
   \EEE^i_{\alpha}(\XXX) = \int^{\sqplus}_{\pP_N} u^i_{\alpha}(\PPP) \odot \PPP \dint{\lambda(\PPP)} &
   \textup{if } (i,\alpha) \neq (\tII,1),\\
   \EEE_{sm}(\XXX) = \int^{\sqplus}_{\pP_N} u^i_{\alpha}(\PPP) \odot \PPP \dint{\lambda(\PPP)} &
   \textup{if } (i,\alpha) = (\tII,1).
   \end{array}\end{equation}
\end{itemize}
Further, it follows from \CORp{ue} that
\begin{equation}\label{eqn:aux129}
\Phi^*(\mu) \ll \lambda \ll \Phi^*(\mu)
\end{equation}
(cf. \eqref{eqn:aux127}). Since $\xxX$ is standard and $\Phi \in \rgS_{loc}(\xxX)$, we may assume that $\Phi$ is a Borel
isomorphism of $\xxX$ onto a measurable domain. Put $g^i_{\alpha} = u^i_{\alpha} \circ \Phi$ for $(i,\alpha) \in J$
and note that $g^i_{\alpha} \in \aaA(\xxX)$. Now \eqref{eqn:aux127}, \eqref{eqn:aux128} and \eqref{eqn:aux129} combined
with (DI4) \pREF{DI4} for every $(i,\alpha) \in J$ yield
\begin{gather}
\EEE^i_{\alpha}(\XXX) = \contSQ{x\in\xxX}{\NnN} g^i_{\alpha}(x) \odot \Phi(x) \quad
\textup{if } (i,\alpha) \neq (\tII,1),\label{eqn:aux131}\\
\EEE_{sm}(\XXX) = \contSQ{x\in\xxX}{\NnN} g^i_{\alpha}(x) \odot \Phi(x) \quad
\textup{if } (i,\alpha) = (\tII,1).\label{eqn:aux132}
\end{gather}
Let $(i,\alpha)$ and $(i',\alpha')$ be distinct elements of $J$. Suppose $s(g^i_{\alpha}) \cap s(g^{i'}_{\alpha'}) \notin
\NnN$ ($s(g^i_{\alpha})$'s are measurable since $g^i_{\alpha}$'s are such). Then there is a nonnull measurable set $\bbB$
which is contained in $s(g^i_{\alpha}) \cap s(g^{i'}_{\alpha'})$. Consequently, thanks to (CS3)
and \eqref{eqn:aux131}--\eqref{eqn:aux132}, $\aleph_0 \odot \bigsqplus^{\NnN}_{x\in\bbB} \Phi(x) \leqsl \aleph_0 \odot
\EEE^i_{\alpha}(\XXX)$ as well as $\aleph_0 \odot \bigsqplus^{\NnN}_{x\in\bbB} \Phi(x) \leqsl \aleph_0 \odot
\EEE^{i'}_{\alpha'}(\XXX)$ which is impossible since $\EEE^i_{\alpha}(\XXX) \disj \EEE^{i'}_{\alpha'}(\XXX)$
and $\bigsqplus^{\NnN}_{x\in\bbB} \Phi(x) \neq \zero$. This proves that $s(g^i_{\alpha}) \cap s(g^{i'}_{\alpha'})
\in \NnN$ for any different members $(i,\alpha)$ and $(i',\alpha')$ of $J$. It then follows from the countability of $J$
that there is $\zzZ \in \NnN$ such that the sets $\ssS^i_{\alpha} = s(g^i_{\alpha}) \setminus \zzZ$ ($(i,\alpha) \in J$)
are pairwise disjoint. Now we define $f\dd \xxX \to \RRR_+ \cup \Card$ by the rules: $f(x) = \alpha \cdot g^i_{\alpha}(x)$
for $x \in \ssS^i_{\alpha}$ with $(i,\alpha) \in J \setminus \{(\tII,1)\}$; $f(x) = g^{\tII}_1(x)$ for $x \in
\ssS^{\tII}_1$ provided $(\tII,1) \in J$; and $f(x) = 0$ for $x \notin \bigcup_{(i,\alpha) \in J} \ssS^i_{\alpha}$.
It follows from the construction that $f \in \aaA(\xxX)$. Finally, \THMp{decomp}, \eqref{eqn:aux131}--\eqref{eqn:aux132},
(CS2) and (CS3) \pREF{CS3} give $\XXX = \bigsqplus^{\NnN}_{x \in \xxX} f(x) \odot \Phi(x)$.
\end{proof}

\THM{reg-id} asserts that $\IiI(\Phi) = \{\bigsqplus^{\NnN}_{x\in\xxX} f(x) \odot \Phi(x)\dd\ f \in \aaA(\xxX)\}$ is
an ideal. We call a quadruple $(\yyY,\Nn,\ZzZ,\Psi)$ or a pair $(\yyY,\Psi)$ a \textit{covering} for an ideal $\AaA \subset
\CDD_N$ iff $(\yyY,\Nn,\ZzZ)$ is a multi-standard measurable space with nullity, $\Psi \in \rgS_{loc}(\yyY)$, $\Psi(\yyY)
\subset \pP_N$ and $\IiI(\Psi) = \AaA$ (with this terminology we are inspired by condition (ii) of \THMP{disj-sum}).
Whenever the ideal $\AaA$ is irrelevant, we shall speak shortly of a \textit{covering}. A \textit{full covering} is
a covering for $\CDD_N$.\par
As usual, whenever $\ddD \in \Mm$, $j_{\ddD}$ stands for the characteristic function of $\ddD$.

\begin{cor}{pd-prp}
Let $f, g, h_1, h_2, \ldots \in \aaA(\xxX)$, $\XXX = \bigsqplus^{\NnN}_{x\in\xxX} f(x) \odot \Phi(x)$ and $\YYY =
\bigsqplus^{\NnN}_{x\in\xxX} g(x) \odot \Phi(x)$.
\begin{enumerate}[\upshape(A)]
\item $\XXX \leqsl \YYY$ iff $f \leqsl g$ a.e.
\item $\XXX \disj \YYY$ iff $f \cdot g = 0$ a.e.
\item $\XXX \ll \YYY$ iff $s(f) \setminus s(g) \in \overline{\NnN}$.
\item $\XXX \leqsl^s \YYY$ iff $f = g \cdot j_{\ddD}$ a.e. for some $\ddD \in \Mm$.
\item $\bigsqplus^{\NnN}_{x\in\xxX} [\sum_{n=1}^{\infty} h_n(x)] \odot \Phi(x) = \bigoplus_{n=1}^{\infty}
   [\bigsqplus^{\NnN}_{x\in\xxX} h_n(x) \odot \Phi(x)]$.
\end{enumerate}
\end{cor}
\begin{proof}
Observe that point (D) is an immediate consequence of \EQp{cont-ords} and \PRO{uniq}; (B) follows from (A)
and \THM{reg-id}; (E) is implied by (A), (CS3) \pREF{CS3} and (AO6) \pREF{AO6};  while (C) follows from (CS3) and (B). It
therefore suffice to prove (A). Implication `$\impliedby$' is a consequence of (CS3) and point (d) of \COR{A(X)}. Finally,
the converse implication follows from \PRO{uniq} and \THM{reg-id}. Indeed, if $\XXX \leqsl \YYY$, there is $\AAA \in
\CDD_N$ such that $\YYY = \XXX \oplus \AAA$. Then $\AAA \in \IiI(\Phi)$ and consequently there is $h \in \aaA(\xxX)$ for
which $\AAA = \bigsqplus^{\NnN}_{x\in\xxX} h(x) \odot \Phi(x)$. We now deduce from (CS3) that $\bigsqplus^{\NnN}_{x\in\xxX}
g(x) \odot \Phi(x) = \bigsqplus^{\NnN}_{x\in\xxX} (f + h)(x) \odot \Phi(x)$ and hence, by \PRO{uniq}, $g = f + h$ a.e.
\end{proof}

For need of the next result, we put $\xxX_{I_n} = \Phi^{-1}(\aA_N(n))$, $\xxX_{\tII_1} = \Phi^{-1}(\sS_N(1))$
and $\xxX_{\tII_{\infty}} = \Phi^{-1}(\sS_N(\infty))$. Observe that all just defined sets are pairwise disjoint,
$\xxX_I = \bigcup_{n=1}^{n=\infty} \xxX_{I_n}$ and $\xxX_{\tII} = \xxX_{\tII_1} \cup \xxX_{\tII_{\infty}}$, and they
are measurable if so is $\Phi$ (and this is our assumption, for simplicity).

\begin{cor}{pd-type}
Let $f \in \aaA(\xxX)$ and $\AAA = \bigsqplus^{\NnN}_{x\in\xxX} f(x) \odot \Phi(x)$.
\begin{enumerate}[\upshape(a)]
\item $\AAA \in \MmM\FfF_N$ (respectively $\AAA \in \HhH\IiI\MmM_N$; $\AAA \in \SsS\MmM_N$) iff $f = j_{\ddD}$ a.e.
   for some measurable $\ddD \subset \xxX_I$ (respectively $f = \aleph_0 \cdot j_{\ddD}$ a.e. for some measurable $\ddD
   \subset \xxX_{\tIII}$; there is $\zzZ \in \NnN$ such that $f((\xxX_I \cup \xxX_{\tIII}) \setminus \zzZ) \subset \{0\}$
   and $f(\xxX_{\tII} \setminus \zzZ) \subset \RRR_+$). In particular, $\bigsqplus^{\NnN}_{x\in\xxX} \Phi(x)$ is the direct
   sum of a minimal $N$-tuple and a semiminimal one.
\item $\AAA \in \SsS\EeE\PpP_N$ (respectively $\AAA \in \aA_N$; $\AAA \in \fF_N$; $\AAA \in \sS_N$; $\AAA \in \Ff_N$) iff
   there is $\zzZ \in \NnN$ such that $s(f) \setminus \zzZ$ is standard and $f(\xxX \setminus \zzZ) \subset I_{\aleph_0}$
   (respectively $f = j_{\{x\}}$ a.e. for some $x \in \xxX_I \cap \xxX^d$; $f = \aleph_0 \cdot j_{\{x\}}$ a.e. for some
   $x \in \xxX_{\tIII} \cap \xxX^d$; $f = t \cdot j_{\{x\}}$ a.e. for some $x \in \xxX_{\tII} \cap \xxX^d$
   and $t \in \RRR_+ \setminus \{0\}$; $f = s \cdot j_{\{x\}}$ a.e. for some $x \in \xxX^d$ and $s \in I_{\aleph_0}
   \setminus \{0\}$).
\item $\AAA$ is type $I$; $I^n$; $\tII$; $\tII^1$; $\tII^{\infty}$; $\tIII$ iff, respectively, $s(f) \setminus \xxX_I$;
   $s(f) \setminus \xxX_{I_n}$; $s(f) \setminus \xxX_{\tII}$; $s(f) \setminus \xxX_{\tII_1}$; $s(f) \setminus
   \xxX_{\tII_{\infty}}$; $s(f) \setminus \xxX_{\tIII}$ is a member of $\overline{\NnN}$.
\item $\AAA^d = \bigsqplus_{x\in\xxX^d} f(x) \odot \Phi(x)$ and $\AAA^c = \bigsqplus^{\NnN}_{x\in\xxX^c} f(x) \odot
   \Phi(x)$.
\item Let $\zzZ \in \NnN$ be such that $f\bigr|_{\xxX \setminus \zzZ}$ is Borel and $\xxX \setminus \zzZ$ is the union
   of a base $\BbB$ consisting of sets each of which is isomorphic either to $([0,1],\Bb([0,1]),\LlL_0)$ or to a one-point
   nontrivial measurable space with nullity (there exists such $\zzZ$). Put $\eeE_{sm} = f^{-1}(\RRR_+ \setminus \{0\})
   \cap \xxX_{\tII} \setminus \zzZ$ and $\eeE^i_{\alpha} = f^{-1}(\{\alpha\}) \cap \xxX_i \setminus \zzZ$ for $(i,\alpha)
   \in \Upsilon_*$. Then $\EeE = \{\eeE^i_{\alpha}\dd\ (i,\alpha) \in \Upsilon_*\} \cup \{\eeE_{sm}\}$ is a base of $\xxX$,
   and $\EEE_{sm}(\AAA) = \bigsqplus^{\NnN}_{x\in\eeE_{sm}} f(x) \odot \Phi(x)$ and $\EEE^i_{\alpha}(\AAA) =
   \bigsqplus^{\NnN}_{x\in\eeE^i_{\alpha}} \Phi(x)$ for $(i,\alpha) \in \Upsilon_*$ with $\alpha \neq 0$.
\end{enumerate}
\end{cor}
\begin{proof}
Points (a)--(d) are left as exercises. They are almost immediate consequences of Propositions \ref{pro:typei}
\pREF{pro:typei}, \ref{pro:pd-transl} \pREF{pro:pd-transl} and the fact that central decompositions of von Neumann algebras
preserve the types. Note also that $\bigsqplus^{\NnN}_{x\in\xxX^d} f(x) \odot \Phi(x) = \bigsqplus_{x\in\xxX^d} f(x) \odot
\Phi(x)$ since $\NnN\bigr|_{\xxX^d} = \{\varempty\}$. Here we shall focus on (e).\par
To prove (e), it suffices to show that $\EeE$ is a base of $\xxX$, since then the remainder will follow from (CS2),
(CS3) \pREF{CS3}, (a) and the uniqueness in \THMp{decomp}. It is clear that $\EeE$ consists of pairwise disjoint,
measurable sets (because $f$ is measurable on $\xxX \setminus \zzZ$) and $\xxX \setminus \bigcup \EeE = \zzZ$. Now assume
$A \subset \xxX \setminus \zzZ$ is such that $A \cap \eeE \in \Mm$ (respectively $A \cap \eeE \in \NnN$) for any $\eeE \in
\EeE$. Let $\BbB$ be as in (e). It follows from the proof of \LEMp{alm-meas} that $f(B) \cap \Card_{\infty}$ is countable
for each $B \in \BbB$. Consequently, also the set $\EeE(B) = \{E \in \EeE\dd\ E \cap B \neq \varempty\}$ is countable and
thus $A \cap B = \bigcup_{E\in\EeE(B)} [(A \cap E) \cap B]$ is a member of $\Mm$ (respectively $\NnN$) for any $B \in
\BbB$. Since $\BbB$ is a base, we obtain $A \in \Mm$ ($A \in \NnN$) and we are done.
\end{proof}

\begin{rem}{axiom}
For need of this remark, for every measurable set $\ddD \subset \xxX$, let $\jJ_{\ddD}$ denote an admissible function
which is equal to $0$ off $\ddD$, $1$ on $\ddD \setminus \xxX_{\tIII}$ and $\aleph_0$ on $\ddD \cap \xxX_{\tIII}$.\par
Using point (E) of \COR{pd-prp} as well as properties (CS2) and (CS3)--(B) \pREF{CS3}, one may show that whenever
$(\xxX,\Mm,\NnN,\Phi)$ is a covering, the regular (continuous) direct sums of the form $\bigsqplus^{\NnN}_{x\in\xxX} f(x)
\odot \Phi(x)$ with $f \in \aaA(\xxX)$ may axiomatically be defined by axioms (AX0)--(AX3) stated below. Namely, it is
now quite easy to prove that if $\Psi\dd \aaA(\xxX) \to \CDD_N$ is an assignment such that
\begin{enumerate}[({A}X1)]\addtocounter{enumi}{-1}
\item for every $\ddD \in \Mm$, $\Psi(\jJ_{\ddD}) = \bigsqplus^{\NnN}_{x\in\ddD} \Phi(x)$,
\item whenever $\BbB$ is a base of $\xxX$, $\Psi(f) = \bigoplus_{B\in\BbB} \Psi(\jJ_B \cdot f)$ for every
   $f \in \aaA(\xxX)$,
\item $\Psi(\alpha \cdot f) = \alpha \odot \Psi(f)$ for any $\alpha \in \Card$ and $f \in \aaA(\xxX)$,
\item $\Psi(\sum_{n=1}^{\infty} f_n) = \bigoplus_{n=1}^{\infty} \Psi(f_n)$ for all $f_1,f_2,\ldots \in \aaA(\xxX)$,
\end{enumerate}
then $\Psi(f) = \bigsqplus^{\NnN}_{x\in\xxX} f(x) \odot \Phi(x)$ for any $f \in \aaA(\xxX)$ (to show this, use
\COR{pd-type}--(e) and the fact that a real-valued measurable function may be written as the series of rational-valued
simple functions). However, at this moment we do not know whether $\Phi$ is uniquely determined (up to a.e. equality)
by `its' continuous direct sums appearing in (AX0). This (and even more) will be proved later, in \THMp{mod-pr}.
\end{rem}

The next result follows from \COR{pd-type} and its proof is left for the reader.

\begin{cor}{cov}
Let $(\yyY,\Psi)$ be a covering for an ideal $\AaA \subset \CDD_N$ and let $\BBB = \JJJ(\AaA)$. Then $\iota^d(\yyY) =
\card(\{\XXX \in \Ff_N\dd\ \XXX \leqsl^s \BBB\})$ and $\iota^c(\yyY) = \dim(\BBB^c)$.
\end{cor}

Our next aim is to establish (in a sense) uniqueness (\THM{uniq-cov} and \COR{uniq-cov} below) and existence
(\PRO{exis-cov}) of coverings for arbitrary ideals in $\CDD_N$.

\begin{thm}{uniq-cov}
Let $(\xxX_1,\Mm_1,\NnN_1,\Phi_1)$ and $(\xxX_2,\Mm_2,\NnN_2,\Phi_2)$ be two coverings such that
\begin{equation}\label{eqn:eq-cov}
\contSQ{x\in\xxX_1}{\NnN_1} \Phi_1(x) = \contSQ{x\in\xxX_2}{\NnN_2} \Phi_2(x).
\end{equation}
Then there are sets $\zzZ_j \in \NnN_j\ (j=1,2)$ and a null-isomorphism $\tau\dd \xxX_1 \setminus \zzZ_1 \to \xxX_2
\setminus \zzZ_2$ such that $\Phi_1(x) = \Phi_2(\tau(x))$ for each $x \in \xxX_1 \setminus \zzZ_1$.
\end{thm}
\begin{proof}
Let $\BbB_j$ be a standard base of $\xxX_j$. For $B \in \BbB_j$ put $\TTT^{(j)}_B = \bigsqplus^{\NnN_j}_{x \in B}
\Phi_j(x)$. It follows from (CS2) and \eqref{eqn:eq-cov} that
\begin{equation}\label{eqn:aux134}
\Bigsqplus_{B\in\BbB_1} \TTT^{(1)}_B = \Bigsqplus_{B\in\BbB_2} \TTT^{(2)}_B.
\end{equation}
Let $I = \{(B_1,B_2) \in \BbB_1 \times \BbB_2|\quad \TTT_{B_1,B_2} := \TTT^{(1)}_{B_1} \wedge \TTT^{(2)}_{B_2} \neq
\zero\}$. We conclude from \eqref{eqn:aux134} that
\begin{gather}
\TTT^{(1)}_B = \Bigsqplus \{\TTT_{B,B'}\dd\ (B,B') \in I\} \qquad (B \in \BbB_1),\label{eqn:aux201}\\
\TTT^{(2)}_B = \Bigsqplus \{\TTT_{B',B}\dd\ (B,B') \in I\} \qquad (B \in \BbB_2).\label{eqn:aux202}
\end{gather}
It follows from \COR{pd-prp} and \eqref{eqn:aux201}--\eqref{eqn:aux202} that sets $I_2(B') = \{B_2 \in \BbB_2\dd\
(B',B_2) \in I\}$ and $I_1(B'') = \{B_1 \in \BbB_1\dd\ (B_1,B'') \in I\}$ are countable (since $\TTT^{(j)}_B \in
\SsS\EeE\PpP_N$) for any $B' \in \BbB_1$ and $B'' \in \BbB_2$ and thus there are families of pairwise disjoint sets
$\{D^1_{B',B}\}_{B \in I_2(B')} \subset \Mm_1$ and $\{D^2_{B,B''}\}_{B \in I_1(B'')} \subset \Mm_2$ such that $B' =
\bigcup_{B \in I_2(B')} D^1_{B',B}$, $B'' = \bigcup_{B \in I_1(B'')} D^2_{B,B''}$ and
\begin{equation}\label{eqn:aux205}
\TTT_{B_1,B_2} = \contSQ[\!\!\!\!\!]{x \in D^1_{B_1,B_2}}{\NnN_1} \Phi_1(x) =
\contSQ[\!\!\!\!\!]{x \in D^2_{B_1,B_2}}{\NnN_2} \Phi_2(x)
\end{equation}
for every $(B_1,B_2) \in I$ (cf. \COR{pd-prp} or \THMP{disj-sum}). We also infer from the countability of the sets
$I_1(B_2)$'s and $I_2(B_1)$'s that
\begin{equation}\label{eqn:aux207}
\{D^j_{B_1,B_2}\dd\ (B_1,B_2) \in I\} \textup{ is a base of } \xxX_j.
\end{equation}
Fix $(B_1,B_2) \in I$. Since $D^j_{B_1,B_2}$ is standard and $\Phi_j \in \rgS_{loc}$, there is a Borel set $G_j \subset
D^j_{B_1,B_2}$ such that $D^j_{B_1,B_2} \setminus G_j \in \NnN_j$, $\Phi_j(G_j)$ is a measurable domain
and $\Phi_j\bigr|_{G_j}$ is a Borel isomorphism of $G_j$ onto $\Phi_j(G_j)$. Let $\mu_j$ be a standard measure
on $\Mm_j\bigr|_{G_j}$ for which $\NnN_j\bigr|_{G_j} = \NnN(\mu_j)$. Relation \eqref{eqn:aux205} yields that
$\int^{\sqplus}_{G_1} \Phi_1(x) \dint{\mu_1(x)} = \int^{\sqplus}_{G_2} \Phi_2(x) \dint{\mu_2(x)}$. Hence \CORp{ue} implies
that $\widehat{\mu}_1 \ll \widehat{\mu}_2 \ll \widehat{\mu}_1$ where $\widehat{\mu}_j(\FfF) = \mu_j(\Phi_j^{-1}(\FfF) \cap
G_j)$ for $\FfF \in \Bb(\pP_N)$. Consequently, $Z^j_{B_1,B_2} = D^j_{B_1,B_2} \setminus [\Phi_j^{-1}(\Phi_1(G_1) \cap
\Phi_2(G_2)) \cap G_j]\in \NnN_j$ and $\tau_{B_1,B_2}\dd D^1_{B_1,B_2} \setminus Z^1_{B_1,B_2} \ni x \mapsto
(\Phi_2\bigr|_{G_2})^{-1}(\Phi_1(x)) \in D^2_{B_1,B_2} \setminus Z^2_{B_1,B_2}$ is a well defined null-isomorphism such
that
\begin{equation}\label{eqn:aux208}\begin{cases}
\tau_{B_1,B_2}\dd D^1_{B_1,B_2} \setminus Z^1_{B_1,B_2} \to D^2_{B_1,B_2} \setminus Z^2_{B_1,B_2}&\\
\Phi_2 \circ \tau_{B_1,B_2} = \Phi_1\bigr|_{D^1_{B_1,B_2} \setminus Z^1_{B_1,B_2}}.&
\end{cases}\end{equation}
Now it suffices to put $\zzZ_j = (\xxX_j \setminus \bigcup \BbB_j) \cup \bigcup_{(B_1,B_2) \in I} Z^j_{B_1,B_2}$ and define
$\tau\dd \xxX_1 \setminus \zzZ_1 \to \xxX_2 \setminus \zzZ_2$ as the union of $\{\tau_{B_1,B_2}\}_{(B_1,B_2) \in I}$. It
follows from \eqref{eqn:aux207} and \eqref{eqn:aux208} that $\zzZ_j \in \NnN_j$ and $\tau$ is a null-isomorphism
we searched for.
\end{proof}

\begin{cor}{uniq-cov}
Let $\AaA \subset \CDD_N$ be an ideal and $(\xxX^1,\Mm^1,\NnN^1,\Phi^1)$ and $(\xxX^2,\Mm^2,\NnN^2,\Phi^2)$ be two
coverings for $\AaA$. Then there are sets $\zzZ^j \in \NnN^j\ (j=1,2)$, a Borel function $u\dd \xxX^1 \to \RRR_+
\setminus \{0\}$ with $u(\xxX^1_I \cup \xxX^1_{\tIII}) \subset \{1\}$, and a null-isomorphism $\tau\dd \xxX^1 \setminus
\zzZ^1 \to \xxX^2 \setminus \zzZ^2$ such that $\Phi^2(\tau(x)) = u(x) \odot \Phi^1(x)$ for every $x \in \xxX^1 \setminus
\zzZ^1$.
\end{cor}
\begin{proof}
Let $\TTT_j = \bigsqplus^{\NnN^j}_{x\in\xxX^j} \Phi^j(x)$. It follows from the assumptions and \THM{reg-id} that
\begin{equation}\label{eqn:aux210}
\TTT_1 \ll \TTT_2
\end{equation}
and there is $f \in \aaA(\xxX^1)$ such that
\begin{equation}\label{eqn:aux211}
\TTT_2 = \contSQ[\!]{x\in\xxX^1}{\NnN^1} f(x) \odot \Phi^1(x).
\end{equation}
Now \COR{pd-type} implies that $\TTT_2$ is the direct sum of a minimal $N$-tuple and a semiminimal one, and consequently
there is $Z \in \NnN^1$ such that $A := s(f) \setminus Z \in \Mm^1$, $f\bigr|_A$ is Borel, $f(A \cap \xxX^1_I) \subset
\{1\}$, $f(A \cap \xxX^1_{\tIII}) \subset \{\aleph_0\}$ and $f(A \cap \xxX^1_{\tII}) \subset \RRR_+ \setminus \{0\}$.
Further, \COR{pd-prp} combined with \eqref{eqn:aux210} yields that $\xxX^1 \setminus s(f) \in \overline{\NnN^1}$ and hence
$\xxX^1 \setminus A \in \NnN^1$. Define $u\dd \xxX^1 \to \RRR_+ \setminus \{0\}$ by $u(x) = f(x)$ for $x \in A \setminus
\xxX_{\tIII}$ and $u(x) = 1$ otherwise. Observe that $u$ is Borel and fits to $\Phi^1$, and $u(x) \odot \Phi^1(x) = f(x)
\odot \Phi^1(x)$ for $x \in A$. So, \eqref{eqn:aux211} gives
\begin{equation}\label{eqn:aux209}
\contSQ[\!]{x\in\xxX^2}{\NnN^2} \Phi^2(x) = \contSQ[\!]{x\in\xxX^1}{\NnN^1} u(x) \odot \Phi^1(x).
\end{equation}
Finally, since $u$ is real-valued, $u \odot \Phi^1\dd \xxX^1 \to \pP_N$ and we deduce from \THMp{disj-sum} that
$(\xxX^1,u \odot \Phi^1)$ is a covering. So, the assertion follows from \THM{uniq-cov}, thanks to \eqref{eqn:aux209}.
\end{proof}

To establish existence of coverings, we need the following

\begin{lem}{disj-meas}
Let $\eeE \subset \rgM(\pP_N)$ be such a family that
\begin{equation}\label{eqn:disj-meas}
\mu \perp_s \nu \quad \textup{if} \quad \mu \neq \nu \quad \textup{and} \quad \mu, \nu \in \eeE.
\end{equation}
Let $(\xxX,\Mm,\NnN) = \bigoplus_{\mu\in\eeE} (\pP_N,\Bb(\pP_N),\NnN(\mu))$ and $\Phi\dd \xxX \to \pP_N$ be the canonical
projection. Then $(\xxX,\Phi)$ is a covering and
\begin{equation}\label{eqn:realiz}
\contSQ{x\in\xxX}{\NnN} \Phi(x) = \Bigsqplus_{\mu\in\eeE} \int^{\sqplus}_{\pP_N} \PPP \dint{\mu(\PPP)}.
\end{equation}
\end{lem}
\begin{proof}
First of all, the usage of `$\bigsqplus_{\mu\in\eeE}$' in the right-hand side expression of \eqref{eqn:realiz} is allowed
by \LEMp{orth-meas}, thanks to \eqref{eqn:disj-meas}. Further, since regularity measures are concentrated on measurable
domains which are Souslin-Borel sets, $(\xxX,\Mm,\NnN)$ is a multi-standard measurable space with nullity and $\{\pP_N
\times \{\mu\}\}_{\mu\in\eeE}$ is a standard base of $\xxX$. Thus, it suffices to check that $\Phi \in \rgS_{loc}(\xxX)$
(then \eqref{eqn:realiz} will automatically be satisfied). It is clear that $\Phi$ is Borel.\par
Let $A \in \Mm$ be standard. We will show that condition (ii) of \THMp{disj-sum} is fulfilled. Since $A$ is standard,
the set $\eeE' = \{\mu \in \eeE\dd\ \Phi(A) \notin \NnN(\mu)\}$ is countable. Observe that $Z_0 = A \cap
[\bigcup_{\mu\notin\eeE'} (\pP_N \times \{\mu\})]$ belongs to $\NnN$. Since $A \setminus Z_0 \subset \pP_N \times \eeE'
\in \Mm$, we may assume that
\begin{equation}\label{eqn:aux225}
A = \pP_N \times \eeE'.
\end{equation}
For $\mu \in \eeE'$ let $\TTT_{\mu} = \int^{\sqplus}_{\pP_N} \PPP \dint{\mu(\PPP)}$. Put $\TTT = \bigsqplus_{\mu \in \eeE'}
\TTT_{\mu}$. It follows from point (C) of \LEMp{pd-sep} that $\TTT_{\mu}$ ($\mu \in \eeE'$) is the direct sum of a minimal
$N$-tuple and a semiminimal one, and thus so is $\TTT$. Moreover, since $\eeE'$ is countable, $\TTT \in \SsS\EeE\PpP_N$
($\TTT \neq \zero$ because standard sets are nonnull). Now point (A) of \LEM{pd-sep} asserts that there is a measure
$\lambda \in \rgM(\pP_N)$ such that $\TTT = \int^{\sqplus}_{\pP_N} \PPP \dint{\lambda(\PPP)}$. Since $\TTT_{\mu} \leqsl^s
\TTT$, we conclude from \CORp{ue} that
\begin{equation}\label{eqn:aux217}
\mu \ll \lambda \qquad (\mu \in \eeE').
\end{equation}
Further, it follows from \eqref{eqn:disj-meas} and the countability of $\eeE'$ that there is a collection
$\{S_{\mu}\}_{\mu\in\eeE'}$ of pairwise disjoint measurable subsets of $\pP_N$ such that $\mu(\pP_N \setminus S_{\mu}) = 0$
for every $\mu \in \eeE'$. Finally, let $\FfF \subset \pP_N$ be a measurable domain such that $\lambda(\pP_N \setminus
\FfF) = 0$. Put $$D = \bigcup_{\mu\in\eeE'} [(S_{\mu} \cap \FfF) \times \{\mu\}].$$
Observe that $D \subset A$ (by \eqref{eqn:aux225}), $A \setminus D \in \NnN$ ($\pP_N \setminus (S_{\mu} \cap \FfF) \in
\NnN(\mu)$ by \eqref{eqn:aux217}), $\Phi\bigr|_D$ is one-to-one (since the sets $S_{\mu}$'s are pairwise disjoint)
and $\Phi(D) \subset \FfF$. So, \REM{weak} \pREF{rem:weak} finishes the proof.
\end{proof}

\begin{pro}{exis-cov}
Let $\TTT \in \CDD_N$ be the direct sum of a minimal $N$-tuple and a semiminimal one. There is a covering
$(\xxX,\Mm,\NnN,\Phi)$ such that
$$
\TTT = \contSQ{x\in\xxX}{\NnN} \Phi(x).
$$
\end{pro}
\begin{proof}
By Zorn's lemma, there is a maximal family $\eeE \subset \rgM(\pP_N)$ such that \eqref{eqn:disj-meas} is satisfied
and $\TTT_{\mu} := \int^{\sqplus}_{\pP_N} \PPP \dint{\mu(\PPP)} \leqsl^s \TTT$ for each $\mu \in \eeE$ (since $\TTT_{\mu}
\neq \zero$; cf. \LEMP{continuum}). It follows from \LEM{disj-meas} and its proof that $\bigsqplus_{\mu\in\eeE} \TTT_{\mu}
\leqsl^s \TTT$ and that it is enough to show that $\XXX := \TTT \sqminus (\bigsqplus_{\mu\in\eeE} \TTT_{\mu})$ is equal to
$\zero$. Suppose, for the contrary, that $\XXX \neq \zero$. Since $\TTT \leqsl \JJJ$, we infer from \PROp{count} that there
is $\YYY \in \SsS\EeE\PpP_N$ such that $\YYY \leqsl^s \XXX$. Then $\YYY$ is the direct sum of a minimal $N$-tuple and
a semiminimal one (because $\XXX \leqsl^s \TTT$). Now \LEMp{pd-sep} yields that there is $\nu \in \rgM(\pP_N)$ such that
$\int^{\sqplus}_{\pP_N} \PPP \dint{\nu(\PPP)} = \YYY (\leqsl^s \TTT)$. Finally, since $\YYY \disj \TTT_{\mu}$ for every
$\mu \in \eeE$, \LEMp{orth-meas} asserts that $\nu \perp_s \mu$ for any $\mu \in \eeE$, contradictory to the fact that
$\eeE$ is maximal.
\end{proof}

The next theorem is an immediate consequence of all previously established properties. This result may be formulated
for arbitrary coverings. Our main interest however are full ones. To make the theorem most transparent, we repeat some
of properties proved earlier.

\begin{thmm}{prime}{\textbf{Prime Decomposition}}
\begin{enumerate}[\upshape(I)]
\item There exists a full covering. What is more, for every $\TTT \in \SsS\MmM_N$ with $\aleph_0 \odot \TTT = \JJJ_{\tII}$
   there is a full covering $(\xxX,\Mm,\NnN,\Phi)$ such that $\bigsqplus^{\NnN}_{x\in\xxX} \Phi(x) = \JJJ_I \sqplus \TTT
   \sqplus \JJJ_{\tIII}$.
\item Let $(\xxX^1,\Mm^1,\NnN^1,\Phi^1)$ and $(\xxX^2,\Mm^2,\NnN^2,\Phi^2)$ be two full coverings. There are a Borel
   function $u\dd \xxX^1 \to \RRR_+ \setminus \{0\}$ such that $u(\xxX^1_I \cup \xxX^{\tIII}) = \{1\}$ and an almost
   null-isomorphism $\tau\dd \xxX^1 \to \xxX^2$ such that $\Phi^2 \circ \tau = u \odot \Phi^1$ a.e. In particular,
   for every $f \in \aaA(\xxX^2)$, $(f \circ \tau)u \in \aaA(\xxX^1)$ and
   $$
   \contSQ[\!]{x\in\xxX^2}{\NnN^2} f(x) \odot \Phi^2(x) = \contSQ[\!]{x\in\xxX^1}{\NnN^1} [(f \circ \tau)u](x) \odot
   \Phi^1(x).
   $$
\item Let $(\xxX,\Mm,\NnN,\{\PPP_x\}_{x\in\xxX})$ be a full covering.
   \begin{enumerate}[\upshape(A)]
   \item For each $\AAA \in \CDD_N$ there is $f \in \aaA(\xxX)$ such that $\AAA = \bigsqplus^{\NnN}_{x\in\xxX} f(x) \odot
      \PPP_x$.
   \item For every $f_1, f_2, f_3, \ldots \in \aaA(\xxX)$, $\bigsqplus^{\NnN}_{x\in\xxX} [\sum_{n=1}^{\infty} f_n(x)] \odot
      \PPP_x = \bigoplus_{n=1}^{\infty} [\bigsqplus^{\NnN}_{x\in\xxX} f_n(x) \odot \PPP_x]$.
   \item Let $f, g \in \aaA(\xxX)$. Put $\XXX = \bigsqplus^{\NnN}_{x\in\xxX} f(x) \odot \PPP_x$ and $\YYY =
      \bigsqplus^{\NnN}_{x\in\xxX} g(x) \odot \PPP_x$. Then:
      \begin{enumerate}[\upshape(a)]
      \item $\XXX = \YYY \iff f = g$ a.e.,
      \item $\XXX \leqsl \YYY \iff f \leqsl g$ a.e.,
      \item $\XXX \leqsl^s \YYY \iff f = g \cdot j_{\ddD}$ a.e. for some $\ddD \in \Mm$,
      \item $\XXX \ll \YYY \iff s(f) \setminus s(g) \in \overline{\NnN}$,
      \item $\XXX \disj \YYY \iff f \cdot g = 0$ a.e. $\iff s(f) \cap s(g) \in \overline{\NnN}$,
      \item $\alpha \odot \XXX = \bigsqplus^{\NnN}_{x\in\xxX} (\alpha \cdot f)(x) \odot \PPP_x$ for any $\alpha \in \Card$,
      \item $\XXX \in \SsS\MmM_N \iff s(f) \setminus \xxX_{\tII} \in \overline{\NnN}$ and $f^{-1}(\Card_{\infty}) \in
         \overline{\NnN}$; if $\XXX \in \SsS\MmM_N$, then $t \odot \XXX = \bigsqplus^{\NnN}_{x\in\xxX} [t \cdot f(x)] \odot
         \PPP_x$ for each $t \in \RRR_+$,
      \item $\XXX \in \SsS\EeE\PpP_N$ iff there is $\zzZ \in \NnN$ such that $s(f) \setminus \zzZ$ is standard
         and $f(\xxX \setminus \zzZ) \subset I_{\aleph_0}$.
      \end{enumerate}
   \end{enumerate}
\end{enumerate}
\end{thmm}

We leave the proofs of point (g) and of a part of point (II) of \THM{prime} as exercises.\par
\THM{prime} says that after fixing $\TTT \in \SsS\MmM_N$ such that $\aleph_0 \odot \TTT = \JJJ_{\tII}$, there is a unique
(up to almost null-isomorphism) full covering $(\xxX,\Mm,\NnN,\Phi)$ such that $\bigsqplus^{\NnN}_{x\in\xxX_{\tII}} \Phi(x)
= \TTT$. Then for every $\AAA \in \CDD_N$ there is a unique (up to almost everywhere equality) function $\mM \in
\aaA(\xxX)$ such that
\begin{equation}\label{eqn:pd}
\AAA = \contSQ{x\in\xxX}{\NnN} \mM(x) \odot \Phi(x).
\end{equation}
The function $\mM$ is called \textit{multiplicity function of $\AAA$ (relative to $\TTT$)} (compare with Chapter~4
of \cite{e}) and the formula \eqref{eqn:pd} is called \textit{prime decomposition of $\AAA$ (relative to $\TTT$)}.
One may check that $\AAA \in \CDD_N$ has multiplicity function (respectively prime decomposition) of a unique
(i.e. independent of the choice of $\TTT$) form iff $\EEE_{sm}(\AAA) = 0$ (respectively $\AAA \disj \JJJ_{\tII}$).\par
Since $\aA_N(n)$ for finite $n$ consists of bounded $N$-tuples, \THM{prime} implies that every $N$-tuple $\XXX$ whose type
I$_{\infty}$, II and III parts vanish admits a decomposition in the form $\XXX = \bigoplus_{n=1}^{\infty} \XXX^{(n)}$ where
each $\XXX^{(n)}$ is bounded. So, under the notation of \EXS{id-clopen} \pREF{exs:id-clopen}, every such $\XXX$ belongs
to $\IiI[\cll\Omega(\operatorname{bd})]$.

\begin{rem}{mod}
\THM{prime} implies that all measurable spaces with nullities being ingredients of full coverings are almost isomorphic.
One may therefore ask of their (common) characteristic numbers $\iota^d$ and $\iota^c$. Using results of the next section
and \COR{cov} one may show that both of them are equal to $2^{\aleph_0}$. Even more: whenever $(\xxX,\Phi)$ is a full
covering, for $\yyY \in \{\xxX,\xxX_I,\xxX_{I_1},\xxX_{I_2},\ldots,\xxX_{I_{\infty}},\xxX_{\tII},\xxX_{\tII_1},
\xxX_{\tII_{\infty}},\xxX_{\tIII}\}$ one has $\iota^d(\yyY) = \iota^c(\yyY) = 2^{\aleph_0}$.
\end{rem}

\begin{rem}{mod-mod}
There is a striking resemblance between Theorems~\ref{thm:model} \pREF{thm:model} and \ref{thm:prime}, and between
the forms of $\Lambda(\Omega)$ (where $\Omega$ is an underlying model space) and of $\aaA(\xxX,\Psi)$ (where $(\xxX,\Psi)$
is a full covering). It is not a coincidence. When $(\xxX,\Mm,\NnN,\Psi)$ is a full covering, $\AAa =
L^{\infty}(\xxX,\Mm,\NnN)$ is a $\WWw^*$-algebra (since $\xxX$ is multi-standard---see the notes of the first paragraph
of \S1.18 in \cite{sak}). Now if $\Omega$ is the Gelfand spectrum of $\AAa$, there is a one-to-one correspondence between
clopen subsets of $\Omega$ and members of $\Mm$, which naturally correspond to $N$-tuples $\XXX$ such that $\XXX \leqsl^s
\widetilde{\TTT} := \JJJ_I \sqplus \TTT \sqplus \JJJ_{\tIII}$ where $\TTT := \bigsqplus^{\NnN}_{x\in\xxX_{\tII}} \Psi(x)$.
Since $\ZZz(\WWw''(\widetilde{\tTT}))$ is isomorphic to $\ZZz(\WWw''(\jJJ))$ (because $\widetilde{\TTT} \ll \JJJ \ll
\widetilde{\TTT}$; cf. (PR6), \PREF{PR6}), $\Omega$ is therefore homeomorphic to the Gelfand spectrum
of $\ZZz(\WWw''(\jJJ))$, that is, $\Omega$ is an underlying model space. Further, using results of Sections~15 and 22, one
may show that there is a `natural' correspondence, $f \mapsto \widehat{f}$, between $\Lambda(\Omega)$ and $\aaA(\xxX)$
(induced by the isomorphism between $\CCc(\Omega)$ and $\AAa$) where in $\aaA(\xxX)$ we identify functions which are equal
almost everywhere. One may then check that the assignment
$$
\Lambda(\Omega) \ni f \mapsto \contSQ{x\in\xxX}{\NnN} \widehat{f}(x) \odot \Psi(x)
$$
is inverse to $\Phi_{\TTT}$ introduced in \THM{model}. Thus $\aaA(\xxX)$ may be considered as a `concrete realization'
of $\Lambda(\Omega)$. With such an approach, the multiplicity function $\mM \in \aaA(\xxX)$ (relative to $\TTT$)
of $\XXX \in \SsS\MmM_N$ corresponds to $\DiNT{\XXX}{\TTT}$.
\end{rem}

\begin{thm}{mod-pr}
Let $(\xxX,\Mm,\NnN)$ be a multi-standard measurable space with nullity.
\begin{enumerate}[\upshape(I)]
\item Let $\Phi\dd \xxX \to \pP_N$ be such that $(\xxX,\Phi)$ is a covering and let $\mu\dd \Mm \to \CDD_N$ be given by
   \begin{equation}\label{eqn:mu-Phi}
   \mu(\aaA) = \contSQ{x\in\aaA}{\NnN} \Phi(x) \qquad (\aaA \in \Mm).
   \end{equation}
   Then:
   \begin{enumerate}[\upshape(M1)]\addtocounter{enumii}{-1}
   \item $\mu(\xxX)$ is the direct sum of a minimal $N$-tuple and a semiminimal one,
   \item for every $\aaA \in \Mm$, $\mu(\aaA) = \zero \iff \aaA \in \NnN$,
   \item whenever $\aaA$ and $\bbB$ are two measurable disjoint sets, $\mu(\aaA \cup \bbB) = \mu(\aaA) \sqplus \mu(\bbB)$,
   \item for every $\AAA \in \CDD_N$ such that $\AAA \leqsl^s \mu(\xxX)$ there exists $\aaA \in \Mm$ for which $\mu(\aaA)
      = \AAA$.
   \end{enumerate}
\item For every function $\mu\dd \Mm \to \CDD_N$ satisfying conditions \textup{(M0)--(M3)} there exists a unique
   (up to almost everywhere equality) function $\Phi\dd \xxX \to \pP_N$ such that $(\xxX,\Phi)$ is a covering
   and \eqref{eqn:mu-Phi} is fulfilled.
\end{enumerate}
\end{thm}
\begin{proof}
Point (I) is left for the reader. Here we focus only on (II).\par
Let $\mu$ be as in (II). Put $\TTT = \mu(\xxX)$. Observe that:
\begin{enumerate}[(M1)]\addtocounter{enumi}{3}
\item for any $\aaA, \bbB \in \Mm$, $\mu(\aaA) \leqsl^s \mu(\bbB)$ iff $\aaA \setminus \bbB \in \NnN$,
\item $\{\mu(\aaA)\dd\ \aaA \in \Mm\} = \{\AAA \in \CDD_N\dd\ \AAA \leqsl^s \TTT\}$.
\end{enumerate}
Indeed, (M5) easily follows from (M2) and (M3), because $\TTT = \mu(\aaA) \sqplus \mu(\xxX \setminus \aaA)$ for every
measurable $\aaA$. To prove (M4), first of all note that
\begin{equation}\label{eqn:aux150}
\mu(\aaA) = \mu(\bbB) \iff (\aaA \setminus \bbB) \cup (\bbB \setminus \aaA) \in \NnN,
\end{equation}
since, by (M2), $\mu(\aaA) = \mu(\aaA \setminus \bbB) \sqplus \mu(\aaA \cap \bbB)$, $\mu(\bbB) = \mu(\bbB \setminus \aaA)
\sqplus \mu(\aaA \cap \bbB)$ and (again thanks to (M2)) $\mu(\aaA \setminus \bbB) \disj \mu(\bbB \setminus \aaA)$.
These combined with (M1) give \eqref{eqn:aux150}. Now if $\aaA \setminus \bbB \in \NnN$, we infer from \eqref{eqn:aux150}
that $\mu(\aaA) = \mu(\aaA \cap \bbB)$ and hence, by (M2), $\mu(\bbB) = \mu(\bbB \setminus \aaA) \sqplus \mu(\aaA)$ which
yields $\mu(\aaA) \leqsl^s \mu(\bbB)$. Conversely, if the latter inequality is fulfilled, we conclude from (M5) that there
is $\ccC \in \Mm$ such that $\mu(\ccC) = \mu(\bbB) \sqminus \mu(\aaA)$. Since then (again by (M2)) $\mu(\ccC \cup \aaA) =
\mu(\ccC) \sqplus \mu(\aaA \setminus \ccC) = \mu(\aaA) \sqplus \mu(\ccC \setminus \aaA)$, $\mu(\ccC) \disj \mu(\aaA)$
and $\mu(\aaA \setminus \ccC) \disj \mu(\ccC \setminus \aaA)$, we get $\mu(\aaA) = \mu(\aaA \setminus \ccC)$
and consequently $\mu(\aaA \cup \ccC) = \mu(\ccC) \sqplus \mu(\aaA) = \mu(\bbB)$. So, \eqref{eqn:aux150} yields
the assertion of (M4).\par
Further, it follows from (M0) and \PRO{exis-cov} that there is a covering $(\xxX',\Mm',\NnN',\Psi)$ such that $\TTT =
\bigsqplus^{\NnN'}_{x\in\xxX'} \Psi(x)$. Put $\mu'\dd \Mm' \ni \aaA \mapsto \bigsqplus^{\NnN'}_{x\in\aaA} \Psi(x) \in
\CDD_N$. It may be infered from \PROp{uniq}, \THMp{reg-id} and \CORp{pd-prp} that conditions (M4) and (M5) as well as
\eqref{eqn:aux150} are satisfied when $\mu$, $\Mm$ and $\NnN$ are replaced by (respectively) $\mu'$, $\Mm'$ and $\NnN'$.
Let $\MmM$ and $\MmM'$ denote the quotient (abstract) Boolean $\sigma$-algebras $\Mm / \NnN$ and $\Mm' / \NnN'$
(respectively). We shall denote the equivalence class in $\MmM$ (in $\MmM'$) of $\aaA \in \Mm$ (of $\aaA \in \Mm'$)
by $[\aaA]_{\NnN}$ (by $[\aaA]_{\NnN'}$). (M4), (M5) and \eqref{eqn:aux150} for both $\mu$ and $\mu'$ imply that
the rule
$$
\tau([\aaA]_{\NnN}) = [\bbB]_{\NnN'} \iff \mu(\aaA) = \mu'(\bbB)
$$
well defines an order isomorphism $\tau\dd \MmM \to \MmM'$. One deduces from this that whenever $\tau([\aaA]_{\NnN}) =
[\bbB]_{\NnN'}$,
$$
\aaA \textup{ is standard} \iff \bbB \textup{ is standard.}
$$
Since $\tau$ is an order isomorphism, it is an isomorphism of Boolean $\sigma$-algebras as well. Now an application
of \cite[Corollary~14.4.12]{roy} separately for every member of a standard base of $\xxX$ yields that there are sets
$\zzZ \in \NnN$ and $\zzZ' \in \NnN'$, and a null-isomorphism $\varphi\dd \xxX \setminus \zzZ \to \xxX' \setminus \zzZ'$
such that $\tau([\aaA]_{\NnN}) = [\varphi(\aaA \setminus \zzZ)]_{\NnN'}$ for every $\aaA \in \Mm$. In particular,
$\mu(\aaA) = \mu'(\varphi(\aaA \setminus \zzZ))$ or, equivalently,
$$
\mu(\aaA) = \contSQ[\!\!\!\!\!\!]{y\in\varphi(\aaA\setminus\zzZ)}{\NnN'} \Psi(y) =
\contSQ[\!\!\!]{x\in\aaA\setminus\zzZ}{\NnN} (\Psi \circ \varphi)(x)
$$
for any $\aaA \in \Mm$. So, to obtain \eqref{eqn:mu-Phi} it suffices to define $\Phi\dd \xxX \to \pP_N$ as an arbitrary
extension of $\Psi \circ \varphi$.\par
Now assume that $\Phi'\dd \xxX \to \pP_N$ is another function such that $(\xxX,\Phi')$ is a covering and $\mu(\aaA) =
\bigsqplus_{x\in\aaA}^{\NnN} \Phi'(x)$ for every $\aaA \in \Mm$. Then $\bigsqplus_{x\in\xxX}^{\NnN} \Phi(x) =
\bigsqplus_{x\in\xxX}^{\NnN} \Phi'(x)$ and consequently---by \THM{uniq-cov}---there is an almost null-isomorphism
$\kappa\dd \xxX \to \xxX$ such that $\Phi' = \Phi \circ \kappa$ almost everywhere. It suffices to check that $\kappa(x)
= x$ for almost all $x \in \xxX$. Take $\zzZ \in \NnN$ such that $\kappa\bigr|_{\xxX \setminus \zzZ}$ is a null
isomorphism of $\xxX \setminus \zzZ$ onto its (measurable) range. For simplicity, for every $\aaA \in \Mm$ put $\aaA_* =
\aaA \setminus \zzZ$. Notice that then
$$
\contSQ{x\in\aaA_*}{\NnN} \Phi(x) = \contSQ[\!\!\!]{x\in\kappa(\aaA_*)}{\NnN} \Phi(x).
$$
This implies (cf. \PROP{uniq}) that $(\aaA_* \setminus \kappa(\aaA_*)) \cup (\kappa(\aaA_*) \setminus \aaA_*) \in \NnN$.
Equivalently, $[\aaA_*]_{\NnN} = [\kappa(\aaA_*)]_{\NnN}$ for every $\aaA \in \Mm$. Since $\xxX$ is multi-standard,
it follows from the uniqueness in \cite[Theorem~14.4.10]{roy} that $\kappa(x) = x$ almost everywhere and we are done.
\end{proof}

\SECT{Classification of ideals up to isomorphism}

This section is the only part of the treatise in which we will compare ideals of tuples of different lengths (that is,
ideals in $\CDD_N$ as well as in $\CDD_{N'}$ with $N' \neq N$).\par
We begin with

\begin{exm}{existence}
It is known that every properly infinite or type I von Neumann algebra acting in a separable Hilbert space is singly
generated (\cite{wog}, \cite{sai}, \cite{g-s}). There are also known examples of singly generated type II$_1$ factors
(\cite{g-s}). Also tensor products of two singly generated von Neumann algebras acting in separable Hilbert spaces are
singly generated (\cite[Corollary~2.1]{sai}). Further, according to \cite[Theorem~2.6.6]{sak}, the $\WWw^*$-tensor product
of a type I$_n$, II$_1$, II$_{\infty}$ or III $\WWw^*$-algebra and $L^{\infty}([0,1])$ is of the same type. Also,
for a factor $\MmM$,
\begin{equation}\label{eqn:aux47}
\ZZz(\MmM \bar{\otimes} L^{\infty}([0,1])) \cong L^{\infty}([0,1]),
\end{equation}
by \cite[Proposition~2.6.7]{sak} or \cite[Corollary~IV.5.11]{tk1}. Finally, if $T$ is a bounded operator and $\tTT =
(T,\ldots,T) \in \CDDc_N$, then $\WWw(T) = \WWw(\tTT)$. All these notices yield that the ideals $\IiI_{I_n}^c$,
$\IiI_{\tII_1}^c$, $\IiI_{\tII_{\infty}}^c$ and $\IiI_{\tIII}^c$ are nonntrivial. (Indeed, take a singly generated factor
$\MmM$ acting in a separable Hilbert space of a fixed type $i$ and let $T$ be a generator of $\MmM \bar{\otimes}
L^{\infty}([0,1])$. Then $\tTT = (T,\ldots,T) \in \IiI_i^c$, by \eqref{eqn:aux47}.)
\end{exm}

\begin{cor}{123}
Let $\Omega$ denote the underlying model space for $\CDD_N$. Each of the spaces $\Omega$, $\Omega_I$, $\Omega_{I_n}$
($n=1,2,\ldots,\infty$), $\Omega_{\tII}$, $\Omega_{\tII_1}$, $\Omega_{\tII_{\infty}}$ and $\Omega_{\tIII}$ is homeomorphic
to the topological disjoint union of $\beta D(2^{\aleph_0})$ and $\beta[D(2^{\aleph_0}) \times \Xx]$
where $D(2^{\aleph_0})$ is the discrete space of power $2^{\aleph_0}$ and $\Xx$ is the Gelfand spectrum
of $L^{\infty}([0,1])$.
\end{cor}
\begin{proof}
By \THMp{homeo}, it suffice to show that $\kappa_c(E) = 2^{\aleph_0}$ where $E$ denotes any of the sets under
the question. Equivalently (cf. \PROP{Dim}), this is to say that $\dim(\JJJ(\AaA)) = 2^{\aleph_0}$ where $\AaA$ is one
of $\IiI_{I_n}^c$, $\IiI_{\tII_1}^c$, $\IiI_{\tII_{\infty}}^c$, $\IiI_{\tIII}^c$. To simplify the argument, let
$(i,k,\Zz)$ be one of $(I,n,\aA_N(n))$ (where $n \in \{1,2,\ldots,\infty\}$), $(\tII,1,\sS_N(1))$,
$(\tII,\infty,\sS_N(\infty))$, $(\tIII,\infty,\fF_N)$ and let $\AaA = \IiI_{i_k}^c$. By \EXM{existence} we know $\AaA$ is
nontrivial. Hence (e.g. by \PROP{count}) there is $\TTT_0 \in \AaA \cap \SsS\EeE\PpP_N$ which is either minimal
or semiminimal. Now according to \LEMp{pd-sep}, there is $\mu_0 \in \rgM(\pP_N)$ such that $\TTT_0 =
\int^{\sqplus}_{\pP_N} \PPP \dint{\mu_0(\PPP)}$. There is a measurable domain $\FfF$ on which $\mu$ concentrates. Since
$\mu$ is standard, we may assume that $\FfF$ is a standard Borel space, and that $\FfF \subset \Zz$,
by \COR{pd-type}--(c) \pREF{cor:pd-type}. We infer from the fact that $\AAA^c = \AAA$ that $\mu_0$ is nonatomic
and consequently that $\FfF$ is uncountable. So, $\FfF$ is Borel isomorphic to $[0,1]$ which implies that there is
a family $\{\lambda_t\}_{t\in\RRR}$ of probabilistic nonatomic Borel measures on $\FfF$ which are mutually singular. Since
every measure on $\FfF$ is a regularity measure, \LEMp{orth-meas} shows that $\XXX_s := \int^{\sqplus}_{\FfF} \PPP
\dint{\lambda_s(\PPP)} \disj \int^{\sqplus}_{\FfF} \PPP \dint{\lambda_t(\PPP)} = \XXX_t$ for any distinct real numbers $s$
and $t$. Finally, again thanks to \COR{pd-type}, $\XXX_s \in \AaA$ (because $\FfF \subset \Zz$ and $\lambda_s$ is
nonatomic) and $\XXX_s$ is minimal or semiminimal for every $s\in\RRR$. Consequently, $\XXX := \bigsqplus_{s\in\RRR}
\XXX_s$ is a minimal or semiminimal member of $\AaA$ as well. This gives $\XXX \leqsl \JJJ(\AaA)$ and therefore
$\dim(\JJJ(\AaA)) \geqsl \dim(\XXX) = 2^{\aleph_0}$ (since $\XXX_s \in \SsS\EeE\PpP_N$ for each $s \in \RRR$).
\end{proof}

As an important consequence of \COR{123} we obtain that the underlying model space for $\CDD_N$ and its `characteristic'
subsets are independent of $N$. This will be crucial in our investigations. Hence, we may shortly speak
of an \textit{underlying model space}.\par
Everywhere below $\AaA$ and $\BbB$ denote arbitrary ideals in $\CDD_N$ and $\CDD_{N'}$ (respectively).

\begin{dfn}{isom}
A function $\Phi\dd \AaA \to \BbB$ is an \textit{isomorphism} iff $\Phi$ is a bijection and $\Phi(\bigoplus_{s \in S}
\AAA_s) = \bigoplus_{s \in S} \Phi(\AAA_s)$ for every collection $\{\AAA_s\}_{s \in S} \subset \AaA$ (where, of course,
 $S$ is a set). An isomorphism $\Phi\dd \AaA \to \BbB$ is
\begin{itemize}
\item an \textit{s-isomorphism} iff $\dim \Phi(\AAA) = \dim \AAA$ for every $\AAA \in \AaA$,
\item a \textit{t-isomorphism} iff for each $\AAA \in \AaA$ the following condition is fulfilled: $\Phi(\AAA)$ is
 of type $i^k$ iff so is $\AAA$ where $i^k$ is one of $I^n$ ($n=1,2,\ldots,\infty$), $\tII^1$, $\tII^{\infty}$, $\tIII$.
\end{itemize}
Two ideals are \textit{isomorphic}, \textit{s-isomorphic} or \textit{t-isomorphic} if there exists an isomorphism
of suitable kind between them.\par
Let `i' be the empty, `s' or `t' prefix. We write $\AaA \cong^i \BbB$ iff $\AaA$ and $\BbB$ are i-isomorphic.
Additionally, we write $\AaA \preccurlyeq^i \BbB$ if $\AaA \cong^i \BbB'$ for some ideal $\BbB' \subset \BbB$.
\end{dfn}

As it is easily seen, every t-isomorphism is a d-isomorphism. Therefore:
\begin{gather*}
\AaA \cong^t \BbB \implies \AaA \cong^s \BbB \implies \AaA \cong \BbB,\\
\AaA \preccurlyeq^t \BbB \implies \AaA \preccurlyeq^s \BbB \implies \AaA \preccurlyeq \BbB.
\end{gather*}
It is also clear that `$\preccurlyeq^i$' is transitive, while `$\cong^i$' is an equivalence.\par
The main tool of this section is the following

\begin{thm}{isom}
If $\Phi\dd \AaA \to \BbB$ is a bijection such that
\begin{equation}\label{eqn:add2}
\Phi(\XXX \oplus \YYY) = \Phi(\XXX) \oplus \Phi(\YYY)
\end{equation}
for any $\XXX, \YYY \in \AaA$, then $\Phi$ is an isomorphism and $\Phi$ preserves all notions, features
and operations appearing in \textup{(ST1)--(ST17) (pp. \pageref{ST*}--\pageref{ST+})}.
\end{thm}

The above result is a generalization of \PROp{fin-infin} and its proof goes similarly (see Section~13).
In particular, for every isomorphism $\Phi\dd \AaA \to \BbB$ and each $\AAA \in \AaA$ one has: $\dim \Phi(\AAA)$
is uncountable iff so is $\dim(\AAA)$ and if this is the case, they are equal. So, $\Phi$ is an s-isomorphism
if $\Phi$ preserves `$\dim$' for members of $\SsS\EeE\PpP$ (the prefix `s' is after `separable'). One may also
check that $\Phi$ preserves atoms, fractals, semiprimes (using their definitions and the note on page
\pageref{note} after \DEF{fractal}), factor tuples (by \PROP{factor}) and types I, II and III. Consequently,
$\Phi(\AaA^d) = \BbB^d$ and $\Phi(\AaA^c) = \BbB^c$.\par
We shall now define characteristics of ideals which turn out to be sufficient for answering the questions whether
$\AaA \cong^i \BbB$ or $\AaA \preccurlyeq^i \BbB$.

\begin{dfn}{charact}
For any $D \in \{I,I_1,I_2,\ldots,I_{\infty},\tII,\tII_1,\tII_{\infty},\tIII\}$ let
$$
\chi_D^d(\AaA) = \card(\{\XXX\dd\ \XXX \in \Ff_N \cap \IiI_D,\ \XXX \leqsl^s \JJJ(\AaA)\}),
$$
$\chi_D^c(\AaA) = \dim(\JJJ(\AaA^c \cap \IiI_D))$ and $\chi_D(\AaA) = (\chi_D^d(\AaA),\chi_D^c(\AaA))$.
Finally, let
\begin{gather*}
\chi(\AaA) = (\chi_I(\AaA);\chi_{\tII}(\AaA);\chi_{\tIII}(\AaA)),\\
\chi_s(\AaA) = (\chi_{I_1}^d(\AaA),\chi_{I_2}^d(\AaA),\ldots,\chi_{I_{\infty}}^d(\AaA)),\\
\chi_t(\AaA) = (\chi_{I_1}(\AaA);\chi_{I_2}(\AaA);\ldots,\chi_{I_{\infty}}(\AaA);\chi_{\tII_1}(\AaA);
\chi_{\tII_{\infty}}(\AaA)).
\end{gather*}
\end{dfn}

When comparing sequences (finite or infinite) of the same length whose entries are cardinals, `$\leqsl$' will
denote the coordinatewise order.\par
Let $\Omega$ be an underlying model space and let $\Psi_N = \Phi_{\TTT}\dd \CDD_N \to \Lambda(\Omega)$ be
as in \THMp{model}. For $E = \supp_{\Omega} \AaA$ we have (under the notation of \DEF{charact}): $\chi_D^d(\AaA)
= \kappa_d(E \cap \Omega_D)$ and $\chi_D^c(\AaA) = \kappa_c(E \cap \Omega_D)$ (cf. \PROP{Dim}). So, according
to \THM{homeo} (page \pageref{thm:homeo}; below `$\cong$' means `homeomorphic'),
\begin{equation}\label{eqn:iso-homeo}
\Omega_D \cap \supp_{\Omega} \AaA \cong \Omega_D \cap \supp_{\Omega} \BbB \iff \chi_D(\AaA) = \chi_D(\BbB).
\end{equation}
As an application of \THM{isom}, \THMp{model}, \COR{123} and \eqref{eqn:iso-homeo} we obtain

\begin{thm}{iso-iso}
Let $N$ and $N'$ be arbitrary positive integers, $\AaA \subset \CDD_N$ and $\BbB \subset \CDD_{N'}$ be ideals.
\begin{enumerate}[\upshape(I)]
\item $\CDD_N \cong^t \CDD_{N'}$. What is more, each entry of $\chi(\CDD_N)$, $\chi_s(\CDD_N)$
   or $\chi_t(\CDD_N)$ is equal to $2^{\aleph_0}$.
\item $\AaA \cong \BbB \iff \chi(\AaA) = \chi(\BbB)$; $\AaA \preccurlyeq \BbB \iff \chi(\AaA) \leqsl \chi(\BbB)$.
\item $\AaA \cong^s \BbB \iff \chi(\AaA) = \chi(\BbB)$ and $\chi_s(\AaA) = \chi_s(\BbB)$; $\AaA \preccurlyeq^s
   \BbB \iff \chi(\AaA) \leqsl \chi(\BbB)$ and $\chi_s(\AaA) \leqsl \chi_s(\BbB)$.
\item $\AaA \cong^t \BbB \iff \chi(\AaA) = \chi(\BbB)$ and $\chi_t(\AaA) = \chi_t(\BbB)$; $\AaA \preccurlyeq^t
   \BbB \iff \chi(\AaA) \leqsl \chi(\BbB)$ and $\chi_t(\AaA) \leqsl \chi_t(\BbB)$.
\item Up to isomorphism (respectively t-isomorphism), there is only $\gamma$ ($2^{\aleph_0}$) different ideals
   where $\gamma = \card(\{\alpha\in\Card\dd\ \alpha \leqsl 2^{\aleph_0}\})$.
\end{enumerate}
\end{thm}
\begin{proof}
The second claim of (I) follows from \COR{123} and \PROp{a-f-s}. Since (V) and the remainder of (I) follow from
(IV), it is sufficient to prove points (II)--(IV). Since their proofs are based on the same idea, we focus
on (IV) and skip proofs for (II) and (III).\par
Since representatives of members of $\IiI^c$ act in infinite-dimensional Hilbert spaces, \THM{isom} yields that
if $\Phi\dd \AaA \to \AaA' \subset \BbB$ is a t-isomorphism, then necessarily $\chi(\AaA) = \chi(\AaA') \leqsl
\chi(\BbB)$ and $\chi_t(\AaA) = \chi_t(\AaA') \leqsl \chi_t(\BbB)$. Conversely, if $\chi(\AaA) \leqsl
\chi(\BbB)$ and $\chi_t(\AaA) \leqsl \chi_t(\BbB)$ (respectively $\chi(\AaA) = \chi(\BbB)$ and $\chi_t(\AaA) =
\chi_t(\BbB)$), there is an ideal $\AaA' \subset \BbB$ ($\AaA' = \BbB$) for which $\chi(\AaA') = \chi(\AaA)$
and $\chi_t(\AaA') = \chi_t(\AaA)$ (this may be deduced e.g. from \eqref{eqn:iso-homeo}; $\AaA'$ may be defined
as $\IiI[F]$ for suitable clopen set $F \subset \supp_{\Omega} \BbB$). Now \THMp{homeo} combined with
\eqref{eqn:iso-homeo} implies that there are homeomorphisms $h_D\dd \Omega_D \cap \supp_{\Omega} \AaA \to
\Omega_D \cap \supp_{\Omega} \AaA'$ where $D$ runs over $I_1,I_2,\ldots,I_{\infty},\tII_1,\tII_{\infty},\tIII$.
Define a homeomorphism $H\dd \supp_{\Omega} \AaA \to \supp_{\Omega} \BbB$ as the unique continuous extension
of the union of all $h_D$'s. Finally let $\Phi\dd \AaA \to \BbB$ be defined as follows. For $\AAA \in \AaA$
put $f = \Psi_N(\AAA) \in \Lambda(\Omega)$. Since $\supp f \subset \supp_{\Omega} \AaA$, the rules
$g = f \circ H^{-1}$ on $\supp_{\Omega} \BbB$ and $g = 0$ elsewhere well defines $g \in \Lambda(\Omega)$
such that $\supp g \subset \supp_{\Omega} \BbB$. We put $\Phi(\AAA) = \Psi_{N'}^{-1}(g)$. It is easily seen
that $\Phi$ is a well defined bijection. What is more, $\Phi$ satisfies condition \eqref{eqn:add2}, by point
(D4') of \THMp{model}. Consequently, \THM{isom} yields that $\Phi$ is an isomorphism. It follows from
the construction that $\Phi$ is in fact a t-isomorphism.
\end{proof}

\begin{cor}{antisym}
If $\AaA \preccurlyeq^i \BbB$ and $\BbB \preccurlyeq^i \AaA$, then $\AaA \cong^i \BbB$.
\end{cor}

\begin{cor}{contract}
$\CDD_N \cong^t \IiI(1)$ where $\IiI(1) \subset \CDD$ is the ideal of all contraction operators.
\end{cor}
\begin{proof}
Thanks to \THM{iso-iso} we may assume that $N=1$. Observe that the $\bB$-transform is a t-isomorphism
of $\CDD$ onto a subideal of $\IiI(1)$. So, the assertion follows from \COR{antisym}.
\end{proof}

\begin{cor}{unitary}
Let $\UuU \subset \CDD$ be the ideal of all unitary operators.
\begin{enumerate}[\upshape(1)]
\item $\IiI_{I_1} \subset \CDD_N$, the ideal of all normal $N$-tuples, is t-isomorphic to $\UuU$.
\item $\IiI_I \subset \CDD_N$, the ideal of all $N$-tuples of type I, is d-isomorphic to $\UuU$.
\end{enumerate}
\end{cor}
\begin{proof}
To see (1), repeat the argument in the proof of \COR{contract}. To show (2), observe that all entries
of the suitable characteristics of both the ideals $\IiI_1$ and $\UuU$ are equal to $2^{\aleph_0}$ and apply
\THM{iso-iso}.
\end{proof}

The above corollaries say that whatever can be said about single (unitary) contraction operators in the language
of `discrete' direct sums, this will have its natural `counterpart' for arbitrary (type I) $N$-tuples.

\begin{rem}{spatial}
Since $\CDD_N \cong^t \CDD_{N'}$ for any $N$ and $N'$, we may also speak of \textit{spatially i-isomorphic}
ideals. Precisely, ideals $\AaA \subset \CDD_N$ and $\AaA' \subset \CDD_{N'}$ are spatially i-isomorphic
(as usual, `i' is the empty, `s' or 't' prefix) iff there is an i-isomorphism $\Phi\dd \CDD_N \to \CDD_{N'}$
which sends $\AaA$ onto $\AaA'$. However, this idea brings nothing new. Indeed, it is quite easy to check that
$\AaA$ and $\AaA'$ are spatially i-isomorphic iff $\AaA \cong^i \AaA'$ and $\AaA^{\perp} \cong^i
(\AaA')^{\perp}$. So, we only have to double the length of characteristics. One information in this subject
however may be interesting: up to spatial isomorphim, there is only $\card(\{\alpha\in\Card\dd\ \alpha \leqsl
2^{\aleph_0}\})$ different ideals. So, under the continuum hypothesis, this number is countable.
\end{rem}

\SECT{Concluding remarks}

\textbf{A.} Results of Sections~20-22, especially Lemmas~\ref{lem:pd-sep} \pREF{lem:pd-sep} and \ref{lem:disj-meas}
\pREF{lem:disj-meas}, prove that it is good to know how to recognize regularity measures. Especially in finite-dimensional
case, since \PROp{finint} simply characterizes summable fields of $N$-tuples. The author is not aware of the existence
of any result in this direction. We suppose that
\begin{quote}\textbf{Conjecture.}
\textit{Every $\sigma$-finite (Borel) measure on $\aA_N(n)$ for finite $n$ is concentrated on a measurable domain.}
\end{quote}
Below we answer the conjecture in the affirmative for $n=1$. (We are convinced this is well known. However, we could not
find anything about this in the literature.) Let us first make some comments on consequences of the conjecture. If only it
is true, every pair $(\xxX,\{\PPP_x\}_{x\in\xxX})$ where $(\xxX,\MmM,\NnN)$ is standard and $\xxX \ni x \mapsto \PPP_x \in
\bigcup_{n=1}^{\infty} \aA_N(n)$ is a one-to-one Borel function is a regular system, i.e. $\PPP_x$'s `form' the prime
decomposition of some $\XXX \in \SsS\EeE\PpP_N$. Indeed, the sets $\xxX_n = \{x\in\xxX\dd\ \dim(\PPP_x) = n\}$
($n=1,2,\ldots$) are measurable and there is a finite Borel measure $\mu_n$ on $\aA_N(n)$ such that the assignment $\xxX_n
\ni x \mapsto \PPP_x \in \aA_N(n)$ is an almost null-isomorphism between $(\xxX_n,\Mm\bigr|_{\xxX_n},\NnN\bigr|_{\xxX_n})$
and $(\aA_N(n),\Bb(\aA_N(n)),\NnN(\mu_n))$. Now it follows from the conjecture that $\mu_n \in \rgM(\aA_N(n))$
and consequently $\{\PPP_x\}_{x\in\xxX_n} \in \rgS_{loc}$. Put $\XXX_n = \bigsqplus^{\NnN}_{x\in\xxX_n} \PPP_x
(= \int^{\sqplus}_{\aA_N(n)} \PPP \dint{\mu_n(\PPP)})$. We conclude from \CORp{pd-type} that $\XXX_n \in \IiI_{I_n}$. So,
$\XXX_n \disj \XXX_m$ for $n \neq m$ and therefore $\mu_n \perp_s \mu_m$, thanks to \LEMp{orth-meas}. Now it suffices
to apply \LEMp{disj-meas} to obtain that $\{\PPP_x\}_{x\in\xxX} \in \rgS_{loc}$ (and $\bigsqplus_{n=1}^{\infty} \XXX_n =
\bigsqplus^{\NnN}_{x\in\xxX} \PPP_x$).\par
Let us add here that the work of Ernest shows that there are standard Borel measures on $\pP_N \cap \SsS\EeE\PpP_N(\infty)$
which are not concentrated on measurable domains (see Propositions~1.53 and 3.13 in \cite{e}).\par
Let us now show that every $\sigma$-finite Borel measure $\mu$ on $\CCC^N$ is concentrated on a measurable domain. Since
there is a finite Borel measure $\nu$ on $\CCC^N$ such that $\mu \ll \nu$, we may assume $\mu$ is finite. First assume
$\mu$ is concentrated on a compact set. Put $\tTT = \int^{\oplus}_{\CCC^N} \xi \dint{\mu(\xi)}$. It follows from
the Stone-Weierstrass theorem that $M_f \in \WWw(\tTT)$ for every $f \in \CCc(K)$ where $M_f$ is the multiplication
operator by $f$. This implies that $M_u \in \WWw(\tTT)$ for every $u \in L^{\infty}(\mu)$ as well. Consequently,
$M_u \in \ZZz(\WWw(\tTT))$ (since $\WWw(\tTT)$ consists of decomposable operators) and hence $\int^{\oplus}_A \xi
\dint{\mu(\xi)} \disj \int^{\oplus}_{\CCC^N \setminus A} \xi \dint{\mu(\xi)}$ which shows that $\TTT =
\int^{\sqplus}_{\CCC^N} \xi \dint{\mu(\xi)}$ and thus $\mu \in \rgM(\CCC^N)$.\par
Now if $\mu$ is arbitrary, there is a sequence $(K_n)_{n=1}^{\infty}$ of compact pairwise disjoint subsets of $\CCC^N$
such that $\mu(\CCC^N \setminus \bigcup_{n=1}^{\infty} K_n) = 0$. The above argument proves that $\mu\bigr|_{K_n} \in
\rgM(\CCC^N)$ for every $n$. Put $\XXX_n = \int^{\sqplus}_{K_n} \xi \dint{\mu(\xi)}$. Now we repeat earlier argument:
$\XXX_n \disj \XXX_m$ for $n \neq m$ (by \LEMP{orth-meas}) and thus $\mu \in \rgM(\CCC^N)$, thanks to \LEMp{disj-meas}.
\vspace{0.3cm}

\textbf{B.} \THM{mod-pr} (\PREF{thm:mod-pr}; cf. also \REM{axiom}, \PREF{rem:axiom}) establishes a one-to-one
correspondence between coverings and functions $\mu\dd \Mm \to \CDD_N$ satisfying conditions (M0)--(M3) (see the statement
of \THM{mod-pr}). These conditions are purely `discrete', i.e. they need no measure-theoretic nor topological background
and are formulated in terms of the direct sum operation of a pair. So, it seems to be interesting (and may turn out to be
relevant) which topological or measure-theoretic notions (operations, features, tools, etc.) are sufficient for
reconstructing from $\mu$ the covering to which it corresponds.\vspace{0.3cm}

\textbf{C.} Similarly as we defined continuous direct sums, one may try to define `continuous' ideals in $\CDD_N$. It may
be done in at least a few ways. Here we propose only one of them. Let us call an ideal $\AaA \subset \CDD_N$
\textit{continuous} if $\AaA$ satisfies the following condition. Whenever $(\xxX,\Mm,\NnN,\Phi)$ is a full covering
and $\AAA = \bigsqplus^{\NnN}_{x\in\xxX} \mM(x) \odot \Phi(x)$ for some $\mM \in \aaA(\xxX)$, then $\AAA \in \AaA$ \iaoi{}
there is a set $\zzZ \in \NnN$ such that $\Phi(x) \in \AaA$ for every $x \in s(\mM) \setminus \zzZ$. Using \THMp{prime} one
may easily check that it suffices to verify the above condition with a one fixed full covering and only for $\AAA \in
\SsS\EeE\PpP_N$. For example, $\IiI_i$ is a continuous ideal for each $i \in \{I,I_1,I_2,\ldots,I_{\infty},\tII,\tII_1,
\tII_{\infty},\tIII\}$, while $\IiI_i^c$ and $\IiI_i^d$ are not. A \textit{p-isomorphism} (the prefix `p' stands for `prime
decomposition') between continuous ideals is such an isomorphism $\Psi\dd \AaA \to \BbB$ that whenever $\AAA =
\bigsqplus^{\NnN}_{x\in\xxX} \mM(x) \odot \PPP_x$ is a prime decomposition of $\AAA \in \AaA$, prime decomposition
of $\Psi(\AAA)$ may be written in the form $\bigsqplus^{\NnN}_{x\in\xxX} \mM(x) \odot \Psi(\PPP_x)$, and the same for
$\Psi^{-1}$. The following problem may be interesting.
\begin{quote}
\textbf{Question.} \textit{Are $\CDD_N$ and $\CDD_{N'}$ p-isomorphic?}
\end{quote}
\vspace{0.3cm}

\textbf{D.} Our last remark is about the length of tuples. Readers interested in sequences (that is, countable infinite
families) of closed densely defined operators acting in common Hilbert spaces may verify that most of the results (with
no changes in proofs) of this treatise remain true also in that case, i.e. for $N = \infty$. (However, when working with
uncountable families, a counterpart of crucial \THMP{common}, fails to be true which causes that the whole theory crashes
in that case.) Since infinite sequences are rarely investigated, we restricted our study to finite collections.

\end{document}